\documentclass[dvipdfmx]{article}
\usepackage{amsmath, amssymb}
\usepackage{amsthm}
\usepackage{multirow,bigdelim}
\usepackage{ascmac}
\usepackage{fancybox}
\usepackage{color}
\usepackage{nccmath}

\usepackage{multicol}

\usepackage[top=30truemm,bottom=30truemm,left=25truemm,right=25truemm]{geometry}
\usepackage{graphicx}
\usepackage{here}
\usepackage{amscd}

\usepackage{tcolorbox}
\tcbuselibrary{breakable}
\usepackage[all]{xy} 
\usepackage{fancyhdr}

\theoremstyle{definition}
\newtheorem{thm}{Theorem}[section]
\newtheorem{defi}[thm]{Definition}
\newtheorem{lemm}[thm]{Lemma}
\newtheorem{prop}[thm]{Proposition}
\newtheorem{example}[thm]{Example}
\newtheorem{remark}[thm]{Remark}

\begin{document}

\title{Antipodal sets of exceptional symmetric spaces}

\author{Yuuki Sasaki}

\date{}

\maketitle

\begin{abstract}
In this paper, we study maximal antipodal sets in simply connected exceptional compact symmetric spaces.
Combining our results with the existing literature, we obtain a complete classification of maximal antipodal sets in all such spaces.
Moreover, through the description of antipodal sets, we disscus the inclusion relations among exceptional symmetric spaces.

\end{abstract}


\section{Introduction}

The study of antipodal sets goes back to the work of Borel and Serre in 1953 \cite{Borel-Serre}.
They investigated an invariant of connected compact Lie groups called the \emph{2-rank}.
For a connected compact Lie group $G$, the 2-rank is defined as the maximal rank of an elementary abelian $2$-subgroup of $G$.
In their paper, they determined the 2-ranks of several connected compact Lie groups ($SO(n), U(n), Sp(n), G_{2}, F_{4}$), established estimates for the 2-rank, and studied its relation to the $2$-torsion of the homology groups.
Motivated by this work, many mathematicians have subsequently studied the 2-rank and maximal elementary abelian $2$-subgroups in other connected compact Lie groups.
For example, Adams used methods from algebraic topology to classify the conjugacy classes of maximal elementary abelian $2$-subgroups in the exceptional compact Lie group $E_{8}$ \cite{Adams}, and Wood investigated maximal elementary abelian $2$-subgroups of spinor groups from a combinatorial theory \cite{Wood}.

In 1988, Chen and Nagano interpreted the 2-rank and maximal elementary abelian $2$-subgroups through the structure of symmetric spaces of compact Lie groups, and arrived at the notion of an antipodal set \cite{Chen-nagano2}.
A symmetric space $M$ is a Riemannian manifold endowed with a geodesic symmetry $s_{x}$ at each point $x \in M$.
If we equip a compact Lie group $G$ with a bi-invariant Riemannian metric, then the map
\[
s_{g} : G \longrightarrow G,\ ;\ h \longmapsto g h^{-1} g
\]
defines the geodesic symmetry at $g$, so that $G$ becomes a symmetric space.
A subset $A \subset M$ is called an \emph{antipodal set} if
\[
s_{x}(y) = y \qquad \text{for all } x,y \in A .
\]
Chen and Nagano observed that the maximal elementary abelian $2$-subgroups of a compact Lie group $G$ are precisely the maximal antipodal sets containing the identity element.
They also showed that the 2-rank of $G$ corresponds to the maximum cardinality of an antipodal set, i.e.\ the cardinality of a \emph{great antipodal set}.
Motivated by this, the maximal cardinality of an antipodal set in a compact symmetric space $M$ is called the \emph{2-number} of $M$ and is denoted by $\#_{2}M$.
Denoting by $r_{2}(G)$ the $2$-rank of $G$, then $2^{r_{2}(G)} = \#_{2}G$.
In the same paper, Chen and Nagano developed the foundamental theory of antipodal sets and determined the 2-numbers of almost all irreducible symmetric spaces of compact type, with only a few exceptions.
They also pointed out connections with topology, showing in particular that every compact symmetric space $M$ satisfies
\[
\chi(M) \leq \#_{2}M .
\]

Subsequently, a number of authors have investigated more deeply the relation between antipodal sets and topology.
Takeuchi proved that for a special class of compact symmetric spaces called \emph{symmetric $R$-spaces} (which include, for instance, real, complex, and quaternionic Grassmannians) one has
\[
\#_{2}M = \dim H_{*}(M;\mathbb{Z}_{2})
\]
in \cite{Takeuchi}.
For symmetric $R$-spaces, it is known that there exists a Morse function whose critical point set is a great antipodal set, and the existence of such a Morse function plays an essential role in the proof.
Berndt-Console-Fino and S\'{a}nchez extended Takeuchi's result to more general K\"{a}hler $C$-spaces and $R$-spaces \cite{Berndt-Console-Fino, Sanchez1, Sanchez2}.
For compact symmetric spaces that are not symmetric $R$-spaces, the situation is more subtle.
Before the notion of antipodal sets was introduced, Kamiya had already studied Morse functions on $G_{2}$, and his results imply that essentially the euqality also holds for $G_{2}$ \cite{Kamiya}.
More recently, the present author studied Morse functions on $G_{2}/SO(4$) and proved that the equality holds for $G_{2}/SO(4$) as well \cite{Sasaki-G}.
On the other hand, Amann used methods from algebraic topology to establish the general inequality
\[
\#_{2}M \leq \dim H_{*}(M;\mathbb{Z}_{2})
\]
for every compact symmetric space $M$ (\cite{Amann}).
The equality case, however, is still not understood in full generality.
In fact, there are known examples where the equality fails; for instance, equality does not hold for the compact Lie groups $E_{6}, E_{7}, E_{8}$.
Thus, while the relation between antipodal sets and topology has been extensively studied, many aspects remain open.

Parallel to these developments, there has also been substantial work on antipodal sets themselves.
Although Chen and Nagano determined the 2-number for many compact symmetric spaces, they did not in general give explicit descriptions of great or maximal antipodal sets.
To remedy this, Tasaki and Tanaka constructed explicit maximal antipodal sets in various compact symmetric spaces \cite{Tanaka-Tasaki, Tanaka-Tasaki2, Tanaka-Tasaki3}.
In particular, they explicitly construced maximal antipodal sets of symmetric $R$-spaces and prove that any maximal antipodal sets are great and congrunet to each other \cite{Tanaka-Tasaki}.
Hence, in the symmetric $R$-space, a great antipodal set and a maximal antipodal set are the same notion.
However, they also showed that, in the bottom space of some classical compact Lie group, there exist non-congruent maximal antipodal sets, thereby demonstrating that great antipodal sets and maximal antipodal sets are essentially different notions in general cases \cite{Tanaka-Tasaki2}.
Their work led to a complete classification of maximal antipodal sets in almost all classical compact symmetric spaces.
In the exceptional cases, Adams's results for $E_{8}$ were available, but no classification of maximal antipodal sets in the other exceptional symmetric spaces was known except for symmetric $R$-spaces.
Recently, Tanaka-Tasaki-Yasukura completed the classification of maximal antipodal sets in $G_{2}$ and $G_{2}/SO(4$) \cite{Tanaka-Tasaki-Yasukura}, and building on this, the present author carried out the classification for $F_{4}, E_{6}$ and their associated symmetric spaces like $FI, EI$, and so on \cite{Sasaki-F, Sasaki-E}.
In another direction, it has been shown that antipodal sets arise as intersections of two real forms in Hermitian symmetric spaces, with applications to computations in Floer homology \cite{Iriyeh-Sakai-Tasaki}.
Kurihara and Okuda pointed out that antipodal sets also play an important role in design theory, and applied them to the study of designs in $U(n$) and complex Grassmannians \cite{Kurihara, Kurihara-Okuda}.
Moreover, Chen extended the Borsuk-Ulam theorem to compact symmetric spaces by focusing on antipodal sets \cite{Chen}.
Sakai, Ohno, Eschenburg, and Quast have studied maximal antipodal sets in $\Gamma$-symmetric spaces, which generalize symmetric spaces \cite{Eschenburg-Sakai, Ohno-Sakai, Quast-Sakai}.

Since their introduction by Chen and Nagano, antipodal sets have thus attracted considerable attention from many geometers and have been studied from a variety of perspectives.
Within this broad context, the aim of the present work is to complete the classification of maximal antipodal sets in all simply connected exceptional compact symmetric spaces.
Throughout the paper, we use Cartan's notations for irreducible symmetric pairs and, by abuse of notation, denote by the same symbols the corresponding simply connected compact symmetric spaces, namely
\[
\begin{array}{lllllllllllll}
FI = F_{4}/Sp(1) \cdot Sp(3), & FII = F_{4}/Spin(9), \\
EI = E_{6}/Sp(4)/(\mathbb{Z}_{2}), & EII = E_{6}/SU(2) \cdot SU(6), & EIII = E_{6}/(U(1) \times Spin(10))/\mathbb{Z}_{4}, & EIV = E_{6}/F_{4}, \\
EV = E_{7}/(SU(8)/\mathbb{Z}_{2}), & EVI = E_{7}/SU(2) \cdot Spin(12), & EVII = E_{7}/(U(1) \times E_{6})/\mathbb{Z}_{3}, \\
EVIII = E_{8}/Ss(16), & EIX = E_{8}/SU(2) \cdot E_{7}.
\end{array}
\]
As for the classification of maximal antipodal sets in exceptional symmetric spaces, previous results are as follows:
the symmetric $R$-spaces $FII, EIII, EIV$ were treated by Tanaka-Tasaki \cite{Tanaka-Tasaki};
the symmetric spaces $G_{2}$ and $G_{2}/SO(4$) were studied by Tanaka-Tasaki-Yasukura \cite{Tanaka-Tasaki-Yasukura};
the group $E_{8}$ was investigated by Adams;
and the spaces $F_{4}, FI, E_{6}, EI, EII, EIV$ were classified by the present author \cite{Sasaki-F, Sasaki-E}.
In this paper we complete the remaining cases, namely $E_{7}, EV, EVI, EVIII, EIX$.
Our method is based on Adams's classification of maximal antipodal sets in $E_{8}$ and on the $(M^{+}, M^{-}$)-method of Chen-Nagano.
Combining our results with the previous work, we obtain the complete classifications of maximal antipodal sets in all simply connected exceptional compact symmetric spaces.
The numbers of congruence classes and the cardinalities of maximal antipodal sets are summarized in Table \ref{intro-1}.
In Table \ref{intro-1}, $M$ denotes the symmetric space, $C_{M}$ denotes the number of congruence classes of maximal antipodal sets in $M$, and ``card.'' stands for the cardinality of a maximal antipodal set.

\ 

\begin{table}[htbp]
\def\arraystretch{1.2}
\arraycolsep=-1pt
\begin{center}
\begin{tabular}{ccc||ccc||cccccccccccccccccccccccccccc} \hline
$M$ & $C_{M}$ & card. & $M$ & $C_{M}$ & card. &$M$ & $C_{M}$ & card. \\ \hline
$G_{2}$ & 1 & 8 & $G$ & 1 & 7  & $EIV$ & 1 & 4 \\
$F_{4}$ & 1 & 32 & $FI$ & 1 & 28 & $EV$ & 3 & 56, 72, 128 \\
$E_{6}$ & 2 & 32, 64 & $FII$ & 1 & 3 & $EVI$ & 2 & 31, 63 \\
$E_{7}$ & 2 & 64, 128 & $EI$ & 2 & 28, 64 & $EVII$ & 1 & 56 \\
$E_{8}$ & 2 & 256, 512 & $EII$ & 2 & 28, 36 & $EVIII$ & 2 & 199, 391 \\
 & & & $EIII$ & 1 & 27 & $EIX$ & 1 & 56, 120 \\ \hline
\end{tabular}
\caption{The number of congruent classes and the cardinalities of maximal antipodal sets}
\label{intro-1}
\end{center}
\end{table}

Every simply connected exceptional compact symmetric space can be realized as a totally geodesic submanifold of $E_{8}$.
In this paper, we describe these realizations explicitly and give a uniform description of maximal antipodal sets in all simply connected exceptional compact symmetric spaces using maximal antipodal sets in $E_{8}$.
Moreover, we consider the orbit of the conjugation action of $F_{4}, E_{6}, E_{7}, E_{8}$ on $E_{8}$ through each element of a maximal antipodal set $A$ of $E_{8}$.
In almost cases, these orbits $L$ become totally geodesic and $A \cap L$ is an antipodal set of $L$.
In general, $A \cap L$ is not necessaliry maximal in $L$.
We study when $A \cap L$ becomes a maximal antipodal set of $L$.
By this study, we obtain the following inclusion relations:
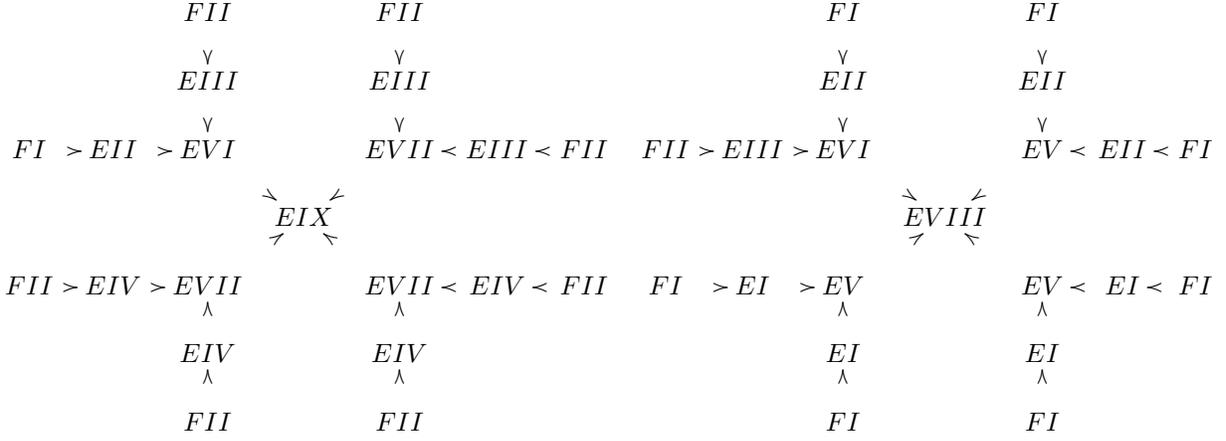
\[
\xymatrix@C=6pt@R=13pt
{
 & & FII \ar[d] & & FII \ar[d] & & \\
 & & EIII \ar[d] & & EIII \ar[d] & & \\
FI \ar[r] & EII \ar[r] & EVI \ar[dr] & & EVII \ar[dl] & EIII \ar[l] & FII \ar[l] \\
 & & & EIX & & & \\
FII \ar[r] & EIV \ar[r] & EVII \ar[ur] & & EVII \ar[ul] & EIV \ar[l] & FII \ar[l] \\
 & & EIV \ar[u] & & EIV \ar[u] & & \\
 & & FII \ar[u] & & FII \ar[u] & & \\
}
\ \ 
\xymatrix@C=6pt@R=13pt
{
 & & FI \ar[d] & & FI \ar[d] & & \\
 & & EII \ar[d] & & EII \ar[d] & & \\
FII \ar[r] & EIII \ar[r] & EVI \ar[dr] & & EV \ar[dl] & EII \ar[l] & FI \ar[l] \\
 & & & EVIII & & & \\
FI \ar[r] & EI \ar[r] & EV \ar[ur] & & EV \ar[ul] & EI \ar[l] & FI \ar[l] \\
 & & EI \ar[u] & & EI \ar[u] & & \\
 & & FI \ar[u] & & FI \ar[u] & & \\
}
\]
From these inclusion relations, we obtain the inclusion relations among exceptional symmetric spaces:
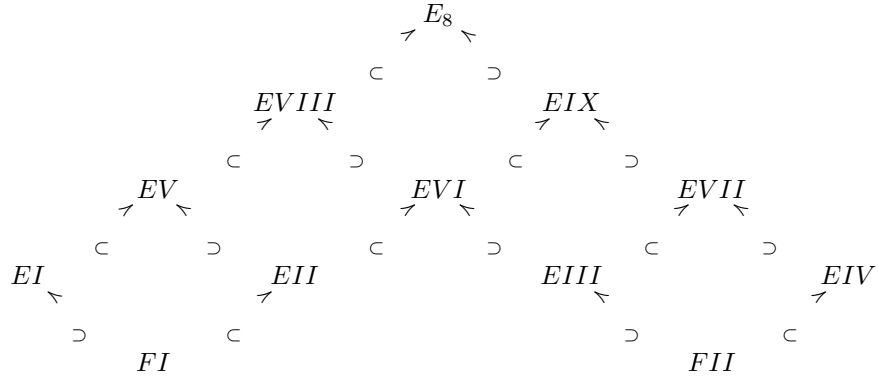
\[
\xymatrix@C=22pt@R=20pt
{
 & & & E_{8} & & & \\
 & & EVIII \ar[ur]_{\subset} & & EIX \ar[ul]^{\supset} & &  \\
 & \ EV \ar[ur]_{\subset} & & EVI \ar[ul]^{\supset} \ar[ur]_{\subset} & & EVII \ar[ul]^{\supset} & \\
 \ EI \ \ar[ur]_{\subset} & & EII \ar[ul]^{\supset} \ar[ur]_{\subset} & & EIII \ar[ul]^{\supset} \ar[ur]_{\subset} & & EIV \ar[ul]^{\supset} \\
 & FI \ar[ul]^{\supset} \ar[ur]_{\subset} & & & & FII \ar[ul]^{\supset} \ar[ur]_{\subset} & \\
}
\]
Note that this inclusion diagram was essencially discovered by Chen-Nagano in \cite{Chen-Nagano1}, and Nagano introduced this diagram in \cite{Nagano} obtained by the different method from the present paper. 
However, they barely discussed this diagram.
It is interesting that there exists the left-right symmetry in this inclusion diagram.
It seems that the left-right symmetry is not merely combinatorial.
It may reflects some geometric properties of exceptional symmetric spaces.
For example, in the above two inclusion relations, every inclusion and the number of orbits match under the left-right symmetry.
Although we do not know the reason for this left-right symmetry, we discuss some phenomenons with respect to the left-right symmetry in the final section of the present paper.

The structure of the paper is as follows.
Section 2 collects the preliminaries needed later.
In Subsection 2.1, we recall the notion of polars, which are important submanifolds in the $(M^{+},M^{-}$)-method for compact symmetric spaces, and we review basic properties of antipodal sets.
Subsection 2.2 is a brief review of root systems.
We also recall some involutive automorphisms arising from roots and consider the properties of the root system of type $E_{8}$.
In Subsection 2.3, we recall the construction of spin groups via Clifford algebras, describe several subgroups of spin groups, and review the basic facts about spin representations.
Subsection 2.4 is devoted to octonions: we define $G_{2}$ using the octonions and recall the $SO(8)$-triality principle, together with the construction of $Spin(8$) and its spin representations via triality.
Much of this material follows \cite{Yokota-g-r}; for the sake of completeness we reproduce the proofs of several propositions from that reference.
Moreover, we recall maximal antipodal sets of $G_{2}$ and $Spin(8)$ in this subsection.
In Subsection 2.5, we recall the construction of the exceptional group $E_{8}$ and its subgroups $E_{7}, E_{6}, F_{4}$ from \cite{Adams-L}, and we review Adams's classification of maximal antipodal sets in $E_{8}$ \cite{Adams}.
We will make explicit descriptions of the Riemannian symmetric pairs corresponding to all simply connected exceptional compact symmetric spaces except for $G_{2}/SO(4)$ and realize them as orbits of conjugation action of $H = F_{4}, E_{6}, E_{7}, E_{8}$ on $E_{8}$.
Subsection 2.6 summarizes the known classification results for maximal antipodal sets in simply connected exceptional compact symmetric spaces and describes each of them uniformly in terms of maximal antipodal sets in $E_{8}$.
We also recall classification results for maximal antipodal sets in certain classical symmetric spaces that will be used later.

In Section 3, we carry out the classification of maximal antipodal sets in the remaining cases, that is, $E_{7}, EV, EVI, EVIII, EIX$.
Subsection 3.1 explains the general strategy, in particular the relation between maximal antipodal sets of a compact symmetric space $M$ and of its polar $M^{+}$.
The cases $EVIII$ and $EIX$ are treated in Subsection 3.2, the case $EVI$ in Subsection 3.3, the case $E_{7}$ in Subsection 3.4, and the case $EV$ in Subsection 3.5.
In each case, we describe the maximal antipodal sets explicitly using a maximal antipodal set of $E_{8}$.
Together with the previous results, this yields a complete classification of maximal antipodal sets in all simply connected exceptional compact symmetric spaces.
At the end of the section, we summarize the all classifications (Table \ref{classification}).

In Section 4, we study the orbits of the conjugation action of $F_{4}, E_{6}, E_{7}, E_{8}$ on $E_{8}$ through each element of a maximal antipodal set of $E_{8}$.
Moreover, we derive the above three inclusion relations in Subsection 4.1-4.4.
In Subsection 4.5, we discuss the geometric properties of the above inclusion relations.
Moreover, we consider some phenomenons with respect to the left-right symmetry.

\section{Preliminaries}

First, we summarize the notations used in the present article.
Let $G$ be a Lie group.
The identity element of $G$ is denoted by $1$ or $e$.
The identity component of $G$ is denoted by $G_{o}$.
The center of $G$ is denoted by $C(G)$.
Let $W$ be a subset of $G$.
The centralizer of $W$ in $G$ is denoted by $C(W,G)$ or $C(W)$, that is, $C(W,G) = \{ g \in G \ ;\ gw = wg \ (w \in W) \}$.
The normalizer of $W$ in $G$ is denoted by $N(W,G)$ or $N(W)$, that is, $N(W,G) = \{ g \in G \ ;\ gwg^{-1} \in W \ (w \in W) \}$.
Let $\frak{g}$ be the Lie algebra of $G$ and $\frak{w}$ be a subset of $\frak{g}$.
Then, the cenralizer of $\frak{w}$ of $G$ and $\frak{g}$ are denoted by $C(\frak{w}, G)$ and $C(\frak{w}, \frak{g})$ respectively.
The identity components of $C(W, G), C(\frak{w}, G), N(W,G), N(\frak{w}, G)$ are denoted by $C_{o}(W,G), C_{o}(\frak{w}, G), N_{o}(W,G), N_{o}(\frak{w}, G)$ respectively.
Let $f : G \rightarrow G$ be an automorphism of a Lie group $G$.
The induced automorphism of $\frak{g}$ induced by $f$ is denoted by the same symbol $f$.
Set $F(f,G) = \{ g \in G \ ;\ f(g) = g \}$.
Let $H$ be a subrgoup of $G$.
If $f(H) \subset H$, the automorphism $f|_{H}$ of $H$ is denoted by the same symbol $f$ simply.
If $f$ is involutive, then we set
\[
F^{+}(f, \frak{g}) = F(f, \frak{g}) = \{ X \in \frak{g} \ ;\ f(X) = X \}, \quad
F^{-}(f, \frak{g}) = \{ X \in \frak{g} \ ;\ f(X) = -X \}.
\]
Moreover, Let $V$ be a vector space and $\rho : G \rightarrow GL(V)$ be a representation.
The induced representation of $\frak{g}$ is denoted by the same symbol.
For any involutive element $g \in G$, we often denote $F^{\pm}(\rho(g), V)$ by $F^{\pm}(g,V)$ simply.
Let $G'$ be a Lie group, $\frak{g}'$ be the Lie algebra of $G'$, and $h:G \rightarrow G'$ be a homomorphism.
Then, the indeuced homomorphims from $\frak{g}$ to $\frak{g}'$ is denoted by the sam symbol $f$.
The restricted homomorphim $f|_{H} : H \rightarrow G'$ is denoted by the same symbol $f$.
If $f:H \rightarrow G$ is one-to-one, then $f(H)$ is often denoted by $H$ simply.
For any real vector space $V$, we denote the complexification of $V$ by $V^{\mathbb{C}}$.
For any 1-from $\alpha$ of $V$ and any subspace $U \subset V$, we denoted $\alpha|_{U}$ by the same symbol $\alpha$.
For any automorphism $g$ of the real Lie algbera $\frak{g}$, we denote the extension of $g$ to $\frak{g}^{\mathbb{C}}$ is denoted by the same symbol.
For any $g \in G$, the conjugation by $g$ is denoted by $\theta_{g}$, that is, $\theta_{g} : G \rightarrow G \ ;\ h \mapsto ghg^{-1}$.
Let $k \leq n$.
The set of all $k$-dimensional subspaces of $\mathbb{R}^{n}$ is denoted by $G_{k}(\mathbb{R}^{n})$.
Moreover, the set of oriented $k$-dimensional subspaces of $\mathbb{R}^{n}$ is denoted by $\tilde{G}_{k}(\mathbb{R}^{n})$.


\subsection{Antipodal sets}

First, we recall the polars introduced by Chen-Nagano \cite{Chen-Nagano1}.
Let $M$ be a connected compact symmetric space.
The symmetry at $p \in M$ is denoted by $s_{p}$.
Fix $o \in M$.

\begin{defi}
Each connected component of $F(s_{o}, M)$ is called a {\it polar} of $o$.
If a polar is a one-point set, then the polar is called a {\it pole}.
By the definition, $\{ o \}$ is a pole.
Hence, $\{ o \}$ is called the trivial pole.
\end{defi}

Since a polar is a connected component of an involutive isometry, any polar is a totally geodesic submanifold and becomes a compact symmetric space.
In each compact symmetric space, Chen and Nagano already classified polars \cite{Chen-Nagano1}.
If the number of polars is $k+1$, then each polar of $o$ is denoted by $(M_{o}^{+})_{i}$ or $M^{+}_{i} \ (0 \leq i \leq k)$.
Then,
\[
F(s_{o}, M) = \bigsqcup_{0 \leq i \leq k} (M_{o}^{+})_{i}.
\]
In the following, we assume that $(M_{o}^{+})_{0}$ is the trivial pole.
Let $G$ be the identity component of the isometry group of $M$ and $K_{o} = \{ g \in G \ ;\ g(o) = o \}$.
Then, it is well known that $K_{o}$ acts on each polar and this action is transitive.
Fix $0 \leq i \leq k$ and set $Z_{i} := \{ k \in K_{o} \ ;\ k|_{M_{i}^{+}} = \mathrm{id}_{M_{i}^{+}} \}$.
Then, $K_{o}/Z_{i}$ is isomorphic to the identity component of the isometry group of $M_{i}^{+}$.
If $M$ is a connceted compact Lie group $H$.
Then, the symmetry at $g \in H$ is given by $s_{g} : H \rightarrow H \ ;\ h \mapsto gh^{-1}g$.
Therefore, 
\[
F(s_{e}, H) = \{ h \in H \ ;\ h^{2} = e \}.
\]
In this case, the action of $K_{o}$ is obtained by the all conjugations.
Hence, any polar $M^{+}$ of the identity element $e$ is the conjugate orbit through an involutive element, that is,
\[
M^{+} = \bigcup_{h \in H}h x h^{-1},
\]
where $x$ is an involutive element of $H$.
Next, we recall antipodal sets of $M$.

\begin{defi}
Let $A \subset M$.
If $s_{p}(q) = q$ for any $q \in A$, then $A$ is called an {\it antipodal set} of $M$.
If an antipodal set $A$ is maximal with respect to the inclusion relation among antipodal sets, then $A$ is called a {\it maximal antipodal set} of $M$.
The supremum of the cardinalities of antipodal sets is called the $2$-number of $M$ and denoted by $\#_{2}M$.
An antipodal set whose cardinality is equal to $\#_{2}M$ is called a great antipodal set of $M$.

\end{defi}

The maximal antipodal set of the compact Lie group containing the identity element is a subrgoup.
Moreover, this subgroup is a maximal elementary abelian $2$-subrgoup and called a maximal antipodal subgroup.
Let $A$ and $B$ be subsets of $M$.
If there exist an element $g$ of the identity component of the isometry group of $M$ such that $g(A) = B$, then we say that $A$ and $B$ are congruent to each other.
Remark that, in the case that $M$ is a compact Lie group, to classiy the congruent classes of maximal antipodal sets is reduce to classify the conjugate classes of maximal antipodal subrgoups.
The purpose of the present paper is to decide the congruent classes of maximal antipodal sets for simply connected compact exceptional symmetic spaces.
By the result by Tanaka-Tasaki, since $FII, EIII, EVI$ are symmetric $R$ space, the classification is complete \cite{Tanaka-Tasaki}.
Tanaka-Tasaki-Yasukura classified the maximal antipodal set for $G_{2}$ and type $G$ \cite{Tanaka-Tasaki-Yasukura}.
Adams classified the maximal antipodal sets of $E_{8}$ by the method of the algebraic topology \cite{Adams}.
Moreover, the author classfied maximal antipodal sets for $F_{4}, E_{6}, EI, EI, EII, EIV$.
Therefore, the remaining cases are $E_{7}, EV, EVI, EVIII, EIX$ and we consider these cases in the present paper.


\newpage

\subsection{Root system}\label{Rs}

In this subsection, we recall fundamental properties about the root system.
Let $G$ be a simply connected simple compact Lie group and $\frak{g}$ be the Lie algebra of $G$.
Let $T$ be a maximal torus of $G$ and the Lie algebra of $T$ is denoted by $\frak{t}$.
The complex conjugation of $\frak{g}^{\mathbb{C}}$ corresponding to $\frak{g}$ is denoted by $\overline{\ \cdot\ }$.
Let $(\ ,\ )$ be the Killing form of $\frak{g}^{\mathbb{C}}$.
Define $\langle \ ,\ \rangle = -(\ ,\ )|_{\frak{g} \times \frak{g}}$.
Then, $\langle \ ,\ \rangle$ is an inavriant inner product of $\frak{g}$.
We define $\Sigma(\frak{g}, \frak{t})$ as the set of all roots of $\frak{g}^{\mathbb{C}}$ with respect to $\frak{t}^{\mathbb{C}}$.
We often denote $\Sigma(\frak{g}, \frak{t})$ by $\Sigma$ simply.
For any $\alpha \in \Sigma$, we define
\[
\frak{r}_{\alpha}(\frak{g}^{\mathbb{C}}, \frak{t}^{\mathbb{C}}) := \{ X \in \frak{g}^{\mathbb{C}} \ ;\ [H, X] = \alpha(H)X \ (H \in \frak{t}^{\mathbb{C}}) \}.
\]
The root space $\frak{r}_{\alpha}(\frak{g}^{\mathbb{C}}, \frak{t}^{\mathbb{C}})$ is often denoted by $\frak{r}^{\mathbb{C}}_{\alpha}$.
Moreover, we define $H_{\alpha} \in i\frak{t}$ such that $(H, H_{\alpha}) = \alpha(H)$ for any $H \in i\frak{t}$.
For any $\alpha, \beta \in \Sigma$, we denote $(H_{\alpha}, H_{\beta})$ by $(\alpha, \beta)$.
Then, we define the coroot $A_{\alpha} = (2/(\alpha, \alpha))H_{\alpha}$.
Let $X_{\alpha} \in \frak{r}_{\alpha}^{\mathbb{C}}\ (X_{\alpha} \not= 0 )$.
We take a linear order of $\frak{t}$ and the set of all positive roots is denoted by $\Sigma^{+}(\frak{g}, \frak{t})$ or $\Sigma^{+}$.
For any subset $B \subset \Sigma$, we set $B^{+} = B \cap \Sigma^{+}$.
Since $\overline{\frak{r}^{\mathbb{C}}_{\alpha}} \subset \frak{r}^{\mathbb{C}}_{-\alpha}$ for any $\alpha \in \Sigma$, we set
\[
\frak{r}_{\alpha}(\frak{g}, \frak{t}) = \frak{g} \cap ( \frak{r}_{\alpha}^{\mathbb{C}} + \frak{r}_{-\alpha}^{\mathbb{C}} )
\]
for any $\alpha \in \Sigma$.
The subspace $\frak{r}_{\alpha}(\frak{g}, \frak{t})$ is often denoted by $\frak{r}_{\alpha}$.
Then, 
\[
\frak{g} = \frak{t} + \sum_{\alpha \in \Sigma^{+}}\frak{r}_{\alpha}.
\]
For each $\alpha \in \Sigma$, we set
\[
\frak{su}^{\alpha}(2) = \mathbb{R}(iA_{\alpha}) + \frak{r}_{\alpha}.
\]
The subspace $\frak{su}^{\alpha}(2)$ is a subalgebra of $\frak{g}$ and isomorphic to $\frak{su}(2)$.
Moreover, let $\alpha, \beta \in \Sigma$ satisfy $(\alpha, \alpha) = (\beta, \beta)$ and $2(\alpha, \beta)/(\alpha, \alpha) = \pm1$.
Then, $\{ \pm \alpha, \pm \beta, \pm (\alpha \mp \beta) \}$ is a root system of type $A_{2}$.
We set
\[
\frak{su}^{\alpha, \beta}(3) = \mathbb{R}(iA_{\alpha}) + \mathbb{R}(iA_{\beta}) + \sum_{\gamma \in \{ \pm \alpha, \pm \beta, \pm (\alpha \mp \beta) \} \cap \Sigma^{+}}\frak{r}_{\gamma}. 
\]
The subspace $\frak{su}^{\alpha, \beta}(3)$ is a subalgebra of $\frak{g}$ and isomorphic to the Lie algebra $\frak{su}(3)$.
We defined the Weyl group $W(T) := N(T)/T$, where $N(T)$ is the normalizer of $T$ in $G$.

\begin{lemm} \cite{Adams}
If there exists $x \in N(T)$ such that $\mathrm{Ad}(x)(H) = -H$ for any $H \in \frak{t}$, then the following (i), (ii), and (iii) are true.

(i)\ The element $x$ is unique up to conjugacy in $G$ fixing $T$.
Any other choice has the form $xu$ for some $u \in T$.

(ii)\ The order of $x$ is $4$ at most.

(iii)\ The square $x^{2}$ is independent of the choice of $x$ and lies in the center of $G$.

\end{lemm}

In the following, we assume that such $x \in N(T)$ exists and the order of $x$ is $2$.
For example, if $G = Spin(8n + \epsilon) \ (n \in \mathbb{N}, \epsilon = 0, \pm1 )$, then such $x$ exists.
We fix such $x \in N(T)$.
The conjugation $\theta_{x}$ is involutive automorphism of $G$.
Then, for any $\alpha \in \Sigma^{+}$ and $X_{\alpha} \in \frak{r}_{\alpha}^{\mathbb{C}}$,
\[
[H, \theta_{x}(X_{\alpha})] = \theta_{x}([\theta_{x}(H), X_{\alpha}]) = \theta_{x}([-H, X_{\alpha}]) = \theta_{x}(-\alpha(H)X_{\alpha}) = -\alpha(H)\theta_{x}(X_{\alpha}).
\]
Hence, $\theta_{x}(\frak{r}_{\alpha}^{\mathbb{C}}) \subset \frak{r}_{-\alpha}^{\mathbb{C}}$.
In particular, $\theta_{x}(\frak{su}_{\alpha}(2)) \subset \frak{su}_{\alpha}(2)$.
Since $\theta_{x}(iA_{\alpha}) = -iA_{\alpha}$, we see $\theta_{x}|_{\frak{su}_{\alpha}(2)} \not= \mathrm{id}_{\frak{su}_{2}(\alpha)}$ and the eigenvalues of $\theta_{x}|_{\frak{su}_{\alpha}(2)}$ are $-1,-1,$ and $1$.
Moreover, $\theta_{x}( \frak{r}_{\alpha} ) \subset \frak{r}_{\alpha}$.
We set
\[
\frak{r}_{\alpha}^{\pm}(\frak{g}, \frak{t}) = \{ X \in \frak{g} \cap (\frak{r}_{\alpha}^{\mathbb{C}} + \frak{r}_{-\alpha}^{\mathbb{C}}) \ ;\ \theta_{x}(X) = \pm X \}.
\]
often denoted by $\frak{r}_{\alpha}^{\pm}$ simply.
Then, $\frak{r}_{\alpha} = \frak{r}_{\alpha}^{+} + \frak{r}_{\alpha}^{-}$.
Moreover, $\dim \frak{r}_{\alpha}^{+} = \dim \frak{r}_{\alpha}^{-} = 1$ and $[\frak{r}_{\alpha}^{+}, \frak{r}_{\alpha}^{-}] \subset \mathbb{R}(iA_{\alpha})$.
We denote $F(\theta_{x}, G)$ by $K_{x}$.
Then, $K_{x}$ is connected since $G$ is simply connected.
Define subspaces $\frak{k}_{x}$ and $\frak{m}_{x}$ as follows:
\[
\frak{k}_{x} = F^{+}(\theta_{x}, \frak{g}) = \{ X \in \frak{g} \ ;\ \theta_{x}(X) = X \}, \quad
\frak{m}_{x} = F^{-}(\theta_{x}, \frak{g}) = \{ X \in \frak{g} \ ;\ \theta_{x}(X) = -X \}.
\]
The subspace $\frak{k}_{x}$ is a Lie algebra of $K_{x}$.
By the definition of $\frak{r}_{\alpha}^{\pm} \ (\alpha \in \Sigma^{+})$, we obtain
\[
\frak{k}_{x} = \frak{t} + \sum_{\alpha \in \Sigma^{+}}\frak{r}_{\alpha}^{+}, \quad
\frak{m}_{x} = \sum_{\alpha \in \Sigma^{+}} \frak{r}_{\alpha}^{-}.
\]
In particular, the symmetric pair $(\frak{g}, \frak{k}_{x})$ is a normal form, that is, $\mathrm{rank}\, G = \mathrm{rank}\, G/K_{x}$.

For any $\alpha \in \Sigma$ and $n \in \mathbb{Z}$, we define
\[
\Sigma_{\alpha, n}(\frak{g}, \frak{t}) = \left\{ \beta \in \Sigma \ ;\ \frac{2(\alpha, \beta)}{(\alpha, \alpha)} = n \right\}.
\]
often denoted by $\Sigma_{\alpha, n}$.
If $\alpha$ is a longer root,
\[
\Sigma = \Sigma_{\alpha, -2} \sqcup \Sigma_{\alpha, -1} \sqcup \Sigma_{\alpha, 0} \sqcup \Sigma_{\alpha, 1} \sqcup \Sigma_{\alpha, 2}.
\]
In particular, $\Sigma_{\alpha, \pm2} = \{ \pm \alpha \}$.
Set $\tau_{\alpha} = \mathrm{exp}\pi(iA_{\alpha})$.
Then, $\tau_{\alpha}$ is an involutive element of $G$.
Therefore, $\theta_{\tau_{\alpha}}$ is an involutive automorphism of $G$.
In the following, $\theta_{\tau_{\alpha}}$ is often denoted by $\theta_{\alpha}$ simply.
We denote $F(\theta_{\alpha}, G)$ by $K_{\alpha}$.
Then, $K_{\alpha}$ is connected.
Define
\[
\frak{k}_{\alpha} = F^{+}(\theta_{\alpha}, \frak{g}) = \{X \in \frak{g} \ ;\ \theta_{\alpha}(X) = X \}, \quad
\frak{m}_{\alpha} = F^{-}(\theta_{\alpha}, \frak{g}) = \{X \in \frak{g} \ ;\ \theta_{\alpha}(X) = -X \}.
\]
Then, the Lie algebra of $K_{\alpha}$ is $\frak{k}_{\alpha}$ and the following decompositions are true:
\[
\frak{k}_{\alpha} = \frak{t} + \frak{r}_{\alpha} + \sum_{\beta \in \Sigma_{\alpha, 0} \cap \Sigma^{+}} \frak{r}_{\beta}, \quad\quad
\frak{m}_{\alpha} = \sum_{\beta \in \Sigma_{\alpha, \pm1} \cap \Sigma^{+}} \frak{r}_{\beta}.
\]
In particular, note that $\frak{su}^{\alpha}(2) \subset \frak{k}_{\alpha}$.
The orthogonal complement of $\mathbb{R}(iA_{\alpha})$ in $\frak{t}$ is denoted by $\frak{t}^{\alpha}$.
Then,
\[
\frak{k}_{\alpha} = \frak{su}^{\alpha}(2) + C(\frak{su}^{\alpha}(2), \frak{k}_{\alpha}), \quad
C(\frak{su}^{\alpha}(2), \frak{k}_{\alpha}) = \frak{t}_{\alpha} + \sum_{\beta \in \Sigma_{\alpha, 0} \cap \Sigma^{+}} \frak{r}_{\beta}^{+}.
\]
The Lie algebra of $C(\frak{su}_{\alpha}(2), G)$ is $C(\frak{su}^{\alpha}(2), \frak{k}_{\alpha})$.
For any root $\alpha, \beta \ (\beta \not= \pm \alpha)$, we set $\frak{t}^{\alpha, \beta}$ as the orthogonal complement of $\mathbb{R}(iA_{\alpha}) + \mathbb{R}(iA_{\beta})$ in $\frak{t}$.

Next, we recall the root system of $E_{8}$ type.
Let $(\ ,\ )_{8}$ be the standard inner product of $\mathbb{R}^{8}$ and $u_{1}, \cdots, u_{8}$ be the standard orthonormal basis of $\mathbb{R}^{8}$.
Then, 
\[
\Sigma_{E_{8}} = \{ \pm u_i \pm u_j \ ;\ 1 \leq i < j \leq 8 \} \cup \left\{ \frac{1}{2}( \epsilon_{1}u_{1} + \cdots \epsilon_{8}u_{8}) \ ;\ \epsilon_{i} = \pm 1 \ (1 \leq i \leq 8), \  \prod_{i=1}^{8}\epsilon_{i} = 1 \right\}
\]
is a root system of type $E_{8}$ and denoted by $\Sigma$ simply in this subsection.
We take a linear order such that the set of all positive roots is
\[
\Sigma^{+} = \{ u_{i} \pm u_{j} \ ;\ 1 \leq i < j \leq 8 \} \cup \left\{ \frac{1}{2}(u_{1} + \epsilon_{2}u_{2} + \cdots + \epsilon_{8}u_{8}) \ ;\ \epsilon_{i} = \pm1 \ (2 \leq i \leq 8),\ \prod_{i=2}^{8}\epsilon_{i} = 1 \right\}.
\]
We consider the lattice $\Gamma = \mathrm{span}_{\mathbb{Z}}\{ 2\gamma \ ;\ \gamma \in \Sigma_{E_{8}} \}$.
Set $\alpha = u_{7} - u_{8}$ and $\beta = u_{6} - u_{7}$.
Then, set
\[
\Lambda_{\alpha, \beta} 
:= \big\{ \alpha + \gamma \ ;\ \gamma \in (\Sigma_{8})^{+}_{\alpha, 0} \big\} 
\cup \big\{ \beta + \delta, \alpha + \beta + \delta \ ;\ \delta \in (\Sigma_{E_{8}})^{+}_{\alpha, 0} \cap \Sigma_{\beta, 0} \big\}.
\]

\begin{lemm} \label{Lattice}
For any $\mu, \nu \in \Lambda_{\alpha, \beta} \ (\mu \not= \nu)$, it is true that $\mu \not\equiv \nu \mod \Gamma$.

\end{lemm}

\begin{proof}
Let $v \in \mathbb{R}^{8}$.
If $v \in \Gamma$, then 
\[
\frac{2(2\lambda, v)}{(2\lambda, 2\lambda)} = \frac{1}{2} \frac{2(\lambda, v)}{(\lambda, \lambda)} = \frac{1}{2}(\lambda, v) \in \mathbb{Z}
\]
for any $\lambda \in \Sigma$ since $\{ 2\lambda \ ;\ \lambda \in \Sigma\}$ is also a root system of $E_{8}$ type.
Hence, if there exist $\lambda \in \Sigma$ such that $(\lambda, v) \not\in 2\mathbb{Z}$, then $v \not\in \Gamma$.
Let $\mu - \nu = \sum_{i=1}^{8}c_{i}u_{i} \ (c_{i} \in \mathbb{R})$.
Then, by direct calculations, we see that there exist $1 \leq i \not= j \leq 8$ such that $(c_{i}, c_{j}) = (\pm1, 0), (\pm3, 0), (\pm\frac{1}{2}, \pm\frac{1}{2})$.
Hence, one of $u_{i} \pm u_{j}$ satisfies $(u_{i} \pm u_{j}, \mu - \nu) \not\in 2\mathbb{Z}$.
Therefore, $\mu - \nu \not\in \Gamma$.
\end{proof}

The Weyl group of $\Sigma$ is denoted by $W(\Sigma)$.
Fix $\alpha \in \Sigma$.
Then, $\Sigma_{\alpha, 0}$ is a root system of $E_{7}$ and the Weyl group of $\Sigma_{\alpha, 0}$ is denoted by $W(\Sigma_{\alpha,0})$.
We can easily verify that $W(\Sigma_{\alpha, 0})$ acts on $\Sigma_{\alpha, 0}, \Sigma_{\alpha, 1}$, and $\Sigma_{\alpha, -1}$ respectively, and these actions are transitive.
Let $\beta \in \Sigma_{\alpha, \pm1}$.
We can assume $\beta \in \Sigma_{\alpha, -1}$ by replacing $\beta$ by $\alpha - \beta$ if necessary.
Then, $\Sigma_{\alpha, 0} \cap \Sigma_{\beta}$ is a root system of type $E_{6}$ and the Weyl group is denoted by $W(\Sigma_{\alpha,0} \cap \Sigma_{\beta,0})$.
We can also verify that $W(\Sigma_{\alpha,0} \cap \Sigma_{\beta,0})$ acts on $\Sigma_{\alpha, 0} \cap \Sigma_{\beta, 0}, \Sigma_{\alpha, 0} \cap \Sigma_{\beta, 1}, \Sigma_{\alpha, 1} \cap \Sigma_{\beta, 0}$, and $\Sigma_{\alpha, 1} \cap \Sigma_{\beta, -1}$, and these actions are transitively.
Note that $\Sigma_{\alpha, 1} \cap \Sigma_{\beta, 1} = \{ \alpha + \beta \}$ and $\Sigma_{\alpha, -1} \cap \Sigma_{\beta, -1} = \{ - \alpha - \beta \}$.


\newpage

\subsection{Spin group}\label{Sg}

In this subsection, we recall the Clifford algebra and $Spin(n)$.
Let $V$ be a non-zero finite dimensional vector space over $\mathbb{R}$.
The tensor algebra over $V$ is denoted by $T(V)$.
Let $(\ ,\ )$ be symmetric bilinear form of $V$.
Set a two-sided ideal $J = \langle v \otimes v + (v,v) 1 \ ;\ v \in V \rangle$ of $T(V)$.
Put $Cl(V) = T(V)/J$ and $Cl(V)$ is called the {\it Clifford algebra} over $V$ with a symmetric bilinear form $(\ ,\ )$.
In $Cl(V)$, the product between $u, v \in Cl(V)$ is denoted by $uv$.
If $V = \mathbb{R}^{n}$ and $(\ ,\ )$ is a standard inner product, then $Cl(V)$ is denoted by $Cl(\mathbb{R}^{n})$.
Let $e_{1}, \cdots, e_{n}$ be the standard orthonormal basis of $\mathbb{R}^{n}$.
Then, for any $1 \leq i, j \leq n$,
\[
e_{i}e_{j} + e_{j}e_{i} = -2\delta_{ij} = -2(e_{i}, e_{j}).
\]
The complexification of $Cl(\mathbb{R}^{n})$ is denoted by $Cl(\mathbb{C}^{n})$.
Let $\pi : T(V) \rightarrow Cl(V)$ be a natural projection.

\begin{prop}
Let $A$ be the algebra over $\mathbb{R}$.
If the homomorphism $j : V \rightarrow A$ satisfies $(j(v))^{2} = -(v,v)1$, there exists unique homomorphims $\theta : Cl(V) \rightarrow A$ such that $\theta|_{V} = j$.
\end{prop}

Define
\[
T(V)^{0} = \bigoplus_{n \geq 0} \otimes^{2n}V, \quad \quad
T(V)^{1} = \bigoplus_{n \geq 0} \otimes^{2n+1}V,
\]
and $Cl(V)^{i} = \pi(T(V)^{i}) \ (i = 1,2)$.
It is known that $Cl(V)$ is a $\mathbb{Z}_{2}$-graded algebra, that is, 
\[
Cl(V) = Cl(V)^{0} \oplus Cl(V)^{1}, \quad \quad
Cl(V)^{i}Cl(V)^{j} \subset Cl(V)^{i+j},
\]
where the indices are $\mathrm{mod}\,2$.
We denote the complexification of $Cl(\mathbb{R}^{n})^{i}$ by $Cl(\mathbb{C}^{n})^{i}$ for $i = 0,1$.
It is well known that 
\[
Cl(\mathbb{R}^{n})^{0} \cong Cl(\mathbb{R}^{n-1}), \quad \quad
Cl(\mathbb{C}^{n})^{0} \cong Cl(\mathbb{C}^{n-1})
\]
as algebras.
For $Cl(\mathbb{C}^{n})$, the following isomorphisms are well known:
\[
\begin{split}
& Cl(\mathbb{C}^{2n}) = M(2^{n}, \mathbb{C}), \\
& Cl(\mathbb{C}^{2n+1}) = M(2^{n}, \mathbb{C}) \oplus M(2^{n}, \mathbb{C}), \\
& Cl(\mathbb{C}^{2n+1})^{0} = \{ (A,A) \ ;\ A \in M(2^{n}, \mathbb{C}) \} = M(2^{n}, \mathbb{C}), \\
& Cl(\mathbb{C}^{2n+1})^{1} = \{ (A, -A) \ ;\ A \in M(2^{n}, \mathbb{C}) \}.
\end{split}
\]
For $Cl(\mathbb{C}^{2m})$,
\[
\rho_{2m} : Cl(\mathbb{C}^{2m}) = M(2^{m}, \mathbb{C}) \ni \phi \mapsto \phi \in M(2^{m}, \mathbb{C}) = \mathrm{End}(\mathbb{C}^{2^{m}})
\]
is a irreducible representation of $Cl(\mathbb{C}^{2m})$ and any irreducible representation is isomorphic to $\rho$.
Moreover, for $Cl(\mathbb{C}^{2m+1})$,
\[
\begin{split}
& \rho^{+}_{2m+1} : Cl(\mathbb{C}^{2m+1}) = M(2^{m}, \mathbb{C}) \oplus M(2^{m}, \mathbb{C}) \ni (\phi, \psi) \mapsto \phi \in M(2^{m}, \mathbb{C}) = \mathrm{End}(\mathbb{C}^{2^{m}}), \\
& \rho^{-}_{2m+1} : Cl(\mathbb{C}^{2m+1}) = M(2^{m}, \mathbb{C}) \oplus M(2^{m}, \mathbb{C}) \ni (\phi, \psi) \mapsto \psi \in M(2^{m}, \mathbb{C}) = \mathrm{End}(\mathbb{C}^{2^{m}})
\end{split}
\]
are irreducible representations of $Cl(\mathbb{C}^{2m+1})$.
Moreover, $\rho^{+}_{2m+1} $ and $\rho^{-}_{2m+1}$ are not isomorphic to each other and any irreducible representation of $Cl(\mathbb{C}^{2m+1})$ is isomorphic to either $\rho^{+}_{2m+1} $ or $\rho^{-}_{2m+1}$.
Since $Cl(\mathbb{C}^{2m})^{0} = Cl(\mathbb{C}^{2m-1})$, the representations of $Cl(\mathbb{C}^{2m})^{0}$ induced by $\rho^{\pm}_{2m-1}$ are denoted by $(\Delta_{2m}^{\pm})^{\mathbb{C}}$.
Similarly, since $Cl(\mathbb{C}^{2m+1})^{0} = Cl(\mathbb{C}^{2m})$, the representation of $Cl(\mathbb{C}^{2m+1})^{0}$ induced by $\rho_{2m}$ is denoted by $(\Delta_{2m+1})^{\mathbb{C}}$.


Next, we recall the spinor group $Spin(n)$ and $Pin(n)$.
For any $u \in \mathbb{R}^{n}$, we denote the element $v$ such that $uv = vu = 1$ by $u^{-1}$.
If $u \not= 0$, then $u^{-1} = (-1/||u||^{2})u$.
In particular, if $||u|| = 1$, then $u^{-1} = -u$.
Let $v \in \mathbb{R}$ be $||v|| = 1$.
Then,
\[
vxv^{-1} = vx(-v) = (-xv - 2(v,x))(-v) = -x + \frac{2(v,x)}{(v,v)}v.
\]
Hence, $\mathbb{R}^{n} \ni x \mapsto vxv^{-1} \in \mathbb{R}^{n}$ is a reflection withe respect to $v$ mutiplicated by $-1$.
In $Cl(\mathbb{R}^{n})$, we define $Pin(n)$ as follows:
\[
Pin(n) := \{ v_{1}v_{2} \cdots v_{p} \ ;\ v_{i} \in \mathbb{R}^{n}, ||v_{i}|| = 1 \ (i = 1,2 \cdots, p) \} \subset Cl(\mathbb{R}^{n}).
\]
Then, $Pin(n)$ is a compact Lie group.
We set a representation $\pi$ of $Pin(n)$ as follws:
For any $g = v_{1} \cdots v_{p}  \in Pin(n)$, 
\[
\pi(g) : \mathbb{R}^{n} \rightarrow \mathbb{R}^{n} \ ;\ u \mapsto gug^{-1} = (v_{1} \cdots v_{p})u((-v_{p}) \cdots (-v_{1})).
\]
Therefore, $\pi(Pin(n))$ is generated by all reflections in $\mathbb{R}^{n}$ and $\pi(Pin(n)) = O(n)$.
In particular, $\pi:Pin(n) \rightarrow O(n)$ is a double covering homomorphism.
The covering transformation of $Pin(n)$ is obtained by $Pin(n) \ni g \mapsto -g \in Pin(n)$.
We define
\[
Spin(n) := \{ v_{1} \cdots v_{p} \in Pin(n) \ ;\ p = 2m \ (m \in \mathbb{N}) \} \subset Cl(\mathbb{R}^{n})^{0}.
\]
Then, $Spin(n)$ is the identity component of $Pin(n)$.
Moreover, $\pi(Spin(n)) = SO(n)$ and $\pi:Spin(n) \rightarrow SO(n)$ is a double covering homomorphism.
Then, the center $C(Spin(n))$ is given by
\[
C(Spin(n)) = 
\begin{cases}
\{ \pm 1, \pm e_{1}e_{2} \cdots e_{n} \} = \mathbb{Z}_{2} \times \mathbb{Z}_{2} & (n = 4m), \\
\{ \pm 1, \pm e_{1}e_{2} \cdots e_{n} \} = \mathbb{Z}_{4} & (n = 4m+2), \\
\{ \pm 1 \} = \mathbb{Z}_{2} & (n = 2m+1). \\
\end{cases}
\]
Let $p,q \in \mathbb{N}$ be $p + q = n$.
Let $Spin(p)$ be defined by $Cl(\mathrm{span}_{\mathbb{R}}\{ e_{1}, \cdots, e_{p} \})$ and $Spin(q)$ be defined by $Cl(\mathrm{span}_{\mathbb{R}}\{ e_{p+1}, \cdots, e_{n} \})$ respectively.
We consider the homomorphim
\[
i_{p,q}\ :\ Spin(p) \times Spin(q) \rightarrow Spin(n) \ ;\ (g,h) \mapsto gh.
\]
Then, $\mathrm{Ker}i_{p,q} = \{ (1,1), (-1,-1) \}$ and $i_{p,q}(Spin(p) \times Spin(q)) = (Spin(p) \times Spin(q)/\{ (1,1), (-1,-1) \}$.
we denote $i_{p,q}(Spin(p) \times Spin(q))$ by $Spin(p) \cdot Spin(q)$.

Next, we consider the Lie algebra $\frak{spin}(n)$ of $Spin(n)$.
The Lie algebra $\frak{spin}(n)$ is given by
\[
\frak{spin}(n) = \left\{ \sum_{1 \leq i < j \leq n}c_{ij}e_{i}e_{j} \in Cl(\mathbb{R}^{n})^{0} \ ;\ c_{ij} \in \mathbb{R} \right\}.
\]
The Lie blacket $[\ ,\ ]$ is given by $[X,Y] = XY - YX$ for any $X,Y \in \frak{spin}(n)$, where $XY$ is product between $X$ and $Y$ in $Cl(\mathbb{R}^{n})$.
For the induced homomorphism $\pi : \frak{spin}(n) \rightarrow \frak{so}(n)$,
\[
\pi(e_{i}e_{j})(e_{k}) = (e_{i}e_{j})e_{k} - e_{k}(e_{i}e_{j}) = 
\begin{cases}
0 & ( k \not= i,j ), \\
2e_{j} & ( k = i ) , \\
-2e_{i} & ( k = j ).
\end{cases}
\]
Let $E_{ij} \in M(n,\mathbb{R}) \ (i < j)$ be the endomorphism of $\mathbb{R}^{n}$ such that
\[
E_{ij}(e_{k}) = 
\begin{cases}
0 & ( k \not= i,j ), \\
e_{j} & ( k = i ) , \\
-e_{i} & ( k = j ).
\end{cases}
\]
Then, $\pi(e_{i}e_{j}) = 2E_{ij}$ and $\pi (\sum_{1 \leq i < j \leq n}c_{ij}e_{i}e_{j}) = \sum_{1 \leq i < j \leq n}2c_{ij}E_{ij}$.
The adjoint representation of $Spin(n)$ is given by
\[
\mathrm{Ad} : Spin(n) \times \frak{spin}(n) \ ;\ (g, X) \mapsto gXg^{-1}.
\]
Let $m = \lfloor n \rfloor$.
For any $\theta_{1}, \cdots, \theta_{m} \in \mathbb{R}$, define 
\[
T_{n}(\theta_{1}, \cdots, \theta_{m}) = \prod_{i=}^{m} \left (\cos \frac{\theta_{i}}{2} + \sin \frac{\theta_{i}}{2} e_{2i-1}e_{2i} \right).
\]
Then, $T_{n} = \{ T_{n}(\theta_{1}, \cdots, \theta_{m}) \ ;\ \theta_{1}, \cdots, \theta_{m} \in \mathbb{R} \}$ is a maximal torus of $\frak{spin}(n)$.
Set 
\[
T'_{n}(\theta_{1}, \cdots, \theta_{m}) = \Big( \mathrm{exp}\theta_{1}E_{12} \Big) \circ \cdots \circ \Big( \mathrm{exp}\theta_{m}E_{2m-1,2m} \Big)
\]
and $T'_{n} = \{ T'_{n}(\theta_{1}, \cdots, \theta_{m}) \ ;\ \theta_{i} \in \mathbb{R} \ (i = 1, \cdots, m) \}$.
Then, $T'_{n}$ is a maximal torus of $SO(n)$ and $\pi(T_{n}(\theta_{1}, \cdots, \theta_{m})) = T'_{n}(\theta_{1}, \cdots, \theta_{m})$ and $\pi(T_{n}) = T'_{n}$.
Let $\frak{t}_{n}$ be the Lie algebra of $T_{n}$.
Set 
\[
t_{n}(c_{1}, \cdots, c_{m}) = \sum_{i=1}^{m}\frac{c_{i}}{2}e_{2i-1}e_{2i} \ (c_i \in \mathbb{R}).
\]
Then, $\frak{t}_{n} = \left\{ t_{n}(c_{1}, \cdots, c_{m}) \ ;\ c_{i} \in \mathbb{R} \ (i = 1, \cdots, m) \right\}$.
Let $x_{i}$ be the one form of $\frak{t}_{n}^{\mathbb{C}}$ such that $x_{i}( t_{n}(c_{1}, \cdots, c_{m}) ) = ic_{i}$.
Then,
\[
\begin{split}
& \Sigma(\frak{spin}(2m), \frak{t}_{2m}) = \{ \pm x_{i} \pm x_{j} \ ;\ 1 \leq i < j \leq m \}, \\
& \Sigma(\frak{spin}(2m+1), \frak{t}_{2m+1}) = \{ \pm x_{i} \pm x_{j} \ ;\ 1 \leq i < j \leq m \} \cup \{ \pm x_{i} \ ;\ 1 \leq i \leq m \}. \\
\end{split}
\]
For example, since
\[
\begin{split}
& \mathrm{ad}(t_{n}(c_{1}, \cdots, c_{m})) \Big( (e_{2i-1}e_{2j-1} + e_{2i}e_{2j}) - i(e_{2i}e_{2j-1} - e_{2i-1}e_{2j}) \Big) \\
& \hspace{40mm} = i(c_{i}- c_{j})\Big( (e_{2i-1}e_{2j-1} + e_{2i}e_{2j}) - i(e_{2i}e_{2j-1} - e_{2i-1}e_{2j}) \Big), \\
& \mathrm{ad}(t_{n}(c_{1}, \cdots, c_{m})) \Big( (e_{2i-1}e_{2j-1} - e_{2i}e_{2j}) - i(e_{2i}e_{2j-1} + e_{2i-1}e_{2j}) \Big) \\
& \hspace{40mm} = i(c_{i} + c_{j})\Big( (e_{2i-1}e_{2j-1} - e_{2i}e_{2j}) - i(e_{2i}e_{2j-1} + e_{2i-1}e_{2j}) \Big), \\
\end{split}
\]
we obtain 
\[
\begin{split}
\frak{r}_{\pm(x_{i}-x_{j})}^{\mathbb{C}} &= \mathbb{C}( (e_{2i-1}e_{2j-1} + e_{2i}e_{2j}) \mp i(e_{2i}e_{2j-1} - e_{2i-1}e_{2j}) ), \\
\frak{r}_{\pm(x_{i}+x_{j})}^{\mathbb{C}} &= \mathbb{C}( (e_{2i-1}e_{2j-1} - e_{2i}e_{2j}) \mp i(e_{2i}e_{2j-1} + e_{2i-1}e_{2j}) ).
\end{split}
\]
The Killing form $\langle \ ,\ \rangle$ of $\frak{spin}(n)$ is given by $\langle X, Y \rangle = (n-2)\mathrm{tr}(\pi(X)\pi(Y))$ for any $X,Y \in \frak{spin}(n)$.
Then, $iA_{x_{i} \pm x_{j}} = (1/2)(e_{2i-1}e_{2i} \pm e_{2j-1}e_{2j})$ for any $1 \leq i < j \leq m$.
We assume that $n = 2m = 4k$.
We take a linear order of $i\frak{t}$ such that $\Sigma^{+} = \{ x_{i} \pm x_{j} \ ;\ 1 \leq i < j \leq m \}$.
Set $x = e_{1}e_{3} \cdots e_{2m-1}$.
Then, $\mathrm{Ad}(x)(\frak{t}_{n}) \subset \frak{t}_{n}$ and $\mathrm{Ad}(x)(X) = -X$ for any $X \in \frak{t}_{n}$.
Moreover, $x^{2} = 1$.
By the definition of $\frak{r}_{x_{i} \pm x_{j}}^{\pm}$,
we obtain
\[
\frak{r}_{x_{i} \pm x_{j}}^{+} = \mathbb{R} \Big( e_{2i-1}e_{2j-1} \mp e_{2i}e_{2j} \Big), \quad
\frak{r}_{x_{i} \pm x_{j}}^{-} = \mathbb{R} \Big( e_{2i}e_{2j-1} \pm e_{2i-1}e_{2j} \Big).
\]


We consider some subgroups of $Spin(n)$.
For each $\alpha \in \Sigma^{+}$, the connected subgroup of $Spin(n)$ whose Lie algebra is $\frak{su}^{\alpha}(2)$ is isomorphic to $SU(2) \cong Sp(1) \cong Spin(3)$.
Therefore, this subgroup is denoted by $SU^{\alpha}(2)$.
For each $\alpha = x_{i} \pm x_{j} \ (1 \leq i < j \leq m)$, we set $\bar{\alpha} = x_{i} \mp x_{j}$.
We consider the centralizer $C(SU^{\alpha}(2), Spin(n))$.
Then,
\[
C(\frak{su}^{\alpha}(2), \frak{spin}(n)) = \frak{t}^{\alpha} + \sum_{\beta \in \Sigma^{+}_{\alpha,0}} \frak{r}_{\beta}.
\]
We see that $C(\frak{su}^{\alpha}(2), \frak{spin}(n))$ is isomorphic to $\frak{su}(2) + \frak{spin}(n-4)$.
The irreducible component isomorphic to $\frak{su}(2)$ is $\frak{su}^{\bar{\alpha}}(2)$ and the other irreducible component is $\frak{spin}(n-4)$ defined by the Clliford algebra over the orthogonal complement of the subspace $\mathrm{span}_{\mathbb{R}}\{ e_{2i-1}, e_{2i}, e_{2j-1}, e_{2j} \}$.
The irreducible component isomorphic to $\frak{spin}(n-4)$ is denoted by $\frak{spin}^{\alpha}(n-4)$.
The connected subgroup whose Lie algebra is $\frak{spin}^{\alpha}(n-4)$ is denoted by $Spin^{\alpha}(n-4)$.
For example, assume that $\alpha = x_{1} - x_{2}$.
Then,
\[
\begin{split}
& C(\frak{su}^{\alpha}(2), \frak{spin}(n)) = \{ t(c_{1}, c_{1}, c_{2}, \cdots, c_{n-1}) \ ;\ c_{i} \in \mathbb{R} \} + \frak{r}_{x_{1} + x_{2}} + \sum_{3 \leq i < j \leq m}\frak{r}_{x_{i} \pm x_{j}}, \\
& \frak{spin}^{\alpha}(n-4) = \{ t(0,0,c_{2}, \cdots, c_{n-1}) \ ;\ c_{i} \in \mathbb{R} \} + \sum_{3 \leq i < j \leq m}\frak{r}_{x_{i} \pm x_{j}}. \\
\end{split}
\]
Hence, $Spin^{\alpha}(n-4)$ is defined by the Clliford algebra over the orthogonal complement of $\mathrm{span}_{\mathbb{R}}\{ e_{1}, e_{2}, e_{3}, e_{4} \}$.
The idenitity component of $C(SU^{\alpha}(2), Spin(n))$ is 
\[
\begin{split}
& \{ gh \ ;\ g \in SU^{\bar{\alpha}}(2), h \in Spin^{\alpha}(n-4) \} = SU^{\bar{\alpha}}(2) \times Spin^{\alpha}(n-4).
\end{split}
\]

Next, let $\alpha = x_{i} \pm x_{j} \ (1 \leq i < j \leq m)$.
Then, $SU^{\alpha}(2) \times SU^{\bar{\alpha}}(2) = Spin(4)$, where $Spin(4)$ is defined by the Clliford algebra $Cl(\mathrm{span}_{\mathbb{R}}\{ e_{2i-1}, e_{2i}, e_{2j-1}, e_{2j} \})$.
We denote $SU^{\alpha}(2) \times SU^{\bar{\alpha}}(2)$ by $Spin(4)^{\alpha}$ or $Spin^{\bar{\alpha}}(4)$.
We consider the centralizer $C(Spin^{\alpha}(4), Spin(n))$.
Then,
\[
C(\frak{spin}^{\alpha}(4), \frak{spin}(n)) = \frak{t}^{\alpha, \bar{\alpha}} + \sum_{\gamma \in \Sigma^{+}_{\alpha,0} \cap \Sigma_{\bar{\alpha},0}} \frak{r}_{\gamma} = \frak{spin}^{\alpha}(n-4)
\]
and the identity component of $C(Spin^{\alpha}(4), Spin^{\alpha}(n))$ is $Spin^{\alpha}(n-4)$.
The subgroup 
\[
\{ gh \ ;\ g \in Spin^{\alpha}(4), h \in Spin^{\alpha}(n-4) \} = Spin^{\alpha}(4) \times Spin^{\alpha}(n-4)/\{ (1,1), (-1,-1) \}
\]
is denoted by $Spin^{\alpha}(4) \cdot Spin^{\alpha}(n-4)$.
We consider the homomorphism
\[
i_{\alpha} : Spin^{\alpha}(4) \times Spin^{\alpha}(12) \rightarrow Spin(16) \ ;\ (g,h) \mapsto gh.
\]
Then, $i_{\alpha}(Spin^{\alpha}(4) \times Spin^{\alpha}(12)) = Spin^{\alpha}(4) \cdot Spin^{\alpha}(12)$ and $\mathrm{Ker}i_{\alpha} = \{(1,1), (-1,-1)\}$.
We consider the conjugation $\theta_{\tau_{\alpha}}$.
Then, $\theta_{\tau_{\alpha}}$ is involutive and $F(\theta_{\tau{\alpha}}, Spin(n)) = Spin^{\alpha}(4) \cdot Spin^{\alpha}(12)$.

Moreover, let $\beta \in \Sigma^{+}$ be $2(\alpha, \beta)/(\alpha, \alpha) = \pm 1$.
Then, $\alpha \mp \beta \in \Sigma^{+}$.
We consider the connected subgroup of $Spin(n)$ whose Lie algebra is $\frak{su}^{\alpha, \beta}(2)$.
Then, this subgroup is isomorphic to $SU(3)$ and denoted by $SU^{\alpha, \beta}(3)$.
We consider the centralizer $C(SU^{\alpha,\beta}(3), Spin(n))$.
Then,
\[
C(\frak{su}^{\alpha,\beta}(3), \frak{spin}(n)) = \frak{t}^{\alpha, \beta} + \sum_{\gamma \in \Sigma^{+}_{\alpha, 0} \cap \Sigma_{\beta, 0}}\frak{r}_{\gamma}.
\]
We see that $C(\frak{su}^{\alpha,\beta}(3), \frak{spin}(n))$ is isomorphic to $\mathbb{R} + \frak{spin}(n-6)$.
Then, the irreducible component of the subalgbera $C(\frak{su}^{\alpha, \beta}(3), \frak{spin}(n))$ isomorphic to $\frak{spin}(n-6)$ is denoted by $\frak{spin}^{\alpha, \beta}(n-6)$.
Moreover, the connected subgroup of $Spin(n)$ whose Lie sublagebra is $\frak{spin}^{\alpha,\beta}(n-6)$ is denoted by $Spin^{\alpha,\beta}(n-6)$.
For example, assume $\alpha = x_{1} - x_{2}$ and $\beta = x_{2} - x_{3}$.
Then,
\[
\begin{split}
& C(\frak{su}^{\alpha,\beta}(3), \frak{spin}(n)) = \{ t(c_{1}, c_{1}, c_{1}, c_{2}, \cdots, c_{n-2}) \ ;\ c_{i} \in \mathbb{R} \} + \sum_{4 \leq i < j \leq m}\frak{r}_{x_{i} \pm x_{j}}, \\
& \frak{spin}^{\alpha, \beta}(n-6) = \{ t(0,0,0, c_{2}, \cdots, c_{n-2}) \ ;\ c_{i} \in \mathbb{R} \} + \sum_{4 \leq i < j \leq m}\frak{r}_{x_{i} \pm x_{j}}. \\
\end{split}
\]
Hence, $Spin^{\alpha, \beta}(n-4)$ is $Spin(n-4)$ defined by the Clliford algebra over the orthogonal complement of $\mathrm{span}_{\mathbb{R}}\{ e_{1}, \cdots, e_{6}\}$.
Then, the identity component of the centralizer $C(SU^{\alpha,\beta}(3), Spin(n))$ is given by
\[
\begin{split}
& \left\{ \left( \cos \frac{\theta}{2} + \sin \frac{\theta}{2}e_{1}e_{2} \right) \left( \cos \frac{\theta}{2} + \sin \frac{\theta}{2}e_{3}e_{4} \right) \left( \cos \frac{\theta}{2} + \sin \frac{\theta}{2}e_{5}e_{6} \right) h \ ;\ \theta \in \mathbb{R}, h \in Spin^{\alpha,\beta}(n-6) \right\} \\
& = (U(1) \times Spin(n-6))/\{ (1,1), (-1,-1) \} \cong U(1) \cdot Spin(n-6).
\end{split}
\]

Next, we consider some subgroups of $Spin(4k)$.
Define 
\[
U(2k) = \{ A \in M(2k,\mathbb{C}) \ ;\ {}^{t}\overline{A} = A^{-1} \}, \quad
SU(2k) = \{ A \in U(2k) \ ;\ \det A = 1\}.
\]
Set
\[
J_{2k} = \sum_{i=0}^{k-1} \Big( E_{2i+1, 2i+2} - E_{2i+2, 2i+1} \Big) \in SO(2k).
\]
Then, 
\[
\sigma : SU(2k) \rightarrow SU(2k) \ ;\ X \mapsto \overline{X}, \quad
\theta_{J_{2k}}\sigma : SU(2k) \rightarrow SU(2k) \ ;\ X \mapsto J_{m} \overline{X} J_{m}^{-1}
\]
are involutive automorphisms respectively and 
\[
F(\sigma, SU(2k)) = SO(2k), \quad
F(\theta_{J_{2k}}\sigma, SU(2k)) = Sp(k).
\]
Set 
\[
J' = \sum_{i=0}^{2k-1} \Big( E_{2i+2, 2i+1} - E_{2i+1. 2i+2} \Big) \in SO(4k).
\]
Then, by the restriction of the map
\[
c:
M(2k, \mathbb{C}) \ni 
\begin{pmatrix}
a_{11} + ib_{11} & \cdots & a_{12k} + ib_{i2k} \\
\vdots & \ddots & \vdots \\
a_{2k1} + ib_{2k1} & \cdots & a_{2k2k} + ib_{2k2k} \\
\end{pmatrix} 
\mapsto 
\begin{pmatrix}
a_{11} & b_{11} & \cdots & a_{12k} & b_{12k} \\
-b_{11} & a_{11} &  & -b_{12k} & a_{12k} \\
\vdots & & & & \\
a_{2k1} & b_{2k1} & & a_{2k2k} & b_{2k2k} \\
-b_{2k1} & a_{2k1} & & -b_{2k2k} & a_{2k2k} \\
\end{pmatrix}
\in M(4k,\mathbb{R})
\]
to $U(2k)$, we see that $U(2k)$ is isomorphic to $F(\theta_{J'}, SO(4k))$.
Hence, 
\[ 
U'(2k) := F(\theta_{J}, SO(4k)).
\] 
The center $C(U'(2k))$ is $\{ \mathrm{exp}\theta J' \ ;\ \theta \in \mathbb{R} \}$.
We denote $c(SU(2k))$ by $SU'(2k)$.
Set $I = \mathrm{d}(1,-1,1,-1, \cdots, 1, -1) \in SO(4k)$.
Then, for any $g \in SU(2k)$, we obtain that $c \circ \sigma (g) = \theta_{I} \circ c (g)$.
Moreover, set
\[
J'_{2k} = \sum_{i=0}^{k-1} \Big( (E_{4i+3, 4i+1} + E_{4i+4, 4i+2}) - (E_{4i+1,4i+3} + E_{4i+2, 4i+4}) \Big) \in SO(4k).
\]
Then, $c(J_{2k}) = J'_{2k}$ and $c \circ \theta_{J_{2k}} (g) = \theta_{J'_{2k}} \circ c (g)$ for any $g \in SU(n)$.
Set $J'' \in Spin(4k)$ as follows:
\[
J'' = \prod_{i=0}^{k-1} \Big( \cos \frac{\pi}{4} + \sin \frac{\pi}{4}e_{2i+1}e_{2i+2} \Big).
\]
Then, $F(\theta_{J''}, Spin(4k))$ is isomorphic to $U(2k)$ and we denote $F(\theta_{J''}, Spin(4k))$ by $U''(2k)$.
Then, $\pi(U''(2k)) = U'(2k)$ and $\pi : U''(2k) \rightarrow U'(2k)$ is an isomorophism.
We set $SU''(2n) = (\pi|_{U''(2k)})^{-1}(SU'(2k))$.
Then, we see $\pi(x) = I$ and $\theta_{I} \circ \pi = \pi \circ \theta_{x}$, where $x = e_{2} e_{4} \cdots e_{4k-2}e_{4k}$.
Set
\[
J''_{2k} = \prod_{i=0}^{k-1} \Big( \cos \frac{\pi}{4} + \sin \frac{\pi}{4} e_{4i}e_{4i+2} \Big) \Big( \cos \frac{\pi}{4} + \sin \frac{\pi}{4} e_{4i+1}e_{4i+3} \Big).
\]
Then, $\pi(J''_{2k}) = J'_{2k}$ and $\theta_{J'_{2k}} \circ \pi (g) = \pi \circ \theta_{J''_{2k}} (g)$ for any $g \in SU''(2k)$.
Therefore, the diagram
\[
\xymatrix{
SU(2k) \ar[r]^{\sigma} \ar[d]^{c} & SU(2k) \ar[r]^{\theta_{J_{2k}}} \ar[d]^{c} & SU(2k) \ar[d]^{c} \\
SU'(2k) \ar[r]^{\theta_{I}} & SU'(2k) \ar[r]^{\theta_{J'_{2k}}} & SU'(2k) \\
SU''(2k) \ar[r]^{\theta_{x}} \ar[u]^{\pi} & SU''(2k) \ar[r]^{\theta_{J''_{2k}}} \ar[u]^{\pi} & SU''(2k) \ar[u]^{\pi} \\
}
\]
is commutative.
Hence,
\[
F \Big( \theta_{J''_{2k}}\theta_{x}, SU''(2k) \Big) \cong F \Big( \theta_{J'_{2k}}\theta_{x}, SU'(2k) \Big) \cong F \Big( \theta_{J_{2k}} \sigma, SU(2k) \Big) = Sp(k).
\]
The Lie algebra of $U'(n), U''(n), Sp'(n), Sp''(n)$ are denoted by $\frak{u}'(n), \frak{u}''(n), \frak{sp}'(n), \frak{sp}''(n)$ resplectively.


Next, we recall the spin representation from \cite{Adams-L}.
By the definition of $Spin(n)$, we see $Spin(n) \subset Cl(\mathbb{C}^{n})^{0}$.
Since, $Cl(\mathbb{C})^{0} = Cl(\mathbb{C}^{n-1})$, we obtain the complex representation of $Spin(n)$ by the irreducible representation of $Cl(\mathbb{C}^{n})^{0}$.
The complex representation of $Spin(2m+1)$ induced by the irreducible representation $\Delta_{2m+1}^{\mathbb{C}}$ of $Cl(\mathbb{C}^{2m})$ is denoted by the same symbol $\Delta_{2m+1}^{\mathbb{C}}$.
The representation space of $\Delta_{2m+1}^{\mathbb{C}}$ is denoted by $V_{2m+1}^{\mathbb{C}}$.
Moreover, the complex representations of $Spin(2m)$ induced by the irreducible representations $(\Delta_{2m}^{\pm})^{\mathbb{C}}$ are denoted by the same symbols $(\Delta_{2m}^{\pm})^{\mathbb{C}}$.
The representation of $(\Delta_{2m}^{\pm})^{\mathbb{C}}$ is denoted by $(V_{2m}^{\pm})^{\mathbb{C}}$.
These representations $\Delta_{2m+1}^{\mathbb{C}}$ and $(\Delta_{2m}^{\pm})^{\mathbb{C}}$ are called the {\it spin representation}.
It is known that $\Delta_{2m+1}^{\mathbb{C}}$ and $(\Delta_{2m}^{\pm})^{\mathbb{C}}$ are irreducible representations.
In $Spin(2m+1)$, the set of all weights of $\Delta_{2m+1}^{\mathbb{C}}$ with respect to $\frak{t}_{2m+1}^{\mathbb{C}}$ is denoted by $\Pi_{2m+1}$.
Then, 
\[
\Pi_{2m+1} = \left\{ \frac{1}{2} \Big( \epsilon_{1}x_{1} + \cdots + \epsilon_{m}x_{m} \Big) \ ;\ \epsilon_{i} = \pm1 \ (i = 1, \cdots, m) \right\}.
\]
In $Spin(2m)$, the set of all weights of $(\Delta_{2m}^{\pm})^{\mathbb{C}}$ with respect to $\frak{t}_{2m}^{\mathbb{C}}$ is denoted by $\Pi_{2m}^{\pm}$.
Then, 
\[
\Pi_{2m}^{\pm} = \left\{ \frac{1}{2} \Big( \epsilon_{1}x_{1} + \cdots + \epsilon_{m}x_{m} \Big) \ ;\ \epsilon_{i} = \pm1 \ (i = 1, \cdots, m), \epsilon_{1} \cdots \epsilon_{m} = \pm 1 \right\}.
\]
Let $p,q \in \mathbb{N}$ such that $n = p + q$.
For any irreduicible representation $\rho$ of $Spin(n)$, the restriction of $\rho$ to $Spin(p) \cdot Spin(q)$ is decomposed into the direct sum of irreducible representations of $Spin(p) \cdot Spin(q)$.
For any $s, r \in \mathbb{N}$, 
\[
\begin{array}{llll}
(\Delta_{2r+ 2s}^{+})^{\mathbb{C}} = (\Delta_{2r}^{+})^{\mathbb{C}} \otimes (\Delta_{2s}^{+})^{\mathbb{C}} + (\Delta^{-}_{2r})^{\mathbb{C}} \otimes (\Delta^{-}_{2s})^{\mathbb{C}}, &
(\Delta^{+}_{2r+2s+2})^{\mathbb{C}} = \Delta_{2r+1}^{\mathbb{C}} \otimes \Delta_{2s+1}^{\mathbb{C}}, \\
(\Delta_{2r+ 2s}^{-})^{\mathbb{C}} = (\Delta_{2r}^{+})^{\mathbb{C}} \otimes (\Delta_{2s}^{-})^{\mathbb{C}} + (\Delta^{+}_{2r})^{\mathbb{C}} \otimes (\Delta^{-}_{2s})^{\mathbb{C}}, &
(\Delta^{-}_{2r+2s+2})^{\mathbb{C}} = \Delta_{2r+1}^{\mathbb{C}} \otimes \Delta_{2s+1}^{\mathbb{C}}, \\
\Delta_{2r + 2s + 1}^{\mathbb{C}} = (\Delta^{+}_{2r})^{\mathbb{C}} \otimes \Delta_{2s + 1}^{\mathbb{C}} + (\Delta^{-}_{2r})^{\mathbb{C}} \otimes \Delta_{2s + 1}^{\mathbb{C}}.
\end{array}
\]
It is known that the complex representations $\Delta_{m}^{\mathbb{C}}$ has the real representation $\Delta_{m}$ if and only if $m \equiv 1, 7 \mod 8$.
The representation space is denoted by $V_{m}$.
Moreover, the representation $(\Delta_{m}^{\pm})^{\mathbb{C}}$ has the real representation $\Delta_{m}^{\pm}$ if and only if $m \equiv 0 \mod 8$.
Similarly, the representation space is denoted by $V_{m}^{\pm}$.


\newpage

\subsection{Octonions}\label{O}

In this subsection, we recall the octonions, $G_{2}$, and $Spin(8)$.
Let $\mathbb{O}= \bigoplus_{i=0}^{7}\mathbb{R}e_{i}$ be the octonions, where $e_{0}, \cdots, e_{7}$ form a basis.
The multiplication between two octonions is defined as follows.
$e_{0}$ is the identity element and denoted by $1$.
For any $1 \leq i \not= j \leq 7$, we have $e_{i}^{2} = -1$ and $e_{i}e_{j} = -e_{j}e_{i}$.
Assume that this multiplication satisfies the distributive law.
In Figure $1$, the multiplication among $e_{1},e_{2},e_{3}$ is defined as $e_{1}e_{2} = e_{3}, \ e_{2}e_{3}=e_{1}, \ e_{3}e_{1}=e_{2}$.
Similarly, the multiplication among any three elements on each of the other lines and the circle is defined in the same manner. 
Note that the associative law does not follow in $\mathbb{O}$.

For any $x = x_{0} + \sum_{i=1}^{7}x_{i}e_{i} \in \mathbb{O} \ (x_{i} \in \mathbb{R})$, the element $\bar{x} = x_{0} - \sum_{i=1}^{7}x_{i}e_{i}$ is referred to as the conjugate of $x$.
Define $\mathrm{Im}\mathbb{O} = \{ x \in \mathbb{O} \ ;\ \bar{x} = -x \}$ and elements of $\mathrm{Im}\mathbb{O}$ are called pure octonions.
For any $x = \sum_{i=0}^{7}x_{i}e_{i}$ and $y = \sum_{i=0}^{7}y_{i}e_{i} \in \mathbb{O}$, set $(x,y) = \sum_{i=0}^{7}x_{i}y_{i}$.
Then, $(\ ,\ )$ is an inner product of $\mathbb{O}$ and satisfies
\[
\begin{split}
(x,y) &= \frac{1}{2}(x\bar{y} + y\bar{x}).
\end{split}
\]
Set $|x| = \sqrt{(x,x)}$.
In the present paper, we often denote $O(\mathbb{O})$ and $SO(\mathbb{O})$ by $O(8)$ and $SO(8)$, respectively.
Moreover, $O(7) = \{ g \in O(8)\ ;\ g(1) = 1 \}$ and $SO(7) = \{ g \in SO(8) \ ;\ g(1) = 1\}$.
If a linear automorphism $f:\mathbb{O} \rightarrow \mathbb{O}$ satisfies $f(xy) = f(x)f(y)$ for any $x,y \in \mathbb{O}$, then $f$ is called an automorphism of $\mathbb{O}$.
The group of all automorphisms of $\mathbb{O}$ is the exceptional compact Lie group $G_{2}$.
For any $g \in G_{2}$, it holds that $g(1) = 1$ and $(g(x), g(y)) = (x,y)$ for any $x,y \in \mathbb{O}$.
Moreover, $G_{2}$ is simply connected and $G_{2} \subset SO(7)$, that is, $G_{2} = \{ g \in SO(7) \ ;\ g(u)g(v) = g(uv) \}$.
Let $\frak{so}(8)$ and $\frak{so}(7)$ denote the Lie algebras of $SO(8)$ and $SO(7)$, respectively.
The Lie algebra of $G_{2}$ is denoted by $\frak{g}_{2}$.
Then, $\frak{g}_{2} = \{ X \in \frak{so}(7) \ ;\ X(u)v + uX(v) = 0 \}$.

Next, we recall some outer linear automorphisms $\alpha, \beta, \gamma$ of $\frak{so}(8)$ from \cite{Yokota-g-r}.
For any $a \in \mathbb{O}$, we set $L_{a}, R_{a}, T_{a} \in \frak{so}(8)$ as follows:
\[
L_{a}(u) = au, \quad
R_{a}(u) = ua , \quad
T_{a}(u) = au + ua = (L_{a} + R_{a})(u).
\]
For any $0 \leq i < j \leq 7$, we define $E_{ij} \in \frak{so}(8)$ as
\[
E_{ij}(e_{k}) = 
\begin{cases}
e_{j} & (k = i), \\
e_{i} & (k = j), \\
0 & (\text{otherwise}).
\end{cases}
\]
Then, $T_{e_{k}} = 2E_{k0}$ and $[T_{e_{k}}, T_{e_{l}}] = 4E_{kl}$.
Since $E_{ij} \ (0 \leq i < j \leq 7)$ forms a basis of $\frak{so}(8)$, we see that $T_{e_{k}}, [T_{e_{k}}, T_{e_{l}}] \ (1 \leq k < l \leq 7)$ form a basis of $\frak{so}(8)$.
Moreover, for any $0 \leq i < j \leq 7$, we set $F_{ij} \in \frak{so}(8)$ for any $0 \leq i,j \leq 7$ as follows.
For any $x \in \mathbb{O}$,
\[
\begin{split}
& F_{i0}(x) = -\frac{1}{2}(L_{e_{i}}x), \quad F_{0i}(x) = \frac{1}{2}(L_{e_{i}}x) \quad (i \not= 0), \\
& F_{ij}(x) = \frac{1}{2}L_{e_{i}}(L_{e_{j}}x) = \frac{1}{4}[ L_{e_{i}}, L_{e_{j}} ](x)\quad (1 \leq i \not= j \leq 7).
\end{split}
\]
Then, 
\[
\begin{matrix}
\begin{array}{lllll} \vspace{2mm}
\begin{array}{llll}
2F_{01} = E_{01} + E_{23} + E_{45} + E_{67}, \\
2F_{23} = E_{01} + E_{23} - E_{45} - E_{67}, \\ 
2F_{45} = E_{01} - E_{23} + E_{45} - E_{67}, \\
2F_{67} = E_{01} - E_{23} - E_{45} + E_{67}, \\
\end{array} &
\begin{array}{llll}
2F_{02} = E_{02} - E_{13} - E_{46} + E_{57}, \\ 
2F_{13} = -E_{02} + E_{13} - E_{46} + E_{57}, \\
2F_{46} = -E_{02} - E_{13} + E_{46} + E_{57}, \\ 
2F_{57} = E_{02} + E_{13} + E_{46} + E_{57}, \\
\end{array} \\ \vspace{3mm}
\begin{array}{llll}
2F_{03} = E_{03} + E_{12} + E_{47} + E_{56}, \\
2F_{12} =  E_{03} + E_{12} - E_{47} - E_{56}, \\ 
2F_{47} = E_{03} - E_{12} + E_{47} - E_{56}, \\ 
2F_{56} = E_{03} - E_{12} - E_{47} + E_{56}, \\
\end{array} &
\begin{array}{llll}
2F_{04} = E_{04} - E_{15} + E_{26} - E_{37}, \\ 
2F_{15} = - E_{04} + E_{15} + E_{26} - E_{37}, \\ 
2F_{26} = E_{04} + E_{15} + E_{26} + E_{37}, \\
2F_{37} = - E_{04} - E_{15} + E_{26} + E_{37}, \\
\end{array} \\ \vspace{3mm}
\begin{array}{llll}
2F_{05} = E_{05} + E_{14} - E_{27} - E_{36}, \\ 
2F_{14} = E_{05} + E_{14} + E_{27} + E_{36}, \\
2F_{27} = - E_{05} + E_{14} + E_{27} - E_{36}, \\ 
2F_{36} = - E_{05} + E_{14} - E_{27} + E_{36}, \\
\end{array} &
\begin{array}{llll}
2F_{06} = E_{06} - E_{17} - E_{24} + E_{35}, \\
2F_{17} = -E_{06} + E_{17} - E_{24} + E_{35}, \\ 
2F_{24} = -E_{06} - E_{17} + E_{24} + E_{35}, \\ 
2F_{35} = E_{06} + E_{17} + E_{24} + E_{35}, \\
\end{array} \\ 
\end{array} \\
\begin{array}{llll}
2F_{07} = E_{07} + E_{16} + E_{25} + E_{34}, \\ 
2F_{16} = E_{07} + E_{16} - E_{25} - E_{34}, \\ 
2F_{25} = E_{07} - E_{16} + E_{25} - E_{34}, \\
2F_{34} = E_{07} - E_{16} - E_{25} + E_{34}.
\end{array}
\end{matrix}
\]
Define $\alpha, \beta, \gamma : \frak{so}(8) \rightarrow \frak{so}(8)$ as follows:
for any $X \in \frak{so}(8)$ and $a \in \mathbb{O}$,
\[
\alpha(X)(a) = \overline{X(\overline{a})}, \quad \beta(E_{ij}) = F_{ij} \ (0 \leq i < j \leq 7), \quad \gamma = \beta \circ \alpha.
\]
We see that $\alpha^{2} = \beta^{2} = \mathrm{id}_{\frak{so}(8)}$.
Remark that $\beta(T_{e_{k}}) = L_{e_{k}}$ and $\beta([T_{e_{k}}, T_{e_{l}}]) = [L_{e_{k}}, L_{e_{l}}]$ for any $1 \leq k < l \leq 7$.

\begin{prop} \cite{Yokota-g-r}
The maps $\alpha, \beta, \gamma$ are outer automorphisms of $\frak{so}(8)$.

\end{prop}

\begin{example} \label{LRT}
For any $a \in \mathrm{Im}\mathbb{O}$, 
\[
\begin{array}{lllll}
\alpha (L_{a}) = -R_{a}, & \alpha (R_{a}) = - L_{a}, & \alpha(T_{a}) = - T_{a}, \\
\beta(L_{a}) = T_{a}, & \beta(R_{a}) = -R_{a}, & \beta(T_{}) = L_{a}, \\
\gamma(L_{a}) = R_{a}, & \gamma(R_{a}) = -T_{a}, & \gamma(T_{a}) = -L_{a}.
\end{array}
\]
\end{example}

Since $E_{ij} \ (0 \leq i < j \leq 7)$ form a basis of $\frak{so}(8)$, it holds that $F_{ij}\ (0 \leq i < j \leq 7)$ also form a basis of $\frak{so}(8)$.
Moreover, since $T_{e_{k}} \ (01 \leq k \leq 7)$ generates $\frak{so}(8)$, $L_{e_{k}} \ (1 \leq k \leq 7)$ generates $\frak{so}(8)$.
By Example \ref{LRT}, we obtain the following.

\begin{thm} \cite{Yokota-g-r}
The subgroup $S_{3}$ of the automorphism group of $SO(8)$ generated by $\alpha$ and $\beta$ is isomorphic to the symmetric group $\frak{S}_{3}$.

\end{thm}

\begin{proof}
The following correspondence gives the isomorphism:
\[
\begin{array}{llll} \vspace{2mm}
1 \mapsto \begin{pmatrix} 1 & 2 & 3 \\ 1 & 2 & 3 \\ \end{pmatrix}, &
\alpha \mapsto \begin{pmatrix} 1 & 2 & 3 \\ 2 & 1 & 3 \end{pmatrix}, &
\beta \mapsto \begin{pmatrix} 1 & 2 & 3 \\ 3 & 2 & 1 \end{pmatrix}, \\
\gamma \mapsto \begin{pmatrix} 1 & 2 & 3 \\ 2 & 3 & 1 \end{pmatrix}, &
\gamma^{2} \mapsto \begin{pmatrix} 1 & 2 & 3 \\ 3 & 1 & 2 \end{pmatrix}, &
\gamma \beta \mapsto \begin{pmatrix} 1 & 2 & 3 \\ 1 & 3 & 2 \end{pmatrix}.
\end{array}
\]
\end{proof}

The following property is called the triality principle of $\frak{so}(8)$.

\begin{thm} \cite{Yokota-g-r} \label{tri-d4}
For any $X_{1} \in \frak{so}(8)$, there exist unique $X_{2}, X_{3} \in \frak{so}(8)$ such that
\[
\Big( X_{1}(u) \Big)v + u \Big( X_{2}(v) \Big) = X_{3}(uv)
\]
for any $u, v \in \mathbb{O}$.
Moreover, $X_{2} = \gamma(X_{1})$ and $X_{3} = \beta(X_{3})$.

\end{thm}

\begin{proof}
Since $L_{e_{k}} \ (1 \leq k \leq 7)$ generate $\frak{so}(8)$, it is sufficient to show the statement follows in tha case of $X_{1} = L_{e_{k}}$ and $X_{1} = [L_{e_{k}}, L_{e_{l}}] \ (1 \leq k < l \leq 7)$.
In the case of $X_{1} = L_{e_{k}}$, we set $X_{2} = R_{e_{k}}$ and $X_{3} = T_{e_{k}}$.
Then,
\[
(L_{e_{k}}u)v + u(R_{e_{k}}u) = (e_{k}u)v + u(ve_{k}) = e_{k}(uv) + (uv)e_{k} = T_{e_{k}}(uv).
\]
In the case of $X_{1} = L_{e_{k}} L_{e_{l}}$, since
\[
\begin{split}
T_{e_{k}}T_{e_{l}}(uv) 
&= T_{e_{k}} \Big( \big( L_{e_{l}}u \big) v + u \big( R_{e_{l}}u \big) \Big) \\
&= \big( L_{e_{k}}L_{e_{l}}u \big) v + \big( L_{e_{l}}u \big) \big( R_{e_{k}} v \big) + \big( L_{e_{k}}u \big) \big( R_{e_{l}}v \big) + u \big( R_{e_{k}} R_{e_{l}} \big), \\
T_{e_{l}}T_{e_{k}}(uv) 
&= T_{e_{l}} \Big( \big( L_{e_{k}}u \big) v + u \big( R_{e_{k}}u \big) \Big) \\
&= \big( L_{e_{l}}L_{e_{k}}u \big) v + \big( L_{e_{k}}u \big) \big( R_{e_{l}} v \big) + \big( L_{e_{l}}u \big) \big( R_{e_{k}}v \big) + u \big( R_{e_{l}} R_{e_{k}} \big),
\end{split}
\]
we obtain
\[
\Big( [L_{e_{k}}, L_{e_{l}}] u \Big) v + u \Big( [R_{e_{k}}, R_{e_{l}}] v \Big) = [T_{e_{k}}, T_{e_{l}}](uv).
\]
Next, we consider the uniqueness of $X_{2}$ and $X_{3}$.
If $X'_{2}$ and $X'_{3}$ satisfies the same condition.
Then, for any $u,v \in \mathbb{O}$,
\[
u \big( (X_{2} -X'_{2})v \big) = (X_{3} - X'_{3})(uv).
\]
If $u = 1$, then $(X_{2} - X'_{2})(v) = (X_{3} - X'_{3})(v)$.
Hence, $X_{2} - X'_{2} = X_{3} - X'_{3}$.
Set $X = X_{2} - X'_{2} = X_{3} - X'_{3}$.
Then, $u \big( X(v) \big) = X(uv)$.
Set $a = X(1)$.
Since $X \in \frak{so}(8)$,
\[
2(a,1) = (a,1) + (1,a) = (X(1), 1) + (1, X(1)) = 0.
\]
Therefore, $a \in \mathrm{Im}\mathbb{O}$ and $X(u) = ua$ for any $u \in \mathbb{O}$.
Since $u(X(v)) = X(uv)$, we obtain $u(va) = (uv)a$ for any 
Hence, $a \in \mathbb{R}$ and $a = 0$.
Thus, $X = 0$ and $X_{2} = X'_{2}, X_{3} = X'_{3}$.

\end{proof}

By Theorem \ref{tri-d4}, we set
\[
\tilde{\frak{d}}_{4} = \{ (X_{1}, X_{2}, X_{3}) \in \frak{so}(8)^{3} \ ;\ (X_{1}u)v + u(X_{2}v) = X_{3}(uv) \ (u,v \in \mathbb{O}) \}.
\]
Then, $\tilde{\frak{d}}_{4}$ is isomorphic to $\frak{so}(8)$ by $\tilde{\frak{d}}_{4} \ni (X_{1}, X_{2}, X_{3}) \mapsto X_{1} \in \frak{so}(8)$.
The corresponding automorphisms to $\alpha, \beta, \gamma, \gamma^{2}, \gamma \beta$ are
\[
\begin{split}
& \alpha : \tilde{\frak{d}}_{4} \rightarrow \tilde{\frak{d}}_{4} \ ;\ (X, \gamma(X), \beta(X)) = (X_{1}, X_{2}, X_{3}) \mapsto ( \alpha(X_{1}), X_{3}, X_{2} ) = (\alpha(X), \beta(X), \gamma(X)), \\
& \beta : \tilde{\frak{d}}_{4} \rightarrow \tilde{\frak{d}}_{4} \ ;\ (X, \gamma(X), \beta(X)) = (X_{1}, X_{2}, X_{3}) \mapsto ( X_{3}, \alpha(X_{2}), X_{1} ) = ( \beta(X), \gamma\beta(X), X ), \\
& \gamma : \tilde{\frak{d}}_{4} \rightarrow \tilde{\frak{d}}_{4} \ ;\ (X, \gamma(X), \beta(X)) = (X_{1}, X_{2}, X_{3}) \mapsto ( X_{2}, \alpha(X_{3}), \alpha(X_{1}) ) = (\gamma(X), \alpha \beta(X), \alpha(X)), \\
& \gamma^{2} : \tilde{\frak{d}}_{4} \rightarrow \tilde{\frak{d}}_{4} \ ;\ (X, \gamma(X), \beta(X)) = (X_{1}, X_{2}, X_{3}) \mapsto ( \alpha(X_{3}), X_{1}, \alpha(X_{2}) ) = (\alpha\beta(X), X, \alpha\gamma(X)), \\
& \gamma \beta  : \tilde{\frak{d}}_{4} \rightarrow \tilde{\frak{d}}_{4} \ ;\ (X, \gamma(X), \beta(X)) = (X_{1}, X_{2}, X_{3}) \mapsto ( \alpha(X_{2}), \alpha(X_{1}), \alpha(X_{3})) = ( \alpha \gamma(X), \alpha(X), \alpha\beta(X) ).
\end{split}
\]
Let $X_{1} \in \frak{so}(8)$ be $X_{1}(1) = 0$, that is, $X_{1} \in \frak{so}(7)$.
Then, $X_{1}(1)u + 1(X_{2}(u)) = X_{3}(1u) = X_{3}(u)$.
Therefore, $X_{2} = X_{3}$.
On the other hand, if $X_{2} = X_{3}$, then $X_{1}(1) = X_{3}(1) - X_{2}(1) = 0$ and $X_{1} \in \frak{so}(7)$.
Set a subalgebra $\tilde{\frak{b}}_{3}$ as follows:
\[
\begin{split}
\tilde{\frak{b}}_{3} 
& = \{ (X_{1}, X_{2}, X_{3}) \in \tilde{\frak{d}}_{4} \ ;\ X_{1}(1) = 0 \} = \{ (X_{1}, X_{2}, X_{3}) \in \tilde{\frak{d}}_{4} \ ;\ X_{2} = X_{3} \} \\
& = \{ X = (X_{1}, X_{2}, X_{3}) \in \tilde{\frak{d}}_{4} \ ;\ \alpha(X) = X \}.
\end{split}
\]
Then, $\tilde{\frak{b}}_{3}$ is isomorphic to $\frak{so}(7)$.
Moreover, we assume that $X_{1} \in \frak{g}_{2}$.
Then, $X_{2} = X_{3} = X_{1}$.
Conversely, if $X_{1} = X_{2} = X_{3}$, then $X_{1} \in \frak{g}_{2}$.
Set a subalgebra $\tilde{\frak{g}}_{2}$ as follows:
\[
\begin{split}
\tilde{\frak{g}}_{2} 
&= \{ (X_{1}, X_{2}, X_{3}) \in \tilde{\frak{d}}_{4} \ ;\ X_{1} \in \frak{g}_{2} \} = \{ (X_{1}, X_{2}, X_{3}) \in \tilde{\frak{d}}_{4} \ ;\ X_{1} = X_{2} = X_{3} \} \\
&= \{ X = (X_{1}, X_{2}, X_{3}) \in \tilde{\frak{d}}_{4} \ ;\ \gamma(X) = X \}.
\end{split}
\]
Then, $\tilde{\frak{g}}_{2}$ is isomorphic to $\frak{g}_{2}$.
The Lie algebra $\frak{g}_{2}$ is spanned by
\[
\begin{matrix}
\begin{array}{lll}
x E_{23} + y E_{45} + z E_{67}, & -x E_{13} -y E_{46} + x E_{57}, \\
x E_{12} + y E_{47} + z E_{56}, & -x E_{15} + y E_{26} - z E_{37}, & x,y,z \in \mathbb{R} \\
x E_{14} - y E_{27} - z E_{36}, & -x E_{17} - y E_{24} + z E_{35}, & x + y + z = 0, \\
\end{array} \\
x E_{16} + y E_{25} + z E_{34}
\end{matrix}
\]

\begin{thm} \cite{Yokota-g-r} \label{tri-D4}
For any $g_{1} \in SO(8)$, there exist $g_{2}, g_{3} \in SO(8)$ such that
\[
\Big( g_{1}(u) \Big) \Big( g_{2}(u) \Big) = g_{3} (uv)
\]
for any $u,v \in \mathbb{O}$.
Moreover, such $(g_{2}, g_{3})$ is either $(g_{2}, g_{3})$ or $(-g_{2}, -g_{3})$.
\end{thm}

\begin{proof}
Since $SO(8)$ is connected, for any $x_{1} \in SO(8)$, there exist $X_{1} \in \frak{so}(8)$ such that $x = \mathrm{exp} X_{1}$.
By Theorem \ref{tri-d4}, there exist unique $X_{2}, X_{3} \in \frak{so}(8)$ such that $(X_{1}u)v + u(X_{2}v) = X_{3}(uv)$ for any $u,v \in \mathbb{O}$.
Then,
\[
X_{3}^{n}(uv) = \sum_{i+j = n} \frac{n!}{i! \ j!} \Big( X_{1}^{i} u \Big) \Big( X_{2}^{j} v \Big).
\]
Therefore,
\[
\begin{split}
\mathrm{exp}X_{3}(uv) 
&= \sum_{n=0}^{\infty} \frac{1}{n!} X_{3}^{n}(uv) = \sum_{n=0}^{\infty} \sum_{i+j = n} \frac{1}{i! \ j!} \Big( X_{1}^{i} u \Big) \Big( X_{2}^{j} v \Big) \\
&= \Big( \sum_{i=0}^{\infty} \frac{1}{i!} X_{1}^{i}(u) \Big) \Big( \sum_{j=0}^{\infty} \frac{1}{j!} X_{2}^{j}(u) \Big) = (\mathrm{exp}X_{1}u)(\mathrm{exp}X_{2}v).
\end{split}
\]
Hence, the existence follows.
Next, we consider the uniqueness.
If $x_{1} =1$, then $u(x_{2}(v)) = x_{3}(uv)$.
By consdiering the case of $u = 1$, we obtain $x_{2} = x_{3}$.
Denote $x_{2}$ and $x_{3}$ by $x$.
Set $a = x(1)$.
Since $x \in SO(8)$, we see $|a| = 1$.
By considering the case of $v = 1$, we obatin $x(u) = ua$.
Therefore,
\[
(uv)a = x(uv) = u(x(v)) = u(va)
\]
and we obtain $a \in \mathbb{R}$.
Hence, $a = \pm1$ and $x_{2} = x_{3} = \pm 1$.
Next, we assume that $x_{1}$ is any element of $SO(8)$.
First, remark that
\[
x_{3}^{-1}(uv) = (x_{1}^{-1}u)(x_{2}^{-1}v)
\]
for any $u,v \in \mathbb{O}$.
Let $x'_{2}, x'_{3} \in SO(8)$ satisfy $(x_{1}u)(x'_{2}v) = x'_{3}(uv)$.
Then,
\[
(x'_{3})^{-1} \big( x_{3}(uv) \big) = (x'_{3})^{-1} \big( (x_{1}u)(x_{2}v) \big) = \big( x_{1}^{-1}x_{1}u \big) \big( (x'_{2})^{-1}x_{2}v \big) = u \big( (x'_{2})^{-1}x_{2}v \big).
\]
By the above argument, $(x'_{2}, x'_{3}) = \pm (x_{2}, x_{3})$ and the latter part of the statement follows.
\end{proof}

This property is called the triality principle of $SO(8)$.
Set
\[
\tilde{D}_{4} = \{ (x_{1}, x_{2}, x_{3}) \in SO(8)^{3} \ ;\ (x_{1}u)(x_{2}v) = x_{3}(uv) \ (u,v \in \mathbb{O}) \}.
\]
Then, $\pi_{\mathbb{O}} ; \tilde{D}_{4} \rightarrow SO(8) \ ;\ (x_{1}, x_{2}, x_{3}) \mapsto x_{1}$ is a double covering and $\tilde{D}_{4}$ is isomorphic to $Spin(8)$.
The subgroup of the automorphism group of $Spin(8)$ corresponding to $S_{3}$ is given by
\[
\begin{split}
& \alpha : \tilde{D}_{4} \rightarrow \tilde{D}_{4} \ ;\ (x_{1}, x_{2}, x_{3}) \mapsto ( \alpha(x_{1}), x_{3}, x_{2} ), \\
& \beta : \tilde{D}_{4} \rightarrow \tilde{D}_{4} \ ;\ (x_{1}, x_{2}, x_{3}) \mapsto ( x_{3}, \alpha(x_{2}), x_{1} ), \\
& \gamma : \tilde{D}_{4} \rightarrow \tilde{D}_{4} \ ;\ (x_{1}, x_{2}, x_{3}) \mapsto ( x_{2}, \alpha(x_{3}), \alpha(x_{1}) ), \\
& \gamma^{2} : \tilde{D}_{4} \rightarrow \tilde{D}_{4} \ ;\ (x_{1}, x_{2}, x_{3}) \mapsto ( \alpha(x_{3}), x_{1}, \alpha(x_{2}) ), \\
& \gamma \beta  : \tilde{D}_{4} \rightarrow \tilde{D}_{4} \ ;\ (x_{1}, x_{2}, x_{3}) \mapsto ( \alpha(x_{2}), \alpha(x_{1}), \alpha(x_{3})),
\end{split}
\]
where $\alpha(x)u = \overline{x(\overline{u})}$ for any $u \in \mathbb{O}$.
Asuume that $x_{1} \in SO(7)$.
Then, $x_{1}(1)x_{2}(u) = x_{3}(u)$ for anu $u \in \mathbb{O}$ and $x_{2} = x_{3}$.
On the other hand, if $x_{2} = x_{3}$, then $x_{1}(1)(u) = u$ for any $u \in \mathbb{O}$. 
Therefore, $x_{1}(1) = 1$.
Thus, set
\[
\begin{split}
\tilde{B}_{3} 
&= \{ (x_{1}, x_{2}, x_{3}) \in \tilde{D}_{4} \ ;\ x_{1} \in SO(7) \} = \{ (x_{1}, x_{2}, x_{3}) \in \tilde{D}_{4} \ ;\ x_{2} = x_{3} \} \\
&= \{ x = (x_{1}, x_{2}, x_{3}) \in \tilde{D}_{4} \ ;\ \alpha(x) = x\}.
\end{split}
\]
Then, $\pi_{\mathbb{O}}|_{\tilde{B}_{3}}$ is a double covering from $\tilde{B}_{3}$ onto $SO(7)$.
Therefore, $\tilde{B}_{3}$ is isomorphic to $Spin(7)$.
Moreover, we assume that $x_{1} \in G_{2}$.
Then, $(x_{2}, x_{3}) = (x_{1}, x_{1})$ or $(x_{2}, x_{3}) = (-x_{1}, -x_{1})$.
Set
\[
\begin{split}
\tilde{G}_{2} 
&= \{ (x_{1}, x_{2}, x_{3}) \in \tilde{D}_{4} \ ;\ x_{1} \in G_{2} \} = \{ (x_{1}, x_{2}, x_{3}) \ ;\ x_{1} = x_{2} = x_{3} \} \\
&= \{ x = (x_{1}, x_{2}, x_{3}) \in \tilde{D}_{4} \ ;\ \gamma(x) = x \}.
\end{split}
\]
Then, $\pi_{\mathbb{O}}|_{\tilde{G}_{2}}$ is an isomorphism between $\tilde{G}_{2}$ and $G_{2}$.
The Lie algebar of $\tilde{D}_{4}, \tilde{B}_{3}, \tilde{G}_{2}$ are $\tilde{\frak{d}}_{4}, \tilde{\frak{b}}_{3}, \tilde{\frak{g}_{2}}$ respectively.
For any $g \in G_{2}$, set
\[
g^{0,0} = (g,g,g), \ g^{1,0} = (g,-g,-g), \ g^{0,1} = (-g,-g,g), \ g^{11} = (-g,g,-g) \in \tilde{D}_{4}.
\]
The centers $C(Spin(8), C(Spin(7))$ of $Spin(8), Spin(7)$ are
\[
\begin{split}
& C(Spin(8)) = \{ 1_{8}^{0,0}, 1_{8}^{1,0}, 1_{8}^{0,1}, 1_{8}^{1,1} \} \cong \mathbb{Z}_{2} \times \mathbb{Z}_{2}, \\
& C(Spin(7)) = \{ 1_{8}^{0,0}, 1_{8}^{1,0} \} \cong \mathbb{Z}_{2}.
\end{split}
\]

Define the isomorphism $\lambda : Spin(8) \rightarrow \tilde{D}_{4}$ such that the induced isomorphim $\lambda : \frak{so}(8) \rightarrow \frak{d}_{4}$ satisfies $\lambda = \pi^{-1} \circ \pi_{\mathbb{O}}$, where $Spin(8)$ is defined by the Clliford algebra over $\mathbb{R}e_{0} + \cdots + \mathbb{R}e_{7}$.
We identify $Spin(8)$ with $\tilde{D}_{4}$ under the isomorphism $\lambda$ and often omit $\lambda$. 
Moreover, $\tilde{D}_{4}, \frak{d}_{4}, \pi_{\mathbb{O}}$ are denoted by $Spin(8), \frak{spin}(8), \pi$.
The subgroup $\tilde{B}_{3}$ and the subalgbera $\tilde{\frak{b}}_{3}$ are also denoted by $Spin(7)$ and $\frak{spin}(7)$.
The subgroup $\tilde{G}_{2}$ and the subalgebra $\tilde{\frak{g}}_{2}$ is denoted by $G_{2}$ and $\frak{g}_{2}$.
Under the identification by $\lambda$,
\[
\begin{split}
1 = 1_{8}^{0,0}, \quad -1 = 1_{8}^{1,0} \quad e_{0}e_{1} \cdots e_{7} = 1_{8}^{0,1}, \quad -e_{0}e_{1} \cdots e_{7} = 1_{8}^{1,1}.
\end{split}
\]
For any $\theta_{1}, \cdots, \theta_{4} \in \mathbb{R}$, we set
\[
\begin{split}
& \tilde{T}_{8}(\theta_{1}, \cdots, \theta_{4}) \\
&=
\Big(
T_{8}'(\theta_{1}, \theta_{2}, \theta_{3}, \theta_{4}), \\
& \quad\quad T'_{8} \left( \frac{-\theta_{1} + \theta_{2} + \theta_{3} + \theta_{4}}{2}, \frac{-\theta_{1} + \theta_{2} - \theta_{3} - \theta_{4}}{2}, \frac{-\theta_{1} - \theta_{2} + \theta_{3} - \theta_{4}}{2}, \frac{-\theta_{1} - \theta_{2} - \theta_{3} + \theta_{4}}{2} \right), \\
& \quad\quad T'_{8} \left( \frac{+\theta_{1} + \theta_{2} + \theta_{3} + \theta_{4}}{2}, \frac{+\theta_{1} + \theta_{2} - \theta_{3} - \theta_{4}}{2}, \frac{+\theta_{1} - \theta_{2} + \theta_{3} - \theta_{4}}{2}, \frac{+\theta_{1} - \theta_{2} - \theta_{3} + \theta_{4}}{2} \right)
\Big) \\
\end{split}
\]
and $\tilde{T}_{8} := \{ \tilde{T}_{8}(\theta_{1}, \cdots, \theta_{4}) \ ;\ \theta_{1}, \cdots, \theta_{4} \in \mathbb{R} \}$.
Then, $\tilde{T}_{8}$ is a maximal torus of $\tilde{D}_{4}$.
Moreover, $\pi(T_{8}(\theta_{1}, \cdots, \theta_{4})) = \pi_{\mathbb{O}}(\tilde{T}(\theta_{1}, \cdots, \theta_{4})) = T'_{8}(\theta_{1}, \cdots, \theta_{4})$ and $\pi(T_{8}) = \pi_{\mathbb{O}}(\tilde{T}_{8})$.
Hence, $\lambda(T_{8}(\theta_{1}, \cdots, \theta_{4})) = \tilde{T}_{8}(\theta_{1}, \cdots, \theta_{4})$.
In the following, we identify $T_{8}(\theta_{1}, \cdots, \theta_{4})$ with $\tilde{T}_{8}(\theta_{1}, \cdots, \theta_{4})$.

Next, we consider the spin representation of $Spin(8)$.
Set real representations $\tilde{\Delta}_{8}^{\pm}$ of $Spin(8)$ as follows:
For any $g = (g_{1}, g_{2}, g_{3}) \in Spin(8)$,
\[
\tilde{\Delta}_{8}^{+}(g) : \mathbb{O} \rightarrow \mathbb{O} \ ;\ u \mapsto g_{3}(u), \quad\quad
\tilde{\Delta}_{8}^{-}(g) : \mathbb{O} \rightarrow \mathbb{O} \ ;\ u \mapsto g_{2}(u).
\]
The complexification of $\tilde{\Delta}_{8}^{\pm}$ is denoted by $(\tilde{\Delta}_{8}^{\pm})^{\mathbb{C}}$.
The set of all weights of $(\tilde{\Delta}^{\pm}_{8})^{\mathbb{C}}$ with respect to $\frak{t}^{\mathbb{C}}$ is denoted by $\tilde{\Pi}^{\pm}_{8}$.
Then,
\[
\begin{split}
\tilde{\Pi}_{8}^{\pm} &= \left\{ \frac{1}{2} \Big( \epsilon_{1}x_{1} + \cdots + \epsilon_{4}x_{4} \Big) \ ;\ \epsilon_{i} = \pm 1 \ (i = 1, \cdots, 4), \epsilon_{1} \cdots \epsilon_{4} = \pm 1 \right\}. \\
\end{split}
\]
The weight space of $(\tilde{\Delta}_{8}^{\pm})^{\mathbb{C}}$ with $\gamma \in \tilde{\Pi}_{8}^{\pm}$ is denoted by $(\tilde{V}_{8}^{\pm})^{\mathbb{C}}_{\gamma}$.
For each element of $\tilde{\Pi}_{8}^{+}$, 
\[
\begin{array}{lllllll}
(\tilde{V}_{8}^{+})_{\frac{1}{2}(x_{1} + x_{2} + x_{3} + x_{4})}^{\mathbb{C}} = \mathbb{C}(e_{0} - ie_{1}), & &
(\tilde{V}_{8}^{+})_{\frac{1}{2}(-x_{1} - x_{2} - x_{3} - x_{4})}^{\mathbb{C}} = \mathbb{C}(e_{0} + ie_{1}), \\
(\tilde{V}_{8}^{+})_{\frac{1}{2}(x_{1} + x_{2} - x_{3} - x_{4})}^{\mathbb{C}} = \mathbb{C}(e_{2} - ie_{3}), & &
(\tilde{V}_{8}^{+})_{\frac{1}{2}(-x_{1} - x_{2} + x_{3} + x_{4})}^{\mathbb{C}} = \mathbb{C}(e_{2} + ie_{3}), \\
(\tilde{V}_{8}^{+})_{\frac{1}{2}(x_{1} - x_{2} + x_{3} - x_{4})}^{\mathbb{C}} = \mathbb{C}(e_{4} - ie_{5}), & &
(\tilde{V}_{8}^{+})_{\frac{1}{2}(-x_{1} + x_{2} - x_{3} + x_{4})}^{\mathbb{C}} = \mathbb{C}(e_{4} + ie_{5}), \\ 
(\tilde{V}_{8}^{+})_{\frac{1}{2}(x_{1} - x_{2} - x_{3} + x_{4})}^{\mathbb{C}} = \mathbb{C}(e_{6} - ie_{7}), & &
(\tilde{V}_{8}^{+})_{\frac{1}{2}(-x_{1} + x_{2} + x_{3} - x_{4})}^{\mathbb{C}} = \mathbb{C}(e_{6} + ie_{7}). \\
\end{array}
\]
For each element of $\tilde{\Pi}_{8}^{+}$, 
\[
\begin{array}{lllllll}
(\tilde{V}_{8}^{-})_{\frac{1}{2}(-x_{1} + x_{2} + x_{3} + x_{4})}^{\mathbb{C}} = \mathbb{C}(e_{0} - ie_{1}), & &
(\tilde{V}_{8}^{-})_{\frac{1}{2}(x_{1} - x_{2} - x_{3} - x_{4})}^{\mathbb{C}} = \mathbb{C}(e_{0} + ie_{1}), \\
(\tilde{V}_{8}^{-})_{\frac{1}{2}(-x_{1} + x_{2} - x_{3} - x_{4})}^{\mathbb{C}} = \mathbb{C}(e_{2} - ie_{3}), & &
(\tilde{V}_{8}^{-})_{\frac{1}{2}(x_{1} - x_{2} + x_{3} + x_{4})}^{\mathbb{C}} = \mathbb{C}(e_{2} + ie_{3}), \\
(\tilde{V}_{8}^{-})_{\frac{1}{2}(-x_{1} - x_{2} + x_{3} - x_{4})}^{\mathbb{C}} = \mathbb{C}(e_{4} - ie_{5}), & &
(\tilde{V}_{8}^{-})_{\frac{1}{2}(x_{1} + x_{2} - x_{3} + x_{4})}^{\mathbb{C}} = \mathbb{C}(e_{4} + ie_{5}), \\
(\tilde{V}_{8}^{-})_{\frac{1}{2}(-x_{1} - x_{2} - x_{3} + x_{4})}^{\mathbb{C}} = \mathbb{C}(e_{6} - ie_{7}), & &
(\tilde{V}_{8}^{-})_{\frac{1}{2}(x_{1} + x_{2} + x_{3} - x_{4})}^{\mathbb{C}} = \mathbb{C}(e_{6} + ie_{7}). \\
\end{array}
\]
Hence, $\tilde{\Delta}_{8}^{\pm}$ and $(\tilde{\Delta}_{8}^{\pm})^{\mathbb{C}}$ are isomorphic to the spin representations $\Delta_{8}^{\pm}$ and $(\Delta_{8}^{+})^{\mathbb{C}}$ as representations respectively.
In the following, we identify $\tilde{\Delta}_{8}^{\pm}$ and $(\tilde{\Delta}_{8}^{\pm})^{\mathbb{R}}$ with $\Delta_{8}^{\pm}$ and $(\Delta_{8}^{\pm})^{\mathbb{R}}$ respectively.
Then, $V_{8}^{\pm}$ and $(V_{8}^{\pm})^{\mathbb{C}}$ are $\mathbb{O}$ and $\mathbb{O}^{\mathbb{C}}$.
When we consider $\mathbb{O}$ and $\mathbb{O}^{\mathbb{C}}$ as the representation spaces of $(\Delta_{8}^{\pm})^{\mathbb{R}}$ and $\Delta_{8}^{\pm}$, then we denote $\mathbb{O}$ and $\mathbb{O}^{\mathbb{C}}$ by $\mathbb{O}^{\pm}$ and $(\mathbb{O}^{\mathbb{C}})^{\pm}$.

Set elements of $G_{2} \subset SO(8)$ as follows:
\[
\begin{split}
\gamma_{0} := \mathrm{id}, \hspace{43mm} &
\gamma_{1} := \mathrm{Id}_{\langle e_{0}, e_{1}, e_{2}, e_{3} \rangle} - \mathrm{id}_{\langle e_{4}, e_{5}, e_{6}, e_{7} \rangle}, \\
\gamma_{2} := \mathrm{Id}_{\langle e_{0}, e_{1}, e_{4}, e_{5} \rangle} - \mathrm{id}_{\langle e_{2}, e_{3}, e_{6}, e_{7} \rangle}, \quad &
\gamma_{3} := \mathrm{Id}_{\langle e_{0}, e_{1}, e_{6}, e_{7} \rangle} - \mathrm{id}_{\langle e_{2}, e_{3}, e_{4}, e_{5} \rangle}, \\
\gamma_{4} := \mathrm{Id}_{\langle e_{0}, e_{2}, e_{4}, e_{6} \rangle} - \mathrm{id}_{\langle e_{1}, e_{3}, e_{5}, e_{7} \rangle}, \quad &
\gamma_{5} := \mathrm{Id}_{\langle e_{0}, e_{2}, e_{5}, e_{7} \rangle} - \mathrm{id}_{\langle e_{1}, e_{3}, e_{4}, e_{6} \rangle}, \\
\gamma_{6} := \mathrm{Id}_{\langle e_{0}, e_{3}, e_{4}, e_{7} \rangle} - \mathrm{id}_{\langle e_{1}, e_{2}, e_{5}, e_{6} \rangle}, \quad &
\gamma_{7} := \mathrm{Id}_{\langle e_{0}, e_{3}, e_{5}, e_{6} \rangle} - \mathrm{id}_{\langle e_{1}, e_{2}, e_{4}, e_{7} \rangle}, \\
\end{split}
\]
where $\langle v_{1}, \cdots, v_{4} \rangle \ (v_{1}, \cdots, v_{4} \in \mathbb{O})$ is the subspace of $\mathbb{O}$ spanned by $v_{1}, \cdots, v_{4}$.
Set $A(G_{2}) := \{ \gamma_{i} \ ;\ 0 \leq i \leq 7\}$.
Then, $\#A(G_{2}) = 8$.

\begin{thm}\cite{Tanaka-Tasaki-Yasukura}
$A(G_{2})$ is a maximal antipodal subgroup of $G_{2}$.
Moreover, any maximal antipodal subgroup of $G_{2}$is conjugate to $A(G_{2})$ and $\#_{2}G_{2} = 8$.

\end{thm}

We define $A(Spin(8))$ as follows:
\[
A(Spin(8)) := \{ \gamma_{i}^{a,b} \in Spin(8) \ ;\ 0 \leq i \leq 7, a,b = 0,1 \}.
\]
Then, $\#A(Spin(8)) = 32$.

\begin{thm} \cite{Wood} 
$A(Spin(8))$ is a maximal antipodal subgroup of $Spin(8)$.
Any maximal antipodal subgroup of $Spin(8)$ is conjugate to $A(Spin(8))$.
Moreover, $\#_{2}Spin(8) = 32$.
\end{thm}

In $Spin(8)$, the number of polars except for the trivial pole is $4$.
The all poles of the identity element is given by $\{ \gamma_{0}^{a,b} \ ;\ a,b = 0,1 \}$ and the other polar of the identity element is isomorphic to $\tilde{G}_{4}(\mathbb{R}^{8})$ denoted by $(\tilde{G}_{4}(\mathbb{R}^{8})_{+}$.
Set
\[
A((\tilde{G}_{4}(\mathbb{R}^{8})_{+}) = A(Spin(8)) \cap (\tilde{G}_{4}(\mathbb{R}^{8})_{+} = \{ \gamma_{i}^{a,b} \ ;\ 1 \leq i \leq 7, a,b = 0,1 \}.
\]
Note $\#A((\tilde{G}_{4}(\mathbb{R}^{12})^{+}) = 28$.

\begin{thm} \cite{Tasaki} 
$A(\tilde{G}_{4}(\mathbb{R}^{8})$ is a maximal antipodal set of $M^{+}(Spin(8))$.
Any maximal antipodal set of $M^{+}(Spin(8))$ is congruent to $A(\tilde{G}_{4}(\mathbb{R}^{8})$.
Moreover, $\#_{2}\tilde{G}_{4}(\mathbb{R}^{8}) = 28$.

\end{thm}

We can easily check that 
\[
\begin{array}{llllllllllllll}\vspace{2mm}
\gamma_{0}^{0,0} = 1, & \gamma_{0}^{1,0} = -1, & \gamma_{0}^{0,1} = e_{0} \cdots e_{7}, & \gamma_{0}^{1,1} = -e_{0} \cdots e_{7}, \\\vspace{2mm}
\gamma_{1}^{0,0} = -e_{4}e_{5}e_{6}e_{7}, & \gamma_{1}^{1,0} = e_{4}e_{5}e_{6}e_{7}, & \gamma_{1}^{0,1} = -e_{0}e_{1}e_{2}e_{3}, & \gamma_{1}^{1,1} = e_{0}e_{1}e_{2}e_{3}, \\\vspace{2mm}
\gamma_{2}^{0,0} = -e_{2}e_{3}e_{6}e_{7}, & \gamma_{2}^{1,0} = e_{2}e_{3}e_{6}e_{7}, & \gamma_{2}^{0,1} = -e_{0}e_{1}e_{4}e_{5}, & \gamma_{2}^{1,1} = e_{0}e_{1}e_{4}e_{5}, \\\vspace{2mm}
\gamma_{3}^{0,0} = -e_{2}e_{3}e_{4}e_{5}, & \gamma_{3}^{1,0} = e_{2}e_{3}e_{4}e_{5}, & \gamma_{3}^{0,1} = -e_{0}e_{1}e_{6}e_{7}, & \gamma_{3}^{1,1} = e_{0}e_{1}e_{6}e_{7}, \\\vspace{2mm}
\gamma_{4}^{0,0} = e_{1}e_{3}e_{5}e_{7}, & \gamma_{4}^{1,0} = -e_{1}e_{3}e_{5}e_{7}, & \gamma_{4}^{0,1} = e_{0}e_{2}e_{4}e_{6}, & \gamma_{4}^{1,1} = -e_{0}e_{2}e_{4}e_{6}, \\\vspace{2mm}
\gamma_{5}^{0,0} = -e_{1}e_{3}e_{4}e_{6}, & \gamma_{5}^{1,0} = e_{1}e_{3}e_{4}e_{6}, & \gamma_{5}^{0,1} = -e_{0}e_{2}e_{5}e_{7}, & \gamma_{5}^{0,1} = e_{0}e_{2}e_{5}e_{7}, \\\vspace{2mm}
\gamma_{6}^{0,0} = -e_{1}e_{2}e_{5}e_{6}, & \gamma_{6}^{1,0} = e_{1}e_{2}e_{5}e_{6}, & \gamma_{6}^{0,1} = -e_{0}e_{3}e_{4}e_{7}, & \gamma_{6}^{1,1} = e_{0}e_{3}e_{4}e_{7}, \\\vspace{2mm}
\gamma_{7}^{0,0} = -e_{1}e_{2}e_{4}e_{7}, & \gamma_{7}^{1,0} = e_{1}e_{2}e_{4}e_{7}, & \gamma_{7}^{0,1} = -e_{0}e_{3}e_{5}e_{6}, & \gamma_{7}^{1,1} = e_{0}e_{3}e_{5}e_{6}. \\
\end{array}
\]
We see that $T \cap A(Spin(8)) =  \{ \gamma_{i}^{a,b} \ ;\ 0 \leq i \leq 3, a,b = 0,1 \}$.
For each element of $T \cap A(Spin(8))$,
\[
\begin{array}{lllllllllll} \vspace{2mm}
\mathrm{exp}\,\pi(iA_{x_{1} - x_{2}}) = \gamma_{1}^{0,1}, & \mathrm{exp}\,\pi(iA_{x_{1} + x_{2}}) = \gamma_{1}^{1,1}, & \mathrm{exp}\,\pi(iA_{x_{3} - x_{4}}) = \gamma_{1}^{0,0}, & \mathrm{exp}\,\pi(iA_{x_{3} + x_{4}}) = \gamma_{1}^{1,0}, \\ \vspace{2mm}
\mathrm{exp}\,\pi(iA_{x_{1} - x_{3}}) = \gamma_{2}^{0,1}, & \mathrm{exp}\,\pi(iA_{x_{1} + x_{3}}) = \gamma_{2}^{1,1}, & \mathrm{exp}\,\pi(iA_{x_{2} - x_{4}}) = \gamma_{2}^{0,0}, & \mathrm{exp}\,\pi(iA_{x_{2} + x_{4}}) = \gamma_{2}^{1,0}, \\ \vspace{2mm}
\mathrm{exp}\,\pi(iA_{x_{1} - x_{4}}) = \gamma_{3}^{0,1}, & \mathrm{exp}\,\pi(iA_{x_{1} + x_{4}}) = \gamma_{3}^{1,1}, & \mathrm{exp}\,\pi(iA_{x_{2} - x_{3}}) = \gamma_{3}^{0,0}, & \mathrm{exp}\,\pi(iA_{x_{2} + x_{3}}) = \gamma_{3}^{1,0}. \\
\end{array}
\]






\newpage

\subsection{Exceptional Lie groups and involutive automorphisms}\label{ELg}

In this subsection, we recall the construction of the simply connected compact exceptional simple Lie groups $E_{8}, E_{7}, E_{6}$.
First, we summarize some propeties of $E_{8}, E_{7}, E_{6}$, and $F_{4}$. 
It is well known that $E_{8}, E_{7}, E_{6}, F_{4}$ are simply connected.
The dimensions and cetnters of each groups are follows:
\[
\begin{array}{|c|c|c|c|c|c|cccccc} \hline
G & E_{8} & E_{7} & E_{6} & F_{4} & G_{2} \\ \hline
\dim G & 248 & 133 & 78 & 52 & 14 \\ \hline
C(G) & \{ e \} & \mathbb{Z}_{2} & \mathbb{Z}_{3} & \{ e \} & \{ e \} \\ \hline
\end{array}
\]
Let $G$ be one of $E_{8}, E_{7}, E_{6}, F_{4}$ and $f$ be an involutive automorphism of $G$.
Then, $F(f,G)$ is connected and isomorphic to one of the following:
\[
\begin{array}{|c|c|c|c|c|c|c|cccccccccccc} \hline
\mathrm{type} & G & F(f, G) & \dim F(f,G) & \dim G/F(f,G) \\ \hline
FI & F_{4} & Sp(1) \cdot Sp(3) & 24 & 28 \\
FII & F_{4} & Spin(9) & 36 & 16 \\ \hline
EI & E_{6} & Sp(4)/\mathbb{Z}_{2} & 36 & 42 \\
EII & E_{6} & Sp(1) \cdot SU(6) & 38 & 40 \\
EIII & E_{6} & (U(1) \times Spin(10))/\mathbb{Z}_{4} & 46 & 32 \\
EIV & E_{6} & F_{4} & 52 & 26 \\ \hline
EV & E_{7} & SU(8)/\mathbb{Z}_{2} & 63 & 70 \\
EVI & E_{7} & Sp(1) \cdot Spin(12) & 69 & 64 \\
EVII & E_{7} & (U(1) \times E_{6})/\mathbb{Z}_{3} & 79 & 54 \\ \hline
EVIII & E_{8} & Sp(1) \cdot E_{7} & 120 & 128 \\
EIX & E_{8} & Ss(16) & 136 & 112 \\ \hline
\end{array}
\]
By this classification, for any involutive automorphism $f$ of $G$, we can decide the isomorphism class of $F(f,G)$ by caluculating $\dim F^{+}(f, \frak{g})$.
For example, let $f$ be an involutive autmorphism of $F_{4}$.
If $\dim F^{+}(f, \frak{g}) = 36$, then $f$ is type $FI$ and $F(f, G) \cong Spin(9)$.
In the present paper, the simply connected compact symmetric space of some type is denoted by the notation of the type,
For example, the simply connected compact symmetric space of type $FI$ is denoted by $FI$.
In the following, we will explicitly describe involutive automorphims of $E_{8}, E_{7}, E_{6}, F_{4}$.


\subsubsection{Lie group $E_{8}$ and Lie algebra $\frak{e}_{8}$} \label{E8}

First, we recall some properties of $Spin(16)$.
Assume that $Spin(16)$ is defined by the Clliford algebra $Cl(\sum_{i=0}^{15}\mathbb{R}e_{i})$.
The spinor groups defined by the Clliford algebras $Cl(\sum_{i=0}^{7}\mathbb{R}e_{i})$ and $Cl(\sum_{i=8}^{15}\mathbb{R}e_{i})$ are denoted by $Spin^{0}(8)$ and $Spin^{1}(8)$ respectively.
The Lie algbera of $Spin^{i}(8) \ (i = 0,1)$ is denoted by $\frak{spin}^{i}(8)$.
The polar of $Spin^{i}(8)$ isomorphic to $\tilde{G}_{4}(\mathbb{R}^{8})$ is denoted by $\tilde{G}_{4}(\mathbb{R}^{8})_{+}^{i}$.
Set $A(\tilde{G}_{4}(\mathbb{R}^{8}))_{+}^{i}) = A(Spin^{i}(8)) \cap (\tilde{G}_{4}(\mathbb{R}^{8}))_{+}^{i}$.
For $i = 1,2$, set
\[
\begin{split}
& T^{i}(\theta_{1}, \theta_{2}, \theta_{3}, \theta_{4}) = \Big( \cos \frac{\theta_{1}}{2} + \sin \frac{\theta_{1}}{2} e_{8i}e_{8i+1} \Big) \Big( \cos \frac{\theta_{1}}{2} + \sin \frac{\theta_{1}}{2} e_{8i+2}e_{8i+3} \Big)  \\
& \hspace{60mm} \Big( \cos \frac{\theta_{1}}{2} + \sin \frac{\theta_{1}}{2} e_{8i+4}e_{8i+5} \Big) \Big( \cos \frac{\theta_{1}}{2} + \sin \frac{\theta_{1}}{2} e_{8i+6}e_{8i+7} \Big). \\
\end{split}
\] 
Then, $T^{i}_{8} = \{ T^{i}(\theta_{1}, \theta_{2}, \theta_{3}, \theta_{4}) \ ;\ \theta_{i} \in \mathbb{R} \ (i = 1,\cdots,4) \}$ is a maximal torus of $Spin^{i}(8)$.
Set
\[
\frak{t}_{8}^{i} = \left\{ t(c_{1}, \cdots, c_{4}) = \sum_{k=1}^{4}\frac{c_{k}}{2} e_{2k-2}e_{k-1} \ ;\ c_{k} \in \mathbb{R} \right\} \quad \ (i = 0,1).
\]
Then, $\frak{t}_{8}^{i}$ is a Lie algebra of $T_{8}^{i}$.
We recall the homomorphism
\[
i := i_{8,8} : Spin^{0}(8) \times Spin^{1}(8) \rightarrow Spin(16)\ ;\ (g,h) \mapsto gh
\]
and the image is denoted by $Spin^{0}(8) \cdot Spin^{1}(8)$,
Set
\[
T_{16}(\theta_{1}, \cdots,\theta_{8}) = i( T^{1}_{8}(\theta_{1}, \cdots, \theta_{4}), T^{2}_{8}(\theta_{5}, \cdots, \theta_{8}))
\]
for any $\theta_{i} \in \mathbb{R} \ (1 \leq i \leq 8)$ and $T_{16} = i(T^{1}_{8}, T^{2}_{8})$ 
Then, $T_{16}$ is a maximal antipodal set of $Spin(16)$.
The Lie algebra $\frak{t}_{16}$ of $T_{16}$ is equal to $\frak{t}_{8}^{0} + \frak{t}_{8}^{1}$.

Next, we consider the spinor representation of $Spin(16)$.
The spinor representation $(\Delta_{16}^{\mathbb{C}})^{\mathbb{C}}$ and $\Delta_{16}^{+}$ are denoted by $\Delta^{\mathbb{C}}$ and $\Delta$.
Moreover, the representation space $\Delta^{\mathbb{C}}$ and $\Delta$ are denoted by $V^{\mathbb{C}}$ and $V$.
The restriction of the spin representations $\Delta^{\mathbb{C}}$ and $\Delta$ to $Spin^{0}(8) \cdot Spin^{1}(8)$ are isomorphic to the direct sum decompositions 
\[
\begin{split}
& \Delta^{\mathbb{C}} = (\Delta_{8}^{+})^{\mathbb{C}} \otimes (\Delta_{8}^{+})^{\mathbb{C}} + (\Delta_{8}^{-})^{\mathbb{C}} \otimes (\Delta_{8}^{-})^{\mathbb{C}}, \quad\quad \Delta = (\Delta_{8}^{+}) \otimes (\Delta_{8}^{+}) + (\Delta_{8}^{-}) \otimes (\Delta_{8}^{-}) \\
\end{split}
\]
respectively.
The representation spaces $(\mathbb{O}^{+})^{\mathbb{C}} \otimes (\mathbb{O}^{+})^{\mathbb{C}}$ and $\mathbb{O}^{+} \otimes \mathbb{O}^{+}$ of $(\Delta_{8}^{+})^{\mathbb{C}} \otimes (\Delta_{8}^{+})^{\mathbb{C}}$ and $\Delta_{8}^{+} \otimes \Delta_{8}^{+}$ are denoted by $\mathbb{O}^{\mathbb{C}} \otimes^{+} \mathbb{O}^{\mathbb{C}}$ and $\mathbb{O} \otimes^{+}\mathbb{O}$ respectively.
Each element of $\mathbb{O}^{\mathbb{C}} \otimes^{+} \mathbb{O}^{\mathbb{C}}$ is denoted by $u \otimes^{+} v \ (u,v \in \mathbb{O}^{\mathbb{C}})$.
We also use the similar notations in the case of $(\Delta_{8}^{-})^{\mathbb{C}} \otimes (\Delta_{8}^{-})^{\mathbb{C}}$.
Then,
\[
V^{\mathbb{C}} = \mathbb{O}^{\mathbb{C}} \otimes^{+} \mathbb{O}^{\mathbb{C}} + \mathbb{O}^{\mathbb{C}} \otimes^{-} \mathbb{O}^{\mathbb{C}}, \quad
V = \mathbb{O} \otimes^{+} \mathbb{O} + \mathbb{O} \otimes^{-} \mathbb{O}.
\]
For any $g_{i} \in Spin^{i}(8) \ (i = 0,1)$ and $u,v,w,z \in \mathbb{O}^{\mathbb{C}}$,
\[
\Delta^{\mathbb{C}}(g_{0}, g_{1})(u \otimes^{+} v + w \otimes^{-} z)
= (\Delta_{8}^{+})^{\mathbb{C}}(g_{0})(u) \otimes^{+} (\Delta_{8}^{+})^{\mathbb{C}}(g_{1})(v) + (\Delta_{8}^{-})^{\mathbb{C}}(g_{0})(w) \otimes^{-} (\Delta_{8}^{-})^{\mathbb{C}}(g_{1})(z).
\]
For calculating $\Delta_{8}^{\pm}$, it is useful to consider that $Spin^{i}(8)$ as $\tilde{D}_{4}$.
Hence, in the followng, we often idenitify $Spin^{i}(8)$ as $\tilde{D}_{4}$. 
The set of all weights of $\Delta^{\mathbb{C}}$ is given by
\[
\Pi_{16} = \left\{ \frac{1}{2} \Big( \epsilon_{1}x_{1} + \cdots + \epsilon_{8}x_{8} \Big) \ ;\ \epsilon_{i} = \pm1 \ (i = 1, \cdots, 8), \ \epsilon_{1} \cdots \epsilon_{8} = 1 \right\}.
\]
We take some linear order of $\frak{t}_{16}$ such that the set $\Pi_{16}^{+}$ of all positive weights is
\[
\Pi_{16}^{+} = \left\{ \frac{1}{2} \Big( \epsilon_{1}x_{1} + \cdots + \epsilon_{8}x_{8} \Big) \ ;\ \epsilon_{i} = \pm1 \ (i = 1, \cdots, 8), \ \epsilon_{1} \cdots \epsilon_{8} = 1,\  \epsilon_{1} = 1 \right\}.
\]
For any $\gamma \in \Pi_{16}$, the weight space $V_{\gamma}^{\mathbb{C}}$ is given by 
\[
V_{\gamma}^{\mathbb{C}} = \mathbb{C} \Big( (e_{2i} \pm ie_{2i+1}) \otimes^{\pm} (e_{2j} \pm ie_{2j+1}) \Big)
\]
for some $0 \leq i, j \leq 3$.
For example, if $\gamma = (++++++++)$, then
\[
V_{\gamma}^{\mathbb{C}} = \mathbb{C} \Big( (e_{0} - ie_{1}) \otimes^{+} (e_{0} - ie_{1}) \Big) = \mathbb{C} \Big( (e_{0} \otimes^{+} e_{0} - e_{1} \otimes^{+} e_{1}) - i ( e_{0} \otimes^{+} e_{1} + e_{1} \otimes^{+} e_{0}) \Big).
\]
Moreover, if $\gamma = (--------)$, then
\[
V^{\mathbb{C}} = \mathbb{C} \Big( (e_{0} + ie_{1}) \otimes^{+} (e_{0} + ie_{1}) \Big) = \mathbb{C} \Big( (e_{0} \otimes^{+} e_{0} - e_{1} \otimes^{+} e_{1}) + i ( e_{0} \otimes^{+} e_{1} + e_{1} \otimes^{+} e_{0}) \Big).
\]
Note that $\overline{V_{\gamma}^{\mathbb{C}}} = V_{-\gamma}^{\mathbb{C}}$ for any $\gamma \in \Pi_{16}$, where $ \overline{ \ \cdot \ }$ is the complex conjugation of $V^{\mathbb{C}}$.
We set
\[
i(\gamma_{4}^{0,1}, \gamma_{4}^{0,1}) = i(e_{0}e_{2}e_{4}e_{6}, e_{8}e_{10}e_{12}e_{14}) = e_{0}e_{2} \cdots e_{12}e_{14} \in Spin(16).
\]
Then, $x = i(\gamma_{4}^{0,1}, \gamma_{4}^{0,1}) \in Spin(16)$, where $x$ is defined in Subsection \ref{Sg}.
We see that $\Delta(x)$ satisfies
\[
\Delta(x) \Big( V \cap \big( V_{\gamma}^{\mathbb{C}} + V_{-\gamma}^{\mathbb{C}} \big) \Big) \subset V \cap \big( V_{\gamma}^{\mathbb{C}} + V_{-\gamma}^{\mathbb{C}} \big).
\]
For each $\gamma \in \Pi^{+}_{16}$, we set
\[
\begin{split}
& V_{\gamma}^{\pm} = \{ X \in V_{16}^{+} \cap \big( (V_{16}^{+})_{\gamma}^{\mathbb{C}} + (V_{16}^{+})_{-\gamma}^{\mathbb{C}} \big) \ ;\ \Delta_{16}^{+}(x)(X) = \pm X \}. \\
\end{split}
\]
Then, for any $\gamma \in \Pi^{+}_{16}$, there exsit some $1 \leq k,j \leq 3$ such that
\[
\begin{split}
& (V_{16}^{+})_{\gamma}^{+} = \mathbb{R} \Big( (e_{2j} \otimes^{\pm} e_{2k}) \pm (e_{2j+1} \otimes^{\pm} e_{2k+1}) \Big), \quad\quad
 (V_{16}^{+})_{\gamma}^{-} = \mathbb{R} \Big( (e_{2j} \otimes^{\pm} e_{2k+1}) \pm (e_{2j} \otimes^{\pm} e_{2k+1}) \Big). \\
\end{split}
\]
For example, if $\gamma = (++++++++)$, then
\[
V_{\gamma}^{+} = \mathbb{R}(e_{0} \otimes^{+} e_{0} - e_{1} \otimes e_{1}), \quad\quad
V_{\gamma}^{-} = \mathbb{R}(e_{0} \otimes^{+} e_{1} + e_{1} \otimes e_{0}).
\]


In the following, we denote $\frak{spin}(16)$ by $L$ and $V$.
First, set an $Spin(16)$-invariant inner product $( \ ,\ )_{V}$ of $V$.
Moreover, we set $Spin(16)$-invariant inner product $(\ ,\ )_{L}$ of $L$ such that
\[
(e_{r}e_{s}, e_{t}e_{u}) = \delta_{rt}\delta_{su} \quad (r < s, t < u).
\]
We define the skew-symmetric bilinear map $[\ ,\ ] : (L + V) \times (L + V) \rightarrow L + V$ as follows:

\begin{itemize}

\item
For any $A, B \in L$, set $[A,B] = [A,B]_{L}$, where $[\ ,\ ]_{L}$ is a bracket of $L = \frak{spin}(n)$.

\item
For any $A \in L$ and $u \in V$, set $[A, u] = \Delta(A)(u)$ and $[u, A] = -[A, u]$.

\item
For any $u,v \in V$, set $[u,v] \in L$ such that $(A, [u,v])_{L} = (\Delta(A)(u), v)_{L}$ for any $A \in L$.

\end{itemize}
Then, $[\ ,\ ]$ is skew symmetric and satisfies the Jacobi identity, and $(L + V, [\ ,\ ])$ is a Lie algebra.
This Lie algebra is called the exceptional Lie algbera $\frak{e}_{8}$.
It is well known that $\frak{e}_{8}$ is compact and simple.
Define the inner product $(\ ,\ )$ of $\frak{e}_{8}$ such that
\[
(\ ,\ ) = (\ ,\ )_{L} + (\ ,\ )_{V}.
\]
Then, $\langle \ ,\ \rangle = -240(\ ,\ )$ is a Killing form of $\frak{e}_{8}$.
The group of all automorphisms of $\frak{e}_{8}$ is a connected simple compact Lie group and denoted by $E_{8}$.
It is well known that $E_{8}$ is simply connected and the center $C(E_{8})$ is trivial.


We define a map $\phi : Spin(16) \rightarrow E_{8}$ as any $g \in Spin(16), A \in L,$ and $u \in V$,
\[
\phi(g)(A + u) = \mathrm{Ad}(g)(A) + \Delta(g)(u).
\]
Since $(\ ,\ )_{L}$ and $(\ ,\ )_{V}$ are invariant under $Spin(16)$, the bracket $[\ ,\ ]$ is invariant under $\phi(Spin(16))$.
Hence, $\phi(Spin(16)) \subset E_{8}$ and $\phi$ is a homomorphism.
Obviously, $\mathrm{Ker}\phi = \{ 1, e_{0} \cdots e_{15} \}$ and $\phi(Spin(16)) \cong Ss(16)$.
Hence, $\phi : L \rightarrow \frak{e}_{8}$ is one-to-one and we denote $\phi(A)$ by $A$ for any $A \in L$ .
Let $p, q$ be even and $p + q = 16$.
Set
\[
\psi_{p,q} = \phi \circ i_{p,q} : Spin(p) \times Spin(q) \rightarrow E_{8}.
\]
Then, $\psi_{p,q}$ is a homomorphism and 
\[
\begin{split}
\mathrm{Ker}\psi_{p,q} &= \{ (1,1), (-1,-1), (e_{0}, \cdots, e_{p}, e_{p+1}, \cdots, e_{16}), (-e_{0} \cdots e_{p}, -e_{p+1} \cdots e_{16}) \}. \\
\end{split}
\]
We denote $\psi_{8,8}$ by $\psi$ simply.
Then,
\[
\begin{split}
\mathrm{Ker}\psi &= \{ (1,1), (-1,-1), (e_{0}, \cdots, e_{p}, e_{p+1}, \cdots, e_{16}), (-e_{0} \cdots e_{p}, -e_{p+1} \cdots e_{16}) \} \\
&= \{ (\gamma_{0}^{a,b}, \gamma_{0}^{a,b}) \ ;\ a,b = 0,1 \}.
\end{split}
\]


Next, we consider the root system of $\frak{e}_{8}$.
We set 
\[
T(\theta_{1}, \cdots, \theta_{8}) = \phi(T_{16}(\theta_{1}, \cdots, \theta_{8})) = \psi(T_{8}(\theta_{1}, \cdots, \theta_{4}), T_{8}(\theta_{5}, \cdots, \theta_{8}))
\]
for any $\theta_{i} \in \mathbb{R} \ (1 \leq i \leq 8)$ and $T = \phi(T_{16}) = \psi(T_{8}, T_{8})$.
It is well known that the rank of $E_{8}$ is $8$, and $T$ is a maximal torus of $E_{8}$.
Set
\[
t(c_{1}, \cdots, c_{8}) = t_{16}(c_{1}, \cdots, c_{8}) = \psi(t_{8}(c_{1}, \cdots, c_{4}), t_{8}(c_{5}, \cdots, c_{8}))
\]
for any $c_{i} \in \mathbb{R} \ (1 \leq i \leq 8)$ and $\frak{t} = \frak{t}_{16} = \psi(\frak{t}_{8}, \frak{t}_{8})$.
Then, $\frak{t}$ is the Lie algebra of $T$ and a maximal abelian subspace of $\frak{e}_{8}$.
We denote $\phi(x) \in \phi(Spin(16)) \subset E_{8}$ by $x$ simply.
Then, by the definition of $x$, it is obvious that $\mathrm{Ad}(x)(\frak{t}) \subset \frak{t}$ and $\mathrm{Ad}(x)(X) = -X$ for any $X \in \frak{t}$.
The root system $\Sigma(\frak{e}_{8}, \frak{t})$ of $\frak{e}_{8}^{\mathbb{C}}$ with respect to $\frak{t}^{\mathbb{C}}$ is obtained by
\[
\Sigma(\frak{e}_{8}, \frak{t}) 
= \Big\{ \pm x_{i} \pm x_{j} \ ;\ 1 \leq i < j \leq 8 \Big\} 
\cup \Big\{ \frac{1}{2} \sum_{i=1}^{8}\epsilon_{i} x_{i} \ ;\ \epsilon_{i} = \pm 1, \ \prod_{i=1}^{8}\epsilon_{i} = 1 \Big\}.
\]
We often denote $(1/2)\sum_{i=1}^{8}\epsilon_{i}x_{i} \in \Sigma$ by $(\epsilon_{1} \cdots \epsilon_{8})$.
For example, $\gamma = (1/2)(x_{1} + x_{2} + x_{3} - x_{4} - x_{5} + x_{6} + x_{7} + x_{8})$ is denoted by $(+++--+++)$.
For $\gamma = \pm x_{i} \pm x_{j} \ (1 \leq i < j \leq 8)$,
\[
\frak{r}_{\gamma}^{\mathbb{C}}(\frak{e}_{8}, \frak{t}) = \frak{r}_{\gamma}^{\mathbb{C}}(L, \frak{t}_{16})
\]
For $\gamma = (1/2)\sum_{i=1}^{8}\epsilon_{i}x_{i} \in \Sigma$,
\[
\frak{r}_{\gamma}^{\mathbb{C}}(\frak{e}_{8}, \frak{t}) = V_{\gamma}^{\mathbb{C}}.
\]
We take some linear order of $\frak{t}$ such that the set $\Sigma^{+}$ of all positive roots is
\[
\Sigma^{+} 
= \Big\{ x_{i} \pm x_{j} \ ;\ 1 \leq i < j \leq 8 \Big\} 
\cup \Big\{ \frac{1}{2} \sum_{i=1}^{8}\epsilon_{i} x_{i} \ ;\ \epsilon_{i} = \pm 1, \ \epsilon_{1} = 1, \ \prod_{i=1}^{8}\epsilon_{i} = 1 \Big\}.
\]
For $\gamma = x_{i} \pm x_{j} \ (1 \leq i < j \leq 8)$,
\[
\frak{r}_{\gamma}(\frak{e}_{8}, \frak{t}) = \frak{r}_{\gamma}(L, \frak{t}_{16}), \quad
\frak{r}_{\gamma}^{\pm}(\frak{e}_{8}, \frak{t}) = \frak{r}_{\gamma}^{\pm}(L, \frak{t}_{16}).
\]
Moreover, for any $\gamma = (1/2)(\sum_{i=1}^{8}\epsilon_{i}x_{i}) \in \Sigma$, there exist some $0 \leq j, k \leq 3$ such that $\frak{r}_{\gamma}$ is one of
\[
\begin{split}
& \mathbb{R}( e_{2k} \otimes^{+} e_{2j} + e_{2k+1} \otimes^{+} e_{2j+1} ) + \mathbb{R}( e_{2k} \otimes^{+} e_{2j+1} - e_{2k+1} \otimes^{+} e_{2j} ), \\
& \mathbb{R}( e_{2k} \otimes^{+} e_{2j} - e_{2k+1} \otimes^{+} e_{2j+1} ) + \mathbb{R}( -e_{2k} \otimes^{+} e_{2j+1} - e_{2k+1} \otimes^{+} e_{2j} ), \\
& \mathbb{R}( e_{2k} \otimes^{-} e_{2j} + e_{2k+1} \otimes^{-} e_{2j+1} ) + \mathbb{R}( e_{2k} \otimes^{-} e_{2j+1} - e_{2k+1} \otimes^{-} e_{2j} ), \\
& \mathbb{R}( e_{2k} \otimes^{-} e_{2j} - e_{2k+1} \otimes^{-} e_{2j+1} ) + \mathbb{R}( -e_{2k} \otimes^{-} e_{2j+1} - e_{2k+1} \otimes^{-} e_{2j} ). \\
\end{split}
\]
Moreover, $\frak{r}_{\gamma}^{+}$ is one of
\[
\begin{split}
& \mathbb{R}( e_{2k} \otimes^{+} e_{2j} + e_{2k+1} \otimes^{+} e_{2j+1} ), \quad \mathbb{R}( e_{2k} \otimes^{+} e_{2j} - e_{2k+1} \otimes^{+} e_{2j+1} ), \\
& \mathbb{R}( e_{2k} \otimes^{-} e_{2j} + e_{2k+1} \otimes^{-} e_{2j+1} ), \quad \mathbb{R}( e_{2k} \otimes^{-} e_{2j} - e_{2k+1} \otimes^{-} e_{2j+1} ), \\
\end{split}
\]
and $\frak{r}_{\gamma}^{-}$ is one of
\[
\begin{split}
& \mathbb{R}( e_{2k} \otimes^{+} e_{2j+1} - e_{2k+1} \otimes^{+} e_{2j} ), \quad \mathbb{R}( e_{2k} \otimes^{+} e_{2j+1} + e_{2k+1} \otimes^{+} e_{2j} ), \\
& \mathbb{R}( e_{2k} \otimes^{-} e_{2j+1} - e_{2k+1} \otimes^{-} e_{2j} ), \quad \mathbb{R}( e_{2k} \otimes^{-} e_{2j+1} + e_{2k+1} \otimes^{-} e_{2j} ). \\
\end{split}
\]
For example, if $\gamma = (++++--++)$, then
\[
\begin{split}
\frak{r}_{\gamma} &= \mathbb{R}( e_{0} \otimes^{+} e_{2} + e_{1} \otimes^{+} e_{3} ) + \mathbb{R}( e_{0} \otimes^{+} e_{3} - e_{1} \otimes^{+} e_{2} ), \\
\frak{r}_{\gamma}^{+} &= \mathbb{R}( e_{0} \otimes^{+} e_{2} + e_{1} \otimes^{+} e_{3} ), \\
\frak{r}_{\gamma}^{-} &= \mathbb{R}( e_{0} \otimes^{+} e_{3} - e_{1} \otimes^{+} e_{2} ).
\end{split}
\]
Moreover, if $\gamma = (+++--+++)$, then
\[
\begin{split}
\frak{r}_{\gamma} &= \mathbb{R}(e_{6} \otimes^{-} e_{0} + e_{7} \otimes^{-} e_{1}) + \mathbb{R}(e_{6} \otimes^{-} e_{1} - e_{7} \otimes^{-} e_{0}), \\
\frak{r}_{\gamma}^{+} &= \mathbb{R}(e_{6} \otimes^{-} e_{0} + e_{7} \otimes^{-} e_{1}), \\
\frak{r}_{\gamma}^{-} &= \mathbb{R}(e_{6} \otimes^{-} e_{1} - e_{7} \otimes^{-} e_{0}).
\end{split}
\]

For each $\alpha = x_{i} \pm x_{j} \ (1 \leq i < j \leq 8)$, we consider the subgroup $SU^{\alpha}(2) \subset Spin(16)$.
Then, $-1$ is not contained in $SU^{\alpha}(2)$.
Therefore, $\phi|_{SU^{\alpha}(2)} : SU^{2}_{\alpha} \rightarrow E_{8}$ is one-to-one and the image is denoted by the same symbol $SU^{\alpha}(2)$.
In the root system of type $E_{8}$, it is well known that the Weyl group acts on all roots transitively.
Hence, for any $\gamma \in \Sigma^{+}$, the connected Lie subgroup whose Lie algebra is $\frak{su}^{\gamma}(2)$ is isomorphic to $SU(2)$.
Therefore, this subgroup is denoted by $SU^{\gamma}(2)$.
Moreover, let $\beta = x_{k} \pm x_{l}\ (1 \leq k < l \leq 8)$ satisfy $\beta \in \Sigma_{\alpha, -1}$. 
Then, we can define the subgroup $SU^{\alpha, \beta}(3) \subset Spin(16)$.
Since it is obvious that $e_{0} \cdots e_{15} \not\in SU^{\alpha, \beta}(3)$, we see that $\phi|_{SU^{\alpha, \beta}(3)} : SU^{\alpha, \beta}(3) \rightarrow E_{8}$ is one-to-one.
Therefore, the image is denoted by $SU^{\alpha, \beta}(3)$.
Let $\gamma, \delta \in \Sigma^{+}$ satisfy $2(\gamma,\delta)/(\gamma,\gamma) = \pm1$.
Then, it is well known that there exist an element $g$ of the Weyl group such that $g(\gamma), g(\delta) \in \Sigma(L, \frak{t}_{16})$.
Hence, the Lie subgroup whose Lie algebra is $\frak{su}^{\gamma,\delta}(2)$ is isomorphic to $SU(3)$ and this subgroup is denoted by $SU^{\gamma,\delta}(3)$.


\subsubsection{Antipodal sets of $E_{8}$}

We recall the maximal antipodal set of $E_{8}$ from \cite{Adams}.
Set $A_{1}(E_{8})$ as follows:
\[
\begin{split}
A_{1}(E_{8}) &:= \psi \big( A(Spin(8)), A(Spin(8)) \big) \\
&= \left\{ \psi( \gamma_{i}^{a,b}, \gamma_{j}^{c,d}) \ ;\ 0 \leq i,j \leq 7, a,b = 0,1 \right\} = \left\{ \psi( \gamma_{i}^{0,0}, \gamma_{j}^{a,b}) \ ;\ 0 \leq i,j \leq 7, a,b = 0,1 \right\}.
\end{split}
\]
Then, since $A(Spin(8))$ is an antipodal subgroup of $Spin(8)$ and $\psi$ is a totally geodesic embedding, $A_{1}(E_{8})$ is an antipodal set of $E_{8}$.
In particular, $A_{1}(E_{8})$ is a maximal antipodal subgroup of $E_{8}$ and $\# A_{1}(E_{8}) = 2^{8} = 256$.
Let $A(T)$ be a maximal antipodal subgroup of the maximal torus $T$ of $E_{8}$.
Then, $A(T) = \{ t \in T \ ;\ t^{2} = e \}$.
We set $\Gamma = \{ A \in \frak{t} \ ;\ \mathrm{exp}A = e \}$.
Since $E_{8}$ is simply connected, we obtain $\Gamma = \mathrm{span}_{\mathbb{Z}}\{ 2\pi (iA_{\gamma}) \ ;\ \gamma \in \Sigma \}$ (\cite{Helgason}).
Let $J$ be a complete system of represetatives of $\mathrm{span}_{\mathbb{Z}}\{ \pi(iA_{\gamma}) \ ;\ \gamma \in \Sigma \}/\Gamma$.
Then $\{ \mathrm{exp}X \ ;\ X \in J \}$ is $A(T)$.
Let $\alpha = x_{7} - x_{8}$ and $\beta = x_{6} - x_{7}$.
Then, by Lemma \ref{Lattice},
\[
\{ 0 \} \cup \Sigma^{+} \cup \{ \gamma + \alpha \ ;\ \gamma \in \Sigma_{\alpha, 0}^{+}\} \cup \{ \delta + \beta, \delta + \alpha + \beta \ ;\ \delta \in \Sigma_{\alpha, 0}^{+} \cap \Sigma_{\beta, 0} \}
\]
is a complete system of representatives.
Therefore,
\[
\begin{split}
A(T) = \{ e \} &\cup \Big\{ \mathrm{exp}\pi(iA_{\gamma}) \ ;\ \gamma \in \Sigma^{+} \Big\} 
\cup \Big\{ \mathrm{exp}\pi(iA_{\gamma})\mathrm{exp}\pi(iA_{\alpha}) \ ;\ \gamma \in \Sigma_{\alpha, 0}^{+} \Big\} \\
&\cup \Big\{ \mathrm{exp}\pi(iA_{\delta})\mathrm{exp}\pi(iA_{\beta}), \mathrm{exp}\pi(iA_{\delta})\mathrm{exp}\pi(iA_{\alpha + \beta})\ ;\ \delta \in \Sigma_{\alpha, 0}^{+} \cap \Sigma_{\beta, 0}  \Big\}. \\
\end{split}
\]
The element $x$ satisfies $x^{2} = e$ and $xtx = t^{-1}$ for any $t \in T$.
Hence, 
\[
A_{2}(E_{8}) = A(T) \cup x(A(T))
\]
 is an antipodal subgroup of $E_{8}$.
In particular, $A_{2}(E_{8})$ is a maximal antipodal subgroup of $E_{8}$ and $\#A_{2}(E_{8}) = 512$.

\begin{thm} \cite{Adams}
$A_{i}(E_{8}) \ (i = 1,2)$ are maximal antipodal subsubgroups of $E_{8}$ and $A_{1}(E_{8})$ and $A_{2}(E_{8})$ are not conjugate to each other.
Moreover, any maximal antipodal subset of $E_{8}$ is congruent to one of $A_{i}(E_{8}) \ (i = 1,2)$.

\end{thm}


\subsubsection{Symmetric subgroup of type $EVIII$}

We consider some involutive automorphisms of $E_{8}$.
First, set
\[
\sigma_{VIII} = \theta_{\phi(-1)} : E_{8} \rightarrow E_{8} \ ;\ g \mapsto \phi(-1) g \phi(-1).
\]
Since $\phi(-1)$ is involutive, $\sigma_{VIII}$ is involutive automorphism.
Note that 
\[
\phi(-1) = \psi(\gamma_{0}^{0,0}, \gamma_{0}^{1,0}) = \mathrm{exp}\pi (iA_{x_{i}- x_{j}}) \mathrm{exp}\pi (iA_{x_{i} + x_{j}}) \quad (1 \leq i < j \leq 8).
\]
In particular, $\mathrm{Ad}(\phi(-1))(A) = A$ for any $A \in L$ and $\mathrm{Ad}(\phi(-1))(u) = -u$ for any $u \in \Delta$.
Therefore,
\[
F^{+}(\sigma_{VIII}, \frak{e}_{8}) = L, \quad
F^{-}(\sigma_{VIII}, \frak{e}_{8}) = V.
\]
Since the connected subgroup of $E_{8}$ whose Lie algebra is $L$ is $\phi(Spin(16))$, we obtain
\[
F(\sigma_{VIII}, E_{8}) = \phi(Spin(16)) \cong Ss(16).
\]
Hence, $\sigma_{VIII}$ is an ivolutive automorphism of type $EVIII$.
The symmetric pair of type $EVIII$ is a normal form.
Next, we consider
\[
\sigma'_{VIII} = \theta_{x} : E_{8} \rightarrow E_{8} \ ;\ g \mapsto xgx.
\]
Since $x$ is involutive, $\sigma'_{VIII}$ is an involutive automorphism of $E_{8}$.
Then,
\[
F^{+}(\sigma'_{VIII}, \frak{e}_{8}) = \sum_{\gamma \in \Sigma^{+}} \frak{r}_{\gamma}^{+}, \quad
F^{-}(\sigma'_{VIII}, \frak{e}_{8}) = \frak{t} + \sum_{\gamma \in \Sigma^{+}}\frak{r}_{\gamma}^{-}.
\]
Therefore, the dimension of a maximal abelian subspace of $F^{-}(\sigma'_{VIII}, \frak{e}_{8})$ is $8$ and $(\frak{e}_{8}, F^{+}(\sigma'_{VIII}, \frak{e}_{8}))$ is a normal form.
Hence, $F^{+}(\sigma'_{VIII}, \frak{e}_{8})$ is isomorphic to $\frak{spin}(16)$ and $F(\sigma'_{VIII}, E_{8})$ is isomorphic to $Ss(16)$.
Note that there exists $g \in E_{8}$ such that $gxg^{-1} = \phi(-1)$ by the classification of Riemmanian symmetric pairs.


\subsubsection{Lie group $E_{7}$, Lie algebra $\frak{e}_{7}$, and symmetric subgroup of type $EIX$} \label{E7}

Fix $\alpha \in \Sigma^{+}$.
The centralizer $C(\frak{su}^{\alpha}(2), \frak{e}_{8})$ is called an exceptional Lie algebra $\frak{e}_{7}$.
We denote $C(\frak{su}^{\alpha}(2), \frak{e}_{8})$ by $\frak{e}_{7}^{\alpha}$.
The idenitity component of the centralizer $C(SU^{\alpha}(2), E_{8})$ is called an exceptional Lie group $E_{7}$ and denoted by $E_{7}^{\alpha}$.
Obviously, the Lie algebra of $E_{7}^{\alpha}$ is $\frak{e}_{7}^{\alpha}$.
It is well known that $E_{7}^{\alpha}$ is a simply connected compact simple Lie group.
The following direct sum decomposition follows:
\[
\frak{e}_{7}^{\alpha} = \frak{t}^{\alpha} + \sum_{\beta \in \Sigma_{\alpha,0}^{+}}\frak{r}_{\beta}(\frak{e}_{8}, \frak{t}).
\]
Obviously, $\frak{t}^{\alpha}$ is a maximal abelian subspace of $\frak{e}_{7}$ and $\mathrm{rank}\frak{e}_{7} = 7$.
The root system $\Sigma(\frak{e}_{7}^{\alpha}, \frak{t}_{\alpha})$ of $(\frak{e}_{7}^{\alpha})^{\mathbb{C}}$ with respect to $(\frak{t}^{\alpha})^{\mathbb{C}}$ is $\Sigma_{\alpha, 0}(\frak{e}_{8}, \frak{t})$.
For example, if $\alpha = x_{7} - x_{8}$, then
\[
\Sigma(\frak{e}_{7}^{\alpha}, \frak{t}^{\alpha}) = \left\{ \pm x_{i} \pm x_{j} \ ;\ 1 \leq i \leq j \leq 6 \right\} \cup \left\{ \pm(x_{7} + x_{8}) \right\} \cup \left\{ \frac{1}{2}\sum_{i=1}^{8}\epsilon_{i}x_{i} \ ;\ \epsilon_{i} = \pm1, \ \prod_{i=1}^{8}\epsilon_{i} = 1, \ \epsilon_{7} = \epsilon_{8} \right\}.
\]
For any $\beta \in \Sigma(\frak{e}_{7}, \frak{t}^{\alpha})$,
\[
\frak{r}_{\beta}(\frak{e}_{7}^{\alpha}, \frak{t}^{\alpha}) = \frak{r}_{\beta}(\frak{e}_{8}, \frak{t}).
\]
Moreover, $\sigma_{x}(\frak{e}_{7}^{\alpha}) \subset \frak{e}_{7}^{\alpha}$ and
\[
\frak{r}_{\beta}^{\pm}(\frak{e}_{7}^{\alpha}, \frak{t}_{\alpha}) = \frak{r}_{\beta}^{\pm}(\frak{e}_{8}, \frak{t}_{\alpha}).
\]
Set $T^{\alpha} = \mathrm{exp}\frak{t}^{\alpha}$.
Then, $T^{\alpha}$ is a maximal torus of $E_{7}^{\alpha}$.
For $\alpha = x_{i} \pm x_{j} \ (1 \leq i < j \leq 8)$, we can define the subgroup $Spin^{\alpha}(12)$ of $Spin(16)$.
Then, since $Spin^{\alpha}(12) \cap C(Spin(16)) = \{ \pm1 \}$, we see that $\phi|_{Spin^{\alpha}(12)} : Spin^{\alpha}(12) \rightarrow E_{8}$ is one-to-one and $Spin^{\alpha}(12)$ is a subgroup of $E_{8}$.
Moreover, each $g \in Spin^{\alpha}(12)$ satisfies $gh = hg$ for any $h \in SU^{\alpha}(2)$.
Therefore, $Spin^{\alpha}(12)$ is a subgroup of $E_{7}^{\alpha}$.
The Lie algebra $\frak{spin}^{\alpha}(12)$ of $Spin^{\alpha}(12)$ is given by
\[
\frak{spin}_{\alpha}(12) = \frak{t}^{\alpha, \bar{\alpha}} + \sum_{x_{k} \pm x_{l} , k,l \not= i,j } \frak{r}_{x_{k} \pm x_{l}}(\frak{e}_{8}, \frak{t}).
\]
It is well known that the center of $E_{7}$ is isomorphic to $\mathbb{Z}_{2}$.
We consider the center $C(E_{7}^{\alpha})$ of $E_{7}^{\alpha}$.
For any $\alpha \in \Sigma^{+}(\frak{e}_{8}, \frak{t})$, there exist $\beta_{1}, \beta_{2}, \beta_{3} \in \Sigma_{\alpha, 0}$ such that
\[
\alpha \equiv \beta_{1} + \beta_{2} + \beta_{3} \mod \Gamma.
\]
For example, if $\alpha = x_{7} - x_{8}$, the roots $\beta_{1} = x_{1} + x_{2}, \ \beta_{2} = x_{3} + x_{4}, \ \beta_{3} = x_{5} - x_{6}$ satisfy this property.
In fact,
\[
\alpha - (\beta_{1} + \beta_{2} + \beta_{3}) = 2 (-----++-).
\]
Hence, $\tau_{\alpha} = \tau_{\beta_{1}} \tau_{\beta_{2}} \tau_{\beta_{3}} \in E_{7}^{\alpha}$.
Since $\tau_{\alpha} \in SU^{\alpha}(2)$, we obtain $g \tau_{\alpha} = \tau_{\alpha} g$ for any $g \in E_{7}^{\alpha}$.
Hence, $\tau_{\alpha} \in C(E_{7}^{\alpha})$ and $C(E_{7}^{\alpha}) = \{ e, \tau_{\alpha} \}$.

For any $\alpha \in \Sigma^{+}(\frak{e}_{8}, \frak{t})$, we set
\[
\sigma_{IX}^{\alpha} = \theta_{\alpha} : E_{8} \rightarrow E_{8} \ ;\ g \mapsto \tau_{\alpha} g \tau_{\alpha}.
\]
Since $\tau_{\alpha}$ is involutive, $\sigma_{IX}^{\alpha}$ is an involutive automorphism.
By Subsection \ref{Rs}, 
\[
F^{+}(\sigma_{IX}^{\alpha}, \frak{e}_{8}^{\alpha}) = \frak{t} + \sum_{\beta \in \Sigma^{+}_{\alpha, 0}} \frak{r}_{\beta} = \frak{su}^{\alpha}(2) + \frak{e}_{7}^{\alpha}.
\]
Moreover, $F(\sigma_{IX}^{\alpha}, E_{8})$ is connected and
\[
F(\sigma^{\alpha}_{IX}, E_{8}) = \{ gh \ ;\ g \in SU^{\alpha}(2), h \in E_{7}^{\alpha} \} = SU^{\alpha}(2) \times E_{7}^{\alpha}/\{ (e,e), (\tau_{\alpha}, \tau_{\alpha}) \} \cong SU(2) \cdot E_{7}.
\]
Therefore, $\sigma_{IX}^{\alpha}$ is an involutive automorphism of type $EIX$.
In the following, $F(\sigma^{\alpha}_{IX}, E_{8})$ is denoted by $SU^{\alpha}(2) \cdot E_{7}^{\alpha}$.
For example, if $\alpha = x_{7} - x_{8}$, then
\[
\frak{e}_{7}^{\alpha} = \frak{spin}^{\alpha}(12) + \frak{su}^{x_{7} + x_{8}}(2) + \mathbb{O} \otimes^{+} \mathbb{H} + \mathbb{O} \otimes^{-} \mathbb{H},
\]
where $\mathbb{H} = \mathrm{sapn}_{\mathbb{R}}\{ e_{0}, e_{1}, e_{2}, e_{3} \}$.
Moreover,
\[
\begin{split}
\frak{t}^{\alpha} = \{ t(c_{1}, \cdots, c_{6}, c_{7}, c_{7}) \ ;\ c_{i} \in \mathbb{R} \}, \quad
T^{\alpha} = \mathrm{exp}\frak{t}^{\alpha} = \{ T(\theta_{1}, \cdots, \theta_{6}, \theta_{7}, \theta_{7}) \ ;\ \theta_{i} \in \mathbb{R} \}.
\end{split}
\]


\subsubsection{Symmetric subgroup of type $EVI$}

Set $\alpha = x_{7} - x_{8} \in \Sigma^{+}$ and consider $E_{7}^{\alpha}$.
Let $\beta \in \Sigma^{+}(\frak{e}_{7}^{\alpha}, \frak{t}^{\alpha}) = \Sigma_{\alpha, 0}^{+}$.
Set an involutive automorphism $\sigma_{VI}^{\beta}$ of $E_{7}^{\alpha}$ as 
\[
\sigma_{VI}^{\beta} = \theta_{\tau_{\beta}} : E_{7} \rightarrow E_{7} \ ;\ g \mapsto \tau_{\beta}g\tau_{\beta}.
\]
Since the Weyl group $W(\Sigma_{\alpha, 0})$ of $\Sigma_{\alpha, 0}$ acts on $\Sigma_{\alpha, 0}$ and this action is transitive, we can assume that $\beta = \bar{\alpha} = x_{7} + x_{8}$.
Then, $\tau_{\beta} = \psi( \gamma_{0}^{0,0}, \gamma_{1}^{1,0} )$ and
\[
\begin{split}
& F^{+}(\sigma_{VI}^{x_{7} + x_{8}}, \frak{e}_{7}) = \frak{su}^{\bar{\alpha}}(2) + \frak{spin}^{\alpha}(12), \quad
F^{-}(\sigma_{VI}^{x_{7} + x_{8}}, \frak{e}_{7}) = \mathbb{O} \otimes^{+} \mathbb{H} + \mathbb{O} \otimes^{-} \mathbb{H}.
\end{split}
\]
Note that the kernel of the restriction of $\phi$ to $SU^{\bar{\alpha}}(2) \times Spin^{\alpha}(12) \subset Spin(16)$ is 
\[
\{ (1,1), (-e_{0}e_{1} \cdots -e_{10}e_{11}, e_{12} \cdots e_{15}) \}.
\]
Therefore, the image is isomorphic to $SU(2) \cdot Spin(12)$ and denoted by $SU^{\bar{\alpha}}(2) \cdot Spin^{\alpha}(12)$.
Then, $SU^{\bar{\alpha}}(2) \cdot Spin^{\alpha}(12)$ is a subgroup of $E_{7}^{\alpha}$ and the Lie algebra is $\frak{su}^{\bar{\alpha}}(2) + \frak{spin}^{\alpha}(12)$.
Therefore,
\[
F(\sigma_{VI}^{x_{7} + x_{8}}, E_{7}) = SU^{\alpha}(2) \cdot Spin^{\alpha}(12).
\]
Next, we set 
\[
\sigma'_{VI} = \theta_{\phi(-1)} : E_{7} \rightarrow E_{7} \ ;\ g \mapsto \phi(-1) g \phi(-1).
\]
Since $F^{+}(\theta_{\phi(-1)}, \frak{e}_{8}) = L$ and $F^{-}(\theta, \frak{e}_{8}) = V$, we obtain
\[
\begin{split}
& F^{+}(\sigma'_{VI}, \frak{e}_{7}) = \frak{su}^{\bar{\alpha}}(2) + \frak{spin}^{\alpha}(12), \quad
F^{-}(\sigma'_{VI}, \frak{e}_{7}) = \mathbb{O} \otimes^{+} \mathbb{H} + \mathbb{O} \otimes^{-} \mathbb{H}.
\end{split}
\]
Note that we obtain the same result by $\phi(-1) = \tau_{\alpha}\tau_{\bar{\alpha}}$ since $\tau_{\alpha}$ is an element of $E_{7}^{\alpha}$. 
Hence,
\[
F(\sigma'_{VI}, E_{7}) = F(\sigma_{VI}^{x_{7} + x_{8}}, E_{7}).
\]
However, there does not exist $g \in E_{7}$ such that $g \phi(-1) g^{-1} = g \tau_{\alpha} \tau_{\bar{\alpha}} g^{-1} = \tau_{\bar{\alpha}}$ since $\tau_{\alpha}$ is an element of the center of $E_{7}^{\alpha}$.


\subsubsection{Symmetric subgroup of type $EV$}

Fix $\alpha \in \Sigma^{+}$ and consider $E_{7}^{\alpha}$.
Recall $x \in E_{8}$.
Then, $\mathrm{Ad}(x)$ leaves $\frak{su}^{\alpha}(2)$ invariant.
Since $\mathrm{Ad}(x)|_{\frak{su}^{\alpha}(2)} \not= \mathrm{id}_{\frak{su}^{\alpha}(2)}$, we obtain $x \not\in E_{7}^{\alpha}$.
However, $x$ is an element of the normalizer of $E_{7}^{\alpha}$.
Set
\[
\sigma_{V} = \theta_{x} : E_{7} \rightarrow E_{7} \ ;\ g \mapsto xg^{-1}x.
\]
Then, $\sigma_{V}$ is an involutive automorphism of $E_{7}^{\alpha}$ and
\[
F^{+}(\sigma_{V}, \frak{e}_{7}) = \sum_{\gamma \in \Sigma^{+}} \frak{r}_{\gamma}^{+}, \quad
F^{-}(\sigma_{V}, \frak{e}_{7}) = \frak{t}^{\alpha} + \sum_{\gamma \in \Sigma^{+}} \frak{r}_{\gamma}^{+}.
\]
Obviously, the dimension of a maximal abelian subspace of $F^{-}(\sigma_{V}, \frak{e}_{7})$ is $7$ and $(\frak{e}_{7}, F^{+}(\sigma_{V}, \frak{e}_{7}))$ is a normal form.
Therefore, $\sigma_{V}$ is an involutive automorphism of type $EV$ and 
\[
F^{+}(\sigma_{V}, \frak{e}_{7}) \cong \frak{su}(8), \quad F(\sigma_{V}, E_{7}) \cong SU(8)/\mathbb{Z}_{2}.
\]

Let $\gamma = (++++++++) \in \Sigma^{+}(\frak{e}_{8}, \frak{t})$ and consider $E_{7}^{\gamma}$.
Note that
\[
\Sigma_{\gamma,0} = \Sigma(\frak{e}_{7}^{\gamma}, \frak{t}^{\gamma}) = \Big\{ \pm(x_{i} - x_{j}) \ ;\ 1 \leq i < j \leq 8 \Big\} \cup \Big\{ \frac{1}{2}\sum_{i=1}^{8}\epsilon_{i}x_{i} \ ;\ \epsilon_{i} = \pm1, \epsilon_{1} + \cdots + \epsilon_{8} = 0 \Big\}.
\]
Since $\theta_{\phi(-1)}(E_{7}^{\gamma}) \subset E_{7}^{\gamma}$,
\[
\sigma'_{V} = \theta_{\phi(-1)} : E_{7}^{\gamma} \rightarrow E_{7}^{\gamma}
\]
is an involutive automorphism.
Then,
\[
F^{+}(\sigma'_{V}, \frak{e}_{7}^{\gamma}) = \frak{t}^{\gamma} + \sum_{1 \leq i < j \leq 8} \frak{r}_{x_{i} - x_{j}} \cong \frak{su}(8).
\]
Recall $J'' \in Spin(16)$ and $SU'(8)$ from Subsection \ref{Sg}.
Then, $\phi(J'') = \tau_{\gamma}$ and $F(\sigma'_{V}, E_{7}^{\gamma}) = \phi(SU''(8))$.
We can easily check that $e_{0}e_{1} \cdots e_{14}e_{15} \in SU''(8)$ and obtain 
\[
F(\sigma'_{V}, E_{7}^{\gamma}) = \phi(SU''(8)) \cong SU(8)/\mathbb{Z}_{2}.
\]
Hence, $\sigma'_{V}$ is also an involutive automorphism of type $EV$.


\subsubsection{Lie group $E_{6}$ and Lie algebra $\frak{e}_{6}$} \label{E6}

Fix $\alpha, \beta \in \Sigma^{+}(\frak{e}_{8}, \frak{t})$ such that $2(\alpha, \beta)/(\alpha,\alpha) = \pm1$.
Then, $\{ \pm \alpha, \pm \beta, \pm(\alpha \mp \beta) \}$ is a root system of type $A_{2}$.
In the following, we can assume $2(\alpha, \beta)/(\alpha,\alpha) = -1$ by replacing $\beta$ to $\alpha - \beta$ if necessary.
We consider $\frak{su}^{\alpha,\beta}(3) \subset \frak{e}_{8}$ and $SU^{\alpha, \beta}(3) \subset E_{8}$.
The identity component $C_{o}(SU^{\alpha, \beta}, E_{8})$ of the centralizer $C(SU^{\alpha, \beta}, E_{8})$ is called a exceptional Lie group $E_{6}$.
It is well known that $E_{6}$ is a simply connceted compact simple Lie group.
The group $C_{o}(SU^{\alpha, \beta}(2), E_{8})$ is denoted by $E_{6}^{\alpha, \beta}$.
The Lie algebra of $E_{6}$ is called a exceptional Lie algebra $\frak{e}_{6}$.
The Lie algebra of $E_{6}^{\alpha, \beta}$ is denoted $\frak{e}_{6}^{\alpha, \beta})$.
Then, $\frak{e}_{6} = C(\frak{su}^{\alpha, \beta}(2), \frak{e}_{8})$.
By the definition, $E_{6}^{\alpha, \beta} \subset E_{7}^{\alpha}, E_{7}^{\beta}$ and $\frak{e}_{6}^{\alpha, \beta} \subset \frak{e}_{7}^{\alpha}, \frak{e}_{7}^{\beta}$.
The following direct sum decomposition holds:
\[
\frak{e}_{6}^{\alpha, \beta} = \frak{t}^{\alpha, \beta} + \sum_{\gamma \in \Sigma^{+}_{\alpha,0} \cap \Sigma_{\beta,0}} \frak{r}_{\gamma}(\frak{e}_{8}, \frak{t}).
\]
Obviously, we see that $\frak{t}^{\alpha,\beta}$ is a maximal abelian subspace of $\frak{e}_{6}^{\alpha,\beta}$ and $\mathrm{rank}\frak{e}_{6}^{\alpha,\beta} = 6$.
The root system $\Sigma(\frak{e}_{6}^{\alpha,\beta}, \frak{t}^{\alpha, \beta})$ of $(\frak{e}_{6}^{\alpha,\beta})^{\mathbb{C}}$ with respect to $(\frak{t}^{\alpha,\beta})^{\mathbb{C}}$ is given by
\[
\Sigma(\frak{e}_{6}^{\alpha, \beta}, \frak{t}^{\alpha,\beta}) = \Sigma_{\alpha, 0}(\frak{e}_{8}, \frak{t}) \cap \Sigma_{\beta,0}(\frak{e}_{8}, \frak{t}).
\]
Obviously, $\theta_{x}(\frak{e}_{6}^{\alpha,\beta}) \subset \frak{e}_{6}^{\alpha,\beta}$, and for any $\gamma \in \Sigma^{+}(\frak{e}_{6}^{\alpha,\beta}, \frak{t}^{\alpha,\beta})$,
\[
\frak{r}^{\pm}_{\gamma}(\frak{e}_{6}^{\alpha,\beta}, \frak{t}^{\alpha,\beta}) = \frak{r}_{\gamma}^{\pm}(\frak{e}_{8}, \frak{t}).
\]

Since the Weyl group of $\Sigma_{\alpha, 0}$ acts on $\Sigma_{\alpha, -1}$ and this action is transitive, we can assume that $\alpha = x_{7} - x_{8}$ and $\beta = x_{6} - x_{7}$.
Then,
\[
\Sigma_{\alpha, 0} \cap \Sigma_{\beta, 0} 
= \{ \pm x_{i} \pm x_{j} \ ;\ 1 \leq i < j \leq 5 \} 
\cup \left\{ \frac{1}{2} \sum_{i=1}^{8}\epsilon_{i}x_{i} \ ;\ \prod_{i=1}^{8}\epsilon_{i} = 1, \ \epsilon_{6} = \epsilon_{7} = \epsilon_{8} \right\}.
\]
We recall $Spin^{\alpha, \beta}(10)$ from Subsection \ref{Sg}.
Then, $Spin^{\alpha, \beta}(10)$ is defined by the Clliford algebra over $\mathrm{span}_{\mathbb{R}}\{ e_{0}, \cdots, e_{9}\}$.
We see that the restiction of $\phi$ to $Spin^{\alpha, \beta}(10)$ is one-to-one and $Spin^{\alpha, \beta}(10)$ is a subgroup of $E_{8}$.
Since $Spin^{\alpha,\beta}(10) \subset C(SU^{\alpha,\beta}(3), Spin(16))$, we see that $Spin^{\alpha,\beta}(10)$ is a subgroup of $E_{6}^{\alpha,\beta}$.
The Lie algebra $\frak{spin}^{\alpha,\beta}(10)$ of $Spin^{\alpha,\beta}(10)$ is 
\[
\frak{spin}^{\alpha,\beta}(10) = \{ t(c_{1}, \cdots, c_{5}, 0, 0, 0) \ ;\ c_{i} \in \mathbb{R} \} + \sum_{1 \leq i < j \leq 5}\frak{r}_{x_{i} \pm x_{j}}. \\
\]
Next, we consider the centralizer $C(SU^{\alpha,\beta}(3), Spin(16)) \cong U(1) \cdot Spin(10)$.
Then, the kernel of the homomorphism
\[
U(1) \times Spin(10) \rightarrow U(1) \cdot Spin(10) \cong C(SU^{\alpha,\beta}(3), Spin(16)) \rightarrow E_{8}
\]
is $\{ (1,1), (-1,-1), (e_{10}e_{11}e_{12}e_{13}e_{14}e_{15}, e_{0}e_{2} \cdots e_{9}), (-e_{10}e_{11}e_{12}e_{13}e_{14}e_{15}, -e_{0}e_{2} \cdots e_{9}) \} \cong \mathbb{Z}_{4}$.
Hence, 
\[
\phi(C(SU^{\alpha,\beta}(3), Spin(16)) \cong \Big( U(1) \times Spin(10) \Big)/\mathbb{Z}_{4}.
\]
The Lie algebra of $\phi(C(SU^{\alpha,\beta}(3), Spin(16)))$ is $C(\frak{su}^{\alpha, \beta}(3), L)$ and
\[
\begin{split}
C(\frak{su}^{\alpha, \beta}(3), L) 
&= \{ t(c_{2}, \cdots, c_{5}, c_{1}, c_{1}, c_{1}) \ ;\ c_{i} \in \mathbb{R} \} + \sum_{1 \leq i < j \leq 5}\frak{r}_{x_{i} \pm x_{j}} \\
& = \mathbb{R}t(0,0,0,0,0,1,1,1) + \frak{spin}^{\alpha, \beta}(10).
\end{split}
\]
Then,
\[
\frak{e}_{6}^{\alpha, \beta} = C(\frak{su}^{\alpha, \beta}(3), L) + \mathbb{O} \otimes^{+} \mathbb{C} + \mathbb{O} \otimes^{-} \mathbb{C}.
\]

Next, we consider the center $C(E_{6}^{\alpha,\beta})$ of $E_{6}^{\alpha,\beta}$.
It is well known that the center of $E_{6}$ is isomorphic to $\mathbb{Z}_{3}$.
Set an element $\omega \in E_{6}^{\alpha, \beta}$ as
\[
\omega = t \left( 0, 0, 0, 0, 0, \frac{4}{3}\pi, \frac{4}{3}\pi, \frac{4}{3}\pi \right), \quad
\omega^{2} = t \left( 0, 0, 0, 0, 0, \frac{8}{3}\pi, \frac{8}{3}\pi, \frac{8}{3}\pi \right). \quad
\]
Then, for any $\gamma = x_{i} \pm x_{j} \ (1 \leq i < j \leq 5)$ and $X \in \frak{r}_{\gamma}^{\mathbb{C}}$,
\[
\mathrm{Ad}(\omega)(X) = (\cos 0 + i \sin 0)X = X.
\]
Moreover, for any $\delta = \frac{1}{2}(\sum_{i=1}^{8}\epsilon_{i}x_{i}) \in \Sigma \ (\epsilon_{6} = \epsilon_{7} = \epsilon_{8})$ and $Y \in \frak{r}_{\delta}^{\mathbb{C}}$,
\[
\mathrm{Ad}(\omega)(Y) = (\cos 2\pi \pm i \sin 2\pi)Y = Y.
\]
Therefore, $\mathrm{Ad}(\omega) = \mathrm{Ad}(\omega^{2}) = \mathrm{id}_{\frak{e}_{6}^{\alpha,\beta}}$.
Since $\{ e, \omega^{2}, \omega^{3} \} \cong \mathbb{Z}_{3}$, we obtain $C(E_{6}^{\alpha, \beta}) = \{ e, \omega, \omega^{2} \}$.


\subsubsection{Symmetric subgroup of type $EVII$}

We consider an involutive automorphism of type $EVII$.
Fix $\alpha = x_{7} - x_{8}$ and consider $E_{7}^{\alpha}$.
Let $\beta \in \Sigma^{+}$ satisfy $2(\alpha,\beta)/(\beta,\beta) = -1$.
Set
\[
\sigma_{VII}^{\beta} = \theta_{\tau_{\beta}} : E^{\alpha}_{7} \rightarrow E^{\alpha}_{7} \ ;\ g \mapsto \tau_{\beta} g \tau_{\beta}.
\]
Then,
\[
F(\sigma_{VII}^{\beta}, \frak{e}_{7}^{\alpha}) = \frak{t}^{\alpha} + \sum_{\gamma \in \Sigma_{\alpha, 0}^{+} \cap \Sigma_{\beta, 0}} \frak{r}_{\gamma}.
\]
Since the Weyl group of $\Sigma_{\alpha, 0}$ acts on $\Sigma_{\alpha, -1}$ and this action is transitive, we can assume $\beta = x_{6} - x_{7}$.
Then,
\[
F(\sigma_{VII}^{\beta}, \frak{e}_{7}^{\alpha}) = \mathbb{R}t(0, \cdots, 0, 2,-1,-1) + \frak{e}_{6}^{\alpha, \beta}.
\]
Since $E_{7}$ is simply connected, we see $F(\sigma_{VII}^{\beta}, E_{7}^{\alpha})$ is connected.
Therefore, we obtain that
\[
\begin{split}
F(\sigma_{VII}^{\beta}, E_{7}^{\alpha}) = & \{ T(0, \cdots, 0, 2\theta, -\theta, -\theta)g \ ;\ \theta \in \mathbb{R}, \ g \in E_{6}^{\alpha,\beta} \} \\
= & \big( \{ T(0, \cdots, 0, 2\theta, -\theta, -\theta) \ ;\ \theta \in \mathbb{R} \} \times E_{6}^{\alpha, \beta} \big) \\
& \quad \quad \big/ \left\{ (1,1),  \left( T \left( 0, \cdots, 0, \frac{8}{3}\pi, -\frac{4}{3}\pi, -\frac{4}{3}\pi \right), \omega \right), \left( T\left(0, \cdots, 0, \frac{4}{3}\pi, -\frac{2}{3}\pi, -\frac{2}{3}\pi \right), \omega^{2} \right) \right\} \\
\cong & \Big( U(1) \times E_{6} \Big)/\mathbb{Z}_{3}.
\end{split}
\]
Hence, $\sigma_{VII}^{\beta}$ is an involutive automorphism of type $EVII$.


\subsubsection{The symmetric subgroup of type $EIII$}

Let $\alpha = x_{7} - x_{8}$ and $\beta = x_{6} - x_{7}$ and consider $E_{7}^{\alpha}$. 
We consider an involutive automorphism of type $EIII$.
We see $\phi(-1) \in E_{6}^{\alpha, \beta}$ since $-1 \in Spin^{\alpha, \beta}(12)$.
Hence, $\theta_{\phi(-1)}(E_{6}^{\alpha, \beta}) \subset E_{6}^{\alpha, \beta}$.
Set
\[
\sigma_{III} = \theta_{\phi(-1)} : E_{6}^{\alpha, \beta} \rightarrow E_{6}^{\alpha, \beta} \ ;\ g \mapsto \phi(-1) g \phi(-1).
\]
Then, $\sigma_{III}$ is an invoultive automorphism of $E_{6}^{\alpha, \beta}$.
Since $F(\theta_{\phi(-1)}, E_{8}) = \phi(Spin(16))$, we obtain
\[
F^{+}(\sigma_{III}, \frak{e}_{7}^{\alpha}) = L \cap C(\frak{su}^{\alpha, \beta}(2), \frak{e}_{6}^{\alpha, \beta}) = C(\frak{su}^{\alpha,\beta}(3), L).
\]
Therefore,
\[
\begin{split}
F^{+}(\sigma_{III}, E_{6}) &= \phi \big( C(SU^{\alpha,\beta}(3), Spin(16)) \big) \cong \big( U(1) \times Spin(10) \big)/\mathbb{Z}_{4} \\
\end{split}
\]
and $\sigma_{III}$ is an involutive automorphism of type $EIII$.


\subsubsection{The symmetric subgroup of type $EII$}

Let $\alpha = 1/2(x_{1} + \cdots x_{8})$ and $\beta = x_{7} + x_{8}$.
Then,
\[
\begin{split}
\Sigma_{\alpha, 0} \cap \Sigma_{\beta, 0}
&= \{ \pm ( x_{i} - x_{j} ) \ ;\ 1 \leq i \leq 6 \} \cup \{ \pm (x_{7} - x_{8}) \} \cup \left\{ \frac{1}{2} \sum_{i=1}^{8} \epsilon_{i}x_{i} \ ;\ \epsilon_{i} = \pm1, \sum_{i=1}^{8}\epsilon_{i} = 0, \epsilon_{7} = - \epsilon_{8} \right\}.
\end{split}
\]
Let $\gamma$ be any element of $\Sigma^{+}(\frak{e}_{6}^{\alpha, \beta}, \frak{t}^{\alpha, \beta}) = \Sigma_{\alpha, 0}^{+} \cap \Sigma_{\beta, 0}$.
Since the Weyl group of $\Sigma(\frak{e}_{6}^{\alpha, \beta}, \frak{t}^{\alpha, \beta})$ acts on $\Sigma(\frak{e}_{6}^{\alpha, \beta}, \frak{t}^{\alpha, \beta})$ and this action is transitive, we can assume $\gamma = x_{7} - x_{8}$.
Set
\[
\sigma_{II}^{\gamma} = \theta_{\tau_{\gamma}} : E_{6}^{\alpha, \beta} \rightarrow E_{6}^{\alpha, \beta} \ ;\ g \mapsto \tau_{\gamma} g \tau_{\gamma}.
\]
Then, we see
\[
F^{+}(\sigma_{II}^{\gamma}, \frak{e}_{6}^{\alpha, \beta}) = \frak{su}^{\gamma}(2) + \left\{ t(c_{1}, \cdots, c_{6}, 0, 0) \ ;\ c_{i} \in \mathbb{R}, \sum_{i=1}^{6}c_{i} = 0 \right\} + \sum_{1 \leq i < j \leq 6} \frak{r}_{x_{i} - x_{j}}.
\]
The orthogonal complement of $\frak{su}^{\gamma}(2)$ of $F^{+}(\sigma_{II}^{\gamma}, \frak{e}_{6}^{\alpha, \beta})$ is isomorphic to $\frak{su}(6)$ and denoted by $\frak{su}^{\gamma}(6)$.
Moreover, the connected subgroup of $Spin(16)$ whose Lie algebra is $\frak{su}^{\gamma}(6)$ is isomorphic to $SU(6)$ and denoted by $SU^{\gamma}(6)$.
In particular, $SU^{\gamma}(6)$ is equal to $SU''(6)$ of $Spin(12)$ defined by the Clifford algebra over $\mathrm{span}_{\mathbb{R}}\{ e_{0}, e_{1}, \cdots, e_{11} \}$.
We consider the homomorphism
\[
\begin{array}{ccccccccccccccc}
SU^{\gamma}(2) \times SU^{\gamma}(6) & \rightarrow & Spin(16) & \rightarrow & E_{8} \\
(g,h) & \mapsto & gh & \mapsto & \phi(gh)
\end{array}.
\]
The kernel of this homomorphism is
\[
\{ (1,1), (- e_{12}e_{13}e_{14}e_{15}, -e_{0}e_{1} \cdots e_{10}e_{11}) \} \cong \mathbb{Z}_{2}.
\]
The image is denoted by $SU^{\gamma}(2) \cdot SU^{\gamma}(6)$.
Since the Lie algebra of $SU^{\gamma}(2) \cdot SU^{\gamma}(6)$ is $F^{+}(\sigma_{II}^{\gamma}, \frak{e}_{6}^{\alpha, \beta})$, we obtain
\[
F(\sigma_{II}^{\gamma}, E_{6}^{\alpha, \beta}) =  SU^{\gamma}(2) \cdot SU^{\gamma}(6).
\]
Therefore, $\sigma_{III}^{\gamma}$ is an involutive automorphism of type $EII$.


\subsubsection{The symmetric subgroup of type $EI$}

Let $\alpha = x_{7} - x_{8}$ and $\beta = x_{6} - x_{7}$, and consider $E_{6}^{\alpha, \beta}$.
Then, $\theta_{x}(E_{6}^{\alpha, \beta}) \subset E_{6}^{\alpha, \beta}$ and set
\[
\sigma_{I} = \theta_{x} : E_{6}^{\alpha, \beta} \rightarrow E_{6}^{\alpha, \beta} \ ;\ g \mapsto x g x^{-1}.
\]
Then, 
\[
F^{+}(\sigma_{I}, \frak{e}_{6}^{\alpha, \beta}) = \sum_{\gamma \in \Sigma^{+}(\frak{e}_{6}, \frak{t}^{\alpha, \beta})} \frak{r}_{\gamma}^{+}, \quad\quad
F^{-}(\sigma_{I}, \frak{e}_{6}^{\alpha, \beta}) = \frak{t}^{\alpha, \beta} + \sum_{\gamma \in \Sigma^{+}(\frak{e}_{6}, \frak{t}^{\alpha, \beta})} \frak{r}_{\gamma}^{-}.
\]
Therefore, the dimension of a maximal abelian subspace of $F^{-}(\sigma_{I}, \frak{e}_{6}^{\alpha, \beta})$ is $6$ and $(\frak{e}_{6}^{\alpha, \beta}, F^{-}(\sigma_{I}, \frak{e}_{6}^{\alpha, \beta}))$ is a normal form.
Therefore, we obtain $F^{+}(\sigma_{I}, \frak{e}_{6}^{\alpha, \beta}) \cong \frak{sp}(4)$ and $F(\sigma_{I}, E_{6}) \cong Sp(4)/\mathbb{Z}_{2}$, and $\sigma_{I}$ is an involutive automorphism of type $EI$.

We more explicitly construct $Sp(4)/\mathbb{Z}_{2}$ in $E_{6}$.
Let $\gamma = 1/2(x_{1} + \cdots + x_{8})$ and consider $E_{7}^{\gamma}$. 
Define the element $J''_{4}$ of $Spin(8)$ as follows:
\[
J''_{4} = \big( E_{02} - E_{13} + E_{46} - E_{57}, \ -E_{02} - E_{13} - E_{46} - E_{57}, \ \pi(e_{1}e_{3}) \big).
\]
Then, 
\[
\begin{split}
i(J''_{4}, J''_{4}) = J''_{8}x = \prod_{i=0}^{1} \Big( -\cos \frac{\pi}{4} + & \sin \frac{\pi}{4}e_{8i}e_{8i+2} \Big) \Big( \cos \frac{\pi}{4} + \sin \frac{\pi}{4}e_{8i+1}e_{8i+3} \Big) \\
& \Big( - \cos \frac{\pi}{4} + \sin \frac{\pi}{4}e_{8i+4}e_{8i+6} \Big) \Big( \cos \frac{\pi}{4} + \sin \frac{\pi}{4}e_{8i+5}e_{8i+7} \Big) \in Spin(16).
\end{split}
\]
Set $\omega = \phi(J''_{8}x)$.
We see that $\theta_{\omega}$ leaves $\frak{su}^{\gamma}(2)$ invariant and $\omega$ is an element of the normalizer $N(E_{7}^{\gamma})$.
Therefore, $\theta_{\omega}(E_{7}^{\gamma}) \subset E_{7}^{\gamma}$.
Note that $\theta_{\omega}$ leaves $L$ and $V$ respectively.
By the arguments of subsection \ref{Sg},
\[
F^{+}(\theta_{\omega}, \frak{e}_{7}^{\gamma} \cap L) = \frak{sp}''(4).
\]
Next, we consider $F^{+}(\theta_{\omega}, \frak{e}_{7}^{\gamma} \cap V)$.
We see that $\phi(\omega) \in N(T)$ and
\[
\begin{split}
\mathrm{Ad}(\omega) t \Big(c_{1}, c_{2}, c_{3}, c_{4}, c_{5}, c_{6}, c_{7}, c_{8} \Big) = t \Big(-c_{2}, -c_{1}, -c_{4}, -c_{3}, c_{5}, -c_{4}, -c_{8}, -c_{7} \Big).
\end{split}
\]
Hence, if $\delta \in \Sigma^{+}(\frak{e}_{7}^{\gamma}, \frak{t}^{\gamma})$ satisfies $\mathrm{Ad}(\omega)\delta = \pm \delta$, that is, $\delta$ is one of 
\[
\begin{array}{llll}
(++++|----), & (++--|++--), & (++--|--++), \\
(+-+-|+--+), & (+-+-|-++-), & (+-+-|-+-+), & (+-+-|+-+-), \\
(+--+|+--+), & (+--+|-++-), & (+--+|-+-+), & (+--+|+-+-), \\
\end{array}
\]
then, $\theta_{\omega}$ leaves $\frak{r}_{\delta}$ invariant.
If $\delta$ satisfied $\mathrm{Ad}(\omega)\delta = \delta$, that is, $\delta$ is one of 
\[
(++++|----), \ (++--|++--), \ (++--|--++),
\]
then, we see that $\dim F^{+}(\theta_{\omega}, \frak{r}_{\delta}) = 1$.
If $\delta$ satisfies $\mathrm{Ad}(\omega)\delta = -\delta$, then we see $\dim F^{+}(\theta_{\omega}, \frak{r}_{\delta}) = 2$.
Moreover, if $\mathrm{Ad}(\omega)\delta \not= \delta$, that is, $\delta$ is one of 
\[
\begin{array}{ccccccccc}
(+++-|+---), & (++-+|-+--), & (+++-|-+--), & (++-+|+---), \\
(+++-|--+-), & (++-+|---+), & (+++-|---+), & (++-+|--+-), \\
(+-++|+---), & (+---|+-++), & (+-++|-+--), & (+---|-+++), \\
(+-++|--+-), & (+---|+++-), & (+-++|---+), & (+---|++-+), \\
(++--|+-+-), & (++--|-+-+), & (++--|+--+), & (++--|-++-), \\
(+-+-|++--), & (+-+-|--++), & (+--+|++--), & (+--+|--++), \\
\end{array}
\]
then, $\theta_{\omega} (\frak{r}_{\delta} ) \subset \frak{r}_{\mathrm{Ad}(\omega)\delta}$.
Hence, $\dim F^{+}(\theta_{\omega}, \frak{r}_{\delta} + \frak{r}_{\mathrm{Ad}(\omega)\delta}) = 2$.
Therefore, 
\[
\dim F^{+}(\theta_{\omega}, \frak{e}_{7}^{\gamma} \cap V) =  3 + 16 + 24 = 43.
\]
Thus, $\dim F^{+}(\theta_{\omega}, \frak{e}_{7}^{\gamma}) = 43 + 36 = 79$ and $\theta_{\omega}$ is an involutive automorphism of type $EVII$.
In particular, $F^{+}(\theta_{\omega}, \frak{e}_{7}^{\gamma}) \cong (U(1) \times E_{6})/\mathbb{Z}_{3}$.
The center of $F^{+}(\theta_{\omega}, \frak{e}_{7}^{\gamma})$ is $1$-dimensional and contained in $V$ since $F^{+}(\theta_{\omega}, \frak{e}^{\gamma}_{7} \cap L) = \frak{sp}(4)$ is simple.
The irreducible component of $F^{+}(\theta_{\omega}, \frak{e}_{7}^{\gamma})$ that is isomorphic to $\frak{e}_{6}$ is denoted by $\frak{e}'_{6}$.
The connected subgroup of $F^{+}(\theta_{\omega}, \frak{e}_{7}^{\gamma})$ whose Lie algebra is $\frak{e}'_{6}$ is denoted by $E'_{6}$.
It is obvious that $\theta_{\phi(-1)}$ leaves $E'_{6}$ invariant.
Set
\[
\sigma'_{I} := \theta_{\phi(-1)} : E'_{6} \rightarrow E'_{6} \ ;\ g \mapsto \phi(-1) g \phi(-1).
\]
Then, $F^{+}(\sigma'_{I}, \frak{e}'_{6}) = F^{+}(\theta_{\omega}, \frak{e}_{7}^{\gamma} \cap L) = \frak{sp}''(4)$.
Since $Sp''(4)$ contains $e_{0}e_{1} \cdots e_{14}e_{15}$, we obtain $\phi(Sp(4)) \cong Sp(4)/\mathbb{Z}_{2}$.
The Lie algebra of $\phi(Sp(4))$ is $\frak{sp}(4)$.
Therefore,
\[
F^{+}(\sigma'_{I}, E'_{6}) = \phi(Sp''(4)) \cong Sp(4)/\mathbb{Z}_{2}
\]
and $\sigma'_{I}$ is an involutive automorphism of type $EI$.


\subsubsection{Lie algebra $\frak{f}_{4}$ and Lie group $F_{4}$}

Recall the homomorphism $\psi : Spin^{0}(8) \times Spin^{1}(8) \rightarrow E_{8}$.
Then, $\psi|_{Spin^{0} \times \{ 1 \}}$ and $\psi |_{\{1\} \times Spin^{1}(8)}$ are one-to-one and $Spin^{i}(8) \ (i = 0,1)$ is a subrgoup of $E_{8}$.
We consider $G_{2} \subset Spin^{2}(8)$ and $\frak{g}_{2} \subset \frak{spin}^{2}(8)$.
The identity component of $C(G_{2}, E_{8})$ is called an exceptional Lie group $F_{4}$.
It is well known that $F_{4}$ is a simply connected compact simple Lie group.
The Lie algebra of $F_{4}$ is denoted by $\frak{f}_{4}$.
Then, $\frak{f}_{4} = C(\frak{g}_{2}, \frak{e}_{8})$.
It is well known that $G_{2}$ fixes the identity element $e_{0}$ of $\mathbb{O}$ and any non-trivial orbit of $G_{2}$ in $\mathrm{Im}\mathbb{O}$ is a round sphere.
Therefore, 
\[
C( \frak{g}_{2}, V) = \mathbb{O} \otimes^{+} (\mathbb{R}e_{0}) + \mathbb{O} \otimes^{-} (\mathbb{R}e_{0}).
\]
In the following, $\mathbb{O} \otimes^{+} (\mathbb{R}e_{0})$ and $\mathbb{O} \otimes^{-} (\mathbb{R}e_{0})$ are denoted by $\mathbb{O}^{+}$ and $\mathbb{O}^{-}$ respectively.
Moreover, in $Spin(16)$, the idenitity component of the centralizer $C(G_{2}, Spin(16))$ is $Spin(9)$, where $Spin(9)$ is defined by the Clifford algebra over $\mathrm{span}_{\mathbb{R}}\{ e_{0}, e_{1}, \cdots, e_{8} \}$.
Hence,
\[
\frak{f}_{4} = \frak{spin}(9) + \mathbb{O}^{+} + \mathbb{O}^{-}.
\]
We see that $\phi|_{Spin(9)}$ is one-to-one and $Spin(9)$ is a subgroup of $F_{4}$.

Set an abelian subspace of $\frak{f}_{4}$ as follows:
\[
\frak{t}_{\frak{f}_{4}} := \{ t(c_{1}, c_{2}, c_{3}, c_{4}, 0, \cdots, 0) \ ;\ c_{i} \in \mathbb{R} \}.
\]
Then, $\frak{t}_{\frak{f}_{4}}$ is a maximal abelian subspace of not only $\frak{spin}(9)$ but also $\frak{f}_{4}$.
The root system $\Sigma(\frak{f}_{4}, \frak{t}_{\frak{f}_{4}})$ of $\frak{f}_{4}^{\mathbb{C}}$ with respect to $\frak{t}_{\frak{f}_{4}}^{\mathbb{C}}$ is
\[
\Sigma(\frak{f}_{4}, \frak{t}_{\frak{f}_{4}}) = \{ \pm x_{i}, \ \pm x_{i} \pm x_{j} \ ;\ 1 \leq i < j \leq 4 \} 
\cup \left\{ \frac{1}{2} \sum_{i=1}^{4} \epsilon_{i}x_{i} \ ;\ \epsilon_{i} = \pm 1 \right\}.
\]
We consider the linear order of $\frak{t}_{\frak{f}_{4}}$ induced by the order of $\frak{t}$.
The set of all positive roots is
\[
\Sigma^{+} = 
\{ x_{i}, \ x_{i} \pm x_{j} \ ;\ 1 \leq i < j \leq 4 \} \cup \left\{ \frac{1}{2} \sum_{i=1}^{4} \epsilon_{i}x_{i} \ ;\ \epsilon_{i} = \pm 1, \epsilon_{1} = 1 \right\}.
\]
Then, for any $1 \leq i < j \leq 4$,
\[
\begin{split}
& \frak{r}_{x_{i} \pm x_{j}}^{\mathbb{C}} (\frak{f}_{4}, \frak{t}_{\frak{f}_{4}}) = \frak{r}_{x_{i} \pm x_{j}}^{\mathbb{C}} (\frak{e}_{8}, \frak{t}) = \frak{r}_{x_{i} \pm x_{j}}(\frak{spin}(9), \frak{t}_{4})^{\mathbb{C}}, \quad \\
& \frak{r}_{x_{i} \pm x_{j}} (\frak{f}_{4}, \frak{t}_{\frak{f}_{4}}) = \frak{r}_{x_{i} \pm x_{j}} (\frak{e}_{8}, \frak{t})= \frak{r}_{x_{i} \pm x_{j}}(\frak{spin}(9), \frak{t}_{4}).
\end{split}
\]
For any $1 \leq i \leq 4$,
\[
\begin{split}
& \frak{r}_{x_{i}}^{\mathbb{C}} (\frak{f}_{4}, \frak{t}_{\frak{f}_{4}}) = \frak{r}_{x_{i}}(\frak{spin}(9), \frak{t}_{9})^{\mathbb{C}}, \quad 
\frak{r}_{x_{i}} (\frak{f}_{4}, \frak{t}_{\frak{f}_{4}}) = \frak{r}_{x_{i}}(\frak{spin}(9), \frak{t}_{9}).
\end{split}
\]
For any $\gamma = (1/2)(x_{1} + \epsilon_{2} x_{2} + \epsilon_{3} x_{3} + \epsilon_{4} x_{4}) \in \Sigma^{+}(\frak{f}_{4}, \frak{t}_{\frak{f}_{4}})$, if $\epsilon_{1}\epsilon_{2}\epsilon_{3}\epsilon_{4} = 1$, then
\[
\frak{r}_{\gamma}^{\mathbb{C}} = (V_{8}^{+})_{\gamma}^{\mathbb{C}}, \quad
\frak{r}_{\gamma} = \mathbb{O} \cap \Big( (V_{8}^{+})_{\gamma}^{\mathbb{C}} + (V_{8}^{+})_{-\gamma}^{\mathbb{C}} \Big),
\]
and if $\epsilon_{1}\epsilon_{2}\epsilon_{3}\epsilon_{4} = -1$, then
\[
\frak{r}_{\gamma}^{\mathbb{C}} = (V_{8}^{-})_{\gamma}^{\mathbb{C}}, \quad
\frak{r}_{\gamma} = \mathbb{O} \cap \Big( (V_{8}^{-})_{\gamma}^{\mathbb{C}} + (V_{8}^{-})_{-\gamma}^{\mathbb{C}} \Big).
\]
It is well known that the center of $F_{4}$ is trivial.


\subsubsection{The symmetric subgroup of type $EIV$}

We consider an involutive automorphism of type $EIV$.
Let $\alpha = x_{7} - x_{8}$ and $\beta = x_{6} - x_{7}$, and consider $E_{6}^{\alpha, \beta}$.
Then, 
\[
\begin{split}
\frak{e}_{6}^{\alpha, \beta} &= \mathbb{R} t(0, \cdots, 0, 1,1,1) + \frak{spin}(10) + \mathbb{O} \otimes^{+} \mathbb{C} + \mathbb{O} \otimes^{-} \mathbb{C},
\end{split}
\]
where $\frak{spin}(10)$ is defined by the Clifford algebra over $\mathrm{span}_{\mathbb{R}}\{ e_{0}, \cdots, e_{7}, e_{8}, e_{9} \}$.
We set $z = \psi( \gamma_{0}^{0,0}, \gamma_{4}^{0,0})$.
Then, we see that $\theta_{z}$ leaves $\frak{su}^{\alpha, \beta}(3)$ invariant and $z$ is an element of the normalizer $N(E_{6}^{\alpha, \beta})$.
Therefore, $\theta_{z}(E_{6}^{\alpha, \beta}) \subset E_{6}^{\alpha, \beta}$.
Set
\[
\sigma_{IV} = \theta_{z} : E_{6}^{\alpha, \beta} \rightarrow E_{6}^{\alpha, \beta}.
\]
Then, $\sigma_{IV}$ leaves $\frak{e}_{6}^{\alpha, \beta} \cap L$ and $\frak{e}_{6}^{\alpha, \beta} \cap V$ respectively.
Moreover,
\[
F^{+}(\sigma_{IV}, \frak{e}_{6}^{\alpha, \beta} \cap L) = \frak{spin}(9),
\]
where $\frak{spin}(9)$ is defined by the Clifford algebra over $\mathrm{span}_{\mathbb{R}}\{ e_{0}, \cdots, e_{7}, e_{8} \}$, and
\[
F^{+}(\sigma_{IV}, \frak{e}_{6}^{\alpha, \beta} \cap V) = \mathbb{O} \otimes^{+} e_{0} + \mathbb{O} \otimes^{-} e_{0}.
\]
Hence, we obtain
\[
F^{+}(\sigma_{IV}, \frak{e}_{6}^{\alpha, \beta}) = \frak{spin}(9) + \mathbb{O} \otimes^{+} e_{0} + \mathbb{O} \otimes^{-} e_{0} = \frak{f}_{4}.
\]
Since $E_{6}^{\alpha, \beta}$ is simply connceted, we obtain $F(\sigma_{IV}, E_{6}^{\alpha, \beta}) = F_{4}$, and $\sigma_{IV}$ is an involutive automorphism of type $EIV$.


\subsubsection{The symmetric subgroup of type $FII$}

Since $-1 \in Spin(9)$, we see $\phi(-1) \in F_{4}$.
Hence, $\theta_{\phi(-1)}(F_{4}) \subset F_{4}$, and set
\[
\sigma_{FII} : = \theta_{\phi(-1)} : F_{4} \rightarrow F_{4} \ ;\ g \mapsto \phi(-1) g \phi(-1).
\]
Then, $F^{+}(\sigma_{FII}, \frak{f}_{4}) = \frak{spin}(9)$ and $F(\sigma_{FII}, F_{4}) = Spin(9)$.
Therefore, $\sigma_{IV}$ is an involutive automorphism of type $FII$.


\newpage

\subsubsection{The symmetric subgroup of type $FI$}

We consider involutive automorphisms of type $FI$.
Recall $x \in E_{8}$.
Then, $x = \psi( \gamma_{4}^{0,0}, \gamma_{4}^{0,0})$.
Since $\gamma_{4}^{0,0} \in G_{2}$, we see that $\theta_{x}$ leaves $C(G_{2}, E_{8})$ invariant.
In particular, $\theta_{x}(F_{4}) \subset F_{4}$, and set
\[
\sigma_{FI} := \theta_{x} : F_{4} \rightarrow F_{4} \ ;\ g \mapsto g x g^{-1}.
\]
Then, $\sigma_{FI}$ leaves $\frak{spin}(9), \mathbb{O}^{+}$, and $\mathbb{O}^{-}$ invariant respectively, and
\[
\begin{split}
& F^{+}(\sigma_{FI}, \frak{spin}(9)) = \frak{spin}(5) + \frak{spin}(4), \\
& F^{+}(\sigma_{FI}, \mathbb{O}^{+}) = \mathrm{span_{\mathbb{R}}\{ e_{0}, e_{2}, e_{4}, e_{6}} \}, \\
& F^{+}(\sigma_{FI}, \mathbb{O}^{-}) = \mathrm{span_{\mathbb{R}}\{ e_{0}, e_{2}, e_{4}, e_{6}} \}, \\
\end{split}
\]
where $\frak{spin}(5)$ and $\frak{spin}(4)$ are defined by the Clifford algebra over $\mathrm{span}_{\mathbb{R}}\{ e_{0}, e_{2}, e_{4}, e_{6}, e_{8} \}$ and $\mathrm{span}_{\mathbb{R}}\{ e_{1}, e_{3}, e_{5}, e_{7} \}$ respectively.
Hence, we obtain 
\[
\dim F(\sigma_{FI}, F_{4}) = 10 + 6 + 8 = 24 = \dim Sp(1) \cdot Sp(3).
\]
Therefore, $F(\sigma_{FI}, F_{4}) \cong Sp(1) \cdot Sp(3)$, and $\sigma_{FI}$ is an involutive automorphism of type $FI$.

Let $\alpha$ be a longer root of $F_{4}$.
Since the Weyl gorup acts on the all longest roots and this action is transitive, we can assume $\alpha = x_{1} - x_{2}$.
Set
\[
\sigma'_{FI} = \theta_{\tau_{\alpha}} : F_{4} \rightarrow F_{4} \ ;\ g \mapsto \tau_{\alpha} g \tau_{\alpha}.
\]
Then, $F^{+}(\sigma'_{FI}, F_{4})$ is given by
\[
F^{+}(\sigma'_{FI}, F_{4}) = \frak{su}^{\alpha}(2) + \frak{t}_{\frak{f}_{4}}^{\alpha} + \sum_{\beta \in \Sigma_{\alpha, 0}^{+}(\frak{f}_{4}, \frak{t}_{\frak{f}_{4}})} \frak{r}_{\gamma},
\]
where $\Sigma_{\alpha, 0}^{+}(\frak{f}_{4}, \frak{t}_{\frak{f}_{4}})$ is given by
\[
\Sigma_{\alpha, 0}^{+} = \{ x_{1} + x_{2}, x_{3} \pm x_{4} \} \cup \{ x_{i} \ ;\ 1 \leq i \leq 4 \} \cup \left\{ \frac{1}{2}(x_{1} + x_{2} \pm x_{3} \pm x_{4}) \right\}.
\]
Note that $\Sigma_{\alpha, 0}(\frak{f}_{4}, \frak{t}_{\frak{f}_{4}})$ is a root system of type $C_{3}$.
Therefore,
\[
F^{+}(\sigma'_{FI}, F_{4}) \cong \frak{sp}(1) + \frak{sp}(3)
\]
and $\sigma'_{FI}$ is an involutive automorphism of type $FI$.

Next, we more explicitly construct $Sp(1) \cdot Sp(3) \subset F_{4}$.
We recall the group $E'_{6}$.
Set $\delta = x_{7} + x_{8}$.
Then, $\tau_{\delta} = \psi(\gamma_{0}^{0,0}, \gamma_{1}^{1,0})$.
We can easily check that $\theta_{\tau_{\delta}}$ leaves $\frak{su}^{\gamma}(2)$ invariant, and $\theta_{\tau_{\delta}}$ induces an involutive automorphism of $E_{7}^{\gamma}$.
Moreover, we see that $\tau_{\delta}$ commutes with $\omega$.
Hence, $\theta_{\tau_{\delta}}$ leaves $E'_{6}$ invariant.
Note that $\theta_{\tau_{\delta}}(L) \subset L$ and $\theta_{\tau_{\delta}}(V) \subset V$.
We see $F^{+}(\theta_{\tau_{\delta}}, \frak{e}'_{6} \cap L) = F^{+}(\theta_{\tau_{\delta}}, \frak{sp}''(4)) = \frak{su}^{x_{7} - x_{8}}(2) + \frak{sp}''(3)$, where $\frak{sp}''(3) \subset \frak{spin}^{x_{7}- x_{8}}(12)$.
Then,
\[
\dim F^{+} \Big( \theta_{\tau_{\beta}}, F^{+}(\theta_{\omega}, \frak{e}_{7}^{\gamma}) \Big) = \dim \Big( \frak{sp}(1) + \frak{sp}(3) \Big) + 16 + 12 = 24 + 28 = 52. 
\]
and $F(\theta_{\tau_{\beta}}, E'_{6}) \cong F_{4}$.
We denote $F(\theta_{\tau_{\beta}}, E'_{6})$ by $F'_{4}$.
The Lie algebra of $F'_{4}$ is denoted by $\frak{f}'_{4}$.
We see that $\theta_{\phi(-1)}$ leaves $F'_{4}$ invariant, and set
\[
\sigma''_{FI} = \theta_{\phi(-1)} : F'_{4} \rightarrow F'_{4} \ ;\ g \mapsto \phi(-1) g \phi(-1).
\]
Then, $F^{+}(\sigma''_{FI}, \frak{f}'_{4}) = F^{+}(\theta_{\tau_{\delta}}, \frak{e}'_{6} \cap L)$.
Since $-e_{12}e_{13}e_{14}e_{15} \in SU^{x_{7} - x_{8}}(2)$ and $-e_{0}e_{1} \cdots e_{7}e_{12}e_{13}e_{14}e_{15} \in Sp''(3)$, we obtain $\phi(SU^{x_{7} - x_{8}}(2) \times Sp''(3))$ is isomorphic to $Sp(1) \cdot Sp(3)$.
The Lie algebra of $\phi(SU^{x_{7} - x_{8}}(2) \times Sp''(3))$ is $\frak{su}^{x_{7} - x_{8}}(2) + \frak{sp}''(3)$.
Therefore, 
\[
F(\sigma''_{FI}, F'_{4}) = \phi(SU^{x_{7} - x_{8}}(2) \times Sp''(3)) \cong Sp(1) \cdot Sp(3)
\] 
and $\sigma''_{FI}$ is an involutive automorphism of type $FI$.

In fact, we consider the following metric vector spaces $(V_{i}, (\ ,\ )_{i})\ (i = 1,2)$ and roots systems $A_{i}\ (i = 1,2)$:
\[
\begin{split} 
& V_{1} = \mathbb{R}^{3}, \\
& \text{$(\ ,\ )_{1}$ is the standard inner product of $\mathbb{R}^{3}$,} \\
& A_{1} = \{ \pm e_{i} \pm e_{j} \ ;\ 1 \leq i < j \leq 3 \} \cup \{ \pm 2e_{i} \ ;\ 1 \leq i \leq 3 \}, \\
& V_{2} = \{ (c_{1}, c_{1}, c_{2}, c_{3}) \in \mathbb{R}^{4} \ ;\ c_{i} \in \mathbb{R} \}, \\
& \text{$(\ ,\ )_{2}$ is the metric of $V$ such that $(1/2)(1,1,0,0), (1/2)(0,0,1,-1), (1/2)(0,0,1,1)$ is an orthonormal basis of $V$}, \\
& A_{2} = \left\{ 
\begin{matrix}
\pm (1/2)(1,1,1,1), & \pm (1/2)(1,1,1,-1), & \pm (1/2)(1,1,-1,1), & \pm (1/2)(1,1,-1,-1), \\
\pm (1,1,0,0), & \pm (0,0,1,1), & \pm (0,0,1,-1), & \pm (0,0,1,0), & \pm (0,0,0,1)
\end{matrix}
\right\}.
\end{split}
\]
Then, the following linear isomorphism 
\[
f : \mathbb{R}^{3} \rightarrow V \ ;\ 
e_{1} \mapsto \frac{1}{2} \begin{pmatrix} 1 \\ 1 \\ 0 \\ 0 \end{pmatrix}, 
e_{2} \mapsto \frac{1}{2} \begin{pmatrix} 0 \\ 0 \\ 1 \\ 1 \end{pmatrix}, 
e_{3} \mapsto \frac{1}{2} \begin{pmatrix} 0 \\ 0 \\ 1 \\ -1 \end{pmatrix}. 
\]
is an isomorphism map between the root system $A_{1}$ and $A_{2}$.
Therefore, $\Sigma_{\alpha, 0}^{+}$ is a root system of type $C_{3}$.
Hence,
\[
F(\sigma'_{FI}, F_{4}) = \frak{t}_{F_{4}} + \frak{r}_{\alpha} + \sum_{\beta \in \Sigma^{+}_{\alpha, 0}} \frak{r}_{\beta} \cong \frak{su}(2) + \frak{sp}(3) = \frak{sp}(1) + \frak{sp}(3).
\]
Thus, $\sigma'_{FI}$ is an involutive automorphism of type $FI$.






\newpage

\subsection{Maximal antipodal sets of some exceptional symmetric spaces}

By the previous studies, the classification of the conguent classes of maximal antipodal sets is complete in the case of $F_{4}, FI, FII, E_{6}, EI, EII, EIII, EIV, EVII, E_{8}$.
We already recalled the result of $E_{8}$ in Subsection \ref{ELg}.
In this subsection, we recall the result of the other cases. 
Moreover, we explicitly describe each maximal antipodal sets using the maximal antipodal set $A_{i}(E_{8}) \ (i = 1,2)$ of $E_{8}$.


\subsubsection{Maximal antipodal sets of $F_{4}, FI$, and $FII$} \label{s-F_{4}}

By the author \cite{Sasaki-F}, the conjugate classes of maximal antipodal subgroups of $F_{4}$ were classified.

\begin{thm}\cite{Sasaki-F} \label{F_{4}}
Any maximal antipodal sets of $F_{4}$ are congruent to each other.
The cardinality of any maximal antipodal set is $32$ and $\#F_{4} = 32$

\end{thm}

We describe maximal antipodal sets of $F_{4}$.
Consider $F_{4}$ as a subgroup of $E_{8}$ as in Subsection \ref{ELg}.
Set $A(F_{4}) = A_{1}(E_{8}) \cap F_{4}$.
Then, 
\[
A_{1}(E_{8}) \cap C(G_{2}, E_{8}) = \{ \psi(\gamma_{i}^{0,0}, \gamma_{0}^{a,b}) \ ;\ 0 \leq i \leq 7, a,b = 0,1 \}.
\]
Since $Spin^{1}(8) \subset F_{4}$, we obtain $A(F_{4}) = A_{1}(E_{8}) \cap C(G_{2}, E_{8})$.
The cardinality of $A(F_{4})$ is $32$ and $A(F_{4})$ is a great antipodal set.
Note $A(F_{4}) = \psi(A(Spin^{1}(8) \times \{ 1 \})$.
Summarizing these arguments, we obtain Propsition \ref{F_{4}-1}.

\begin{prop} \label{F_{4}-1}
Set $A(F_{4}) = A_{1}(E_{8}) \cap F_{4}$.
Then, $A(F_{4})$ is a maximal antipodal set of $F_{4}$.

\end{prop}

\begin{remark}
We see that
\[
\begin{split}
A_{2}(E_{8}) \cap F_{4}
&= \{ e \} \cup \{ \tau_{x_{i} \pm x_{j}} \ ;\ 1 \leq i < j \leq 4 \} \cup \{ \tau_{x_{1} \pm x_{2}} \tau_{x_{3} + x_{4}}, \tau_{x_{1} - x_{2}}\tau_{x_{1} + x_{2}} \} \\
&= \{ \psi(\gamma_{i}^{0,0}, \gamma_{0}^{a,b}) \ ; \ 0 \leq i \leq 3, a,b = 0,1 \}. 
\end{split}
\]
Therefore, $A_{2}(E_{8}) \cap F_{4} \subset A(F_{4})$, and $A_{2}(E_{8}) \cap F_{4}$ is not a maximal antipodal set of $F_{4}$.

\end{remark}


\subsubsection{Maximal antipodal sets of $FI$ and $FII$} \label{s-FIFII}

By the author \cite{Sasaki-F}, the congruent classes of maximal antipodal sets of $FI$ were classified.

\begin{thm}\cite{Sasaki-F} \label{FI}
Any maximal antipodal sets of $FI$ are congruent to each other.
The cardilnality of any maximal antipodal set is $28$ and $\#FI = 28$.

\end{thm}

We describe maximal antipodal sets of $FI$.
It is well known that the number of polars of except for the trivial pole is $2$ in $F_{4}$.
Such polars are isomorphic to either $FI$ or $FII$.
Consider the following orbits:
\[
\begin{split}
& \bigcup_{g \in F_{4}}g \tau_{x_{1} - x_{2}} g^{-1} = F_{4}/F(\sigma'_{FI}, F_{4}) \cong FI, \\
& \bigcup_{g \in F_{4}}g \phi(-1) g^{-1} = F_{4}/F(\sigma_{FII}, F_{4}) \cong FII. \\
\end{split}
\]
These are polars of the identity element,
The first polar is denoted by $FI_{+}$ and the second is denoted by $FII_{+}$.
Set $A(FI_{+}) = A_{1}(E_{8}) \cap FI_{+}$.
Then, 
\[
A(FI_{+}) = \{ \psi(\gamma_{i}^{0,0}, \gamma_{0}^{a,b}) \ ;\ 1 \leq i \leq 7 \}
\]
and the cardinality of $A(FI_{+})$ is $28$.
Therefore, $A(FI_{+})$ is a great antipodal set of $FI_{+}$.
Note that $A(FI_{+}) = \psi( A((\tilde{G}_{4}(\mathbb{R}^{8})_{+}) \times \{ 1 \} )$.
Summarizing these argument, we obtain Proposition \ref{FI-1}.

\begin{prop}\label{FI-1}
Set $A(FI_{+}) = A_{1}(E_{8}) \cap FI_{+}$.
Then, $A(FI_{+})$ is a maximal antipodal sets of $FI_{+}$. 

\end{prop}

\begin{remark}
We see 
\[
A_{2}(E_{8}) \cap FI_{+} = \{ \tau_{x_{i} \pm x_{j}} \ ;\ 1 \leq i < j \leq 4 \} = \{ \psi(\gamma_{i}^{0,0}, \gamma_{0}^{a,b}) \ ;\ 1 \leq i \leq 3, a.b = 0,1 \}.
\]
Hence, $A_{2}(E_{8}) \cap FI_{+} \subset A(FI_{+})$ and $A_{2}(E_{8}) \cap FI_{+}$ is not a maximal antipodal set of $FI_{+}$.

\end{remark}

By Tanaka and Tasaki \cite{Tanaka-Tasaki}, the congruent classes of maximal antipodal sets of symmetric $R$-spaces were classified.
Any maximal antipodal sets of a symmetric $R$-space are congruent to each other and any maximal antipodal sets is an orbit of a Weyl group.
It is well known that the simply connected compact symmetric space of type $FII$ is a symmetric $R$ space.

\begin{thm} \cite{Tanaka-Tasaki} \label{FII}
Any maximal antipodal sets of $FII$ are congruent to each other.
The cardinality of any maximal antipodal set is $3$ and $\#_{2}FII = 3$.

\end{thm}

Set $A(FII_{+}) = A_{1}(E_{8}) \cap FII_{+}$.
Then, 
\[
A(FII_{+}) = \{ \psi(\gamma_{0}^{0,0}, \gamma_{0}^{a,b}) \ ;\ (a,b) = (1,0), (0,1), (1,1) \}.
\]
The cardinlaity of $A(FII_{+})$ is $3$ and $A(FII_{+})$ is a great antipodal set of $FII_{+}$.
Moreover,
\[
A_{2}(E_{8}) \cap FII_{+} = \{ \tau_{x_{1} + x_{2}}\tau_{x_{3} + x_{4}}, \ \tau_{x_{1} - x_{2}}\tau_{x_{3} + x_{4}},\ \tau_{x_{1} + x_{2}}\tau_{x_{1} - x_{2}} \} = A(FII).
\]
Summarizing these arguments, we obtain Propostion \ref{FII-1}.

\begin{prop} \label{FII-1}
Set $A(FII_{+}) = A_{1}(E_{8}) \cap FII_{+} = A_{2}(E_{8}) \cap FII_{+}$.
Then, $A(FII_{+})$ is a maximal antipodal set of $FII_{+}$.

\end{prop}


\subsubsection{Maximal antipodal sets of $E_{6}$} \label{s-E_{6}}

By the author \cite{Sasaki-E}, the congruent classes of maximal antipodal sets of $E_{6}$ were classified.

\begin{thm} \cite{Sasaki-E}
The number of congruent classes of maximal antipodal subgroups of $E_{6}$ is $2$.
The cardinality of any maximal antipodal set is either $32$ or $64$.
\end{thm}

We describe maximal antipodal sets of $E_{6}$.
In this subsection, let $\alpha = x_{7} - x_{8}$ and $\beta = x_{6} - x_{7}$, and consider $E_{6}^{\alpha, \beta}$.
Set $A_{i}(E_{6}^{\alpha, \beta}) = A_{i}(E_{8}) \cap E_{6}^{\alpha, \beta} \ (i = 1,2)$.
First, 
\[
\begin{split}
A_{1}(E_{8}) \cap C(SU^{\alpha, \beta}(3), E_{8}) &= 
\{ \psi(\gamma_{i}^{0,0}, \gamma_{0}^{a,b}) \ ;\ 0 \leq i \leq 7, a,b = 0,1 \}.
\end{split}
\]
Then, $A_{1}(E_{8}) \cap C(SU^{\alpha, \beta}(3), E_{8}) = A(F_{4}) \subset E_{6}^{\alpha, \beta}$.
Therefore, $A_{1}(E_{6}^{\alpha, \beta}) = A_{1}(E_{8}) \cap C(SU^{\alpha, \beta}(3), E_{8})$.
The cardinality of $\#A_{1}(E_{6}^{\alpha, \beta})$ is $32$, and $A_{1}(E_{6}^{\alpha, \beta})$ is a maximal antipodal set of $E_{6}^{\alpha, \beta}$.
Next, 
\[
A_{2}(E_{8}) \cap C(SU^{\alpha, \beta}(3), E_{8}) 
= \{ e \} \cup \{ \tau_{\delta} \ ;\ \delta \in (\Sigma_{\alpha, 0} \cap \Sigma_{\beta, 0})^{+} \} \cup \{ \tau_{\delta}\tau_{\alpha} \ ;\ \delta \in \Sigma_{\alpha, 0} \cap \Sigma_{\beta, 1} \}.
\]
We can easily verify that, for any $\delta \in \Sigma_{\alpha,0}^{+} \cap \Sigma_{\beta, \pm1}$, it is true that $\tau_{\delta}\tau_{\alpha} \in T^{\alpha, \beta} = \mathrm{exp} \frak{t}^{\alpha, \beta}$.
Therefore, $A_{2}(E_{8}) \cap C(SU^{\alpha, \beta}(3), E_{8}) \subset E_{6}^{\alpha, \beta}$ and $A_{2}(E_{6}^{\alpha, \beta}) = A_{2}(E_{8}) \cap C(SU^{\alpha, \beta}(3), E_{8})$.
The cardinality of $A_{2}(E_{6}^{\alpha, \beta})$ is $64$ and $A_{2}(E_{6}^{\alpha, \beta})$ is a great antipodal set of $E_{6}^{\alpha, \beta}$.
Note that $A_{2}(E_{6}^{\alpha, \beta})$ is a maximal antipodal set of the maximal torus $T^{\alpha, \beta}$ of $E_{6}^{\alpha, \beta}$.
Summarizing these arguments, we obtain Proposition \ref{E_{6}-1}.

\begin{prop} \label{E_{6}-1}
Set $A_{i}(E_{6}^{\alpha, \beta}) = A_{i}(E_{8}) \cap E_{6}^{\alpha, \beta} \ (i = 1,2)$.
Then, $A_{1}(E_{6}^{\alpha, \beta})$ and $A_{2}(E_{6}^{\alpha, \beta})$ are maximal antipodal subgroups of $E_{6}^{\alpha, \beta}$ and the cardinalities are $32$ and $64$.

\end{prop}


\subsubsection{Maximal antipodal sets of $EI$} \label{s-EI}

By the author \cite{Sasaki-E}, the congruent classes of maximal antipodal sets of $EI$ were classified.

\begin{thm} \cite{Sasaki-E}
The number of congruent classes of maximal antipodal sets of $EI$ is $2$.
The cardinality of any maximal antipodal set is either $28$ or $64$.

\end{thm}

We describe maximal antipodal sets of $EI$.
Since $\sigma_{I} = \theta_{x}$ is an involutive automorphims of type $EI$, 
\[
\bigcup_{g \in E_{6}^{\alpha, \beta}}g x g^{-1} = E_{6}^{\alpha, \beta}/F(\sigma_{I}, E_{6}^{\alpha, \beta}) \cong EI
\]
and this orbit is denoted by $EI_{x}$.
Set $A_{i}(EI_{x}) = A_{i}(E_{8}) \cap EI_{x} \ (i = 1,2)$.
First, 
\[
\{ g \in A_{1}(E_{8}) \ ;\  \mathrm{Ad}(g)|_{\frak{su}^{\alpha, \beta}(3)} = \mathrm{Ad}(x)|_{\frak{su}^{\alpha, \beta}(3)} \} = \{ \psi(\gamma_{i}^{0,0}, \gamma_{4}^{a,b}) \ ;\ 0 \leq i \leq 7, \ a,b = 0,1 \}.
\]
By calculating the eigenvalues, we see that $\{ \psi(\gamma_{0}^{0,0}, \gamma_{4}^{a,b}) \ ;\ a,b = 0,1 \}$ is not contained in $EI_{x}$.
Since $\{ \gamma_{i}^{a,b} \ ;\ 1 \leq i \leq 7 \} \subset (\tilde{G}_{4}(\mathbb{R}^{8}))_{+}$ and $Spin^{1}(8) \subset E_{6}^{\alpha, \beta}$, we obtain
\[
A_{1}(EI_{x}) = \{ \psi(\gamma_{i}^{0,0}, \gamma_{4}^{a,b}) \ ;\ 1 \leq i \leq 7 \}.
\]
Then, the cardinality of $A_{1}(EI_{x})$ is $28$ and $A_{1}(EI_{x})$ is a maximal antipodal set of $EI_{x}$.
Note 
\[
A_{1}(EI_{x}) = A(FI_{x}) = \psi(A((\tilde{G}_{4}(\mathbb{R}^{8})^{+}) \times \{ \gamma_{4}^{0,0} \} ).
\]
Next, we consider $A_{2}(EI_{x})$.
Since
\[
A_{2}(E_{6}^{\alpha, \beta})x = \{ x \} \cup \{ \tau_{\delta}x \ ;\ \delta \in (\Sigma_{\alpha, 0} \cap \Sigma_{\beta, 0})^{+} \} \cup \{ \tau_{\delta}\tau_{\alpha}x \ ;\ \delta \in \Sigma_{\alpha, 0} \cap \Sigma_{\beta, 1} \} \subset A_{2}(EI_{x})
\]
and the cardinality of $A_{2}(E_{6}^{\alpha, \beta})x$ is $64$, we obtain $A_{2}(EI_{x}) = A_{2}(E_{6}^{\alpha, \beta})x$, and $A_{2}(EI_{x})$ is a great antipodal set.
Note that $A_{2}(EI_{x})$ is a maximal antipodal set of a maximal flat torus of $EI_{x}$.
Summarizing these arguments, we obtain Proposition \ref{EI-1}.

\begin{prop} \label{EI-1}
Set $A_{i}(EI_{x}) = A_{i}(E_{8}) \cap EI_{x} \ (i = 1,2)$.
Then, $A_{1}(EI_{x})$ and $A_{2}(EI_{x})$ are maximal antipodal sets of $EI_{x}$ and the cardinalities are $28$ and $64$.

\end{prop}


\subsubsection{Maximal antipodal sets of $EII$} \label{s-EII}

By the author \cite{Sasaki-E}, the congruent classes of maximal antipodal sets of $EII$ were classified.

\begin{thm} \cite{Sasaki-E}
The number of congruent classes of maximal antipodal sets of $EII$ is $2$.
The cardinality of any maximal antipodal set is either $28$ or $36$.

\end{thm}

We describe maximal antipodal sets of $EII$.
Let $\gamma \in \Sigma^{+}(\frak{e}_{6}^{\alpha, \beta}, \frak{t}^{\alpha, \beta})$.
Since $\sigma_{II}^{\gamma}$ is an involutive automorphism of type $EII$, 
\[
\bigcup_{g \in E_{7}}g \tau_{\gamma} g^{-1} \cong E_{6}^{\alpha, \beta}/F(\sigma_{II}^{\gamma}, E_{6}^{\alpha, \beta}) \cong EII.
\]
This orbit is a polar of the identity element in $E_{6}^{\alpha, \beta}$ and denoted by $EII_{+}$.
Then, $FI_{+}$ is a totally geodesic submanifold of $EII_{+}$ and $FII_{+} \cap EII_{+} = \phi$.
Set $A_{i}(EII_{+}) = A_{i}(E_{8}) \cap EII_{+} \ (i = 1,2)$.
First, we consider $A_{1}(EII_{+})$.
Since $FII_{+} \cap EII_{+} = \phi$, we see that $\{ \psi(\gamma_{0}^{0,0}, \gamma_{0}^{a,b}) \ ;\ a,b = 0,1 \}$ is not contained in $A_{1}(EII_{+})$.
Moreover, since $FI_{+} \subset EII_{+}$, 
\[
A_{1}(EII_{+}) = A(FI_{+}) = \{ \psi(\gamma_{i}^{0,0}, \gamma_{0}^{a,b}) \ ;\ 1 \leq i \leq 7, a,b = 0,1 \}.
\]
The cardinality of $A_{1}(EII_{+})$ is $28$ and $A_{1}(EII_{+})$ is a maximal antipodal set of $EII_{+}$.
Next, we consider $A_{2}(EII_{+})$.
Since $EII_{+}$ is a polar of $E_{6}^{\alpha, \beta}$, it is obvious that $\{ \tau_{\delta} \ ;\ \delta \in \Sigma_{\alpha,0}^{+} \cap \Sigma_{\beta, 0} \} \subset A_{2}(EII_{+})$.
The cardinality of the left hand side is $36$, and 
\[
A_{2}(EII_{+}) = 
\{ \tau_{\delta} \ ;\ \delta \in (\Sigma_{\alpha,0} \cap \Sigma_{\beta, 0})^{+} \} =
\{ \tau_{\delta} \ ;\ \delta \in \Sigma^{+}(\frak{e}_{6}^{\alpha, \beta}, \frak{t}^{\alpha, \beta}) \}.
\]
Therefore, $A_{2}(EII_{+})$ is a great antipodal set of $EII_{+}$.
Let $W(E_{6}^{\alpha, \beta})$ be the Weyl group of the root system $\Sigma(\frak{e}_{6}^{\alpha, \beta}, \frak{t}^{\alpha, \beta}) = \Sigma_{\alpha, 0} \cap \Sigma_{\beta, 0}$.
Then, $W(E_{6}^{\alpha, \beta})$ acts on $A_{2}(EII_{+})$ and this action is transitive.
Summarizing these arguments, we obtain Proposition \ref{EII-1}.

\begin{prop}\label{EII-1}
Set $A_{i}(EII_{+}) = A_{i}(E_{8}) \cap EII_{+}\ (i = 1,2)$.
Then, $A_{1}(EII_{+})$ and $A_{2}(EII_{+})$ are maximal antipodal sets of $EII_{+}$ and the cardinalities are $28$ and $36$.

\end{prop}


\subsubsection{Maximal antipodal sets of $EIII$} \label{s-EIII}

It is well known that $EIII$ is a Hermitian symmetric space.
In particular, $EIII$ is a symmetric $R$-space.
Therefore, the congruent classes of maximal antipodal sets of $EIII$ were classified by Tanaka-Tasaki \cite{Tanaka-Tasaki}.

\begin{thm}\cite{Tanaka-Tasaki} \label{anti-EIII}
Any maximal antipodal sets of $EIII$ are congruent to each other.
The cardinality of any maximal antipodal set is $27$.

\end{thm}

We describe maximal antipodal sets of $EIII$.
Since $\sigma_{III} = \theta_{\phi(-1)}$ is an involutive automorphism of type $EIII$,
\[
\bigcup_{g\in E_{6}^{\alpha, \beta}} g \phi(-1) g^{-1} = E_{6}^{\alpha, \beta}/F(\sigma_{III}, E_{6}^{\alpha, \beta}) \cong EIII.
\]
This orbit is a polar of the identity element of $E_{6}^{\alpha, \beta}$ and denoted by $EIII_{+}$.
The polars except for the trivial pole of the identity element are $EII_{+}$ and $EIII_{+}$ in $E_{6}^{\alpha, \beta}$.
Then, $FI_{+} \cap EIII_{+} = \phi$ and $FII_{+}$ is a totally geodesic submanifold of $EIII_{+}$.
Set $A(EIII_{+}) = A_{2}(E_{8}) \cap EIII_{+}$.
Then, 
\[
A(EIII_{+}) = \{ \tau_{\delta}\tau_{\alpha} \ ;\ \delta \in \Sigma_{\alpha, 0} \cap \Sigma_{\beta, 1}^{+} \}.
\]
The cardinlality of $A(EIII_{+})$ is $27$ and $A(EIII_{+})$ is a great antipodal set of $EIII_{+}$.
The Weyl group $W(E_{6}^{\alpha, \beta})$ acts on $\Sigma_{\alpha, 0} \cap \Sigma_{\beta, 1}$ and this action is transitive.
Hence, $W(E_{6}^{\alpha, \beta})$ acts on $A_{2}(EIII_{+})$ and this action is transitive.
Summarizing these arguments, we obtain Proposition \ref{EIII-1}.

\begin{prop} \label{EIII-1}
Set $A(EIII_{+}) = A_{2}(E_{8}) \cap EIII_{+}$.
Then, $A(EIII_{+})$ is a maximal antipodal set of $EIII_{+}$.

\end{prop}

\begin{remark}
We consider $A_{1}(E_{8}) \cap EIII_{+}$.
We can easily check that
\[
A_{1}(E_{8}) \cap EIII_{+} = \{ \psi(\gamma_{0}^{0,0}, \gamma_{0}^{a,b}) \ ;\ (a,b) = (1,0), (0,1), (1,1) \}
\]
and $A_{1}(E_{8}) \cap EIII_{+} = A(FII_{+})$.
Since the cardinality is $3$, the antipodal set $A_{1}(E_{8}) \cap EIII_{+}$ is not maximal.

\end{remark}


\subsubsection{Maximal antipodal sets of $EIV$} \label{s-EIV}

The congruent classes of maximal antipodal sets of $EIV$ were classified by the author \cite{Sasaki-E}.

\begin{thm} \cite{Sasaki-E}
Any maximal antipodal sets of $EIV$ are congruent to each other.
The cardinality of any maximal antipodal set is $4$.

\end{thm}

We describe maximal antipodal sets of $EIV$.
Since $\sigma_{IV} = \theta_{z} \ (z = \psi(\gamma_{0}^{0,0}, \gamma_{4}^{0,0}))$ is an involutive automorphism of type $EIV$,
\[
\bigcup_{g \in E_{6}^{\alpha, \beta}}g x g^{-1} = E_{6}^{\alpha, \beta}/F(\sigma_{IV}, E_{6}^{\alpha, \beta}) \cong EIV.
\]
This orbit is denoted by $EIV_{z}$.
Set $A(EIV_{z}) = A_{1}(E_{8}) \cap EIV_{z}$.
Then, 
\[
A(EIV_{z}) = \{ \psi(\gamma_{0}^{0,0}, \gamma_{4}^{a,b}) \ ;\ a,b = 0,1 \}.
\]
The cardinality of $A(EIV_{z})$ is $4$ and $A(EIV_{z})$ is a maximal antipodal set of $EIV_{z}$.
Note that $A(EIV_{z})$ is a maximal antipodal set of a maximal flat torus.
Summarizing these arguments we obtain Proposition \ref{EIV-1}.

\begin{prop} \label{EIV-1}
Set $A(EIV_{z}) = A_{1}(E_{8}) \cap EIV_{z}$.
Then, $A(EIV_{z})$ is a maximal antipodal set of $EIV_{z}$.

\end{prop}


\subsubsection{Maximal antipodal sets of $EVII$}

It is well known that $EVII$ is a Hermitian symmetric space.
Therefore, $EVII$ is a symmetric $R$ space and the congruent classes of maximal antipodal sets of $EVII$ were classefied by Tanaka-Tasaki \cite{Tanaka-Tasaki}.

\begin{thm}\cite{Tanaka-Tasaki} \label{anti-EVII}
Any maximal antipodal sets of $EVII$ are congruent to each other.
The cardinality of any maximal antipodal set is $56$.

\end{thm}

We describe maximal antipodal sets of $EVII$.
Consider $E_{7}^{\alpha}$.
Since $\sigma_{VII}^{\beta}$ is an involutive automorphism of type $EVII$,
\[
\bigcup_{g \in E_{7}^{\alpha}} g \tau_{\beta} g^{-1} = E_{7}^{\alpha}/F(\sigma_{VII}^{\beta}, E_{7}^{\alpha}) \cong EVII.
\]
This orbit is denoted by $EVII_{\beta}$.
Set $A(EVII_{\beta}) = A_{2}(E_{8}) \cap EVII_{\beta}$.
The Weyl group $W(E_{7}^{\alpha})$ of the root system $\Sigma(\frak{e}_{7}^{\alpha}, \frak{t}^{\alpha})$ acts on $\Sigma_{\alpha, 1}$ and this action is transitive.
Therefore, $\{ \tau_{\gamma} \ ;\ \gamma \in \Sigma_{\alpha, -1} \} \subset A(EVII_{\beta})$.
Since the cardinality of the left hand side is $56$, we obtain 
\[
A(EVII_{\beta}) = \{ \tau_{\gamma} \ ;\ \gamma \in \Sigma_{\alpha, -1} \}.
\]
Hence, $A(EVII_{\beta})$ is a great antipodal set of $EVII_{\beta}$.
Note that $W(E_{7}^{\alpha})$ acts on $A(EVII_{\beta})$ and this action is transitive.
Summarizing these argument, we obtain Proposition \ref{EVII-1}.

\begin{prop}\label{EVII-1}
Set $A(EVII_{\beta}) = A_{2}(E_{8}) \cap EVII_{\beta}$.
Then, $A(EVII_{\beta})$ is a maximal antipodal set.

\end{prop}

\begin{remark}
We see 
\[
A_{1}(E_{8}) \cap EVII_{\beta} = \{ \psi(\gamma_{0}^{0,0}, \gamma_{2}^{a,b}), \psi(\gamma_{0}^{0,0}, \gamma_{3}^{a,b}) \ ;\ a,b = 0,1 \}.
\]
However, $A_{1}(E_{8}) \cap EVII_{\beta} = \{ \tau_{x_{i} \pm x_{7}}, \tau_{x_{i} \pm x_{8}} \ ;\ i = 4,5 \}$.
Therefore, $A_{1}(E_{8}) \cap EVII_{\beta} \subset A(EVII_{\beta})$ and $A_{1}(E_{8}) \cap EVII_{\beta}$ is not a maximal antipodal set of $EVII_{\beta}$.

\end{remark}


\subsection{Maximal antipodal sets of some classical symmetric space}

In this subsection, we recall some results about maximal antipodal sets of some classical symmetric spaces that we will use later.
Set $v_{0}, \cdots, v_{n-1}$ as an oriented orthonormal basis of $\mathbb{R}^{n}$.
For any subspace $V \subset \mathbb{R}^{n}$, we denote by $\pm V$ the oriented subspace $V$ with the positive (resp. negative) orientation with respect to $v_{0}, \cdots, v_{n-1}$.
The set of all subsets of $\{ 0, 1, \cdots n-1 \}$ whose caridinalities are $k$ is denoted by $P_{k}(n)$.
Set $A_{1}, A_{2} \subset P_{4}(12)$ as follows:
\[
\begin{split}
A_{1} &=
\small 
\left\{
\begin{array}{ccccccccccccccc}
\{0,1,2,3\}, & \{0,1,4,5\}, & \{0,1,6,7\}, & \{0,2,4,6\}, & \{0,2,5,7\}, & \{0,3,4,7\}, & \{0,3,5,6\}, \\
\{4,5,6,7\}, & \{2,3,6,7\}, & \{2,3,4,5\}, & \{1,3,5,7\}, & \{1,3,4,6\}, & \{1,2,5,6\} & \{1,2,4,7\}, & \{8,9,10,11\} \\
\end{array}
\right\}, \\
\normalsize
A_{2} &= \Big\{ \{ 2i, 2i+1, 2j, 2j+1 \} \ ;\ 0 \leq i < j \leq 5 \Big\}.
\end{split}
\]
Set subsets $A_{i}(\tilde{G}_{4}(\mathbb{R}^{12})) \ (i = 1,2)$ of $\tilde{G}_{4}(\mathbb{R}^{12})$ as follows:
\[
A(\tilde{G}_{4}(\mathbb{R}^{12})) = \Big\{ \pm( \mathbb{R}v_{a} + \mathbb{R}v_{b} + \mathbb{R}v_{c} + \mathbb{R}v_{d}) \ ;\ \{ a,b,c,d \} \in A_{i} \Big\} \ (i = 1,2).
\]
Note that $\#V_{1} = \#V_{2} = 30$ and $A_{1}(\tilde{G}_{4}(\mathbb{R}^{12})$ is not congruent to $A_{2}(\tilde{G}_{4}(\mathbb{R}^{12})$.

\begin{thm}\cite{Tasaki}\label{Ori-Grass}
The subsets $A_{i}(\tilde{G}_{4}(\mathbb{R}^{12})) \ (i = 1,2)$ are great antipodal sets of $\tilde{G}_{4}(\mathbb{R}^{12})$. 
Moreover, any maximal antipodal set of $\tilde{G}_{4}(\mathbb{R}^{12})$ is congruent to one of them.

\end{thm}

Next, we recall the maximal antipodal sets of $SU(8)/\mathbb{Z}_{2} = SU(8)/\{ \pm I_{8} \}$.
Let $\pi_{SU(8)} : SU(8) \rightarrow SU(8)/\mathbb{Z}_{2}$ be the natural projection.
It is well known that a polar of the identity element of $SU(8)$ is a pole or isomorphic to one of $G_{2}(\mathbb{C}^{8}), G_{4}(\mathbb{C}^{8}), G_{6}(\mathbb{C}^{8})$.
Hence, the polar of the identity element isomorphic to $G_{4}(\mathbb{C}^{8})$ is denoted by $G_{4}(\mathbb{C}^{8})_{+}$.
Then,
\[
G_{4}(\mathbb{C}^{8})_{+} = \bigcup_{g \in SU(8)}g \mathrm{d}(1,1,1,1,-1,-1,-1,-1) g^{-1}.
\]
Set $G_{4}^{*}(\mathbb{C}^{8}) = \pi_{SU(8)}(G_{4}(\mathbb{C}^{8})_{+})$.
Then, $G_{4}^{*}(\mathbb{C}^{8})$ is a polar of the identity element in $SU(8)/\mathbb{Z}_{2}$.
Set
\[
D_{4} = \left\{ 
\pm 1_{2} = \pm \begin{pmatrix} 1 & 0 \\ 0 & 1 \\ \end{pmatrix}, \ 
\pm I_{2} = \pm \begin{pmatrix} -1 & 0 \\ 0 & 1 \\ \end{pmatrix}, \ 
\pm J_{2} = \pm \begin{pmatrix} 0 & -1 \\ 1 & 0 \\ \end{pmatrix}, \ 
\pm K_{2} = \pm \begin{pmatrix} 0 & 1 \\ 1 & 0 \end{pmatrix}
\right\}.
\]
The set of all diagonal matrices of $O(n)$ is denoted by $\Delta(n)$ and set $\Delta^{\pm}(n) = \{ A \in \Delta(n) \ ; \ \mathrm{det}A = \pm 1 \}$.
Define subgroups of $U(8)$ and $SU(8)$ as follows:
\[
\begin{split}
& D_{0} = \{ 1, i \}\Delta(8), \quad
D_{1} = \{ 1, i \}D(4) \otimes \Delta(4), \quad
D_{2} = \{ 1, i \}D(4) \otimes D(4) \otimes D(4), \\
& D_{i}^{+} = \{ A \in D_{i} \ ;\ \det A = 1 \} \ (i = 0,1,2).
\end{split}
\]
Set $A_{i}(SU(8)/\mathbb{Z}_{2}) = \pi_{SU(8)}( D_{i}^{+} )\ (i = 1,2,3)$.

\begin{thm} \cite{Tanaka-Tasaki2}\label{Bottom-U(8)}
The subgroups $A_{i}(SU(8)/\mathbb{Z}_{2}) \ (i = 1,2,3)$ are maximal antipodal subrgoup of $SU(8)/\mathbb{Z}_{2}$ and they are not congruent to each other.
Moreover, any maximal antipodal subgroup is conjugate to one of them.

\end{thm}

Set $A_{i}(G_{4}^{*}(\mathbb{C}^{8})) \ (i = 1,2,3)$ as follows;
\[
\begin{split}
A_{i}(G_{4}^{*}(\mathbb{C}^{8})) = \pi_{SU(8)} \Big( D_{i}^{+} \cap G_{4}(\mathbb{C}^{8})_{+}  \Big) = A_{i}(SU(8)/\mathbb{Z}_{2}) \cap G_{4}^{*}(\mathbb{C}^{8}).
\end{split}
\]
Then,
$\# A_{1}(G_{4}^{*}(\mathbb{C}^{8}) = 35, \# A_{2}(G_{4}^{*}(\mathbb{C}^{8}) = 27$, and $\# A_{3}(G_{4}^{*}(\mathbb{C}^{8}) = 63$.

\begin{thm}\cite{Tanaka-Tasaki3}\label{Bottom-Grass}
The subsets $A_{i}(G_{4}^{*}(\mathbb{C}^{8})) \ (i =1,2,3)$ are maximal antipodal sets of $G_{4}^{*}(\mathbb{C}^{8})$ and not congruent to each other.
The cadinalities are $35, 27, 63$.
Moreover, any maximal antipodal set of $G_{4}^{*}(\mathbb{C}^{8})$ is congruent to one of them.

\end{thm}

\newpage

\section{Classification}

In this section, we classify the congruent classes of maximal antipodal sets of the simpley connected compact exceptional symmetric spaces $EV, EVI, EVIII, EIX$ and the compact Lie group $E_{7}$.

\subsection{Strategies} \label{strategies}

In this subsection, we consider a method to calssify maximal antipoldal sets later.
Let $M$ be a compact symmetric space and fix $o \in M$.
Let $M^{+}$ be a polars of $o$ and $A(M^{+})$ be a maximal antipodal set of $M^{+}$.
Since $\{ o \} \cup A(M^{+})$ is an antipodal set of $M$ by the definition of a polar, there exists a maximal antipodal set $A(M)$ of $M$ such that $\{ o \} \cup A(M^{+}) \subset A(M)$.
Note that $A(M) \subset F(s_{o}, M)$ since $o \in A(M)$.
By the maximality of $A(M^{+})$, since $A(M) \cap M^{+}$ is an antipodal set of $M^{+}$, we obtain
\[
A(M) \cap M^{+} = A(M_{i}^{+}).
\]
Therefore, any maximal antipodal set of $M^{+}$ is obtained by the intersection between a maximal antipodal set of $M$ and $M^{+}$.
Let $A_{j}(M) \ (1 \leq j \leq l)$ be maximal antipodal sets of $M$.
Assume that for any $1 \leq i \not= j \leq l$, there does not exist $k \in K_{o}$ such that $k(A_{i}(M)) = A_{j}(M)$, where $K_{o}$ is the isotropy group of the idenitity component of the isometry group at $o$. 
Moreover, assume that for any maximal antipodal set $A$ of $M$ containing $o$, there exist some $k \in K_{o}$ and $1 \leq j \leq l$ such that $A = k(A_{j}(M))$.
Set
\[
A_{j}(M^{+}) = A_{j}(M) \cap M^{+} \ (1 \leq j \leq l).
\]
In general, each $A_{j}(M^{+})$ is not a maximal antipodal set of $M^{+}$.
Set $1 \leq j_{1}, < \cdots < j_{m} < l$ such that $A_{j_{s}}(M^{+}) \ (1 \leq s \leq m)$ is a maximal antipodal set of $M^{+}$.
Then, any maximal antipodal set of $M^{+}$ is congruent to one of $A_{j_{s}}(M^{+}) \ (1 \leq s \leq m)$.
Therefore, to classify the congruent classes of maximal antipodal sets of $M^{+}$, it is sufficient to decide $A_{j_{s}}(M^{+}) \ (1 \leq s \leq m)$ and study the congruent classes of $A_{j_{s}}(M^{+}) \ (1 \leq s \leq m)$.

Next, we consider special cases like $E_{7}$ and $EV$.
Assume that the number of polars is $4$, that is, denoting the polars except for the trivial pole by $M_{i}^{+}\ (1 \leq i \leq 3)$
\[
F(s_{o}, M) = \{ o \} \sqcup M_{1}^{+} \sqcup M_{2}^{+} \sqcup M_{3}^{+}.
\]
Moreover, assume that $M_{3}^{+}$ is a pole and set $M_{3}^{+} = \{ p \}$.
If $M$ has a non-trivial pole, then it is known that there exists a covering transformation $\tau$ of $M$ such that $\tau(p)$ is a non-trivial pole of $p$ for any $p \in M$.
It is well known that $\tau$ commutes any isometry $g$ of $M$ and $s_{p} = s_{\tau(p)}$ for any $p \in M$.
Therefore, $\tau$ leaves any maximal antipodal set of $M$ invariant.
Assume that $\tau(M_{1}^{+}) = M_{2}^{+}$.
Let $A_{u}(M_{1}^{+}) \ (1 \leq u \leq p)$ be maximal antipodal sets of $M_{1}^{+}$, and assume that $A_{u}(M_{1}^{+})$ are not congruent to each other and any maximal antipodal set of $M_{1}^{+}$ is congruent to one of $A_{u}(M_{1}^{+})$.
For any $1 \leq u \leq p$, set
\[
A_{u}(M) = \{ o \} \cup A_{u}(M_{1}^{+}) \cup \tau \big( A_{u}(M_{1}^{+}) \big)  \cup \{ p \}.
\]
Then, $A_{u}(M)$ are maximal antipodal sets of $M$.
Note that there does not exist any element $k$ of the identity component of $K_{o}$ such that $k(A_{i}(M)) = A_{j}(M)$ for any $1 \leq i \not= j \leq p$.
Moreover, any maximal antipodal set of $M$ is congruent to one of $A_{u}(M) \ (1 \leq u \leq p)$.
Therefore, to classify the congruent classes of maximal antipodal set of $M$, it is sufficient to classify $A_{u}(M_{1}^{+}) \ (1 \leq u \leq p)$ and study whether there exist an element of $g \in G$ such that $g(A_{i}(M)) \subset A_{j}(M)$ for any $1 \leq i \not= j \leq p$.


\subsection{Maximal antipodal sets of $EVIII$ and $EIX$}\label{s-EVIIIEIX}

In this subsection, we classify the congruent classes of maximal antipodal sets of $EVIII$ and $EIX$.
Frist, we recall polars of $E_{8}$.
It is well known that the number of polars except for the trivial pole is $2$.
One of polars is isomorphic to the compact symmetric space of type $EVIII$ and the other is isomorphic to type $EIX$.
Hence, the polar of the identity element isomorphic to $EIX$ is denoted by $EIX_{+}$ and the other polar is denoted by $EVIII_{+}$.
Then, 
\[
\begin{split}
& EIX_{+} = \{ g \in E_{8} \ ;\ g^{2} = e, \ \dim F^{+}(\mathrm{Ad}(g), \frak{e}_{8}) = 136 \}, \\
& EVIII_{+} = \{ g \in E_{8} \ ;\ g^{2} = e, \ \dim F^{+}(\mathrm{Ad}(g), \frak{e}_{8}) = 120 \}. \\
\end{split}
\]
Fix $\alpha \in \Sigma^{+}(\frak{e}_{8}, \frak{t})$.
Then,
\[
\begin{split}
& EIX_{+} = \bigcup_{g \in E_{8}}g \tau_{\alpha} g^{-1} \cong E_{8}/F(\sigma_{IX}^{\alpha}, E_{8}) \cong E_{8}/SU(2) \cdot E_{7}, \\
& EVIII_{+} = \bigcup_{g \in E_{8}}g \phi(-1) g^{-1} \cong E_{8}/F(\sigma_{VIII}, E_{8}) \cong E_{8}/Ss(16).
\end{split}
\]
Note that $EIX_{+}$ is the closest polar to the identity element and $EVIII_{+}$ is the farthest polar from the identity element.
To classify the maximal antipodal sets of $EIX_{+}$ and $EVIII_{+}$, we consider $A_{i}(E_{8}) \cap EIX_{+}$ and $A_{i}(E_{8}) \cap EVIII_{+} \ (i = 1,2)$.

First, we consider $A_{1}(E_{8}) \cap EIX_{+}$ and $A_{1}(E_{8}) \cap EVIII_{+}$.
For any $g = \psi(a,b) \in A_{1}(E_{8})$, since $\mathrm{Ad}(g)$ leaves $L$ and $V$ invariant respectively, we obtain $F^{+}(\mathrm{Ad}(g), \frak{e}_{8}) = F^{+}(\mathrm{Ad}(g)|_{L}, L) + F^{+}(\mathrm{Ad}(g)|_{V}, V)$.
Then,
\[
\begin{split}
F^{+}(\mathrm{Ad}(g)|_{V}, V) 
&= F^{+}(\Delta_{16}^{+}(\psi(a,b)), V) \\
&= F^{+}(\Delta_{8}^{+}(a), V_{8}^{+}) \otimes F^{+}(\Delta_{8}^{+}(b), V_{8}^{+}) + F^{-}(\Delta_{8}^{+}(a), V_{8}^{+}) \otimes F^{-}(\Delta_{8}^{+}(b), V_{8}^{+}) \\
&+ F^{+}(\Delta_{8}^{-}(a), V_{8}^{-}) \otimes F^{+}(\Delta_{8}^{-}(b), V_{8}^{-}) + F^{-}(\Delta_{8}^{-}(a), V_{8}^{-}) \otimes F^{-}(\Delta_{8}^{-}(b), V_{8}^{-}). \\
\end{split}
\]
For any $\gamma_{i}^{a,b} \in A(Spin(8))$, we see that
\[
\begin{matrix}
\begin{array}{lll} \vspace{1mm}
\dim F^{+}(\Delta_{8}^{+}(\gamma_{0}^{0,0}), V_{8}^{+}) = \dim F^{+}(\Delta_{8}^{-}(\gamma_{0}^{0,0}), V_{8}^{-}) = 8, &
\dim F^{-}(\Delta_{8}^{+}(\gamma_{0}^{0,0}), V_{8}^{+}) = \dim F^{-}(\Delta_{8}^{-}(\gamma_{0}^{0,0}), V_{8}^{-}) = 0, \\\vspace{1mm}
\dim F^{-}(\Delta_{8}^{+}(\gamma_{0}^{1,0}), V_{8}^{+}) = \dim F^{-}(\Delta_{8}^{-}(\gamma_{0}^{1,0}), V_{8}^{-}) = 8, &
\dim F^{+}(\Delta_{8}^{+}(\gamma_{0}^{1,0}), V_{8}^{+}) = \dim F^{+}(\Delta_{8}^{-}(\gamma_{0}^{1,0}), \Delta_{8}^{-}) = 0, \\\vspace{1mm}
\dim F^{+}(\Delta_{8}^{+}(\gamma_{0}^{0,1}), V_{8}^{+}) = \dim F^{-}(\Delta_{8}^{-}(\gamma_{0}^{0,1}), V_{8}^{-}) = 8, &
\dim F^{-}(\Delta_{8}^{+}(\gamma_{0}^{0,1}), V_{8}^{+}) = \dim F^{+}(\Delta_{8}^{-}(\gamma_{0}^{0,1}), \Delta_{8}^{-}) = 0, \\\vspace{1mm}
\dim F^{-}(\Delta_{8}^{+}(\gamma_{0}^{1,1}), V_{8}^{+}) = \dim F^{+}(\Delta_{8}^{-}(\gamma_{0}^{1,1}), V_{8}^{-}) = 8, &
\dim F^{+}(\Delta_{8}^{+}(\gamma_{0}^{1,1}), V_{8}^{+}) = \dim F^{-}(\Delta_{8}^{-}(\gamma_{0}^{1,1}), V_{8}^{-}) = 0, \\
\end{array}
\\
\dim F^{\pm}(\Delta_{8}^{\pm}(\gamma_{i}^{a,b}), V_{8}^{\pm}) = 4 \ \  (1 \leq i \leq 7, a,b = 0,1).
\end{matrix}
\]
For $F^{+}(\mathrm{Ad}(g)|_{L}, L)$, by direct calculations, we obtain
\[
\dim F^{+}( \mathrm{Ad}(\psi(\gamma_{i}^{0,0}, \gamma_{j}^{a,b}))|_{L}, L) =
\begin{cases}
72 & (i = 0, \ j \not= 0, \ a,b = 0,1 \ \text{or} \ i \not= 0, \ j = 0, \ a,b = 0,1 ), \\
56 & (i,j \not= 0, \ a,b = 0,1 \ \text{or} \ i = j = 0, \ (a,b) = (0,1), (1,1)), \\
120 & (i = j = 0, (a,b) = (0,0), (1,1)). \\
\end{cases}
\]
Therefore, we obtain
\[
\begin{split}
& A_{1}(E_{8}) \cap EIX_{+} = \Big\{ \psi( \gamma_{0}^{0,0}, \gamma_{i}^{a,b} ), \psi( \gamma_{i}^{0,0}, \gamma_{0}^{a,b}) \ ;\ 1 \leq i \leq 7, a,b = 0,1 \Big\}, \\
& A_{1}(E_{8}) \cap EVIII_{+} = \Big\{ \psi( \gamma_{i}^{0,0}, \gamma_{j}^{a,b} ) \ ;\ 1 \leq i, j \leq 7, a,b = 0,1 \Big\} \cup \{ \psi(\gamma_{0}^{0,0}, \gamma_{0}^{1,0}), \psi(\gamma_{0}^{0,0}, \gamma_{0}^{0,1}), \psi(\gamma_{0}^{0,0}, \gamma_{0}^{1,1}) \Big\}.
\end{split}
\]
The cardinalities of $A_{1}(E_{8}) \cap EIX_{+}$ and $A_{1}(E_{8}) \cap EVIII_{+}$ are $56$ and $199$.
Next, we consider $A_{2}(E_{8}) \cap EIX_{+}$ and $A_{2}(E_{8}) \cap EVIII_{+}$.
For any $\alpha \in \Sigma^{+}$ and $\beta \in \Sigma_{\alpha, 0}^{+}$, we obtain
\[
\begin{split}
& F^{+}(\mathrm{Ad}(x), \frak{e}_{8}) = \sum_{\gamma \in \Sigma^{+}}\frak{r}_{\gamma}^{+}, \\
& F^{+}(\mathrm{Ad}(\tau_{\alpha}), \frak{e}_{8}) = \frak{t} + \frak{r}_{\alpha} + \sum_{\gamma \in \Sigma_{\alpha, 0}^{+}}\frak{r}_{\gamma}, \\
& F^{+}(\mathrm{Ad}(\tau_{\alpha}x), \frak{e}_{8}) = \frak{r}_{\alpha}^{+} + \sum_{\gamma \in \Sigma_{\alpha, 0}^{+}}\frak{\gamma}^{+} + \sum_{\delta \in \Sigma_{\alpha, 1}}\frak{r}_{\delta}^{-}, \\
& F^{+}(\mathrm{Ad}(\tau_{\alpha}\tau_{\beta}), \frak{e}_{8}) = \frak{t} + \frak{r}_{\alpha} + \frak{r}_{\beta} + \sum_{\gamma \in \Sigma_{\alpha, 0}^{+} \cap \Sigma_{\beta, 0}}\frak{r}_{\gamma} + \sum_{\delta \in \Sigma_{\alpha, 1} \cap \Sigma_{\beta, \pm1}}\frak{r}_{\delta}, \\
&  F^{+}(\mathrm{Ad}(\tau_{\alpha}\tau_{\beta}x), \frak{e}_{8}) = \frak{r}_{\alpha}^{+} + \frak{r}_{\beta}^{+} 
+ \sum_{\gamma \in \Sigma_{\alpha, 0}^{+} \cap \Sigma_{\beta, 0}}\frak{r}_{\gamma}^{+} 
+ \sum_{\delta \in \Sigma_{\alpha, 1} \cap \Sigma_{\beta, \pm1}}\frak{r}_{\delta}^{+} 
+ \sum_{ \epsilon \in \Sigma_{\alpha, 0} \cap \Sigma_{\beta, 1}}\frak{r}_{\epsilon}^{-} 
+ \sum_{\zeta \in \Sigma_{\alpha, 1} \cap \Sigma_{\beta, 0}}\frak{r}_{\zeta}^{-}.
\end{split}
\]
Fix $\alpha \in \Sigma^{+}$ and $\beta \in \Sigma_{\alpha, 1}^{+}$.
Then,
\[
\begin{split}
& A_{2}(E_{8}) \cap EIX_{+} = \big\{ \tau_{\gamma} \ ;\ \gamma \in \Sigma^{+} \big\}, \\
& A_{2}(E_{8}) \cap EVIII_{+} = xA(T) \cup \big\{ \tau_{\gamma}\tau_{\alpha} \ ;\ \gamma \in \Sigma_{\alpha, 0}^{+} \big\} \cup \big\{ \tau_{\delta}\tau_{\beta}, \tau_{\delta}\tau_{\alpha + \beta} \ ;\ \delta \in \Sigma_{\alpha,0}^{+} \cap \Sigma_{\beta, 0} \big\}.
\end{split}
\]
The cardinalities of $A_{2}(E_{8}) \cap EIX_{+}$ and $A_{2}(E_{8}) \cap EVIII_{+}$ are $120$ and $391$.

\begin{thm} \label{EIX}
Set $A_{i}(EIX_{+}) = A_{i}(E_{8}) \cap EIX_{+} \ (i = 1,2)$.
Then, $A_{1}(EIX_{+})$ and $A_{2}(EIX_{+})$ are maximal antipodal sets of $EIX_{+}$ whose cardinalities are $56$ and $120$.
Moreover, any maximal antipodal set of $EIX_{+}$ is congruent to one of them, and $\#_{2}EIX = 120$.

\end{thm}

\begin{proof}
By the argument of subsection \ref{strategies}, it is sufficient to show that there does not exist $g \in E_{8}$ such that $g(A_{1}(EIX_{+}))g^{-1} \subset A_{2}(EIX_{+})$.
However, we see that $A_{1}(EIX_{+})$ generates $A_{1}(E_{8})$.
Therefore, if there exists such $g \in E_{8}$, then $gA_{1}(E_{8})g^{-1} \subset A_{2}(E_{8})$ and this contradicts to the maximality of $A_{1}(E_{8})$.
Hence, there does not exist $g \in E_{8}$ such that $g(A_{1}(EIX_{+}))g^{-1} \subset A_{2}(EIX_{+})$.
\end{proof}

Note that 
\[
A_{1}(EIX_{+}) = \psi \Big( A \big( \tilde{G}_{4}(\mathbb{R}^{8})_{+} \big) \times \{ \gamma_{0}^{0,0} \} \Big) \cup \psi \Big( \{ \gamma_{0}^{0,0} \} \times A \big( \tilde{G}_{4}(\mathbb{R}^{8})_{+} \big) \Big).
\]
Therefore, $A_{1}(EIX_{+})$ is constructed by a maximal antipodal set of $\tilde{G}_{4}(\mathbb{R}^{8})$.
Let $W(E_{8})$ be the Weyl group of the root system $\Sigma(\frak{e}_{8}, \frak{t})$.
Since $W(E_{8})$ acts on $\Sigma(\frak{e}_{8}, \frak{t})$ transitively and this action is transitive, we see $W(E_{8})$ acts on $A_{2}(EIX)$ and this action is transitive.

\begin{thm} \label{EVIII}
Set $A_{i}(EVIII_{+}) = A_{i}(E_{8}) \cap EVIII_{+} \ (i = 1,2)$.
Then, $A_{1}(EVIII_{+})$ and $A_{2}(EVIII_{+})$ are maximal antipodal sets of $EVIII_{+}$ and the cardinalities are $199$ and $391$.
Moreover, any maximal antipodal set of $EVIII_{+}$ is congruent to one of them and $\#_{2}EVIII = 391$.

\end{thm}

\begin{proof}
Since $A_{i}(EVIII_{+})$ generates $A_{i}(E_{8}) \ (i = 1,2)$, we can prove the statement by the similar argument to the proof of Theorem \ref{EIX}.
\end{proof}

Note that
\[
A_{1}(EVIII_{+}) = \psi \Big( A \big( \tilde{G}_{4}(\mathbb{R}^{8})_{+} \big) \times A \big( \tilde{G}_{4}(\mathbb{R}^{8})_{+} \big) \Big) \cup \{ \psi( \gamma_{0}^{0,0}, \gamma_{0}^{a,b}) \ ;\ (a,b) = (1,0),(0,1),(1,1) \}.
\]
Therefore, $A_{1}(EVIII_{+})$ is constructed by a maximal antipodal set of $\tilde{G}_{4}(\mathbb{R}^{8})$.
In $A_{2}(EVIII_{+})$, since $t x t^{-1} = t^{2} t$ for any $t \in T$, we see that $xT \subset EVIII_{+}$ is a maximal flat torus.
Moreover, $x( A(T) )$ is a maximal antipodal set of $xT$.
We can easily check that the Weyl group $W(E_{8})$ acts on
\[
\{ \pm \delta \pm \alpha \ ;\ \delta \in \Sigma_{\alpha, 0}^{+} \} \cup \{ \pm \delta \pm \beta, \pm \delta \pm (\alpha + \beta) \ ;\ \delta \in \Sigma_{\alpha,0}^{+} \cup \Sigma_{\beta, 0} \}
\]
and this action is transitive.
Therefore, $W(E_{8})$ acts on $A_{1}(EVIII_{+}) - xA(T)$ and this action is transitive.
Hence, $A_{2}(EVIII_{+})$ is the union of a maximal antipodal set of a maximal flat torus and an orbit of $W(E_{8})$.


\subsection{Maximal antipodal sets of $EVI$} \label{s-EVI}

In this subsection, we classify the congruent classes of maximal antipodal sets of $EVI$.
Let $\alpha = x_{7} - x_{8} \in \Sigma^{+}(\frak{e}_{8}, \frak{t})$ and consider $E_{7}^{\alpha}$.
Then, $\tau_{\alpha} \in EIX$.
We denote $\tau_{\alpha}$ by $p$ in this subsection.
Then,
\[
\{ g \in E_{8} \ ;\ gpg^{-1} = p \} = F(\theta_{\tau_{\alpha}}, E_{8}) = SU^{\alpha}(2) \cdot E_{7}^{\alpha}.
\]
Therefore, any polar of $p$ in $EIX_{+}$ is an $(SU^{\alpha}(2) \cdot E_{7}^{\alpha})$-orbit.
It is well known that the number of polars except for the trivial pole in $EIX_{+}$ is $2$ and one of them is isomorphic to $S^{2} \cdot EVII$ and the other is isomorphic to $EVI$.
Denote the polar of $p$ isomorphic to $S^{2} \cdot EVII$ by $(S^{2} \cdot EVII)_{+}$ and the other is denoted by $EVI_{+}$.
Then,
\[
\begin{split}
(S^{2} \cdot EVII)_{+} &= \left\{ q \in F(s_{p}, EIX) \ ;\ \dim \Big( F^{+}(\mathrm{Ad}(p), \frak{e}_{8}) \cap F^{+}(\mathrm{Ad}(q), \frak{e}_{8}) \Big) = 80 \right\} \\
&= \left\{ q \in F(s_{p}, EIX) \ ;\ \dim F^{+}(\mathrm{Ad}(q), \frak{su}^{\alpha}(2) + \frak{e}_{7}^{\alpha}) = 80 \right\}, \\
EVI_{+} &= \left\{ q \in F(s_{p}, EIX) \ ;\ \dim \Big( F^{+}(\mathrm{Ad}(p), \frak{e}_{8}) \cap F^{+}(\mathrm{Ad}(q), \frak{e}_{8}) \Big) = 72 \right\} \\
&= \left\{ q \in F(s_{p}, EIX) \ ;\ \dim F^{+}(\mathrm{Ad}(q), \frak{su}^{\alpha}(2) + \frak{e}_{7}^{\alpha}) = 72 \right\}. \\
\end{split}
\]
Moreover, if $\beta \in \Sigma_{\alpha, \pm1}^{+}$ and $\gamma \in \Sigma_{\alpha,0}^{+} = \Sigma^{+}(\frak{e}_{7}^{\alpha}, \frak{t}^{\alpha})$, then
\[
\begin{split}
(S^{2} \cdot EVII)_{+} &= \bigcup_{g \in SU^{\alpha}(2) \cdot E_{7}^{\alpha}}g \tau_{\beta} g^{-1}, \quad\quad
EVI_{+} = \bigcup_{g \in SU^{\alpha}(2) \cdot E_{7}^{\alpha}}g \tau_{\gamma} g^{-1}.
\end{split}
\]
Note that $EVI_{+} \subset E_{7}^{\alpha}$.
Moreover, $EVI_{+}$ is a polar of the identity element in $E_{7}^{\alpha}$.
To classify maximal antipodal sets of $EVI$, we consider $A_{i}(EIX_{+}) \cap EVI_{+} \ (i = 1,2)$.
First, we consider $A_{1}(EIX_{+}) \cap EVI_{+}$.
Since $p = \tau_{\alpha} = \psi(\gamma_{0}^{0,0}, \gamma_{1}^{0,0})$, we see
\[
F^{+}(\mathrm{Ad}(p), \frak{e}_{8}) = \frak{spin}(12) + \frak{spin}(4) + \mathbb{O} \otimes^{+} \mathbb{H} + \mathbb{O} \otimes^{-} \mathbb{H},
\]
where $\frak{spin}(12)$ and $\frak{spin}(4)$ are defined by the Clifford algbera over $\mathrm{span}_{\mathbb{R}}\{ e_{0}, \cdots, e_{11}\}$ and $\mathrm{span}_{\mathbb{R}}\{ e_{12}, \cdots, e_{15}\}$ respectively, and $\mathbb{H} = \mathrm{span}_{\mathbb{R}}\{ e_{0}, e_{1}, e_{2}, e_{3} \}$.
Therefore, for any $\psi(a,b) \in A_{1}(EIX_{+})$, 
\[
\begin{split}
F^{+}(\mathrm{Ad}(\psi(a,b)), \frak{su}^{\alpha}(2) + \frak{e}_{7}^{\alpha}) &= F^{+} \big( \mathrm{Ad}(\psi(a,b)), \frak{spin}(12) + \frak{spin}(4) \big) + F^{+} \big( \mathrm{Ad}(\psi(a,b)), \mathbb{O} \otimes^{+} \mathbb{H} + \mathbb{O} \otimes^{-} \mathbb{H} \big), \\
\end{split}
\]
Then,
\[
\begin{split}
F^{+} \big( \mathrm{Ad}(\psi(a,b)), & \ \mathbb{O} \otimes^{+} \mathbb{H} + \mathbb{O} \otimes^{-} \mathbb{H} \big) \\
&= F^{+}(\Delta_{8}^{+}(a), \mathbb{O}) \otimes^{+} F^{+}(\Delta_{8}^{+}(b), \mathbb{H}) + F^{-}(\Delta_{8}^{+}(a), \mathbb{O}) \otimes^{+} F^{-}(\Delta_{8}^{+}(b), \mathbb{H}) \\
&+ F^{+}(\Delta_{8}^{-}(a), \mathbb{O}) \otimes^{-} F^{+}(\Delta_{8}^{-}(b), \mathbb{O}) + F^{-}(\Delta_{8}^{-}(a),\mathbb{O}) \otimes^{-} F^{-}(\Delta_{8}^{-}(b), \mathbb{H}).
\end{split}
\]
For any $\psi(\gamma_{i}^{0,0}, \gamma_{j}^{a,b}) \in A_{1}(EIX)$,
\[
F^{+} \big( \psi(\gamma_{i}^{0,0}, \gamma_{j}^{a,b}), \frak{spin}(12) + \frak{spin}(4) \big) \cong
\begin{cases}
\frak{spin}(10) + \mathbb{R}^{3} & (i = 0, \ 2 \leq j \leq 7, \ a,b = 0,1), \\
\frak{spin}(12) + \frak{spin}(4) & (i = 0, \ j = 1, \ (a,b) = (1,0)), \\
\frak{spin}(8) + \frak{spin}(4) + \frak{spin}(4) & (\text{the others}). \\
\end{cases}
\]
Therefore, 
\[
\dim F^{+} \big( \psi(\gamma_{i}^{0,0}, \gamma_{j}^{a,b}), \frak{spin}(12) + \frak{spin}(4) \big) \cong
\begin{cases}
48 & (i = 0, \ 2 \leq j \leq 7, \ a,b = 0,1), \\
72 & (i = 0, \ j = 1, \ (a,b) = (1,0)), \\
40 & (\text{the others}). \\
\end{cases}
\]
Moreover, we can verify that
\[
\begin{split}
& F^{+} \big( (\gamma_{i}^{0,0}, \gamma_{j}^{a,b}), \mathbb{O} \otimes^{+} \mathbb{H} + \mathbb{O} \otimes^{-} \mathbb{H} \big)  = 
\begin{cases}
\mathbb{O} \otimes^{+} V_{1} + \mathbb{O} \otimes^{-} V_{2} & (i = 0, 2 \leq j \leq 7, a,b = 0,1), \\
\{ 0 \} & (i = 0, j = 1, (a,b) = (1,0)), \\
\mathbb{O} \otimes^{+} \mathbb{H} & (i = 0, j = 1, (a,b) = (0,1)), \\
\mathbb{O} \otimes^{-} \mathbb{H} & (i = 0, j = 1, (a,b) = (1,1)), \\
W_{1} \otimes^{+} \mathbb{H} + W_{2} \otimes^{-} \mathbb{H} & (\text{the others}),
\end{cases}
\end{split}
\]
where $V_{1}$ and $V_{2}$ are $2$-dimensional subspaces of $\mathbb{H}$ and $W_{1}$ and $W_{2}$ are $4$-dimensional subspaces of $\mathbb{O}$.
Hence, 
\[
\dim F^{+} \big( (\gamma_{i}^{0,0}, \gamma_{j}^{a,b}), \mathbb{O} \otimes^{+} \mathbb{H} + \mathbb{O} \otimes^{-} \mathbb{H} \big) =
\begin{cases}
0 & (i = 0, j = 1, (a.b) = (1,0)), \\
32 & (\text{the others}). \\
\end{cases}
\]
Therefore, we obtain that
\[
\begin{split}
& A_{1}(EIX_{+}) \cap (S^{2} \cdot EVII)_{+} = \left\{ \psi(\gamma_{0}^{0,0}, \gamma_{i}^{a,b}) \ ;\ 2 \leq i \leq 7, a,b = 0,1 \right\}, \\
& A_{1}(EIX_{+}) \cap EVI_{+} = \left\{ \psi(\gamma_{i}^{0,0}, \gamma_{0}^{a,b}) \ ;\ 1 \leq i \leq 7, a,b = 0,1 \right\} \\
& \hspace{40mm} \cup \left\{ \psi(\gamma_{0}^{0,0}, \gamma_{1}^{0,1}), \psi(\gamma_{0}^{0,0}, \gamma_{1}^{1,0}), \psi(\gamma_{0}^{0,0}, \gamma_{1}^{1,1}) \right\}.
\end{split}
\]
Next, we consider $A_{2}(EIX_{+}) \cap EVI_{+}$.
Let $\beta = \Sigma_{\alpha, 1}$.
Then,
\[
F^{+}( \mathrm{Ad}(\tau_{\delta}), \frak{su}^{\alpha}(2) + \frak{e}_{7}^{\alpha} ) \cong \mathbb{R} + F(\sigma_{VII}^{\beta}, \frak{e}_{7}^{\alpha}) \cong \mathbb{R} + \mathbb{R} + \frak{e}_{6}^{\alpha, \beta}.
\]
Therefore, $\dim F^{+}( \mathrm{Ad}(\tau_{\delta}), \frak{su}^{\alpha}(2) + \frak{e}_{7}^{\alpha} ) = 80$ and $\tau_{\beta} \in (S^{2} \cdot EVII)_{+}$.
Let $\gamma \in \Sigma_{\alpha, 0}^{+}$.
Then,
\[
F^{+}( \mathrm{Ad}(\tau_{\delta}), \frak{su}^{\alpha}(2) + \frak{e}_{7}^{\alpha} ) = F(\sigma^{\beta}_{VII}, \frak{e}_{7}^{\alpha}) \cong \frak{su}(2) + \frak{spin}(12).
\]
Hence, $\dim F^{+}( \mathrm{Ad}(\tau_{\gamma}), \frak{su}^{\alpha}(2) + \frak{e}_{7}^{\alpha} ) = 72$ and $\tau_{\gamma} \in EVI_{+}$.
Thus, 
\[
\begin{split}
A_{2}(EIX_{+}) \cap (S^{2} \cdot EVI_{+}) = \{ \tau_{\beta} \ ;\ \beta \in \Sigma_{\alpha, 1} \}, \quad 
A_{2}(EIX_{+}) \cap EVI_{+} = \{ \tau_{\gamma} \ ;\ \gamma \in \Sigma_{\alpha, 0}^{+} = \Sigma(\frak{e}_{7}^{\alpha}, \frak{t}^{\alpha}) \}.
\end{split}
\]
Then, the Weyl group $W(E_{7}^{\alpha})$ acts on $A_{2}(EIX_{+}) \cap (S^{2} \cdot EVII)_{+}$ and $A_{2}(EIX_{+}) \cap EVI_{+}$, and these actions are transitive.
We set
\[
A_{i}(EVI_{+}) = A_{i}(EIX_{+}) \cap EVI_{+} \quad (i = 1,2).
\]
The cardinalities of $A_{1}(EVI_{+})$ and $A_{2}(EVI_{+})$ are $31$ and $63$.

Next, we study whether there exists $k \in E_{7}^{\alpha}$ such that $kA_{1}(EVI_{+})k^{-1} \subset A_{2}(EVI_{+})$.
We consider polars of $EVI_{+}$.
It is well known that the number of the polars except for the trivial pole is $2$ and one of them is isomorphic to $S^{2} \cdot DIII(3)$ and the other is isomorphic to $\tilde{G}_{4}(\mathbb{R}^{12})$.
Set $q = \psi(\gamma_{0}^{0,0}, \gamma_{1}^{1,0}) = \tau_{\bar{\alpha}} = \tau_{x_{7} + x_{8}}$.
The polar of $q$ isomorphic to $S^{2} \cdot DIII(3)$ is denoted by $(S^{2} \cdot DIII(3))_{+}$ and the other polar isomorphic to $\tilde{G}_{4}(\mathbb{R}^{12})$ is denoted by $\tilde{G}_{4}(\mathbb{R}^{12})_{+}$.
The isotropy group of $E_{7}^{\alpha}$ at $q$ is given by
\[
F^{+}(\sigma^{\bar{\alpha}}_{VI}, E_{7}^{\alpha}) = SU^{\bar{\alpha}}(2) \cdot Spin^{\alpha}(12).
\]
Then, $\tilde{G}_{4}(\mathbb{R}^{12})_{+}$ is a polar of the identity element of $Spin^{\alpha}(12) \subset E_{7}^{\alpha}$.
Let $r \in EVI_{+}$ be $s_{q}(r) = r$.
Then,
\[
\begin{array}{ccl}
r \in (S^{2} \cdot DIII(3))_{+} & \Longleftrightarrow &
F^{+}(\mathrm{Ad}(r), \frak{su}^{\bar{\alpha}}(2) + \frak{spin}^{\alpha}(12)) \cong \frak{su}(6) + \mathbb{R}^{2}, \\
r \in (\tilde{G}_{4}(\mathbb{R}^{8}))_{+} & \Longleftrightarrow &
F^{+}(\mathrm{Ad}(r), \frak{su}^{\bar{\alpha}}(2) + \frak{spin}^{\alpha}(12)) \cong \frak{spin}(8) + \frak{spin}(4) + \frak{su}(2).
\end{array}
\]
Note that $\dim (\frak{su}(6) + \mathbb{R}^{2}) = \dim (\frak{so}(8) + \frak{so}(4) + \frak{su}(2)) = 37$.
We study $A_{i}(EVI_{+}) \cap (S^{2} \cdot DIII(3))_{+}$ and $A_{i}(EVI_{+}) \cap \tilde{G}_{4}(\mathbb{R}^{12})_{+}$ for $i = 1,2$.
For any $\psi(a,b) \in A_{1}(EVI_{+})$, we easily see that
\[
F^{+} \big( \mathrm{Ad}(\psi(a,b)), \frak{su}^{\bar{\alpha}}(2) + \frak{spin}^{\alpha}(12) \big) \cong \frak{spin}(8) + \frak{spin}(4) + \frak{su}(2).
\]
Hence, we obtain
\[
\begin{split}
& A_{1}(EVI_{+}) \cap (S^{2} \cdot DIII(3))_{+} = \phi, \\
& A_{1}(EVI_{+}) \cap \tilde{G}_{4}(\mathbb{R}^{12})_{+} = A_{1}(EVI_{+}) - \{ \psi(\gamma_{0}^{0,0}, \gamma_{1}^{1,0}) \}.
\end{split}
\]
We see that 
\[
\tilde{G}_{4}(\mathbb{R}^{12})_{+} \ni r \mapsto F^{+} \big( \pi(r), \ \mathrm{span}_{\mathbb{R}}\{ e_{0}, \cdots, e_{11} \} \big) \in G_{4}(\mathbb{R}^{12}) 
\]
is a double covering, and $A_{1}(EVI_{+}) \cap \tilde{G}_{4}(\mathbb{R}^{12})_{+}$ is a maximal antipodal set of $\tilde{G}_{4}(\mathbb{R}^{12})$ congruent to $A_{1}(\tilde{G}_{4}(\mathbb{R}^{12}))$ in Theorem \ref{Ori-Grass}.
Next, we consider $A_{2}(EVI_{+}) \cap (S^{2} \cdot DIII(3))_{+}$ and $A_{2}(EVI_{+}) \cap (\tilde{G}_{4}(\mathbb{R}^{12}))_{+}$.
Note that
\[
\Sigma^{+}_{\alpha, 0} = \big\{ x_{i} \pm x_{j} \ ;\ 1 \leq i < j \leq 6 \} \cup \{ x_{7} + x_{8} \} \cup \left\{ \frac{1}{2}\sum_{i=1}^{8}\epsilon_{i}x_{i} \ ;\ \epsilon_{i} = \pm 1, \epsilon_{1} = 1, \epsilon_{7} = \epsilon_{8} \right\}.
\]
Then, we can verify that
\[
\begin{split}
& A_{2}(EVI_{+}) \cap (S^{2} \cdot DIII(3))_{+} = \left\{ \tau_{\gamma} \ ;\ \gamma = \frac{1}{2}(\sum_{i=1}^{8}\epsilon_{i}x_{i}) \in \Sigma_{\alpha, 0}^{+} \right\}, \\
& A_{2}(EVI_{+}) \cap \tilde{G}_{4}(\mathbb{R}^{12})_{+} = \left\{ \tau_{x_{i} \pm x_{j}} \ ;\ 1 \leq i < j \leq 6 \right\}. \\
\end{split}
\]
We see that $A_{2}(EVI_{+}) \cap \tilde{G}_{4}(\mathbb{R}^{12})_{+}$ is a maximal antipodal set of $\tilde{G}_{4}(\mathbb{R}^{12})$ congruent to $A_{2}(\tilde{G}_{4}(\mathbb{R}^{12}))$ in Theorem \ref{Ori-Grass}.

\begin{thm}\label{EVI}
The antipodal set $A_{1}(EVI_{+})$ and $A_{2}(EVI_{+})$are maximal antipodal sets of $EVI_{+}$ and the cardinalities are $31$ and $63$.
Moreover, any maximal antipodal set of $EVI_{+}$ is congruent to one of them and $\#_{2}EVI = 63$.

\end{thm}

\begin{proof}
It is sufficient to show that there does not exist $g \in E_{7}^{\alpha}$ such that $g(A_{1}(EVI_{+}))g^{-1} \subset A_{2}(EVI_{+})$.
We assume that there exists such $g$.
Since the Weyl group $W(E_{7}^{\alpha})$ acts on $A_{2}(EVI_{+})$ and this action is transitive, there exists $h \in E_{7}^{\alpha}$ such that
\[
hqh^{-1} = q, \quad
h(A_{1}(EVI_{+}))h^{-1} \subset A_{2}(EVI_{+}).
\]
Therefore, $h \in SU^{\bar{\alpha}}(2) \cdot Spin^{\alpha}(12)$ and 
\[
h \Big( A_{1}(EVI_{+}) \cap \tilde{G}_{4}(\mathbb{R}^{12})_{+} \Big) h^{-1} \subset A_{2}(EVI_{+}) \cap \tilde{G}_{4}(\mathbb{R}^{12})_{+}.
\]
However, $A_{i}(EVI_{+}) \cap \tilde{G}_{4}(\mathbb{R}^{12})_{+} \ (i = 1,2)$ are maximal antipodal sets of $\tilde{G}_{4}(\mathbb{R}^{12})_{+}$ and they are not congruent to each other by Theorem \ref{Ori-Grass}.
This is a contradiction.
Hence, there does not exist such $g \in E_{7}^{\alpha}$ and the statement follows by the arguments of Subsection \ref{strategies}.
\end{proof}

We see that 
\[
A_{1}(EVI_{+}) = \psi \Big( A \big( \tilde{G}_{4}(\mathbb{R}^{8})_{+} \big) \times \{ \gamma_{0}^{0,0} \} \Big) \cup \{ \psi( \gamma_{0}^{0,0}, \gamma_{1}^{a,b} ) \ ;\ (a,b) = (1,0),(0,1),(1,1) \}.
\]
Hence, $A_{1}(EVI_{+})$ is constructed by a maximal antipodal set of $\tilde{G}_{4}(\mathbb{R}^{8})$.
As mentioned above, $A_{2}(EVI_{+})$ is an orbit of the Weyl group $W(E_{7}^{\alpha})$.


\subsection{Maximal antipodal sets of $E_{7}$}\label{s-E_{7}}

In this subsection, we classify the congruent classes of the maximal antipodal sets of $E_{7}$.
It is well known that the number of polars except for the trivial pole is $3$ and one of them is a pole and the others are isomorphic to $EVI$.
Note that the pole of the identity element is the non-trivial element $z$ of the center of $E_{7}$.
Let $\tau$ be the covering transformation of $E_{7}$, that is,
\[
\tau : E_{7} \rightarrow E_{7} \ ;\ g \mapsto gz = zg.
\]
Let $M_{i}^{+}(E_{7}) \ (i = 1,2)$ be polars of the identity element isomorphic to $EVI$.
Then, $M_{2}^{+}(E_{7}) = \tau( M_{1}^{+}(E_{7}) ) = M_{1}^{+}(E_{7})z$.
Define $A_{j}(M_{1}^{+}(E_{7})) \ (j = 1,2)$ as maximal antipodal sets of $M_{1}^{+}(E_{7})$ whose cardinalities are $31$ and $63$ respectively.
Set
\[
A_{j}(E_{7}) = \{ e \} \cup A_{j}(M_{1}^{+}(E_{7})) \cup A_{j}(M_{1}^{+}(E_{7}))z \cup \{ z \} \quad (j = 1,2).
\]
Then, $\#A_{1}(E_{7}) = 64$ and $\#A_{2}(E_{7}) = 128$.
We obtain Theorem \ref{E_{7}} by Theorem \ref{EVI}.

\begin{thm} \label{E_{7}}
$A_{j}(E_{7}) \ (j = 1,2)$ are maximal antipodal subgroups of $E_{7}$ and $A_{1}(E_{8})$ and $A_{2}(E_{8})$ are not congruent to each other.
The cardinalities of $A_{1}(E_{7})$ and $A_{2}(E_{7})$ are $64$ and $128$.
Moreover, any maximal antipodal subset of $E_{7}$ is congruent to one of them and $\#_{2}E_{7} = 128$.

\end{thm}

In the following, we describe maximal antipodal sets of $E_{7}$ explicitly.
Let $\alpha = x_{7} - x_{8}$ and consider $E_{7}^{\alpha}$.
Recall that $EVI_{+}$ is a polar of the identity element in $E_{7}^{\alpha}$.
The pole of the identity element is $\tau_{\alpha} = \psi(\gamma_{0}^{0,0}, \gamma_{1}^{0,0})$.
The other polar is $\tau_{\alpha}EVI_{+}$.
Then, $EVI_{+}$ is a totally geodesic submanifold of $EIX_{+}$ and $\tau_{\alpha}EVI_{+}$ is a totally geodesic submanifold of $EVIII_{+}$.

\begin{prop}
Set $A_{1}(E_{7}^{\alpha}) = A_{1}(E_{8}) \cap E_{7}^{\alpha}$.
Then, $A_{1}(E_{7}^{\alpha})$ is a maximal antipodal set of $E_{7}^{\alpha}$ whose cardinality is $64$.

\end{prop}

\begin{proof}
We see that
\[
\begin{split}
A_{1}(E_{7}^{\alpha}) 
& \subset A_{1}(E_{8}) \cap C(SU^{\alpha}(2), E_{8}) 
= \big\{ \psi(\gamma_{i}^{0,0}, \gamma_{0}^{a,b}), \psi(\gamma_{i}^{0,0}, \gamma_{1}^{a,b}) \ ; \ 0 \leq i \leq 7, a,b = 0,1 \big\}.
\end{split}
\]
Then, the right hand side is contained in $SU^{\bar{\alpha}} \cdot Spin^{\alpha}(12) = F(\sigma_{VI}^{\bar{\alpha}}, E_{7}^{\alpha})$.
Therefore, $A_{1}(E_{7}^{\alpha}) = A_{1}(E_{8}) \cap C(SU^{\alpha}(2), E_{8})$.
On the other hand,
\[
\begin{split}
& \{ e \} \cup A_{1}(EVI_{+}) \cup \tau_{\alpha}A_{1}(EVI_{+}) \cup \{ \tau_{\alpha} \} \\
& = \{ \psi(\gamma_{0}^{0,0}, \gamma_{0}^{0,0}) \} 
\cup \big\{ \psi(\gamma_{i}^{0,0}, \gamma_{0}^{a,b}), \ \psi(\gamma_{i}^{0,0}, \gamma_{1}^{0,0}) \ ;\ 1 \leq i \leq 7, a,b = 0,1 \big\} 
\cup \{ \psi(\gamma_{0}^{0,0}, \gamma_{1}^{0,0}) \} = A_{1}(E_{7}^{\alpha}).
\end{split}
\]
Hence, $A_{1}(E_{7}^{\alpha})$ is a maximal antipodal set of $E_{7}^{\alpha}$ whose cardinality is $64$.
\end{proof}

\begin{prop}
Set $A_{2}(E_{7}^{\alpha}) = A_{2}(E_{8}) \cap E_{7}^{\alpha}$.
Then, $A_{2}(E_{7}^{\alpha})$ is a maximal antipodal set of $E_{7}^{\alpha}$ whose cardinality is $128$.
\end{prop}

\begin{proof}
It is obvious that $A_{2}(E_{7}^{\alpha})$ is an antipodal set of $E_{7}^{\alpha}$.
On the other hand, 
\[
A = \{ e \} \cup \{ \tau_{\beta} \ ;\ \beta \in \Sigma^{+}_{\alpha,0} \} \cup \{ \tau_{\beta}\tau_{\alpha} \ ;\ \beta \in \Sigma^{+}_{\alpha, 0} \} \cup \{ \tau_{\alpha} \}.
\]
is an antipodal set of $E_{7}^{\alpha}$ and $A \subset A_{2}(E_{7}^{\alpha})$.
Since the cardinality of $A$ is $128$, we obtain $A_{2}(E_{7}^{\alpha}) = A$ by Theorem \ref{E_{7}}.
Therefore, $A_{2}(E_{7}^{\alpha})$ is a maximal antipodal set of $E_{7}^{\alpha}$ whose cardinality is $128$.
\end{proof}


\subsection{Maximal antipodal sets of $EV$}\label{s-EV}

In this subsection, we consider maximal antipodal sets of $EV$.
It is well known that the number of polars of $EV$ except for the trivial pole is $3$ and one of them is a pole and the others are isomorphic to $G_{4}^{*}(\mathbb{C}^{8})$.
Let $\tau$ be the covering transformation of $EV$.
Fix $p \in EV$ and denote the pole of $p$ by $q$.
Moreover, the polars of $p$ isomorphic to $G_{4}^{*}(\mathbb{C}^{8})$ is denoted by $M_{i}^{+}(EV)\ (i = 1,2)$.
Then, $M_{2}^{+}(EV) = \tau(M_{1}^{+}(EV))$.
Let $A_{j}(M_{1}^{+}(EV)) \ (i = 1,2,3)$ be maximal antipodal sets of $M_{1}^{+}(EV)$ whose cardinalities are $27, 35, 63$.
Set
\[
A_{j}(EV) = \{ p \} \cup A_{j}(M_{1}^{+}(EV)) \cup \tau( A_{j}(M_{1}^{+}(EV) ) \cup \{ q \}.
\]
Then, $\#A_{1}(EV) = 56, \#A_{2}(EV) = 72, \#A_{3}(EV) = 128$.

\begin{thm} \label{EV}
The subsets $A_{j}(EV)\ (j = 1,2,3)$ are maximal antipodal sets of $EV$ whose cardinalities are $56, 72, 128$.
Moreover, any maximal antipodal set is congruent to one of $A_{i}(EV) \ (i = 1,2,3)$.

\end{thm}

\begin{proof}
It is sufficient to study whether there does not exist $g \in E_{7}$ such that $g(A_{i}(EV)) \subset A_{j}(EV) \ (1 \leq i < j \leq 3)$.
The rank of $EV$ is equal to $7$, so $A_{3}(EV)$ is a maximal antipodal set of a maximal flat torus.
In particular, there exists some subgroup of $E_{7}$ acting $A_{3}(EV)$ transitively.
For each $i = 1,2$, we assume that there exists $g \in E_{7}$ such that $g(A_{i}(EV)) \subset A_{3}(EV)$.
Then, there exists some element $h$ of the isotorpy group of $E_{7}$ at $p$ such that $h(A_{i}(EV)) \subset A_{3}(EV)$.
In particular, $h(A_{i}(M_{1}^{+}(EV)) \subset A_{3}(M_{1}^{+}(EV))$.
However, this contradicts to Theorem \ref{Bottom-Grass} and there does not exist such $g \in E_{7}$.
Next, we study that whether there exists $g \in E_{7}$ such that $g(A_{1}(EV)) \subset A_{2}(EV)$.
By the result of the author \cite{Sasaki} (Theorem 4.19), we see that there exists $g \in E_{7}$ such that $g(A_{1}(EV)) \subset A_{2}(EV)$ and $g(p) = p$ or $g(p) = q$.
In the both cases, in $G_{4}^{*}(\mathbb{C}^{8})$, there exists some element $h \in SU(8)/\mathbb{Z}_{2}$ such that $g(A_{1}(G_{4}^{*}(\mathbb{C}^{8}))) \subset A_{2}(G_{4}^{*}(\mathbb{C}^{8}))$.
However, this contradicts to Theorem \ref{Bottom-Grass}.
Hence, there does not exist such $g \in E_{7}$ and we complete the proof of the statement.
\end{proof}

We describe maximal antipodal sets of $EV$ more explicitly.
Let $\alpha = x_{7} - x_{8}$ and consider $E_{7}^{\alpha}$.
Then, since 
\[
\bigcup_{g \in E_{7}^{\alpha}}g x g^{-1} = E_{7}^{\alpha}/F(\sigma_{V}, E_{7}^{\alpha}) \cong E_{7}/SU(8)/\mathbb{Z}_{2},
\]
we denote this conjugate orbit by $EV_{x}$.
Then, $EV_{x}$ is a totally geodesic submanifold of $EVIII_{+}$.
Since $x \not\in C(SU^{\alpha}(2), E_{8})$, we see $x \not\in E_{7}^{\alpha}$ and $EV \not\subset E_{7}^{\alpha}$.

\begin{prop}
Set $A_{1}(EV_{x}) = A_{1}(E_{8}) \cap EV_{x}$.
Then, $A_{1}(EV_{x})$ is a maximal antipodal set of $EV_{x}$ whose cardinality is $56$.

\end{prop}

\begin{proof}
First, we describe $A_{1}(EV_{x})$ explicitly.
We see that
\[
\begin{split}
A_{1}(EV_{x}) & \subset \{ a \in A_{1}(EVIII_{+}) \ ;\ \mathrm{Ad}(a)|_{\frak{su}^{\alpha}(2)} = \mathrm{Ad}(x)|_{\frak{su}^{\alpha}(2)} \}  \\
& = \{ \psi( \gamma_{i}^{0,0}, \gamma_{4}^{a,b}), \psi( \gamma_{i}^{0,0}, \gamma_{5}^{a,b}) \ ;\ 1 \leq i \leq 7, a,b = 0,1 \}.
\end{split}
\]
On the other hand, for any $1 \leq i \leq 7$ and $a,b = 0,1$, there exists $g \in Spin(8)$ such that $g(\gamma_{i}^{a,b})g^{-1} = \gamma_{1}^{0,0}$.
Moreover, since
\[
\begin{split}
\Big( \frac{1}{\sqrt{2}}(1 + e_{4}e_{5}) & \frac{1}{\sqrt{2}}(1 + e_{6}e_{7}) \Big) \gamma_{4}^{0,0} \Big( \frac{1}{\sqrt{2}}(1 - e_{6}e_{7}) \frac{1}{\sqrt{2}}(1 - e_{4}e_{5}) \Big) \\
& = -\frac{1}{4}(1 + e_{4}e_{5})(1 + e_{6}e_{7})(e_{1}e_{3}e_{5}e_{7})(1 - e_{6}e_{7})(1 - e_{4}e_{5}) \\
& = \cdots = -e_{1}e_{3}e_{4}e_{6} = \gamma_{5}^{0,0},
\end{split}
\]
there exists $h \in SU^{\bar{\alpha}}(2) \subset E_{7}^{\alpha}$ such that $h \psi(\gamma_{i}^{a,b}, \gamma_{4}^{0,0}) h^{-1} = \psi(\gamma_{i}^{a,b}, \gamma_{5}^{0,0})$ for any $1 \leq i \leq 7$ and $a,b = 0,1$.
Therefore,
\[
A_{1}(EV_{x}) = \{ \psi( \gamma_{i}^{0,0}, \gamma_{4}^{a,b}), \psi( \gamma_{i}^{0,0}, \gamma_{5}^{a,b}) \ ;\ 1 \leq i \leq 7, a,b = 0,1 \}.
\]
Next, we show that $A_{1}(EV_{x})$ is a maximal antipodal set of $EV_{x}$.
We consider the Cartan embedding
\[
EV_{x} \ni gxg^{-1} \mapsto gxg^{-1}x \in E_{7}^{\alpha}.
\]
Then, $A_{1}(EV_{x})x$ is an antipodal set of $(EV_{x})x$ and $E_{7}^{\alpha}$, and
\[
A_{1}(EV_{x})x = \{ \psi(\gamma_{i}^{0,0}, \gamma_{0}^{a,b}), \psi(\gamma_{i}^{0,0}, \gamma_{1}^{a,b}) \ ;\ 1 \leq i \leq 7, i \not= 4, a,b = 0,1 \}.
\]
In particular, the subgroup $\langle A_{1}(EV_{x})x \rangle$ generated by $A_{1}(EV_{x})x$ is $A_{1}(E_{7}^{\alpha})$.
Let $B$ be a maximal antipodal set of $EV_{x}$ containing $A_{1}(EV_{x})$.
Then, $Bx$ is an antipodal set of $E_{7}^{\alpha}$ and $A_{1}(E_{7}^{\alpha}) \subset \langle Bx \rangle$.
Note that the identity element is contained in $Bx$.
By the maximality of $A_{1}(E_{7}^{\alpha})$, it is true that $A_{1}(E_{7}^{\alpha}) = \langle Bx \rangle$.
Hence, $Bx \subset A_{1}(E_{7}^{\alpha})$.
Since
\[
A_{1}(E_{7}^{\alpha}) - A_{1}(EV_{x})x = \{ \psi(\gamma_{4}^{0,0}, \gamma_{0}^{a,b}), \psi(\gamma_{4}^{0,0}, \gamma_{1}^{a,b}) \ ;\ a,b = 0,1 \}
\]
and $Bx \subset A_{1}(E_{7}^{\alpha}) \cap (EV_{x})x$, we obtain $Bx = A_{1}(EV_{x})x$ and $B = A_{1}(EV_{x})$.
Therefore, $A_{1}(EV_{x})$ is a maximal antipodal set of $EV_{x}$.
\end{proof}

Since $Spin^{1}(8) \subset E_{7}^{\alpha}$, we see that $\psi( (\tilde{G}_{4}(\mathbb{R}^{8}))_{+}^{1} \times \{ \gamma_{i}^{0,0} \}) \ (i = 4,5)$ are totally geodesic submanifolds of $EV_{x}$.
Then,
\[
A_{1}(EV_{x}) = \psi \Big( A \big( (\tilde{G}_{4}(\mathbb{R}^{8}))_{+}^{1} \big) \times \{ \gamma_{4}^{0,0} \} \Big) \cup \psi \Big( A \big( (\tilde{G}_{4}(\mathbb{R}^{8}))_{+}^{1} \big) \times \{ \gamma_{5}^{0,0} \} \Big).
\]
Therefore, $A_{1}(EV_{x})$ is constructed by a maximal antipodal set of $\tilde{G}_{4}(\mathbb{R}^{8})$.

\begin{prop}
Set $A_{3}(EV_{x}) = A_{2}(E_{8}) \cap EV_{x}$.
Then, $A_{3}(EV_{x})$ is a maximal antipodal set of $EV_{x}$ whose cardinality is $128$.

\end{prop}

\begin{proof}
It is obvious that $A_{3}(EV_{x})$ is an antipodal set of $EV$.
Moreover,
\[
\{ x, \tau_{\alpha}x \} \cup \{ \tau_{\beta}x, \ \tau_{\beta}\tau_{\alpha}x \ ;\ \beta \in \Sigma_{\alpha, 0}^{+} \}
\]
is contained in $A_{3}(EV_{x})$ and the cardinality is $128$.
By Theorem \ref{EV}, we see that $A_{3}(EV_{x})$ is equal to this and the statement follows.
\end{proof}

Since the rank of $EV$ is $7$, we see that $T^{\alpha} x$ is a maximal flat torus of $EV_{x}$.
Then, $A_{3}(EV_{x})$ is a maximal antipodal set of $T^{\alpha}x$.

Next, let $\gamma = (1/2)(x_{1} + \cdots + x_{8}) \in \Sigma^{+}(\frak{e}_{8}, \frak{t})$ and consider $E_{7}^{\gamma}$.
Then, since
\[
\bigcup_{g \in E_{7}^{\gamma}} g \phi(-1) g^{-1} = E_{7}^{\gamma} / F^{+}(\sigma'_{V}, E_{7}^{\gamma}) \cong E_{7}/(SU(8)/\mathbb{Z}_{2}),
\]
this orbit is denoted by $EV_{\phi(-1)}$.
It is obvious that $EV_{\phi(-1)}$ is also an totally geodesic submanifold of $EVIII$.
Set $A_{2}(EV_{\phi(-1)}) = A_{2}(E_{8}) \cap EV_{\phi(-1)}$.
Let $\alpha = x_{7} - x_{8}$ and $\beta = x_{6} - x_{7}$.
Then,
\[
\begin{split}
A_{2}(EV_{\phi(-1)}) \subset & \{ g \in A_{2}(EVIII_{+}) \ ;\ \mathrm{Ad}(g)|_{\frak{su}^{\gamma}(2)} = \mathrm{Ad}(\phi(-1))|_{\frak{su}^{\gamma}(2)} \} \\
= & A_{2}(E_{8}) - (((A_{2}(EVIII_{+}) \cap A(T^{\alpha})) \cup \{ e \} ) \\
= & \{ \tau_{\delta} \tau_{\alpha} \ ;\ \delta \in \Sigma_{\alpha, 0}^{+} \cap \Sigma_{\gamma, \pm1} \} \cup \{ \tau_{\delta} \tau_{\beta}, \tau_{\delta} \tau_{\alpha + \beta} \ ;\ \delta \in \Sigma_{\alpha,0}^{+} \cap \Sigma_{\beta, 0} \cap \Sigma_{\gamma, \pm1} \} \\
= & \{ \phi(-1) \} \cup \{ \tau_{x_{i} + x_{j}} \tau_{\alpha} \ ;\ 1 \leq i < j \leq 6 \} \cup \{ \tau_{x_{i} + x_{j}} \tau_{\beta}, \ \tau_{x_{i} + x_{j}}\tau_{\alpha + \beta} \ ;\ 1 \leq i < j \leq 5 \} \\
& \cup \left\{ \tau_{\delta}\tau_{\alpha} \ ;\ \delta = \frac{1}{2} \sum_{i=1}^{8}\epsilon_{i}x_{i} \in \Sigma, \ \epsilon_{1} = 1, \ \epsilon_{7} = \epsilon_{8} = \pm 1, \ \sum_{i=1}^{8}\epsilon_{i} = \pm 4 \right\} \\
& \cup \left\{ \tau_{\delta}\tau_{\beta} \ ;\ \delta = \frac{1}{2} \sum_{i=1}^{8}\epsilon_{i}x_{i} \in \Sigma, \ \epsilon_{1} = 1, \ \epsilon_{6} = \epsilon_{7} = \epsilon_{8} = \pm1, \ \sum_{i=1}^{8}\epsilon_{i} = \pm4  \right\} \\
& \cup \left\{ \tau_{\delta}\tau_{\alpha + \beta} \ ;\ \delta = \frac{1}{2} \sum_{i=1}^{8}\epsilon_{i}x_{i} \in \Sigma, \ \epsilon_{1} = 1, \ \epsilon_{6} = \epsilon_{7} = \epsilon_{8} = \pm1, \ \sum_{i=1}^{8}\epsilon_{i} = \pm4  \right\}. \\
\end{split}
\]
Since the Weyl group $W(E_{7}^{\gamma})$ of the root system $\Sigma(\frak{e}_{7}^{\gamma}, \frak{t}^{\gamma})$ acts on $A_{2}(EVI_{+})$ and this action is transitive, we see that $W(E_{7}^{\gamma})$ acts on the right hand side of the above and this action is transitive.
Therefore,
\[
A_{2}(EV_{\phi(-1)}) = \{ \tau_{\delta} \tau_{\alpha} \ ;\ \delta \in \Sigma_{\alpha, 0}^{+} \cap \Sigma_{\gamma, \pm1} \} \cup \{ \tau_{\delta} \tau_{\beta}, \tau_{\delta} \tau_{\alpha + \beta} \ ;\ \delta \in \Sigma_{\alpha,0}^{+} \cap \Sigma_{\beta, 0} \cap \Sigma_{\gamma, \pm1} \}
\]
and $\#A_{2}(EV_{\phi(-1)}) = 72$.
We consider polars of $\phi(-1)$ in $EV_{\phi(-1)}$.
The non-trivial pole of $\phi(-1)$ is $\tau_{\gamma}\phi(-1) = \tau_{(++++++--)}\tau_{\alpha}$ and the covering transformation of $EV_{\phi(-1)}$ is $EV_{\phi(-1)} \ni g \mapsto \tau_{\gamma}g \in EV_{\phi(-1)}$.
The isotropy group of $E_{7}^{\gamma}$ at $\phi(-1)$ is $F^{+}(\sigma'_{V}, E_{7}^{\gamma}) = \phi(SU''(8))$.
Therefore, 
\[
\bigcup_{g \in \phi(SU''(8))} g \tau_{x_{1} + x_{2}} \tau_{\alpha} g^{-1}
\]
is a polar of $\phi(-1)$ isomorphic to $G_{4}^{*}(\mathbb{C}^{8})$.
This polar is denoted by $G_{4}^{*}(\mathbb{C}^{8})_{+}$.
Then, 
\[
A_{2}(EV_{\phi(-1)}) \cap G_{4}^{*}(\mathbb{C}^{8})_{+} = \{ \tau_{x_{i} + x_{j}} \tau_{\alpha} \ ;\ 1 \leq i < j \leq 6 \} \cup \{ \tau_{x_{i} + x_{j}} \tau_{\beta}, \ \tau_{x_{i} + x_{j}}\tau_{\alpha + \beta} \ ;\ 1 \leq i < j \leq 5 \}.
\]
We can verify that
\[
\tau_{\gamma} \Big( A_{2}(EV_{\phi(-1)}) \cap (G_{4}^{*}(\mathbb{C}^{8}))_{+} \Big) = A_{2}(EV_{\phi(-1)}) - \Big( A_{2}(EV_{\phi(-1)}) \cap (G_{4}^{*}(\mathbb{C}^{8}))_{+} \cup \{ \phi(-1), \tau_{\gamma}\phi(-1) \} \Big).
\]

\begin{prop}
The subset $A_{2}(EV_{\phi(-1)})$ is a maximal antipodal set of $EV_{\phi(-1)}$ whose cardinality is $72$.

\end{prop}

\begin{proof}
We consider the Cartan embedding $EV_{\phi(-1)} \ni g \mapsto g \phi(-1)  \in E_{7}^{\gamma}$.
Then, $G_{4}^{*}(\mathbb{C}^{8})_{+}\phi(-1) $ is contained in $F(\sigma'_{V}, E_{7}^{\gamma}) = \phi(SU''(8))$ since $G_{4}(\mathbb{C}^{8})_{+}$ is an conjugate orbit by $\phi(SU''(8))$.
In particular, $G_{4}^{*}(\mathbb{C}^{8})_{+}\phi(-1)$ is a polar of $\phi(SU''(8))$.
Let $G_{4}(\mathbb{C}^{8})''_{+}$ be the polar of the identity element in $SU''(8)$ isomorphic to $G_{4}(\mathbb{C}^{8})$.
Then, $G_{4}^{*}(\mathbb{C}^{8})_{+} \phi(-1) = \phi( G_{4}(\mathbb{C}^{8})''_{+} )$.
To prove that the the statement is true, it is sufficient to show that $\big( A_{2}(EV_{\phi(-1)}) \cap (G_{4}^{*}(\mathbb{C}^{8}))_{+} \big) \phi(-1)$ is a maximal antipodal set of $G_{4}^{*}(\mathbb{C}^{8})$.
Then, 
\[
\big( A_{2}(EV_{\phi(-1)}) \cap G_{4}^{*}(\mathbb{C}^{8})_{+} \big)\phi(-1) 
= \phi \circ \pi^{-1} \circ c ( D_{0}^{+} \cap G_{4}(\mathbb{C}^{8})_{+} ).
\]
By Theorem \ref{Bottom-Grass}, we obtain that $\big( A_{2}(EV_{\phi(-1)}) \cap G_{4}^{*}(\mathbb{C}^{8})_{+} \big)\phi(-1) $ is a maximal antipodal set of $G_{4}^{*}(\mathbb{C}^{8})_{+}''$ whose cardinality is $35$.
Hence, the statement follows.

\end{proof}

Note that the Weyl group $W(E_{7}^{\gamma})$ acts on $A_{2}(EV_{\phi(-1)})$ and this action is transitive as mentioned above.


\subsection{The classification}

For each simply connected compact exceptional symmetric spaces, the classification of the congruent classes of maximal antipodal set has been complete.
In this subsection, we summarize these classifications (Table 1).
Let $M$ be one of simply connected compact exceptional symmetric spaces.
In Table 1, $C_{M}$ implies the number of congruent classes of maximal antipodal sets of $M$.
The ``card." implies the all cardinalities of maximal antipodal sets.
If the number of the congruent classes of maximal antipodal sets of $M$ is $k$, then we set $A_{i}(M) \ (1 \leq i \leq k)$ as the complete system of representative and assume that $\#A_{i} \leq \#A_{j}(M) \ (i < j)$.
If $k=1$, then we denote $A_{1}(M)$ by $A(M)$.
In ``Construction'', we introduce the construction of each maximal antipodal set briefly.
In the table, ``MAS'' means maximal antipodal set.

\ 

\begin{table}[htbp]
\small
\def\arraystretch{1.5}
\arraycolsep=-1pt
\begin{center}
\begin{tabular}{|c|c|c|l|c|c|ccc} \hline
$M$ & $C_{M}$ & card. & Construction \\ \hline
$G_{2}$ & 1 & 8 & $\begin{array}{c}\text{$A(G_{2})$ is obtained by Fano plane.} \end{array}$\\ \hline
$F_{4}$ & 1 & 32 & $\begin{array}{c}\text{$A(F_{4})$ is a MAS of $Spin(8) \subset F_{4}$.} \end{array}$ \\ \hline
$E_{6}$ & 2 & $\begin{matrix}\vspace{-2mm} 32 \\ 64 \end{matrix}$ & 
$\begin{array}{lll} \vspace{-2mm}
\text{$A_{1}(E_{6})$ is the image of a MAS under $Spin(8) \times Spin(2) \rightarrow E_{6}$.} \\ 
\text{$A_{2}(E_{6})$ is a MAS of a maximal torus.}
\end{array}$ \\ \hline
$E_{7}$ & 2 & $\begin{matrix}\vspace{-2mm} 64 \\ 128 \end{matrix}$ & 
$\begin{array}{lll} \vspace{-2mm}
\text{$A_{1}(E_{7})$ is the image of a MAS under $Spin(8) \times Spin(4) \rightarrow E_{7}$.} \\ 
\text{$A_{2}(E_{7})$ is a MAS of a maximal torus.}
\end{array}$ \\ \hline
$E_{8}$ & 2 & $\begin{matrix}\vspace{-2mm} 256 \\ 512 \end{matrix}$ & 
$\begin{array}{lll} \vspace{-2mm}
\text{$A_{1}(E_{8})$ is the image of a MAS under $Spin(8) \times Spin(8) \rightarrow E_{8}$.} \\ 
\text{$A_{2}(E_{8})$ is ``a MAS of a maximal torus" $\cup$ ``a MAS of a maximal torus".}
\end{array}$ \\ \hline
$G$ & 1 & 7 & $\begin{array}{c}\text{$A(G)$ is obtained by Fano plane.}\end{array}$ \\ \hline
$FI$ & 1 & 28 & $\begin{array}{c}\text{$A(FI)$ is a MAS of $\tilde{G}_{4}(\mathbb{R}^{8}) \subset F_{4}$.}\end{array}$ \\ \hline
$FII$ & 1 & 3 & $\begin{array}{c}\text{$A(FII)$ is an orbit of the Weyl group $W(F_{4})$.} \end{array}$ \\ \hline
$EI$ & 2 & $\begin{matrix}\vspace{-2mm} 28 \\ 64 \end{matrix}$ & 
$\begin{array}{ll} \vspace{-2mm}
\text{$A_{1}(EI)$ is a MAS of $\tilde{G}_{4}(\mathbb{R}^{8}) \subset EI$.}  \\ 
\text{$A_{2}(EI)$ is a MAS of a maximal flat torus.} 
\end{array}$ \\ \hline
$EII$ & 2 & $\begin{matrix}\vspace{-2mm} 28 \\ 36 \end{matrix}$ & 
$\begin{array}{ll} \vspace{-2mm}
\text{$A_{1}(EII)$ is a MAS of $\tilde{G}_{4}(\mathbb{R}^{8}) \subset EII$.} \\ 
\text{$A_{2}(EII)$ is an orbit of the Weyl group $W(E_{6})$.} 
\end{array}$ \\ \hline
$EIII$ & 1 & 27 & $\begin{array}{c}\text{$A(EIII)$ is an orbit of the Weyl group $W(E_{6})$.} \end{array}$ \\ \hline
$EIV$ & 1 & 4 & $\begin{array}{c}\text{$A(EIV)$ is an orbit of the Weyl group $W(E_{6})$.}\end{array}$ \\ \hline
$EV$ & 3 & $\begin{matrix}\vspace{-2mm} 56 \\\vspace{-2mm} 72 \\ 128 \end{matrix}$ & 
$\begin{array}{ll} \vspace{-2mm}
\text{$A_{1}(EV)$ is a MAS of $\tilde{G}_{4}(\mathbb{R}^{8}) \sqcup \tilde{G}_{4}(\mathbb{R}^{8}) \subset EV$.} \\ \vspace{-2mm} 
\text{$A_{2}(EV)$ is an orbit of the Weyl group $W(E_{7})$.} \\ 
\text{$A_{3}(EV)$ is a MAS of a maximal flat torus.} 
\end{array}$  \\ \hline
$EVI$ & 2 & $\begin{matrix}\vspace{-2mm} 31 \\ 63 \end{matrix}$ & 
$\begin{array}{ll} \vspace{-2mm}
\text{$A_{1}(EVI)$ is ``a MAS of $G^{+}_{4}(\mathbb{R}^{8})$" $\cup$ ``3 points".} \\ 
\text{$A_{2}(EVI)$ is an orbit of the Weyl group $W(E_{7})$.}  
\end{array}$ \\ \hline
$EVII$ & 1 & 56 & $\begin{array}{c}\text{$A(EVII)$ is an orbit of the Weyl group $W(E_{7})$.}\end{array}$ \\ \hline
$EVIII$ & 2 & $\begin{matrix} \vspace{-2mm}199 \\ 391 \end{matrix}$ & 
$\begin{array}{ll} \vspace{-2mm}
\text{$A_{1}(EVIII)$ is ``an image of a MAS under $G^{+}_{4}(\mathbb{R}^{8}) \times G^{+}_{4}(\mathbb{R}^{8}) \rightarrow EVIII$'' $\cup$ `` 3 points''.} \\ 
\text{$A_{2}(EVIII)$ is `` an orbit of the Weyl group $W(E_{8})$" $\cup$ ``a MAS of a maximal flat torus". }  
\end{array}$ \\ \hline
$EIX$ & 2 & $\begin{matrix}\vspace{-2mm} 56 \\ 120 \end{matrix}$ &
$\begin{array}{ll}\vspace{-2mm}
\text{$A_{1}(EIX)$ is a MAS of $G^{+}_{4}(\mathbb{R}^{8}) \sqcup G^{+}_{4}(\mathbb{R}^{8}) \subset EIX$.} \\
\text{$A_{2}(EIX)$ is an orbit of the Weyl group $W(E_{8})$.}
\end{array}$ \\ \hline
\end{tabular}
\caption{Maximal antipodal sets of exceptional compact symmetric space} \label{classification}
\end{center}
\end{table}






\newpage

\section{Inclusion relations}

In this section, we observe some duality between exceptional symmetric spaces via antipodal sets.
Let $\alpha = x_{7} - x_{8}$ and $\beta = x_{6} - x_{7}$.
We denote $E_{7}^{\alpha}$ and $E_{6}^{\alpha, \beta}$ by $E_{7}$ and $E_{6}$ respectively.

First, we consider a decomposition of maximal antipodal sets.
Let $H$ be one of $E_{7}, E_{6}, F_{4}$.
Then, $\theta_{g_{1}}(H) \subset H$ for any $g_{1} \in A_{1}(E_{8})$.
Therefore, the $H$-orbit $H(g_{1}) = \cup_{h \in H} h g_{1} h^{-1}$ is a compact symmetric space since $\theta_{g_{1}}$ is an involutive automorphism.
Moreover, since $EVIII_{+}$ and $EIX_{+}$ are $E_{8}$-orbits, we see that $H(g_{1})$ is a totally geodesic submanifold of either $EVIII_{+}$ or $EIX_{+}$ except for $g_{1} = e$.
We can consider the decomposition of $A_{1}(E_{8})$ as
\[
\begin{split}
& A_{1}(E_{8}) = \bigcup_{g_{1} \in A_{1}(E_{8})} \Big( A_{1}(E_{8}) \cap H(g_{1}) \Big). \\
\end{split}
\]
Then, $A_{1}(E_{8}) \cap H(g_{1})$ is an antipodal set of $H(g_{1})$.
However, in general, this is not necessarily maximal.
First, we study what orbits the intersection is maximal.
Since $H(g_{1})$ is a conjugate oribt, we consider two cases, that is, one is the case of $H(g_{1}) \subset A(EIX_{+})$, and the other is the case of $H(g_{1}) \subset EVIII_{+}$.
Morover, we also study the case of $A_{2}(E_{8})$.
By summarizing these results, we consider some inclusion relations of exceptional symmetric spaces.


\subsection{$A_{1}(EIX_{+})$ and orbits of $E_{7}, E_{6}, F_{4}$} \label{A_{1}(EIX_{+})}

In this subsection, we consider $H(g_{1})$ for each $g_{1} \in A_{1}(EIX_{+})$.
Fist, we consider $E_{7}$-orbits.
Since $\psi(\gamma_{0}^{0,0}, \gamma_{1}^{0,0}) = \tau_{\alpha}$ is a pole of the identity element in $E_{7}$, we obtain $E_{7}(\psi(\gamma_{0}^{0,0}, \gamma_{1}^{0,0})) = \{ \psi(\gamma_{0}^{0,0}, \gamma_{1}^{0,0}) \}$.
Recall the polar $EVI_{+}$ of the identity element in $E_{7}$.
Then,
\[
EVI_{+} = E_{7} \big( \psi(  \gamma_{1}^{0,0}, \gamma_{0}^{0,0}) \big)
\]
Note that $EVI_{+}$ is isomorphic to $EVI$.
Define the conjugate orbits of $E_{7}$ as 
\[
EVII_{1} = E_{7} \big( \psi( \gamma_{0}^{0,0}, \gamma_{2i}^{0,0} ) \big) \ (1 \leq i \leq 3).
\]
Then, for any $1 \leq i \leq 3$, since $\dim F(\theta_{\psi(\gamma_{0}^{0,0}, \gamma_{2i}^{0,0})}, E_{7}) = 79$ and $F(\theta_{\psi(\gamma_{0}^{0,0}, \gamma_{2i}^{0,0})}, E_{7}) \cong (U(1) \times E_{6})/\mathbb{Z}_{3}$, we obtain $EVII_{i}$ is isomorphic to $EVII$.
Recall $A(EVI_{+}) = A_{1}(EIX_{+}) \cap EVI_{+}$ from Subsection \ref{s-EVI}.
Then,
\[
\begin{split}
A_{1}(EVI_{+}) = \left\{ \psi(\gamma_{k}^{0,0}, \gamma_{0}^{a,b}) \ ;\ 1 \leq k \leq 7, a,b = 0,1 \right\} \cup \left\{ \psi(\gamma_{0}^{0,0}, \gamma_{1}^{a,b}) \ ;\ (a,b) \not= (0,0) \right\}
\end{split}
\]
and $A_{1}(EVI_{+})$ is a maximal antipodal set of $EVI_{+}$.
We study $A_{1}(EIX_{+}) \cap EVII_{i} \ (i = 1,2,3)$.
Let $U$ be a subspace of $\frak{e}_{8}$ invariant under $A_{1}(E_{8})$.
Then, we set a equivalence relation $\sim_{U}$ such that $g \sim_{U} h$ if and only if $\mathrm{Ad}(g)|_{U} = \mathrm{Ad}(h)|_{U}$.
We consider the equivalence relation $\sim_{\frak{su}^{\alpha}(2)}$ in $A_{1}(EIX_{+})$.
Set subsets of $A_{1}(EIX_{+})$ as follows:
\[
\begin{split}
A_{1}^{\circ}(EVII_{1}) &= \left\{ \psi(\gamma_{0}^{0,0}, \gamma_{i}^{a,b}) \ ;\ i = 2,3, \ a,b = 0,1 \right\}, \\
A_{1}^{\circ}(EVII_{2}) &= \left\{ \psi(\gamma_{0}^{0,0}, \gamma_{i}^{a,b}) \ ;\ i = 4,5, \ a,b = 0,1 \right\}, \\
A_{1}^{\circ}(EVII_{3}) &= \left\{ \psi(\gamma_{0}^{0,0}, \gamma_{i}^{a,b}) \ ;\ i = 6,7, \ a,b = 0,1 \right\}. \\
\end{split}
\]
Then, $A_{1}^{\circ}(EVII_{i}) \ (i = 1,2,3)$ are all equivalent classes and $A_{1}(EIX) \cap EVII_{i} \subset A_{1}^{\circ}(EVII_{i})$. 
On the other hand, for any $i = 1,2,3$ and $g, g' \in A_{1}(EVII_{i})$, we can verify that there exists $h \in SU^{\bar{\alpha}}(2) \cdot Spin^{\alpha}(12)$ such that $h g h^{-1} = g'$.
Therefore,
\[
A_{1}^{\circ}(EVII_{i}) = A_{1}(EIX) \cap EVII_{i} \quad (1 \leq i \leq 3).
\]
By Theorem \ref{anti-EVII}, we see $A_{1}^{\circ}(EVII_{i})$ is not a maximal antipodal set of $EVII_{i}$ since $\# A_{1}^{\circ}(EVII_{i}) = 8$.
Summarizing these arguments, we obtain the following.

\begin{prop}
For any $p \in A_{1}(EIX)$, the conjugate orbit of $E_{7}^{\alpha}$ through $p$ is one of 
\[
\{ \psi(\gamma_{0}^{0,0}, \gamma_{1}^{0,0}) \}, \ EVI_{+}, \ EVII_{i} \ (i = 1,2,3).
\]
Then, $EVI_{+}$ is a polar of the identity element in $E_{7}$.
Let $L$ be one of $EVI_{+}$ and $EVII_{i} \ (1 \leq i \leq 3)$.
Then, $A(EIX_{+}) \cap L$ is a maximal antipodal set of $L$ if and only if $L = EVI_{+}$.
If $L$ is one of the others, then $\# A(EIX_{+}) \cap L = 3$ and $A(EIX_{+}) \cap L$ is not maximal.

\end{prop}

Next, we consider the case of $H = E_{6}$.
We see that $E_{6} \big( \psi(\gamma_{0}^{0,0}, \gamma_{1}^{0,0}) \big) = \{\psi(\gamma_{0}^{0,0}, \gamma_{1}^{0,0})\}$ since $E_{6} \subset E_{7}$.
Moreover, since $\psi(\gamma_{0}^{0,0}, \gamma_{2}^{0,0}) = \tau_{\alpha + \beta}$ and $\psi(\gamma_{0}^{0,0}, \gamma_{3}^{0,0}) = \tau_{\beta}$, by the definition of $E_{6}$,
\[
E_{6} \big( \psi(\gamma_{0}^{0,0}, \gamma_{k}^{0,0}) \big) = \{ \psi(\gamma_{0}^{0,0}, \gamma_{k}^{0,0}) \} \ (k = 2,3).
\]
Recall the polar $EII_{+}$ of the identity element in $E_{6}$,.
Then, 
\[
EII_{+} = E_{6} \big( \psi(\gamma_{1}^{0,0}, \gamma_{0}^{0,0}) \big).
\]
Note that $EII_{+}$ is isomorphic to $EII$.
Define the conjugate orbits of $E_{6}$ as
\[
\begin{split}
EIII_{i} = E_{6} \big( \psi(\gamma_{0}^{0,0}, \gamma_{i}^{1,0}) \big) \ (i = 1,2,3), \quad
EIV_{j} = E_{6} \big( \psi(\gamma_{0}^{0,0}, \gamma_{j}^{0,0}) \big) \ (j = 4,5,6,7). \\
\end{split}
\]
For any $i = 1,2,3$, since $\dim F^{+}( \theta_{\psi(\gamma_{0}^{0,0}, \gamma_{i}^{0,0})}, E_{6}) = 46$ and $F^{+}( \theta_{\psi(\gamma_{0}^{0,0}, \gamma_{i}^{0,0})}, E_{6}) \cong (U(1) \times Spin(10))/\mathbb{Z}_{4}$, we see that $EIII_{i}$ is isomorphic to $EIII$.
Similarly, for any $4 \leq j \leq 7$, since $\dim F^{+}( \theta_{\psi(\gamma_{0}^{0,0}, \gamma_{j}^{0,0})}, E_{6}) = 52$ and $F^{+}( \theta_{\psi(\gamma_{0}^{0,0}, \gamma_{j}^{0,0})}, E_{6}) \cong F_{4}$, we see that $EIV_{j}$ is isomorphic to $EIV$.
Recall $A_{1}(EII_{+}) = A_{1}(EIX) \cap EII_{+}$ from subsection \ref{s-EII}.
Then,
\[
A_{1}(EII_{+}) = \{ \psi(\gamma_{k}^{0,0}, \gamma_{0}^{a,b}) \ ;\ 1 \leq k \leq 7, a,b = 0,1 \}
\]
and $A_{1}(EII_{+})$ is a maximal antipodal set of $EII_{+}$.
Set subsets of $A_{1}(EIX_{+})$ as
\[
\begin{split}
& A_{1}^{\circ}(EIII_{i}) = \{ \psi(\gamma_{0}^{0,0}, \gamma_{i}^{1,0}), \ \psi(\gamma_{0}^{0,0}, \gamma_{i}^{0,1}), \ \psi(\gamma_{0}^{0,0}, \gamma_{i}^{1,1}) \} \quad (i = 1,2,3), \\
& A_{1}(EIV_{j}) = \{ \psi(\gamma_{0}^{0,0}, \gamma_{i}^{a,b}) \ ;\ a, b = 0,1 \} \quad (j = 4,5,6,7). \\
\end{split}
\]
We consider the equivalence relation $\sim_{\frak{su}^{\alpha, \beta}(3)}$.
Then, $A_{1}(EII_{+}), \{ \psi(\gamma_{0}^{0,0}, \gamma_{i}^{0,0}) \} \cup A_{1}^{\circ}(EIII_{i}) \ (i = 1,2,3)$, and $A_{1}(EIV_{j}) \ (j = 4,5,6,7)$ are all equivalence classes.
Obviously, $A_{1}(EIX_{+}) \cap EIII_{i} \subset A_{1}^{\circ}(EIII_{i})$ for any $i = 1,2,3$.
Since
\[
A_{1}(EIII_{1}) = \{ \tau_{x_{5} \pm x_{6}}, \tau_{x_{7} + x_{8}} \}, \ 
A_{1}(EIII_{2}) = \{ \tau_{x_{5} \pm x_{7}}, \tau_{x_{6} + x_{8}} \}, \ 
A_{1}(EIII_{3}) = \{ \tau_{x_{5} \pm x_{8}}, \tau_{x_{6} + x_{7}} \},
\]
by considering the Weyl group $W(E_{6})$, we obtain $A_{1}^{\circ}(EIII_{i}) = A_{1}(EIX_{+}) \cap EIII_{i}$ for any $1 \leq i \leq 3$.
Note that $A_{1}^{\circ}(EIII_{i})$ is not a maximal antipodal set of $EIII_{i}$.
Moreover, for any $4 \leq j \leq 7$ and $g, g' \in A_{1}(EIV_{j})$, there exists $h \in F(\sigma_{III}, E_{6}) \cong (Spin(10) \times U(1))/\mathbb{Z}_{4}$ such that $hgh^{-1} = g'$.
Hence, $A_{1}(EIV_{j}) = A_{1}(EIX_{+}) \cap EIV_{j}$.
We see $A_{1}(EIV_{j})$ is a maximal antipodal set of $EIV_{j}$ since $\#_{2}EIV = 4$.
Summarizing these arguments, we obtain the following.

\begin{prop}
For any $p \in A_{1}(EIX_{+})$, the conjugate orbit of $E_{6}$ through $p$ is one of 
\[
\{ \psi(\gamma_{0}^{0,0}, \gamma_{i}^{0,0}) \}, \ EII_{+}, \ EIII_{i}, \ EIV_{j} \ (1 \leq i \leq 3, 4 \leq j \leq 7).
\]
Then, $EII_{+}$ is a polar of the identity element in $E_{6}$.
Let $L$ be one of $EII_{+}, EIII_{i}, EIV_{j} \ (1 \leq i \leq 3, 4 \leq j \leq 7)$.
Then, $A(EIX_{+}) \cap L$ is a maximal antipodal set of $L$ if and only if $L$ is either $EII_{+}$ or $EIV_{j} \ (j = 4,5,6,7)$.
If $L$ is one of the others, $A(EIX_{+}) \cap L$ is not maximal and the cardinality is $3$.

\end{prop}

Next, we consider the case of $H = F_{4}$.
By the definition of $F_{4}$, 
\[
F_{4} \big( \psi( \gamma_{0}^{0,0}, \gamma_{i}^{0,0}) \big) = \{ \psi( \gamma_{0}^{0,0}, \gamma_{i}^{0,0}) \} \ (1 \leq i \leq 7).
\]
Recall the polar $FI_{+}$ of the identity element in $F_{4}$ from Subsection \ref{s-FIFII}.
Then,
\[
FI_{+} = F_{4} \big( \psi(\gamma_{1}^{0,0}, \gamma_{0}^{0,0}) \big).
\]
Note that $FI_{+}$ is isomorphic to $FI$.
Define the conjugate orbits as
\[
FII_{i} = F_{4} \big( \psi(\gamma_{0}^{0,0}, \gamma_{i}^{1,0}) \big) \ (1 \leq i \leq 7).
\]
Then, for any $1 \leq i \leq 7$, we see that $\dim F^{+}(\theta_{\psi(\gamma_{0}^{0,0}, \gamma_{i}^{1,0})}, F_{4}) = 36$ and $F^{+}(\theta_{\psi(\gamma_{0}^{0,0}, \gamma_{i}^{1,0})}, F_{4}) \cong Spin(9)$.
Hence, $FII_{i}$ is isomorphic to $FII$.
Recall $A(FI_{+}) = A_{1}(EXI_{+}) \cap FI_{+}$.
Then,
\[
A(FI_{+}) = \{ \psi(\gamma_{k}^{0,0}, \gamma_{0}^{a,b}) \ ;\ 1 \leq k \leq 7, a,b = 0,1 \}
\]
and $A(FI_{+})$ is a maximal antipodal set of $FI_{+}$.
Set $A_{1}(FII_{i}) = A_{1}(EIX_{+}) \cap FII_{i}\ (1 \leq i \leq 7)$.
If $4 \leq i \leq 7$, then $F(\theta_{\psi(\gamma_{0}^{0,0}, \gamma_{i}^{0,0})}, E_{6}) = F_{4}$ and $FII_{i}$ is a polar of $\psi(\gamma_{0}^{0,0}, \gamma_{i}^{0,0})$ in $EVI_{i}$.
Therefore, 
\[
A_{1}(FII_{i}) = \{ \psi(\gamma_{0}^{0,0}, \gamma_{i}^{1,0}), \psi(\gamma_{0}^{0,0}, \gamma_{i}^{0,1}), \psi(\gamma_{0}^{0,0}, \gamma_{i}^{1,1}) \} \ (i = 4,5,6,7).
\]
Moreover, by the definition of $F_{4}$, 
\[
FII_{1} = \psi(\gamma_{0}^{0,0}, \gamma_{5}^{0,0}) FII_{4}, \quad 
FII_{2} = \psi(\gamma_{0}^{0,0}, \gamma_{6}^{0,0}) FII_{4}, \quad 
FII_{3} = \psi(\gamma_{0}^{0,0}, \gamma_{7}^{0,0}) FII_{4}. \quad 
\]
Hence, we obtain
\[
A_{1}(FII_{i}) = \{ \psi(\gamma_{0}^{0,0}, \gamma_{i}^{1,0}), \psi(\gamma_{0}^{0,0}, \gamma_{i}^{0,1}), \psi(\gamma_{0}^{0,0}, \gamma_{i}^{1,1}) \} \ (i = 1,2,3).
\]
Since $\#_{2}FII = 3$, we see that $A_{1}(FII_{i})$ is a maximal antipodal set of $FII_{i}$ for any $1 \leq i \leq 7$.
Summarizing these arguments, we obtain the following.

\begin{prop}
For any $p \in A_{1}(EIX_{+})$, the conjugate orbit of $F_{4}$ through $p$ is one of
\[
\{ \psi(\gamma_{0}^{0,0}, \gamma_{i}^{0,0}) \}, \ FI_{+}, \ FII_{i} \ (i = 1, \cdots, 7).
\]
Then, $FI_{+}$ is a polar of the identity element in $F_{4}$.
Let $L$ be one of $FI_{+}$ and $FII_{i} \ (1 \leq i \leq 7)$.
Then, $A_{1}(EIX_{+}) \cap L$ is a maximal antipodal set of $L$.
\end{prop}


\subsection{$A_{1}(EVIII_{+})$ and orbits of $E_{7}, E_{6}, F_{4}$} \label{A_{1}(EVIII_{+})}

We consider the conjugate orbits of $H = E_{7}, E_{6}, F_{4}$ through points of $A_{1}(EVIII_{+})$.
First, we consider the case of $H = E_{7}$.
Since $\tau_{\alpha} = \psi(\gamma_{0}^{0,0}, \gamma_{1}^{0,0})$ is a non-trivial element of the center of $E_{7}$,
\[
EVI'_{+} = E_{7} \big( \psi(\gamma_{1}^{0,0}, \gamma_{1}^{0,0}) \big) = \psi(\gamma_{0}^{0,0}, \gamma_{1}^{0,0}) EVI_{+}
\]
is a polar of the identity element in $E_{7}$ different from $EVI_{+}$.
Obviously, $EVI'_{+}$ is isomorphic to $EVI$.
Define the conjugate orbits of $E_{7}$ as 
\[
EV_{i} = E_{7} \big( \psi(\gamma_{1}^{0,0}, \gamma_{2i}^{0,0}) \big) \quad (1 \leq i \leq 3).
\]
Then, for any $1 \leq i \leq 3$, we see that $\dim F( \theta_{\psi(\gamma_{1}^{0,0}, \gamma_{2i}^{0,0})}, E_{7}) = 63$ and $F^{+}( \theta_{\psi(\gamma_{1}^{0,0}, \gamma_{2i}^{0,0})}, E_{7}) \cong SU(8)/\mathbb{Z}_{2}$.
Therefore, $EV_{i}$ is isomorphic to $EV$.
Note $EV_{2} = EV_{x}$.
Set $A_{1}(EV'_{+}) = A_{1}(EVIII_{+}) \cap EVI'_{+}$.
Then, since $\tau_{\alpha} A_{1}(EVI_{+}) = A_{1}(EVI'_{+})$, we obtain
\[
\begin{split}
A_{1}(EVI'_{+})
&= \big\{ \psi(\gamma_{i}^{0,0}, \gamma_{1}^{0,0}) \ ;\ 1 \leq i \leq 7, a,b = 0,1 \big\} 
\cup \big\{ \psi(\gamma_{0}^{0,0}, \gamma_{0}^{a,b}) \ ;\ (a,b) = (0,0), (1,0), (0,1) \big\}.
\end{split}
\]
It is obvious that $A_{1}(EVI'_{+})$ is a maximal antipodal set of $EVI'_{+}$.
Consider the equivalence relation $\sim_{\frak{su}^{\alpha}(2)}$ on $A_{1}(EVIII_{+})$.
Set 
\[
A_{1}(EV_{i}) = \big\{ \psi(\gamma_{k}^{0,0}, \gamma_{2i}^{a,b}), \ \psi(\gamma_{k}^{0,0}, \gamma_{2i+1}^{a,b}) \ ;\ 1 \leq k \leq 7, a,b = 0, 1 \big\} \ (1 \leq i \leq 3). \\
\]
Then, $A_{1}(EVI'_{+})$ and $A_{1}(EV_{i}) \ (1 \leq i \leq 3)$ are all equivalence classes, and  $A_{1}(EVIII_{+}) \cap EV_{i} \subset A_{1}(EV_{i})$.
On the other hand, we can easily verify that for any $g_{1}, g_{2} \in A_{1}(EV_{i})$, there exists $h \in E_{7}$ such that $hg_{1}h^{-1} = g_{2}$.
Therefore, $A_{1}(EVIII_{+}) \cap EV_{i} = A_{1}(EV_{i})$ for any $1 \leq i \leq 3$.
Then, $A_{1}(EV_{i})$ is a maximal antipodal set of $EV_{i}$ whose cardinality is $56$.
Summarizing these arguments we obtain the following.

\begin{prop}
For any $p \in A_{1}(EVIII_{+})$, the conjugate orbit of $E_{7}$ through $p$ is one of
\[
EVI'_{+}, \ EV_{i} \ (1 \leq i \leq 3).
\]
Then, $EVI'_{+}$ is a polar of the identity element in $E_{7}^{\alpha}$.
Let $L$ be one of these $E_{7}$-orbits.
Then, $A_{1}(EVIII_{+}) \cap L$ is a maximal antipodal set of $L$.

\end{prop}

Next, we consider the case of $H = E_{6}$.
Recall the polar $EIII_{+}$ of the identity element in $E_{6}$.
Then,
\[
EIII_{+} = E_{6} \big( \psi(\gamma_{0}^{0,0}, \gamma_{0}^{1,0}) \big)
\]
and $EIII_{+}$ is isomorphic to $EIII$.
Note $\psi(\gamma_{0}^{0,0}, \gamma_{1}^{0,0})EIII_{1} = EIII_{+}$.
Set the conjugate orbits of $E_{6}$ as 
\[
\begin{split}
& EII_{i} = E_{6} \big( \psi(\gamma_{1}^{0,0}, \gamma_{i}^{0,0}) \big) \ (1 \leq i \leq 3), \quad
 EI_{j} = E_{6} \big( \psi(\gamma_{1}^{0,0}, \gamma_{j}^{0,0}) \big) \ (4 \leq j \leq 7).
\end{split}
\]
For any $1 \leq i \leq 3$, since $\dim F( \theta_{\psi(\gamma_{1}^{0,0}, \gamma_{i}^{0,0})}, E_{6}) = 38$ and $F( \theta_{\psi(\gamma_{1}^{0,0}, \gamma_{i}^{0,0})}, E_{6}) \cong Sp(1) \cdot Spin(10)$, we see that $EII_{i}$ is isomorphic to $EII$.
Similarly, for any $4 \leq j \leq 7$, since $\dim F( \theta_{\psi(\gamma_{1}^{0,0}, \gamma_{j}^{0,0})}, E_{6}) = 36$ and $F( \theta_{\psi(\gamma_{1}^{0,0}, \gamma_{j}^{0,0})}, E_{6}) \cong Sp(4)/\mathbb{Z}_{2}$, we see that $EI_{j}$ is isomorphic to $EI$.
Note that $EII_{1} = \psi(\gamma_{0}^{0,0}, \gamma_{1}^{0,0})EII_{+}$ and $EI_{4} = EI_{x}$.
We consider the equivalence relation $\sim_{\frak{su}^{\alpha, \beta}(3)}$ on $A_{1}(EVIII_{+})$.
Set 
\[
\begin{split}
& A_{1}^{\circ}(EIII_{+}) = \big\{ \psi(\gamma_{0}^{0,0}, \gamma_{0}^{a,b}) \ ;\ (a,b) = (1,0), (0,1), (1,1) \big\}, \\
& A_{1}(EII_{i}) = \big\{ \psi(\gamma_{k}^{0,0}, \gamma_{i}^{a,b}) \ ;\ 1 \leq k \leq 7, a,b = 0,1 \big\} \quad (1 \leq i \leq 3) \\
& A_{1}(EI_{j}) = \big\{ \psi(\gamma_{k}^{0,0}, \gamma_{j}^{a,b}) \ ;\ 1 \leq k \leq 7, a,b = 0,1 \big\} \quad (4 \leq j \leq 7).
\end{split}
\]
Then, the all equivalent classes are $A_{1}^{\circ}(EIII_{+}), A_{1}(EII_{i}), A_{1}(EI_{j}) \ (1 \leq i \leq 3, 4 \leq j \leq 7)$.
We see that for any $g_{1}, g_{2} \in \{ \gamma_{0}^{1,0}, \gamma_{0}^{0,1}, \gamma_{0}^{1,1} \}$, there exists $h_{1} \in F_{4}$ such that $h_{1} g_{1} h_{1}^{-1} = g_{2}$.
Moreover, for any $g_{3}, g_{4} \in \{ \gamma_{k}^{a,b} \ ;\ 1 \leq k \leq 7, a,b = 0,1 \}$, there exists $h_{2} \in F_{4}$ such that $h_{2} g_{3} h_{2}^{-1} = g_{4}$.
Hence, we obtain
\[
\begin{split}
& A_{1}(EVIII_{+}) \cap EIII_{+} = A_{1}(EIII_{+}), \\
& A_{1}(EVIII_{+}) \cap EII_{i} = A_{1}(EII_{i}) \quad (1 \leq i \leq 3), \\
& A_{1}(EVIII_{+}) \cap EI_{j} = A_{1}(EI_{j}) \quad (4 \leq j \leq 7).
\end{split}
\]
By Theorem \ref{anti-EIII}, $A_{1}(EIII_{+})$ is not a maximal antipodal set of $EIII_{+}$ since $\#A_{1}(EIII_{+}) = 3$.
On the other hand, $A_{1}(EII_{1}) = \psi(\gamma_{0}^{0,0}, \gamma_{1}^{0,0})A_{1}(EII_{+})$, and $A_{1}(EII_{1})$ is a maximal antipodal set of $EII_{1}$.
Moreover, there exist $g_{l} \in N(E_{6}^{\alpha, \beta}, E_{8}) \ (l = 2,3)$ such that $g_{l} EII_{1} g_{l}^{-1} = EII_{l}$ and $g_{l} A_{1}(EII_{1}) g_{l}^{-1} = A_{1}(EII_{l})$.
Therefore, $A_{1}(EII_{l})$ is a maximal antipodal set of $EII_{l}$.
By the similar arguments, we see that $A_{1}(EI_{j})$ is a maximal antipodal set of $EI_{j}$ for any $4 \leq j \leq 7$.
Summarizing these arguments, we obtain the following.

\begin{prop}
For any $p \in A_{1}(EVIII_{+})$, the conjugate orbit of $E_{6}$ through $p$ is one of
\[
EIII_{+}, \ EII_{i} \ (1 \leq i \leq 3), \ EI_{j}\ (4 \leq j \leq 7).
\]
Then, $EIII_{+}$ is a polar of the identity element in $E_{6}$.
Let $L$ be one of $EII_{i} \ (1 \leq i \leq 3)$ and $EI_{j} \ (4 \leq j \leq 7)$.
Then, $A_{1}(EVIII_{+}) \cap L$ is a maximal antipodal set of $L$.
If $L = EIII_{+}$, then $A_{1}(EVIII_{+}) \cap L$ is not a maximal antipodal set of $L$ and the cardinality is $3$.

\end{prop}

Next, we consider the case of $H = F_{4}$.
Recall the polar $FII_{+}$ of the identity element in $F_{4}$.
Then,
\[
FII_{+} = F_{4} \big( \psi(\gamma_{0}^{0,0}, \gamma_{0}^{1,0}) \big)
\]
and $FII_{i}$ is isomorphic to $FII$.
Set the conjugate orbits of $F_{4}$ as
\[
FI_{i} = F_{4} \big( \psi(\gamma_{1}^{0,0}, \gamma_{i}^{0,0}) \big) \quad (1 \leq i \leq 7).
\]
For any $1 \leq i \leq 7$, since we $\dim F(\theta_{\psi(\gamma_{1}^{0,0}, \gamma_{i}^{0,0})}, F_{4}) = 24$ and $F(\theta_{\psi(\gamma_{1}^{0,0}, \gamma_{i}^{0,0})}, F_{4}) \cong Sp(1) \cdot Sp(3)$, we see that $FI_{i}$ is isomorphic to $EI$.
We consider the equivalence relation $\sim_{\frak{g}_{2}}$ on $A_{1}(EVIII_{+})$.
Recall $A(FII_{+}) = A_{1}(EVIII_{+}) \cap FII_{+}$.
Then,
\[
A_{1}(FI_{+}) = \{ \psi(\gamma_{0}^{0,0}, \gamma_{0}^{a,b}) \ ;\ (a,b) = (1,0),(0,1),(1,1) \}.
\]
Set the subsets as
\[
A_{1}(FI_{i}) = \{ \psi(\gamma_{k}^{0,0}, \gamma_{i}^{a,b}) \ ;\ 1 \leq k \leq 7, a,b = 0,1 \} \quad (1 \leq i \leq 7). \\
\]
Then, the all equivalence classes are $A_{1}(FI_{i})$ and $A_{1}(FI_{i}) \ (1 \leq i \leq 7)$.
By the similar arguments to the case of $H = E_{6}$, we obtain
\[
A_{1}(EVIII_{+}) \cap FII_{+} = A_{1}(FII_{+}), \quad
A_{1}(EVIII_{+}) \cap FI_{i} = A_{1}(FI_{i}).
\]
Since the cardinalities of $A_{1}(FII_{+})$ and $A_{1}(FI_{i})$ are $3$ and $28$, we see that $A_{1}(FII_{+})$ and $A_{1}(FI_{i})$ is a maximal antipodal set of $FII_{+}$ and $FI_{i}$ for any $1 \leq i \leq 7$.
Note that 
\[
\begin{array}{lllllll}
FII_{+} = \psi(\gamma_{0}^{0,0}, \gamma_{i}^{0,0}) FII_{i} , & A_{1}(FII_{+}) = \psi(\gamma_{0}^{0,0}, \gamma_{i}^{0,0}) A_{1}(FII_{i}), \\
FI_{+} = \psi(\gamma_{0}^{0,0}, \gamma_{i}^{0,0}) FI_{i}, & A_{1}(FI_{+}) = \psi(\gamma_{0}^{0,0}, \gamma_{i}^{0,0}) A_{1}(FI_{i}),  \\
\end{array}
\quad (i = 1,\cdots, 7). \\
\]
Summarizing these arguments, we obatin the following.

\begin{prop}
For any $p \in A_{1}(EVIII_{+})$, the conjugate orbit of $F_{4}$-orbit through $p$ is one of 
\[
FII_{+}, \ 
FI_{i} \quad (1 \leq i \leq 7).
\]
Then, $FII_{+}$ is a polar of the identity element in $F_{4}$.
Let $L$ be one of these $F_{4}$-orbits.
Then, $A_{1}(EVIII_{+}) \cap L$ is a maximal antipodal set of $L$.

\end{prop}

Summarizing the arguments in Subsection \ref{A_{1}(EIX_{+})}, \ref{A_{1}(EVIII_{+})}, we obtain the following.

\begin{thm}
Let $p \in A_{1}(E_{8})$ and $p \not= e$.
If $H = E_{8}, E_{7}, E_{6}, F_{4}$, then the conjugate orbit of $H$ through $p$ is a totally geodesic submanifold of $E_{8}$.
The conjugate orbit of $E_{8}$ through $p$ is either $EVIII_{+}$ or $EIX_{+}$.
The conjugate orbit of $E_{7}$ through $p$ is one of $\{ p_{1} \}, EV_{i}, EVI_{+}, EVI'_{+}, EVII_{i} \ (1 \leq i \leq 3)$.
The conjugate orbit of $E_{6}$ through $p$ is one of $\{ p_{i} \}, EI_{j}, EII_{+}, EII_{i}, EIII_{+}, EIII_{i}, EIV_{j} \ (1 \leq i \leq 3, 4 \leq j \leq 7)$.
The conjugate orbit of $F_{4}$ through $p$ is one of $\{ p_{k} \}, FI_{+}, FI_{k}, FII_{+}, FII_{k} \ (1 \leq k \leq 7)$.
Let $L$ be one of these orbits and $\dim L \not= 0$.
If $L$ is a Hermitian symmetric space, that is, $L$ is one of $EVII_{i}, EIII_{+}, EIII_{i} \ (1 \leq i \leq 3)$, then $A_{1}(E_{8}) \cap L$ is not a maximal antipodal set of $L$.
If $L$ is one of $FI_{+}, FI_{k}, FII_{+}, FII_{k}, EIV_{j} \ (1 \leq k \leq 7, 4 \leq j \leq 7)$, then $A_{2}(E_{8}) \cap L$ is a great antipodal set of $L$.
If $L$ is one of the others, then $A_{2}(E_{8}) \cap L$ is not a maximal antipodal set but a great antipodal set of $L$.

\end{thm}

Denoting $\psi(\gamma_{0}^{0,0}, \gamma_{i}^{0,0})$ by $p_{i}$ for any $1 \leq i \leq 7$, the inclusions of above orbits contained $EIX_{+}$ are as follows:
\[
\footnotesize
\xymatrix@C=8pt@R=10pt
{
 & & FII_{1} \ar[d]^{\cap} & & FII_{2} \ar[d]^{\cap} & & \\
 & & EIII_{1} \ar[d]^{\cap} & & EIII_{2} \ar[d]^{\cap} & p_{2}, p_{3} \ar[dl]^{\supset} & \\
FI_{+} \ar[r]^{\subset} & EII_{+} \ar[r]^{\subset} & EVI_{+} \ar[dr]^{\subset} & p_{1} \ar[d]^{\cap} & EVII_{1} \ar[dl]^{\supset} & EIII_{3} \ar[l]^{\supset} & FII_{3} \ar[l]^{\supset} \\
 & & & EIX_{+} & & & \\
FII_{7} \ar[r]^{\subset} & EIV_{7} \ar[r]^{\subset} & EVII_{3} \ar[ur]^{\subset} & & EVII_{2} \ar[ul]^{\supset} & EIV_{4} \ar[l]^{\supset} & FII_{4} \ar[l]^{\supset} \\
p_{7} \ar[ur]^{\subset} & & EIV_{6} \ar[u]^{\cup} & & EIV_{5} \ar[u]^{\cup} & & p_{4} \ar[ul]^{\supset} \\
 & p_{6} \ar[ur]^{\subset} & FII_{6} \ar[u]^{\cup} & & FII_{5} \ar[u]^{\cup} & p_{5} \ar[ul]^{\supset} & \\
}
\]
The inclusions of above orbits contained $EVIII_{+}$ are as follows:
\[
\footnotesize
\xymatrix@C=8pt@R=10pt
{
 & & FI_{1} \ar[d]^{\cap} & & FI_{2} \ar[d]^{\cap} & & \\
 & & EII_{1} \ar[d]^{\cap} & & EII_{2} \ar[d]^{\cap} & & \\
FII_{+} \ar[r]^{\subset} & EIII_{+} \ar[r]^{\subset} & EVI'_{+} \ar[dr]^{\subset} & & EV_{1} \ar[dl]^{\supset} & EII_{3} \ar[l]^{\supset} & FI_{3} \ar[l]^{\supset} \\
 & & & EVIII_{+} & & & \\
FI_{7} \ar[r]^{\subset} & EI_{7} \ar[r]^{\subset} & EV_{3} \ar[ur]^{\subset} & & EV_{2} \ar[ul]^{\supset} & EI_{1} \ar[l]^{\supset} & FI_{4} \ar[l]^{\supset} \\
 & & EI_{6} \ar[u]^{\cup} & & EI_{5} \ar[u]^{\cup} & & \\
 & & FI_{6} \ar[u]^{\cup} & & FI_{5} \ar[u]^{\cup} & & \\
}
\]


\subsection{$A_{2}(EIX_{+})$ and orbits of $E_{7}, E_{6}, F_{4}$} \label{A_{2}(EIX_{+})}

In this subsection, we consider $H(p_{2})$ for any $p_{2} \in A_{2}(EIX_{+})$, where $H = E_{7}, E_{6}, F_{4}$.
Note that there exists $g \in A_{2}(E_{8})$ such that $g F_{4} g^{-1} \not\subset F_{4}$ and $F_{4}(g)$ is not a totally geodesic submanifold.

Let $H = E_{7}$.
Since $\tau_{\alpha}$ is a pole of the identity element in $E_{7}$, we obtain $E_{7}(\tau_{\alpha}) = \{ \tau_{\alpha} \}$.
Moreover, since $\tau_{\bar{\alpha}} = \psi(\gamma_{0}^{0,0}, \gamma_{1}^{1,0})$ and $\tau_{\beta} = \psi(\gamma_{0}^{0,0}, \gamma_{2}^{0,0})$,
\[
\begin{split}
& EVI_{+} = \bigcup_{g \in E_{7}^{\alpha}} g \tau_{\alpha} g^{-1}, \quad\quad
EVII_{1} = \bigcup_{g \in E_{7}^{\alpha}}g \tau_{\beta} g^{-1}.
\end{split}
\]
We consider $\Sigma_{\alpha, 0}$ and $\Sigma_{\alpha, -1}$.
Recall $A_{2}(EVI_{+}) = A_{2}(EXI_{+}) \cap EVI_{+}$ from Subsection \ref{s-EVI}.
Then,
\[
A_{2}(EVI_{+}) = \{ \tau_{\gamma} \ ;\ \gamma \in \Sigma_{\alpha, 0}^{+} \}
\]
and $A_{2}(EVI_{+})$ is a maximal antipodal set of $EVI_{+}$ whose cardinality is $63$.
Moreover, set
\[
A_{2}(EVII_{1}) = A_{2}(EIX_{+}) - ( A_{2}(EVI_{+}) \cup \{ \tau_{\alpha} \} ) = \{ \tau_{\gamma} \ ;\ \gamma \in \Sigma_{\alpha, -1} \}.
\]
The Weyl group $W(E_{7}^{\alpha})$ acts on $\Sigma_{\alpha, -1}$ and this action is transitive.
Therefore,
\[
A_{2}(EVII_{1}) = A_{2}(EIX_{+}) \cap EVII_{1}.
\]
Since $\#A_{2}(EVII_{1}) = 56$, by Theorem \ref{anti-EVII}, we see that $A_{2}(EVII_{1})$ is a maximal antipodal set of $EVII_{1}$.
Summarizing these arguments, we obtain the following.

\begin{prop}
For any $p \in A_{2}(EIX_{+})$, the conjugate orbit of $E_{7}$ through $p$ is one of
\[
\{ \tau_{\alpha} \}, \ EVI_{+}, \ EVII_{1}.
\]
If $L$ is either $EVI_{+}$ or $EVII_{1}$, then $A_{2}(EIX_{+}) \cap L$ is a great antipodal set of $L$ and $\#( A_{2}(EIX_{+}) \cap EVI_{+} ) = 63, \#(A_{2}(EIX_{+}) \cap EVII_{1}) = 56$.

\end{prop}

Next, let $H = E_{6}$.
By the definition of $E_{6}$, we obtain $E_{6}(\tau_{\gamma}) = \{ \tau_{\gamma} \}$ for any $\gamma \in \{ \alpha, \beta, \alpha + \beta \}$.
Since 
\[
\begin{array}{lll} \vspace{2mm}
\tau_{x_{3} - x_{4}} = \psi(\gamma_{1}^{0,0}, \gamma_{0}^{0,0}), &
\tau_{x_{7} + x_{8}} = \psi(\gamma_{0}^{0,0}, \gamma_{1}^{1,0}), \\
\tau_{x_{6} + x_{8}} = \psi(\gamma_{0}^{0,0}, \gamma_{2}^{1,0}), & 
\tau_{-x_{6} - x_{7}} = \psi(\gamma_{0}^{0,0}, \gamma_{3}^{1,0}), \\
\end{array}
\]
we obtain
\[
\begin{split}
EII_{+} = \bigcup_{g \in E_{6}^{\alpha, \beta}}g \tau_{x_{3} - x_{4}} g^{-1}, & \quad
EIII_{1} = \bigcup_{g \in E_{6}^{\alpha, \beta}}g \tau_{x_{7} + x_{8}} g^{-1}, \\
EIII_{2} = \bigcup_{g \in E_{6}^{\alpha, \beta}}g \tau_{x_{6} + x_{8}} g^{-1}, & \quad 
EIII_{3} = \bigcup_{g \in E_{6}^{\alpha, \beta}}g \tau_{-x_{6} - x_{7}} g^{-1}. \\
\end{split}
\]
Note that
\[
x_{3} - x_{4} \in \Sigma_{\alpha, 0} \cap \Sigma_{\beta, 0}, \quad
x_{7} + x_{8} \in \Sigma_{\alpha, 0} \cap \Sigma_{\beta, -1}, \quad
x_{6} + x_{8} \in \Sigma_{\alpha, -1} \cap \Sigma_{\beta, 1} \quad
-x_{6} - x_{7} \in \Sigma_{\alpha, 1} \cap \Sigma_{\beta, 0}.
\]
Recall a great antipodal set $A_{2}(EII_{+})$ of $EII_{+}$.
Then,
\[
A_{2}(EII_{+}) = \{ \tau_{\gamma} \ ;\ \gamma \in \Sigma_{\alpha, 0}^{+} \cap \Sigma_{\beta, 0} \}.
\]
For $1 \leq i \leq 3$, set
\[
\begin{split}
& A_{2}(EIII_{1}) = \{ \tau_{\gamma} \ ;\ \gamma \in \Sigma_{\alpha, 0} \cap \Sigma_{\beta, 1} \}, \\
& A_{2}(EIII_{2}) = \{ \tau_{\gamma} \ ;\ \gamma \in \Sigma_{\alpha, -1} \cap \Sigma_{\beta, 1} \},\\
& A_{2}(EIII_{3}) = \{ \tau_{\gamma} \ ;\ \gamma \in \Sigma_{\alpha, -1} \cap \Sigma_{\beta, 0} \}. \\
\end{split}
\]
Obviously,
\[
A_{2}(EIX_{+}) = \{ \tau_{\alpha}, \tau_{\beta}, \tau_{\alpha + \beta} \} \sqcup A_{2}(EII_{+}) \sqcup \bigcup_{i=1}^{3}A_{2}(EIII_{i}).
\]
Since the Weyl group $W(E_{6})$ acts on $\Sigma_{\alpha, 0} \cap \Sigma_{\beta,0}, \Sigma_{\alpha, 0} \cap \Sigma_{\beta, 0}, \Sigma_{\alpha, -1} \cap \Sigma_{\beta, 1}, \Sigma_{\alpha, -1} \cap \Sigma_{\beta, 0}$ respectively and these actions are transitive, for any $1 \leq i \leq 3$, we obtain
\[
A_{2}(EIII_{i}) = A_{2}(EIX_{+}) \cap EIII_{i}.
\]
Then, since the cardinality of $A_{2}(EIII_{i})$ is $27$, by Theorem \ref{anti-EIII}, $A_{2}(EIII_{i})$ is a great antipodal set.
Summarizing these argumenets, we obtain the following.

\begin{prop}
For any $p \in E_{6}^{\alpha, \beta}$, the conjugate orbit of $E_{6}$ through $p$ is one of
\[
\{ \tau_{\alpha} \}, \ \{ \tau_{\beta} \}, \ \{ \tau_{\alpha + \beta} \}, \ EII_{+}, \ EIII_{i} \ (i = 1,2,3).
\]
Let $L$ is one of $EII_{+}$ and $EIII_{i} \ (1 \leq i \leq 3)$, then $A_{2}(EIX_{+}) \cap L$ is a great antipodal set.
Moreover, $\#(A_{2}(EIX_{+}) \cap EII_{+}) = 36$ and $\#(A_{2}(EIX_{+}) \cap EIII_{i}) = 27$.

\end{prop}

For any $p \in A_{2}(EIX_{+})$, it is not necessarily true that $\theta_{p}(F_{4}) \subset F_{4}$.
However, each element $p$ of
\[
\big\{ \tau_{x_{i} \pm x_{j}} \ ;\ 1 \leq i < j \leq 4 \big\} \cup \big\{ \tau_{x_{k} \pm x_{l}} \ ;\ 5 \leq k < l \leq 8 \big\}
\]
satisfies $\theta_{p}(F_{4}) \subset F_{4}$.
Hence, the conjugate orbits of $F_{4}$ through these points are totally geodesic.
Then, these orbits are one of the following:
\[
\{ \tau_{\alpha} \} , \ \{ \tau_{\beta} \}, \ \{ \tau_{\alpha + \beta} \}, \ FI_{+}, \ FII_{i} \ (i = 1,2,3).
\]
Moreover,
\[
\begin{split}
& A_{2}(EIX_{+}) \cap FI_{+} = \{ \tau_{x_{i} \pm x_{j}} \ ;\ 1 \leq i < l \leq 4 \}, \\
& A_{2}(EIX_{+}) \cap FII_{1} = \{ \tau_{x_{5} \pm x_{6}}, \tau_{x_{7} + x_{8}} \}, \\
& A_{2}(EIX_{+}) \cap FII_{2} = \{ \tau_{x_{5} \pm x_{7}}, \tau_{x_{6} + x_{8}} \}, \\
& A_{2}(EIX_{+}) \cap FII_{3} = \{ \tau_{x_{6} \pm x_{7}}, \tau_{x_{5} + x_{8}} \}.
\end{split}
\]
Then, by Theorem \ref{FI}, $A_{2}(EIX_{+}) \cap FI_{+}$ is not a maximal antipodal set of $FI_{+}$ since $\#(A_{2}(EIX_{+}) \cap FI_{+}) = 6$.
On the other hand, $A_{2}(EIX_{+}) \cap FII_{i} = A_{1}(FII_{i})$ for each $i = 1,2,3$.
Hence, $A_{2}(EIX_{+}) \cap FII_{i}$ is a maximal antipodal set of $FII_{i}$.
Summarizing these arguments, we obtain the following.

\begin{prop}
The totally geodesic conjugate orbit of $F_{4}$ through a point of $A_{2}(EIX_{+})$ is one of
\[
FI_{+}, \ FII_{i} \ (1 \leq i \leq 3).
\]
Then, $A_{2}(EIX_{+}) \cap FI_{+}$ is not a maximal antipodal set of $FI_{+}$ and the cardinality is $3$.
Moreover, for any $1 \leq i \leq 3$, $A_{2}(EIX_{+}) \cap FII_{i}$ is a great antipodal set of $FII_{i}$.

\end{prop}


\subsection{$A_{2}(EVIII_{+})$ and orbits of $E_{7}, E_{6}$} \label{A_{2}(EVIII_{+})}

In this subsection, we consider $H(p)$ for each $p \in A_{2}(EVIII_{+})$.
As in the case of $A_{2}(EIX_{+})$, note that there exist $g_{2} \in A_{2}(EVIII_{+})$ such that $g_{2} F_{4} \subset g_{2} \subset F_{4}$.

Let $H = E_{7}$.
Set subsets of $A_{2}(EVIII_{+})$ as
\[
\begin{split}
& A_{2}(EVI'_{+}) = \{ \tau_{\gamma}\tau_{\alpha} \ ;\ \gamma \in \Sigma_{\alpha, 0}^{+} \}, \\
& A_{2}(EV_{1}) = \{ \tau_{\gamma}\tau_{\beta}, \tau_{\gamma}\tau_{\alpha + \beta} \ ;\ \gamma \in \Sigma_{\alpha, 0}^{+} \cap \Sigma_{\beta, 0} \}, \\
& A_{3}(EV_{2}) = A_{3}(EV_{x}) = x A_{2}(E_{7}) = \{ x, x\tau_{\alpha}, x\tau_{\gamma}, x\tau_{\gamma}\tau_{\alpha} \ ;\ \gamma \in \Sigma_{\alpha, 0}^{\pm} \}, \\
& A_{3}(EV_{3}) = xA(T) - A_{3}(EV_{2}) = x A_{2}(E_{7}) \tau_{\beta} = \{ x\tau_{\beta}, x\tau_{\alpha + \beta}, x\tau_{\gamma}\tau_{\beta}, x\tau_{\gamma}\tau_{\alpha + \beta} \ ;\ \gamma \in \Sigma_{\alpha,0}^{+} \cap \Sigma_{\beta, 0} \}.
\end{split}
\]
Then, these are the all equivalence classes of $\sim_{\frak{su}^{\alpha}(2)}$ on $A_{2}(EVIII_{+})$.
Since $\tau_{\alpha}$ is an element of the center $C(E_{7})$, the Weyl group $W(E_{7})$ acts on $A_{2}(EVI'_{+})$ and this action is transitive.
Moreover, for example, $x_{3} - x_{4} \in \Sigma_{\alpha, 0}$ and $\tau_{x_{3} - x_{4}}\tau_{\alpha} = \psi(\gamma_{1}^{0,0}, \gamma_{1}^{0,0}) \in EVI'_{+}$.
Therefore, for any $p \in EVI'$,
\[
E_{7}(p) = EVI'_{+}, \quad
A_{2}(EVI'_{+}) = A_{2}(EVIII_{+}) \cap EVI'_{+}.
\]
Since $A_{2}(EVI)$ is a great antipodal set of $EVI$, it is obvious that $A_{2}(EVI'_{+})$ is a great antipodal set of $EV'_{+}$.
Recall that $A_{2}(EV_{\phi(-1)})$ is an orbit of the Weyl group $W(E_{7}^{\gamma})$ (Subsection \ref{s-EV}).
Hence, we see that $W(E_{7})$ acts on $A_{2}(EV_{1})$ and this action is transitive.
Moreover, $x_{3} - x_{4} \in \Sigma_{\alpha,0} \cap \Sigma_{\beta,0}$ and $\tau_{x_{3} - x_{4}}\tau_{\beta} = \psi(\gamma_{1}^{0,0}, \gamma_{2}^{0,0}) \in EV_{1}$.
Therefore, for any $p \in A_{2}(EV_{1})$,
\[
E_{7}(p) = EV_{1}, \quad
A_{2}(EV_{1}) = A_{2}(EVIII_{+}) \cap EV_{1}.
\]
Then, $A_{2}(EV_{1})$ is a maximal antipodal set of $EV_{1}$ and the cardinality is $72$.
Since $A_{3}(EV_{2}) = x A_{2}(E_{7})$, by the definition of $x$, these exists some subgroup of $T^{\alpha} \subset E_{7}$ acting on $A_{3}(EV_{2})$ transitively.
Moreover, since $x = \psi(\gamma_{4}^{0,0}, \gamma_{4}^{0,0}) \in A_{3}(EV_{2})$, for any $p \in A_{3}(EV_{2})$,
\[
E_{7}(p) = EV_{2}, \quad
A_{3}(EV_{2}) = A_{2}(EVIII_{+}) \cap EV_{2}.
\]
By similar argumenets, for any $p \in A_{3}(EV_{3})$,
\[
E_{7}(p) = EV_{3}, \quad
A_{3}(EV_{3}) = A_{2}(EVIII_{+}) \cap EV_{3}.
\]
For any $i = 2,3$, we see that $A_{3}(EV_{i})$ is a great antipodal set of $EV_{i}$.
Summarizing these arguments, we obtain the following.

\begin{prop}
For any $p \in A_{2}(EVIII_{+})$, the conjugate orbit of $E_{7}$ through $p$ is one of
\[
EVI'_{+}, \ EV_{i} \ (i = 1,2,3).
\]
If $L$ is $EV_{1}$, then $A_{2}(EVIII_{+}) \cap L$ is not a great antipodal set but a maximal antipodal set, and the cardinality is $72$.
If $L$ is one of $EVI'_{+}$ and $EV_{i} \ (i = 2,3)$, then $A_{2}(EVIII_{+}) \cap L$ is a great antipodal set of $L$.

\end{prop}

Next, let $H = E_{6}$.
Set subsets of $A_{2}(EVIII_{+})$ as
\[
\begin{array}{lllll}
A_{2}(EIII_{+}) = \{ \tau_{\gamma} \tau_{\alpha} \ ;\ \gamma \in \Sigma_{\alpha, 0} \cap \Sigma_{\beta, 1} \}, & 
A_{2}(EII_{1}) = \{ \tau_{\gamma} \tau_{\alpha} \ ;\ \gamma \in \Sigma_{\alpha,0}^{+} \cap \Sigma_{\beta, 0} \}, \\
A_{2}(EII_{2}) = \{ \tau_{\gamma} \tau_{\alpha + \beta} \ ;\ \gamma \in \Sigma_{\alpha,0}^{+} \cap \Sigma_{\beta, 0} \}, &
A_{2}(EII_{3}) = \{ \tau_{\gamma} \tau_{\beta} \ ;\ \gamma \in \Sigma_{\alpha,0}^{+} \cap \Sigma_{\beta, 0} \}, \\
A_{2}(EI_{4}) = xA(T^{\alpha, \beta}), &
A_{2}(EI_{5}) = xA(T^{\alpha, \beta})\tau_{\alpha}, \\
A_{2}(EI_{6}) = xA(T^{\alpha, \beta})\tau_{\alpha + \beta}, & 
A_{2}(EI_{7}) = xA(T^{\alpha, \beta})\tau_{\beta}.
\end{array}
\]
Since the Weyl group $W(E_{6}^{\alpha, \beta})$ acts on $\Sigma_{\alpha, 0} \cap \Sigma_{\beta, 0}$ and $\Sigma_{\alpha, 0} \cap \Sigma_{\beta, \pm1}$ respectively and these actions are transitive, $W(E_{6})$ acts on $A_{2}(EIII_{+})$ and $A_{2}(EII_{i}) \ (i = 1,2,3)$ respectively and these actions are transitive.
Morover, for each $4 \leq j \leq 7$, there exist a subgroup of $T^{\alpha, \beta} \subset E_{6}^{\alpha, \beta}$ acting on $A_{2}(EI_{j})$ transitively.
Since
\[
\begin{array}{lll}
A_{2}(EIII_{+}) \ni \tau_{x_{7} + x_{8}}\tau_{\alpha} = \psi(\gamma_{0}^{0,0}, \gamma_{0}^{1,0}) \in \tau_{x_{7} + x_{8}}\tau_{\alpha} \in EIII_{+}, \\
A_{2}(EII_{1}) \ni \tau_{x_{3} - x_{4}}\tau_{\alpha} = \psi(\gamma_{1}^{0,0}, \gamma_{1}^{0,0}) \in A_{2}(EII_{1}), \\
A_{2}(EII_{2}) \ni \tau_{x_{3} - x_{4}}\tau_{\alpha + \beta} = \psi(\gamma_{1}^{0,0}, \gamma_{2}^{0,0}) \in A_{2}(EII_{2}), \\
A_{2}(EII_{3}) \ni \tau_{x_{3} - x_{4}}\tau_{\beta} = \psi(\gamma_{1}^{0,0}, \gamma_{3}^{0,0}) \in A_{2}(EII_{3}), \\
A_{2}(EI_{4}) \ni x = \psi(\gamma_{4}^{0,0}, \gamma_{4}^{0,0}) \in A_{2}(EI_{4}), \\
A_{2}(EI_{5}) \ni x \tau_{\alpha} = \psi(\gamma_{4}^{0,0}, \gamma_{5}^{0,0}) \in A_{2}(EI_{5}), \\
A_{2}(EI_{6}) \ni x \tau_{\alpha + \beta} = \psi(\gamma_{4}^{0,0}, \gamma_{6}^{0,0}) \in A_{2}(EI_{6}), \\
A_{2}(EI_{7}) \ni x \tau_{\beta} = \psi(\gamma_{4}^{0,0}, \gamma_{7}^{0,0}) \in A_{2}(EI_{7}), \\
\end{array}
\]
we obtain 
\[
E_{6}(p) =
\begin{cases}
EIII_{+} & (p \in A_{2}(EIII_{+}), \\
EII_{i} & (p \in EII_{i}, i = 1,2,3), \\
EI_{j} & (q \in EI_{j}, j = 4,5,6,7). \\
\end{cases}
\]
Moreover, for any $1 \leq i \leq 3$ and $4 \leq j \leq 7$,
\[
\begin{split}
& A_{2}(EIII_{+}) = A_{2}(EVIII_{+}) \cap EIII_{+}, \\
& A_{2}(EII_{i}) = A_{2}(EVIII_{+}) \cap EII_{i}, \\
& A_{2}(EI_{j}) = A_{2}(EVIII_{+}) \cap EI_{j}. \\
\end{split}
\]
Since the cardinalities of $A_{2}(EIII_{+}), A_{2}(EII_{i}) \ (i = 1,2,3),$ and $A_{2}(EI_{j}) \ (j = 4,5,6,7)$ are $27, 36,$ and $64$, these are great antipodal sets.
Summaizing these arguments, we obtain the following.

\begin{prop}
For any $p \in A_{2}(EVIII_{+})$, the conjugate orbit of $E_{6}$ through $p$ is one of
\[
EIII_{+}, \ EII_{i} \ (i = 1,2,3), \ EI_{j} \ (j = 4,5,6,7).
\]
If $L$ is one of them, then $A_{2}(EVIII_{+}) \cap L$ is a great antipodal set of $L$.

\end{prop}

For any $p \in A_{2}(EVIII_{+})$, it is not necessarily true that $\theta_{p}(F_{4}) \subset F_{4}$.
However, each element $p$ of
\[
\begin{split}
& \big\{ \tau_{x_{5} \pm x_{6}} \tau_{\alpha}, \ \tau_{x_{7} + x_{8}}\tau_{\alpha}, \ \tau_{x_{i} \pm x_{j}}\tau_{\alpha}, \tau_{x_{i} \pm x_{j}}\tau_{\beta}, \tau_{x_{i} \pm x_{j}}\tau_{\alpha + \beta} \ ;\ 1 \leq i < j \leq 4 \big\} \\
& \quad \cup \big\{ x, x\tau_{x_{i} \pm x_{j}}, x\tau_{x_{7} + x_{8}}\tau_{\alpha}, x\tau_{x_{5} \pm x_{6}}\tau_{\alpha} \ ;\ 1 \leq i < j \leq 4 \big\} \{ e, \tau_{\alpha}, \tau_{\beta}, \tau_{\alpha + \beta} \} \\
& = \{ \psi(\gamma_{0}^{0,0}, \gamma_{0}^{a,b}) \ ;\ (a,b) = (1,0),(0,1),(0,0) \} \cup \{ \psi(\gamma_{k}^{0,0}, \gamma_{i}^{a,b}) \ ;\ 1 \leq k, i \leq 3, a,b = 0,1 \} \\
& \quad \cup \{ \psi( \gamma_{l}^{0,0}, \gamma_{j}^{a,b}) \ ;\ 4 \leq l, j \leq 7, a, b = 0,1 \}.
\end{split}
\]
satisfies $\theta_{p}(F_{4}) \subset F_{4}$.

For any $p \in A_{2}(EVIII_{+})$, it is true that $\theta_{p}$ leaves $F_{4}$ invariant if and only if $p$ is contained in
\[
\begin{split}
& \big\{ \tau_{x_{5} \pm x_{6}} \tau_{\alpha}, \ \tau_{x_{7} + x_{8}}\tau_{\alpha}, \ \tau_{x_{i} \pm x_{j}}\tau_{\alpha}, \tau_{x_{i} \pm x_{j}}\tau_{\beta}, \tau_{x_{i} \pm x_{j}}\tau_{\alpha + \beta} \ ;\ 1 \leq i < j \leq 4 \big\} \\
& \quad \cup \big\{ x, x\tau_{x_{i} \pm x_{j}}, x\tau_{x_{7} + x_{8}}\tau_{\alpha}, x\tau_{x_{5} \pm x_{6}}\tau_{\alpha} \ ;\ 1 \leq i < j \leq 4 \big\} \{ e, \tau_{\alpha}, \tau_{\beta}, \tau_{\alpha + \beta} \}. \\
\end{split}
\]
Then, the conjugate orbit of $F_{4}$ through these points are
\[
FI_{i} \ (i = 1, \cdots, 7), \ FII_{+}.
\]
Moreover,
\[
\begin{split}
& A_{2}(EVIII_{+}) \cap FII_{+} = \{ \tau_{x_{5} \pm x_{6}} \tau_{\alpha}, \tau_{x_{7} + x_{8}}\tau_{\alpha} \}, \\
& A_{2}(EVIII_{+}) \cap FI_{1} = \{ \tau_{x_{i} \pm x_{j}} \tau_{\alpha} \ ;\ 1 \leq i < j \leq 4 \}, \\
& A_{2}(EVIII_{+}) \cap FI_{2} = \{ \tau_{x_{i} \pm x_{j}} \tau_{\alpha + \beta} \ ;\ 1 \leq i < j \leq 4 \}, \\
& A_{2}(EVIII_{+}) \cap FI_{3} = \{ \tau_{x_{i} \pm x_{j}} \tau_{\beta} \ ;\ 1 \leq i < j \leq 4 \}, \\
& A_{2}(EVIII_{+}) \cap FI_{4} = \{ x, x\tau_{x_{i} \pm x_{j}}, x\tau_{x_{5} \pm x_{6}}\tau_{\alpha}, x\tau_{x_{7} + x_{8}}\tau_{\alpha} \ ;\ 1 \leq i < j \leq 4 \}, \\
& A_{2}(EVIII_{+}) \cap FI_{5} = \big( A_{2}(EVIII_{+}) \cap FI_{4} \big) \tau_{\alpha}, \\
& A_{2}(EVIII_{+}) \cap FI_{6} = \big( A_{2}(EVIII_{+}) \cap FI_{4} \big) \tau_{\alpha + \beta}, \\
& A_{2}(EVIII_{+}) \cap FI_{7} = \big( A_{2}(EVIII_{+}) \cap FI_{4} \big) \tau_{\beta}. \\
\end{split}
\]
We see that $A_{2}(EVIII_{+}) \cap FII_{+}$ is a maximal antipodal set of $FII_{+}$.
On the other hand, for any $1 \leq i \leq 7$, we see that $A_{2}(EVIII_{+}) \cap FII_{i}$ is not a maximal antipodal set of $FI_{i}$ since the cardinality is $3$.
Summarizing these arguments, we obtain the following.

\begin{prop}
The totally geodesic conjugate orbit of $F_{4}$ through a point of $A_{2}(EVIII_{+})$ is one of
\[
FI_{k} \ ( 1 \leq k \leq 7), \ FII_{+}.
\]
Then, for any $1 \leq k\leq 7$, $A_{2}(EVIII_{+}) \cap FI_{k}$ is not a maximal antipodal set of $FI_{k}$ and the cardinality is $3$.
On the other hand, $A_{2}(EVIII_{+}) \cap FII_{+}$ is a great antipodal set of $FII_{+}$.

\end{prop}

Summarizing the arguments in Subsection \ref{A_{2}(EIX_{+})}, \ref{A_{1}(EVIII_{+})}, we obtain the following.

\begin{thm}
Let $p \in A_{2}(E_{8})$ be $p \not= e$.
If $H = E_{8}, E_{7}, E_{6}$, then the conjugate orbit through $p$ is a totally geodesic submanifold of $E_{8}$.
The conjugate orbit of $E_{8}$ through $p$ is either $EVIII_{+}$ or $EIX_{+}$.
The conjugate orbit of $E_{7}$ through $p$ is one of $\{\tau_{\alpha}\} \ (\gamma \in \{ \alpha, \beta, \alpha + \beta \}), EV_{i} \ (1 \leq i \leq 3), EVI_{+}, EVI'_{+}$.
The conjugate orbit of $E_{6}$ through $p$ is one of $\{\tau_{\alpha}\} \ (\gamma \in \{ \alpha, \beta, \alpha + \beta \}), EI_{j}, EII_{+}, EII_{i}, EIII_{+}, EIII_{i} \ (1 \leq i \leq 3, 4 \leq j \leq 7)$.
The conjugate orbit of $F_{4}$ through $p$ that is a totally geodesic submanifold is one of $\{\tau_{\alpha}\} \ (\gamma \in \{ \alpha, \beta, \alpha + \beta \}), FI_{+}, FI_{k}, FII_{+}, FII_{i} \ (1 \leq k \leq 7, 1 \leq i \leq 3)$.
Let $L$ be one of these orbits and $\dim L \not= 0$.
If $L$ is $EV_{1}$, then $A_{2}(E_{8}) \cap L$ is not a great antipodal set but a maximal antipodal set of $L$ and the cardinality is $72$. 
If $L$ is either $FI_{+}$ or $FI_{k} \ (1 \leq k \leq 7)$, the $A_{2}(E_{8}) \cap L$ is not a maximal antipodal set of $L$ and the cardinality is $3$.
If $L$ is one of the others, then $A_{2}(E_{8}) \cap L$ is a great antipodal set of $L$.

\end{thm}

The inclusions among the above orbits contained in $EIX_{+}$ are as follows:
\[
\footnotesize
\xymatrix@C=8pt@R=10pt
{
 & & FII_{1} \ar[d]^{\cap} & & FII_{2} \ar[d]^{\cap} & & \\
 & & EIII_{1} \ar[d]^{\cap} & & EIII_{2} \ar[d]^{\cap} & \tau_{\beta}, \tau_{\alpha + \beta} \ar[dl]^{\supset} & \\
FI_{+} \ar[r]^{\subset} & EII_{+} \ar[r]^{\subset} & EVI_{+} \ar[dr]^{\subset} & \tau_{\alpha} \ar[d]^{\cap} & EVII_{1} \ar[dl]^{\supset} & EIII_{3} \ar[l]^{\supset} & FII_{3} \ar[l]^{\supset} \\
 & & & EIX_{+} & & & \\
}
\]
This daigram is the upper half part of the diagram of Subsection \ref{A_{1}(EIX_{+})}.
The inclusions among the above orbits contained in $EVIII_{+}$ are as follows:
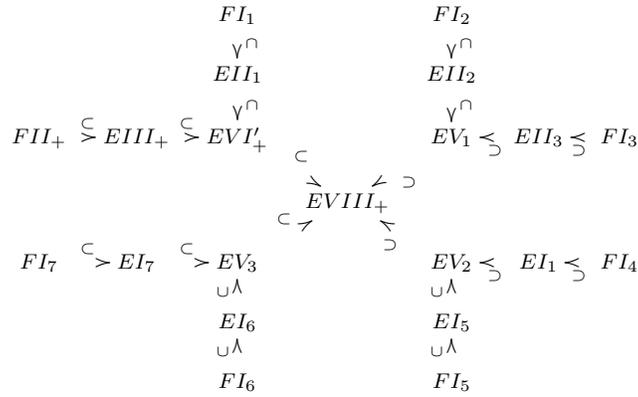
\[
\footnotesize
\xymatrix@C=8pt@R=10pt
{
 & & FI_{1} \ar[d]^{\cap} & & FI_{2} \ar[d]^{\cap} & & \\
 & & EII_{1} \ar[d]^{\cap} & & EII_{2} \ar[d]^{\cap} & & \\
FII_{+} \ar[r]^{\subset} & EIII_{+} \ar[r]^{\subset} & EVI'_{+} \ar[dr]^{\subset} & & EV_{1} \ar[dl]^{\supset} & EII_{3} \ar[l]^{\supset} & FI_{3} \ar[l]^{\supset} \\
 & & & EVIII_{+} & & & \\
FI_{7} \ar[r]^{\subset} & EI_{7} \ar[r]^{\subset} & EV_{3} \ar[ur]^{\subset} & & EV_{2} \ar[ul]^{\supset} & EI_{1} \ar[l]^{\supset} & FI_{4} \ar[l]^{\supset} \\
 & & EI_{6} \ar[u]^{\cup} & & EI_{5} \ar[u]^{\cup} & & \\
 & & FI_{6} \ar[u]^{\cup} & & FI_{5} \ar[u]^{\cup} & & \\
}
\]
This diagram is same to the diagram of Subsetion \ref{A_{1}(EVIII_{+})}


\subsection{Symmetry}

In this section we highlight a certain symmetry among the exceptional Riemannian symmetric spaces, as manifested through the totally geodesic inclusion relations obtained in the previous subsections.
From those inclusions, we obtain the following diagram describing the totally geodesic inclusion relations among the exceptional symmetric spaces:
\begin{figure}[htpb]
\[
\xymatrix@C=22pt@R=20pt
{
 & & & E_{8} & & & \\
 & & EVIII \ar[ur]_{\subset} & & EIX \ar[ul]^{\supset} & &  \\
 & \ EV \ar[ur]_{\subset} & & EVI \ar[ul]^{\supset} \ar[ur]_{\subset} & & EVII \ar[ul]^{\supset} & \\
 \ EI \ \ar[ur]_{\subset} & & EII \ar[ul]^{\supset} \ar[ur]_{\subset} & & EIII \ar[ul]^{\supset} \ar[ur]_{\subset} & & EIV \ar[ul]^{\supset} \\
 & FI \ar[ul]^{\supset} \ar[ur]_{\subset} & & & & FII \ar[ul]^{\supset} \ar[ur]_{\subset} & \\
}
\]
\caption{Inclusion relations among exceptional symmetric spaces}
\end{figure}
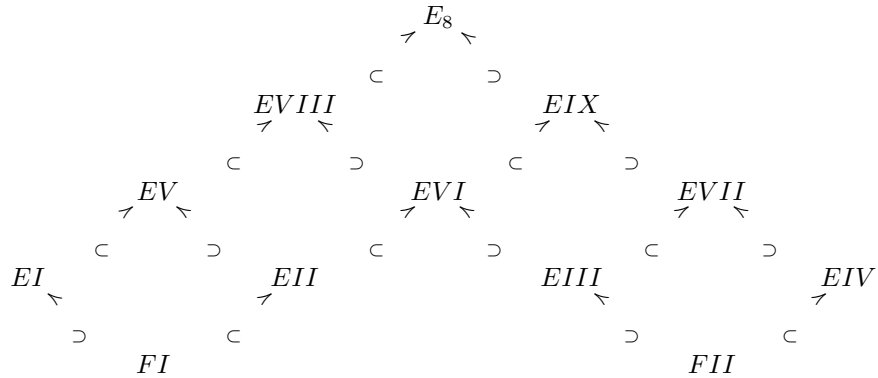
This daigram was essentialy discovered by Chen-Nagano in \cite{Chen-Nagano1} using the $(M^{+}, M^{-})$-method.
Moreover, Nagano obtained this diagram by the different method from the present paper in \cite{Nagano}.
However, they barely discussed this diagram.
We consider some properties of this diagram in the following.
First, we show that no other totally geodesic embedding between the other than pair of the above exists.

\begin{lemm}
No totally geodesic embedding exists from $FI$ into $EIII, EIV$, or $EVII$; from $EI$ into $EVI, EVII$, or $EIX$; from $EII$ into $EVII$; from $EV$ into $EIX$.
\end{lemm}

\begin{proof}
Note that
\[
\begin{split}
& \mathrm{rank}EIII = \mathrm{rank}EIV = 2, \quad
\mathrm{rank}EVII = 3, \\
& \mathrm{rank}FI = \mathrm{rank}EII = \mathrm{rank}EVI = \mathrm{rank}EIX = 4, \\
& \mathrm{rank}EI = 6, \quad \mathrm{rank}EV = 7.
\end{split}
\]
By comparing the rank, we complete the proof.
\end{proof}

By the similar method to \cite{Chen-Nagano1}, we can prove the following.

\begin{lemm}
No totally geodesic embedding exists from $FII$ into $EI, EII$, or $EV$; from $EIII$ into $EV$; from $EIV$ into $EV, EVI$, or $EVIII$; from $EVII$ into $EVIII$.
\end{lemm}

\begin{proof}
If there exists a totally geodesic embedding from $EIV$ into $EVIII$, then $F_{4}$ is a subgroup of $\phi(Spin(16)) \cong Ss(16)$.
Therefore, $F_{4}$ is a subgroup of $Spin(16)$.
However, it is well known that the dimension of the smallest real representation of $F_{4}$ is $26$ and of $Spin(16)$ is $16$.
Hence, there does not exist such embedding.
We already constructed a totally geodesic embedding $EV \subset EVIII, EVI \subset EVIII$, and $EIV \subset EVII$.
Thus, there does not exist a totally geodesic embedding from $EIV$ into $EV, EVI$ and from $EVII$ into $EVIII$.

Moreover, if there exists a totally geodesic embedding from $FII$ into $EV$, then $Spin(9)$ is a subrgoup of $\phi(SU''(8)) = SU(8)/\mathbb{Z}_{2}$.
Therefore, $Spin(9)$ is a subgroup of $SU(8)$.
It is well known the complex dimension of the smallest complex representation of $Spin(9)$ is $9$ and of $SU(8)$ is $8$.
Hence, there does not exist a totally geodesic emnbedding from $FII$ into $EV$.
We already constructed a totally geodesic embedding $EI \subset EV, EII \subset EV$, and $FII \subset EIII$.
Thus, there does not exist a totally geodesic embedding form $FII$ into $EI, EII$ and from $EIII$ into $EV$. 
\end{proof}

This diagram is not merely a convenient arrangement of the exceptional symmetric spaces.
Rather, it reflects several geometric features of these spaces.
Moreover, the left-right symmetry of the diagram has the geometric meaning with respect to the exceptioanl symmetric spaces.
First, we consider the line
\[
FI \subset EII \subset EVI \subset EIX.
\]
In $F_{4}, E_{6}, E_{7}, E_{8}$, these correspond precisely to the polar of the identity element closest to the identity element other than the trivial pole.
From the viewpoint of geometric structures, each of these is a quaternion K\"{a}hler symmetric space, and every inclusion in this chain is a reflective quaternionic embedding.
Here, an inclusion $M \subset N$ is reflective if $M$ is a reflective submanifold of $N$.
Next, consider the chain
\[
FII \subset EIII \subset EVI \subset EVIII.
\]
In $F_{4}, E_{6}, E_{7}, E_{8}$, these correspond to the polar of the unit element farthest from the identity element other than poles.
Moreover, these symmetric spaces are the well-known Rosenfeld projective planes:
\[
FII = \mathbb{O}P^{2}, \ 
EIII = (\mathbb{O} \otimes \mathbb{C})P^{2}, \ 
EVI = (\mathbb{O} \otimes \mathbb{H})P^{2}, \ 
EVIII = (\mathbb{O} \otimes \mathbb{O})P^{2}.
\]
Their dimensions satisfy
\[
\dim EVIII = 2 \dim EVI = 4 \dim EIII = 8 \dim FII.
\]
From the viewpoint of geometric structures, $EIII$ is a Hermitian symmetric space.
The inclusion $FII \subset EIII$ is a reflective Lagrangian embedding (a real form), and $EIII \subset EVI$ is a reflective totally complex embedding in the sense of Funabashi (\cite{Funabashi}), analogously to totally real embeddings into K\"{a}hler manifolds.
It is well known that half-dimensional totally complex submanifolds of compact quaternion-K"{a}hler symmetric spaces are reflective (\cite{Takeuchi}).
Thus, the two chains
\[
FI \subset EII \subset EVI \subset EIX, \quad
FII \subset EIII \subset EVI \subset EVIII,
\]
exhibit a parallel not only in their placement in the diagram but also in their intrinsic geometric structures.
Moroianu and Semmelmann introduced the even-Clifford structure in \cite{Moroianu-Semmelmann}, which is a generalization of a K\"{a}hler structure and a quaternionic K\"{a}hler structure, and many mathematician pay attention to this structure.
By their result, in simply connected exceptional compact symmetric spaces, only
\[
FI, FII, EII, EIII, EVI, EVIII, EIX
\]
have the even-Clifford structure.
These facts constitute the indication of an underlying symmetry between the left and right sides of the diagram.

The chain
\[
FI \subset EI \subset EV \subset EVIII
\]
forms the normal forms, namely the class of symmetric spaces whose rank coincides with the rank of their full isometry group.
Only $FI \subset EI$ is reflective; the inclusions $EI \subset EV$ and $EV \subset EVIII$ are not reflective.
For the chain
\[
EIV \subset EVII \subset EIX,
\]
the intermediate space $EVII$ is Hermitian.
The inclusion $EIV \subset EVII$ is totally geodesic and totally real, but not a real form, since $\dim EVII = 2(\dim EIV +1)$.
The inclusion $EVII \subset EIX$ is totally complex but not reflectively so, since $\dim EIX = 2(\dim EVII + 2)$.
The inclusions
\[
FII \subset EIV, \quad 
EIII \subset EVII, \quad
EVI \subset EIX
\]
are reflective and, moreover, polars.
The inclusion $EIII \subset EVII$ is also a complex embedding.
The inclusion $EII \subset EV$ is reflective.

We now examine the symmetry exhibited by the reflective inclusions in the diagram.
Denoting each reflective inclusion by $\subset_{r}$ or $\supset_{r}$, we obtain the following diagram:
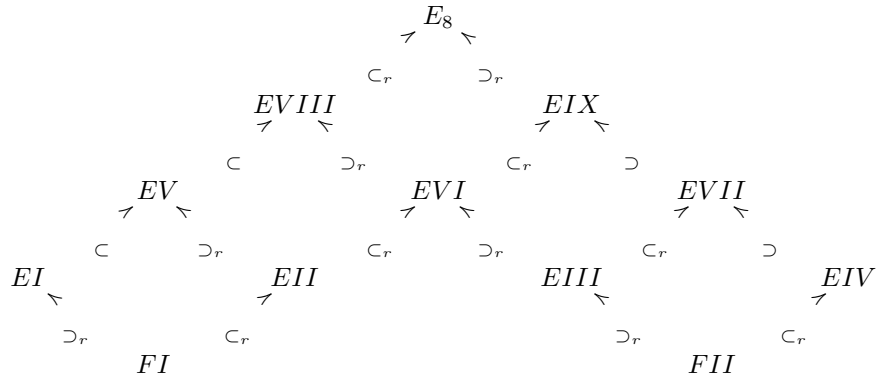
\[
\xymatrix@C=22pt@R=20pt
{
 & & & E_{8} \ar@{.}[dd] & & & \\
 & & EVIII \ar[ur]_{\subset_{r}} & & EIX \ar[ul]^{\supset_{r}} & &  \\
 & \ EV \ar[ur]_{\subset} & & EVI \ar[ul]^{\supset_{r}} \ar[ur]_{\subset_{r}} \ar@{.}[dd] & & EVII \ar[ul]^{\supset} & \\
 \ EI \ \ar[ur]_{\subset} & & EII \ar[ul]^{\supset_{r}} \ar[ur]_{\subset_{r}} & & EIII \ar[ul]^{\supset_{r}} \ar[ur]_{\subset_{r}} & & EIV \ar[ul]^{\supset} \\
 & FI \ar[ul]^{\supset_{r}} \ar[ur]_{\subset_{r}} & & & & FII \ar[ul]^{\supset_{r}} \ar[ur]_{\subset_{r}} & \\
}
\]
This diagram makes the left-right symmetry of the reflective inclusions immediately apparent.
Although the non-reflective inclusions also occur in symmetric positions, Leung's results (\cite{Leung}) provide an explanation:
\[
\begin{array}{lll}
(S^{1} \times EI)/\mathbb{Z}_{3} \subset EV, & (S^{2} \times EV)/\mathbb{Z}_{2} \subset EVIII, \\
(S^{1} \times EIV)/\mathbb{Z}_{3} \subset EVII, & (S^{2} \times EVII)/\mathbb{Z}_{2} \subset EIX
\end{array}
\]
are reflective inclusions.
In every case, there exists an additional irreducible factor that turns a non-reflective inclusion into a reflective one, and this irreducible factor coincides for each corresponding pair.
This demonstrates that the symmetry is not merely combinatorial: it reflects a geometry of the reflective submanifolds.

The diagram also induces a correspondence between the symmetric spaces themselves; for example, 
$EVIII$ corresponds to $EIX$, and $EVI$ corresponds to itself.
This correspondence has a meaning beyond the layout of the diagram.
Indeed, recent work of Kollross and Rodr\'{i}guez-V\'{a}zquez (\cite{Kollross}) on the classification of maximal totally geodesic submanifolds in exceptional symmetric spaces shows the following phenomenon.
From their classification, the maximal totally geodesic submanifolds of each above exceptional symmetric spaces whose irreducible components include above exceptional symmetric spaces  are in Table \ref{maximal-totally-geodesic}.
Then, in the corresponding pair under the left-right symmetry, each maximal totally geodesic submanifold including exceptional irreducible components corresponds to each other under the left-right symmetry.
Note that considering the irreducible component as $G_{2}/SO(4)$ the left-right symmetry breaks.

\begin{table}[htbp]
\begin{center}
\begin{tabular}{|c|c|ccc} \hline
$EVIII$ & $EIX$ \\ \hline
$EV$ & $EVII$ \\
$EVI$ & $EVI$ \\
$\big( SU(3)/SO(3)) \times EI \big)/\mathbb{Z}_{3}$ & $\big( SU(3)/SO(3) \times EIV \big)/\mathbb{Z}_{3}$ \\
$\mathbb{C}P^{2} \times EII$ & $\mathbb{C}P^{2} \times EIII$ \\
$G_{2}/SO(4) \times FI$ & $G_{2}/SO(4) \times FII$ \\ \hline
\end{tabular}

\vspace{4mm}

\begin{tabular}{|c|c|c|cc} \hline
$EV$ & $EVI$ & $EVII$ \\ \hline
$(S^{1} \times EI)/\mathbb{Z}_{3}$ & $EII$ & $(S^{1} \times EIV)/\mathbb{Z}_{3}$ \\
$EII$ & $EIII$ & $EIII$ \\
$\mathbb{R}P^{2} \times FI$ & & $\mathbb{R}P^{2} \times FII$ \\ \hline
\end{tabular}

\vspace{4mm}

\begin{tabular}{|c|c|c|c|c} \hline
$EI$ & $EII$ & $EIII$ & $EIV$ \\ \hline
$FI$ & $FI$ & $FII$ & $FII$ \\ \hline
\end{tabular}

\caption{Maximal totally geodesic submanifolds including exceptional irreducible component}
\label{maximal-totally-geodesic}
\end{center}
\end{table}

Finally, recall the inclusion diagram of orbits in Subsection \ref{A_{1}(EVIII_{+})}.
Removing the fixed-point orbits, we obtain
\[
\xymatrix@C=6pt@R=13pt
{
 & & FII \ar[d] & & FII \ar[d] & & \\
 & & EIII \ar[d] & & EIII \ar[d] & & \\
FI \ar[r] & EII \ar[r] & EVI \ar[dr] & & EVII \ar[dl] & EIII \ar[l] & FII \ar[l] \\
 & & & EIX & & & \\
FII \ar[r] & EIV \ar[r] & EVII \ar[ur] & & EVII \ar[ul] & EIV \ar[l] & FII \ar[l] \\
 & & EIV \ar[u] & & EIV \ar[u] & & \\
 & & FII \ar[u] & & FII \ar[u] & & \\
}
\quad
\xymatrix@C=6pt@R=13pt
{
 & & FI \ar[d] & & FI \ar[d] & & \\
 & & EII \ar[d] & & EII \ar[d] & & \\
FII \ar[r] & EIII \ar[r] & EVI \ar[dr] & & EV \ar[dl] & EII \ar[l] & FI \ar[l] \\
 & & & EVIII & & & \\
FI \ar[r] & EI \ar[r] & EV \ar[ur] & & EV \ar[ul] & EI \ar[l] & FI \ar[l] \\
 & & EI \ar[u] & & EI \ar[u] & & \\
 & & FI \ar[u] & & FI \ar[u] & & \\
}
\]
Then, every inclusion and the number of orbits match under the left–right symmetry.
These observations support the existence of a certain symmetry amog exceptional symmetric spaces.
Although the precise geometric reason for this left-right symmetry remains unknown, uncovering its origin is an important problem for future research.

\end{document}